\documentclass[11pt,twoside]{article}

\usepackage{geometry}

\input{commands}

\begin{document}

\begin{center}

  {\bf{\LARGE{ \mbox{Alternating minimization for generalized rank one matrix} sensing: Sharp predictions from a random initialization}}}

\vspace*{.2in}

{\large{
\begin{tabular}{ccc}
Kabir Aladin Verchand$^{*,\circ,\dagger}$, Mengqi Lou$^{*,\dagger}$, Ashwin Pananjady$^{\dagger, \ddagger}$
\end{tabular}
}}
\vspace*{.2in}

\begin{tabular}{c}
Department of Pure Mathematics and Mathematical Statistics$^\circ$, University of Cambridge\\
Schools of Industrial and Systems Engineering$^\dagger$ and Electrical and Computer Engineering$^\ddagger$ \\
Georgia Institute of Technology
\end{tabular}
\def\thefootnote{*}\footnotetext{Kabir Aladin Chandrasekher and Mengqi Lou contributed equally to this work.}
\def\thefootnote{\arabic{footnote}}

\vspace*{.2in}

\today

\vspace*{.2in}

\begin{abstract}
  We consider the problem of estimating the factors of a rank-$1$ matrix with i.i.d. Gaussian, rank-$1$ measurements that are nonlinearly transformed and corrupted by noise. Considering two prototypical choices for the nonlinearity, we study the convergence properties of a natural alternating update rule for this nonconvex optimization problem starting from a random initialization. We show sharp convergence guarantees for a sample-split version of the algorithm by deriving a deterministic one-step recursion that is accurate even in high-dimensional problems. Notably, while the infinite-sample population update is uninformative and suggests exact recovery in a single step, the algorithm---and our deterministic one-step prediction---converges geometrically fast from a random initialization. Our sharp, non-asymptotic analysis also exposes several other fine-grained properties of this problem, including how the nonlinearity and noise level affect convergence behavior.	

  On a technical level, our results are enabled by showing that the empirical error recursion can be predicted by our deterministic one-step updates within fluctuations of the order $n^{-1/2}$ when each iteration is run with $n$ observations. Our technique leverages leave-one-out tools originating in the literature on high-dimensional $M$-estimation and provides an avenue for sharply analyzing complex iterative algorithms from a random initialization in other high-dimensional optimization problems with random data.
\end{abstract}
\end{center}

\small
\addtocontents{toc}{\protect\setcounter{tocdepth}{2}}
\tableofcontents

\normalsize

\section{Introduction} \label{sec:intro}

Worst-case efficiency theory for nonconvex optimization suggests that convergence to global near-optimality is prohibitive~\citep{nemirovski1983problem} and that iterative algorithms converge sub-linearly fast to stationary points~\citep[e.g.,][]{vavasis1993black,nesterov2012make,carmon2020lower}. Despite these pessimistic results, primitive and natural iterative algorithms are routinely applied to nonconvex problems in many modern data science applications, and several of them are known to converge quite quickly to accurate solutions. For instance, iterative algorithms often converge to statistically useful solutions when the data in the problem is suitably random (see~\cite{jain2013low,loh2012high} for some early examples), and the resulting random ensemble of optimization problems only exhibits benign forms of nonconvexity. At the same time, the observed convergence behavior for these algorithms is often delicate, and the presence or absence of convergence---as well as the rate of convergence---depends critically on how well the algorithm is initialized (see~\cite{jain2017non,chi2019nonconvex,chen2018harnessing,zhang2020symmetry,sun2022} for a slew of such examples). This wide range of possible behavior motivates the need for a \emph{sharp, average-case} theory of efficiency for iterative nonconvex optimization from a random initialization in settings with random data.

With this broad goal in mind, we consider the concrete problem of sensing a rank-one matrix under a generalized bilinear model. This is defined by two unknown coefficient vectors $\bcoefX_{\star}, \bcoefZ_{\star} \in \mathbb{R}^d$ and i.i.d. observations $(y_i, \bx_i, \bz_i)$ drawn according to
\begin{align} \label{eq:model}
	y_i = \psi \left( \langle \bx_i, \bcoefX_{\star} \rangle \cdot \langle \bz_i, \bcoefZ_{\star} \rangle \right) + \epsilon_i, \quad i = 1,\ldots, N.
\end{align}
Here $\bx_i \in \mathbb{R}^d$ and $\bz_i \in \mathbb{R}^d$ are \emph{sensing vectors}, typically drawn i.i.d. from some distribution, $\epsilon_i$ denotes zero-mean noise in the measurements, and $\psi: \mathbb{R} \rightarrow \mathbb{R}$ is some potentially unknown nonlinearity. Throughout, we make the assumption that $\| \bcoefX_{\star} \|_2 = \| \bcoefZ_{\star} \|_2 = 1$ for convenience, noting that this assumption can be straightforwardly relaxed.

The problem of solving systems of bilinear equations finds applications in diverse areas of science and engineering, including astronomy, medical imaging and communications~\citep{jefferies1993restoration,wang1998blind,campisi2017blind}.

For example, the model~\eqref{eq:model} with $\psi$ taken to be the identity map is an example of the \emph{blind deconvolution} problem in statistical signal processing (see, e.g.,~\cite{recht2010guaranteed,ahmed2013blind} and the references therein for several applications of this problem). In addition, the nonlinearity $\psi$ models cases in which we have some additional misspecification or quantization, which is common in many applications involving matrix estimation problems and their relatives~\cite[see, e.g.,][]{davenport20141,plan2016generalized,ganti2015matrix,yang2019misspecified,thrampoulidis2019lifting,ongie2021tensor}.

We are interested in the model-fitting problem, and the natural least squares objective \sloppy\mbox{$F_N: \mathbb{R}^d \times \mathbb{R}^d \rightarrow \mathbb{R}$} corresponding to the scaled negative log-likelihood of our observations under Gaussian noise\footnote{Note that $F_N$ is also a reasonable loss function to minimize when the noise is sub-Gaussian.} can be written as
\begin{align} \label{eq:global-loss}
	F_N(\bcoefX, \bcoefZ) = \frac{1}{N} \sum_{i = 1}^N \bigr(y_i - \psi(\langle \bx_i, \bcoefX \rangle \cdot \langle \bz_i, \bcoefZ \rangle) \bigl)^2.
\end{align}
From an optimization-theoretic standpoint, $F_N$ is a jointly nonconvex function in the parameters $(\bcoefX, \bcoefZ)$---even in special cases where $\psi$ is linear---and gives rise to a nonconvex problem on which the behavior of iterative algorithms can be analyzed. Indeed, when $\psi$ is the identity function, gradient methods and their variants~\citep{bhojanapalli2016global,ma2020implicit,davis2019stochastic,chen2021convex,soltani2017improved}, composite optimization methods~\citep{charisopoulos2021low,charisopoulos2021composite}, alternating update methods~\citep{zhong2015efficient}, and other methods in the Burer--Monteiro~\citep{burer2003nonlinear} family of algorithms~\citep{park2017non,chen2015fast} have all been analyzed for this problem, typically from ``good” initializations in the neighborhood of the unknown pair $(\bcoefX_{\star}, \bcoefZ_{\star})$. The aforementioned papers have proved upper bounds on the local rates of convergence of all these algorithms that hold with high probability---provided the pair $(\bx_i, \bz_i)$ is drawn from a suitable random ensemble. In this paper, we study the rank one generalized bilinear sensing problem under the canonical assumption that the vectors $\bx_i$ and $\bz_i$ are Gaussian with $\bx_i, \bz_i \overset{\mathsf{i.i.d.}}{\sim} \mathsf{N}(0, \bI_{d})$ and the noise $\epsilon_i$ is distributed as $\epsilon_i \overset{\mathsf{i.i.d.}}{\sim} \mathsf{N}(0, \sigma^2)$.

As previously mentioned, our goal is to establish \emph{sharp efficiency} estimates for natural iterative algorithms on this problem, and our investigations are informed by the following two questions:
\begin{itemize}
	\item Do complex iterative algorithms—that are based on more than first-order information at every iteration—converge from a random initialization to a statistically useful solution?
	\item Given the wide variation in the behavior of algorithms on related problems---for instance, their rates of convergence can vary from linear to superlinear in closely related problems~\citep{chandrasekher2021sharp,ghosh2020alternating}---can we establish tight efficiency estimates (both upper and lower bounds) on their convergence behavior?
\end{itemize}
Inspired by analogous one-shot procedures for single-index models~\citep{brillinger2012generalized,plan2016generalized}, our focus is on answering these questions for a natural alternating minimization algorithm that ignores the nonlinearity and alternates between estimates of the two factors $(\bcoefX_{\star}, \bcoefZ_{\star})$ by solving least squares problems. The algorithm is defined precisely in Section~\ref{sec:setup1} to follow, and our analysis proceeds under a natural \emph{sample-splitting} assumption that is pervasive in the nonconvex optimization literature~\citep[see, e.g.,][]{jain2013low,hardt2014fast,netrapalli2015phase,kwon2019global}. While operating in this specific setting allows us to state sharp and concrete results, we expect our technique itself to be much more broadly applicable to answering similarly posed questions in other nonconvex optimization problems with random data.

\subsection{Setup, contributions, and techniques} \label{sec:setup1}

One method to minimize the nonconvex loss $F_N$~\eqref{eq:global-loss} is to use the alternating minimization (AM) heuristic, which is a natural algorithm with classical roots~\citep{von1949rings} and the focus of our paper. The version of AM that we consider proceeds by fixing the coefficients $\bcoefX$ at the current iteration, and computes an estimate for the coefficients $\bcoefZ$ by solving a least squares problem that ignores the nonlinearity $\psi$.
 
Subsequently, it fixes the set of coefficients $\bcoefZ$ to find the next estimate of the coefficients $\bcoefX$ by solving another least squares problem. This iteration is executed iteratively until some stopping criterion is satisfied. 
As mentioned above, we analyze this iteration under a \emph{sample-splitting} assumption. More precisely, suppose we initialize the algorithm at some $\bcoefX_0$ and draw $2n$ \emph{fresh} observations $(y_i, \bx_i, \bz_i)$ i.i.d. per iteration $t$. Then we consider the following procedure run for each $t = 0, \ldots, T - 1$: 

\begin{subequations}\label{eq:AM}
	\begin{align} 
		\bcoefZ_{t+1} &= \argmin_{\bcoefZ \in \mathbb{R}^d}\; \sum_{i=1}^{n} \bigr(y_i - \langle \bx_i, \bcoefX_t \rangle \langle \bz_i, \bcoefZ \rangle \bigl)^2,\label{eq:coefZ-update}\\
		\bcoefX_{t+1} &=\argmin_{\bcoefX \in \mathbb{R}^d}\; \sum_{i=n+1}^{2n} \bigr(y_i - \langle \bx_i, \bcoefX \rangle \langle \bz_i, \bcoefZ_{t+1} \rangle \bigl)^2. \label{eq:coefX-update} 
	\end{align}
\end{subequations}
Our motivation for ignoring the link function in the algorithm comes from the fact that there are many situations where the link function is unknown to us. For example, in a closely related class of low-rank matrix completion problems~\citep{ganti2015matrix}, each entry of the low-rank matrix is observed after applying an unknown nonlinear but monotone transformation corresponding to a type of quantization. An unknown link function is also a critical component of single-index models, which have been applied in nonlinear dimensionality reduction and statistical signal processing (see, e.g.,~\citet{li1989regression,plan2016generalized} and the references therein). Nonetheless, in the sequel, we demonstrate that non-trivial estimation of the coefficients is still possible with iteration~\eqref{eq:AM} when $\psi$ satisfies a natural assumption.

Note that owing to our sample-splitting assumption, every appearance of $(\bx_i, \bz_i)$ in the iteration~\eqref{eq:AM} is independent of everything else. 
We abuse notation slightly and write $(\bX, \bZ) \in \real^{n \times d} \times \real^{n \times d}$ for the pair of data matrices used in each substep, so that we may write, e.g., Eq.~\eqref{eq:coefX-update} in closed form as
\begin{align} \label{eq:equiv-step}
	\bcoefX_{t+1} = \bigl(\bX^{\top} \bW_{t+1}^2 \bX\bigr)^{-1} \bX^{\top} \bW_{t+1}\by, \qquad \text{ where } \qquad \bW_{t+1} = \diag(\bZ\bcoefZ_{t+1}). 
\end{align}
Taking stock, we have let $n$ denote the \emph{per-substep} sample size, and we define $\Lambda := n/d$ to be the per-substep oversampling ratio. We assume throughout\footnote{Unlike results in high-dimensional asymptotic statistics that operate in the regime where $\Lambda$ is constant and $d \to \infty$, our treatment is fully non-asymptotic, requiring only that $\Lambda > 1$. The reader should think of $\Lambda$ just as convenient shorthand for the quantity $n/d$.} that $\Lambda > 1$, so that each matrix inversion in the update is well-defined with probability $1$. Note that if $T$ steps of the algorithm are run with each step defined by the pair of operations~\eqref{eq:AM}, we sample $2T$ pairs of data matrices.\footnote{In reality, each pair of data matrices is a fresh sample and ought to be indexed by the iteration count, but we drop this dependence for brevity. Also note that the total sample size in the problem is given by $N = 2T n$.} 

Recall that our goal is to estimate the ground-truth pair $(\bcoefX_{\star}, \bcoefZ_{\star})$. In order to assess the convergence of the AM iterations to this pair, it is convenient to define the following two-dimensional \emph{state evolution} for the parameter estimates over iterations, given by
\begin{subequations} \label{eq:comps}
	\begin{align} \label{eq:right-comps}
		\parcompZ_{t+1} &= \langle \bcoefZ_{t+1}, \bcoefZ_{\star} \rangle, \qquad  \perpcompZ_{t+1} = \bigl\| \bP_{\bcoefZ_{\star}}^{\perp} \bcoefZ_{t+1} \bigr \|_{2}, \qquad \text{ and } \\
		\parcompX_{t+1} &= \langle \bcoefX_{t+1}, \bcoefX_{\star} \rangle, \qquad  \perpcompX_{t+1} = \bigl \| \bP_{\bcoefX_{\star}}^{\perp} \bcoefX_{t+1} \bigr \|_{2}. \label{eq:left-comps}
	\end{align}
\end{subequations}
In words---and taking Eq.~\eqref{eq:left-comps} as an example---the scalar $\parcompX_{t+1}$ tracks the component of $\bcoefX_{t + 1}$ that is parallel to its ground truth analog $\bcoefX_{\star}$, and $\perpcompX_{t+1}$ tracks the norm of the component of $\bcoefX_{t + 1}$ perpendicular to the ground truth $\bcoefX_{\star}$. State evolutions---first introduced in the AMP literature~\citep{donoho2009message,bayati2011lasso}---have been used in several papers studying iterative optimization problems with random data both from a random initialization (e.g.~\cite{wu2019randomly,chen2019gradient}) as well as from local initializations (e.g.~\cite{donoho2009message,celentano2020estimation}) as succinct representations of an algorithm's behavior over iterations. Note that several pertinent error metrics for parameter estimation (e.g., the $\ell_2$ distance or angular distance to the ground truth) can be expressed purely in terms of the state evolution.
At this juncture, it is also useful to mention that the individual parameters $(\bcoefX_{\star}, \bcoefZ_{\star})$ are only identifiable from the observations~\eqref{eq:model} up to scale factors, i.e., the identifiable quantity is actually the matrix $\bcoefX_{\star}\bcoefZ^\top_{\star}$. 
This motivates us to look at functionals of the state evolution (e.g., the ratio $\perpcompX_{t}/\parcompX_{t}$) that are proxies for the angle between estimated quantities and their ground-truth analogues.

With this setup in hand, we are now in a position to describe our contributions, which are answers under this setup to the two questions posed in Section~\ref{sec:intro}.
\begin{enumerate}
	\item \textbf{Sharp, deterministic one-step predictions from a random initialization.} For any choice of the nonlinearity $\psi$ satisfying a set of regularity conditions (see Assumption~\ref{assptn:Y-psi} to follow), we derive explicit \emph{deterministic} one-step predictions (depending on $\psi$) for $(\parcompZ^{\mathsf{det}}_{t}, \perpcompZ^{\mathsf{det}}_{t}, \parcompX^{\mathsf{det}}_{t}, \perpcompX^{\mathsf{det}}_{t})$ that closely track their empirical counterparts, showing that with high probability, we have
	\begin{align} \label{eq:conc-intro}
		\max \left\{ | \parcompZ_{t} - \parcompZ^{\mathsf{det}}_{t} | , |\perpcompZ_{t} - \perpcompZ^{\mathsf{det}}_{t}|, |\parcompX_t - \parcompX^{\mathsf{det}}_{t}|, |\perpcompX_t - \perpcompX^{\mathsf{det}}_{t} | \right\} \lesssim n^{-1/2},
	\end{align}
	where the notation $\lesssim$ hides polylogarithmic factors in $n$. Note that this guarantee is fully non-asymptotic; see Theorem~\ref{thm:one-step} for the formal claim.  This presents a significant improvement over previous work~\citep{chandrasekher2021sharp}, which applies Gaussian comparison inequalities to obtain a sub-optimal rate of $n^{-1/4}$ on the orthogonal component $\perpcompX_{t}$.  Crucially, this improvement enables a sharp analysis from a random initialization, when the parallel component satisfies $\parcompX_{0}\asymp d^{-1/2}$.  See Section~\ref{sec:one-step-main} for a detailed discussion. 
	
	\item \textbf{Fine-grained convergence guarantees.} We use our deterministic one-step predictions to execute an iterate-by-iterate analysis of the algorithm from a random initialization for two canonical choices of the function $\psi$: the identity function, corresponding to the blind deconvolution problem, and the sign function, corresponding to such a problem with one-bit measurements. 
	This analysis reveals several similarities and differences between these models.  On the one hand, both models exhibit sharp linear convergence (see Definition~\ref{def:linear-convergence}) with rate $\asymp d/n$, thus rigorously justifying the phenomenon that ``larger problems are harder" to optimize; this was observed by~\cite[][see Figure 1]{agarwal2012fast} on a related problem, but only based on proving upper bounds on the rate.  On the other hand, the two models that we consider exhibit distinct behavior in low noise problems, in which the linear model converges to the level $\sigma^2 d/n$ and the non-linear model converges to the level $(1 + \sigma^2) d/n$.  See Theorems~\ref{thm:global-conv-linear} and~\ref{thm:global-conv-nonlinear} for precise statements. 
\end{enumerate}

From a technical perspective, there are many steps involved in 
arriving at the bounds~\eqref{eq:conc-intro}, several of which may be of independent interest. 
To show that our deviation bound~\eqref{eq:conc-intro} enjoys the $1/ \sqrt{n}$ rate---which in turn simultaneously captures both low-dimensional and high-dimensional problems---we proceed from a coordinate-by-coordinate characterization of the estimate $\bcoefX_{t + 1}$ that is derived using leave-one-out techniques.\footnote{Our leave-one-out techniques are distinct from those that have appeared in the recent nonconvex optimization literature~\citep{ma2020implicit,chen2019gradient}, and it would be interesting to combine the two approaches.} While a technique of this form has appeared before in the context of studying (one-shot) $M$-estimators in high dimensions~\citep{el2013robust,karoui2013asymptotic,el2018impact}, our analysis is significantly more involved given the iterative nature of the problem, the non-asymptotic character of our results, and the nonlinear dependence of the response on the planted signal. In particular: 
\begin{itemize}
	\item Obtaining the solution $\bcoefX_{t + 1}$ involves a least squares estimation task with the random design $\bX$, but also includes contributions from the other ``current" parameter, i.e., the diagonal matrix $\bW_{t+1} = \diag(\bZ\bcoefZ_{t+1})$ in Eq.~\eqref{eq:equiv-step}.  The presence of these interactions presents two difficulties in comparison to the work~\citep{el2013robust,karoui2013asymptotic,el2018impact}.  First, the effective data matrix---formed by the product $\bW_{t+1}\bX$---no longer consists of i.i.d. entries.  Second, the nonlinear observations preclude the previously used strategy of removing the signal; indeed, the signal component must be carefully handled. 
	
	\item Given the fact that we track a two-dimensional state evolution, we must prove non-asymptotic concentration bounds at the $n^{-1/2}$ rate on both the parallel and perpendicular components. As alluded to before, deviation bounds of the order $n^{-1/2}$ on both components are critical for us to obtain sharp guarantees from a random initialization.  We use a careful truncation argument in conjunction with Warnke's typical bounded differences inequality~\citep{warnke2016method} to obtain these guarantees.  Along the way, we develop some non-asymptotic random matrix theory which may be of independent interest (see Section~\ref{sec:non-asymptotic-rmt}). 
\end{itemize}
Overall, we hope that many of these technical tools---and the general perspective of reducing iterative nonconvex optimization to a sequence of convex $M$-estimation problems---will find broader application in the analysis of other such algorithms.

\subsection{Related work} \label{sec:related}

This work touches upon several themes in iterative optimization in statistical settings, which has a formidable literature. We discuss those papers that are most closely related to (and contextualize) our contributions.

\paragraph{Deterministic predictions beyond first-order methods.}
A natural approach to understanding the iterations of any algorithm in high-dimensional problems with random data is to understand these iterations in the infinite sample limit---using what is commonly known as the \emph{population update}. This approach has been influential in analyzing several iterative algorithms, specifically complex and higher-order algorithms~\citep{balakrishnan2017statistical,daskalakis2017ten,xu2018EM,kwon2019global,wu2019randomly,klusowski2019estimating,ho2020instability} but also several first-order algorithms~\citep{chen2019gradient,tian2017analytical}. To take the specific example of AM for our model, we may evaluate the update~\eqref{eq:coefX-update} in the  limit via a straightforward calculation, obtaining
\begin{align} \label{eq:pop}
	\lim_{n \to \infty} \Bigl(\frac{1}{n} \cdot \bX^{\top} \bW_{t+1}^2 \bX\Bigr)^{-1} \left( \frac{1}{n} \bX^{\top} \bW_{t+1}\by \right) = \frac{1}{\parcompZ_{t+1}^2 + \perpcompZ_{t+1}^2} \cdot \EE \bX^{\top} \bW_{t+1}\by = \widetilde{\tau}_{t + 1} \cdot \bcoefX_{\star},
\end{align}
where the constant of proportionality $\widetilde{\tau}_{t + 1}$ depends on $(\parcompZ_{t+1},\perpcompZ_{t+1})$ and the nonlinearity $\psi$.
In other words, after \emph{one} step, the population limit of the update~\eqref{eq:equiv-step} obtains an estimate which is a constant multiplied by the ground truth coefficient $\bcoefX_{\star}$. This suggests exact recovery in a single step, which is inconsistent with the empirical performance of the algorithm (see Figure~\ref{fig:intro-fig}).

\begin{figure}[!h]
	\centering
	\begin{subfigure}[b]{0.48\textwidth}
		\centering
		\includegraphics[scale=0.49]{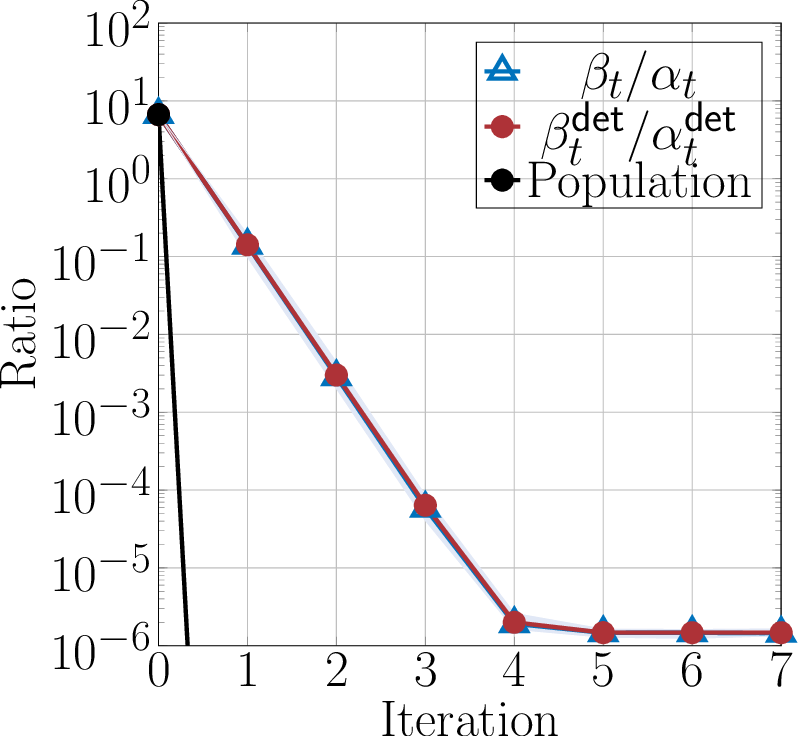} 
		\caption{Linear model ($\psi(w) = w$).} 
	\end{subfigure}
	\hfill
	\begin{subfigure}[b]{0.48\textwidth}  
		\centering 
		\includegraphics[scale=0.49]{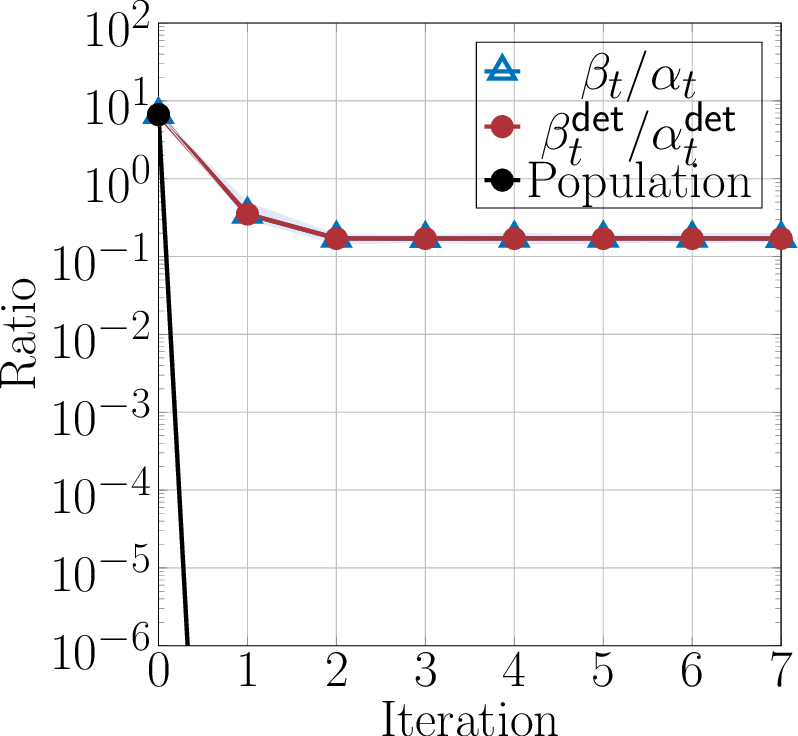}
		\caption{One-bit model ($\psi(w) = \sign(w)$).}
	\end{subfigure}
	\caption{Comparison of the population (infinite sample) prediction with our deterministic one-step prediction.  The (barely visible) blue shaded region denotes the minimum and maximum values over $50$ independent trials.  Each trial was run from a random initialization and used parameters $d=200, \Lambda=50$, and $\sigma=10^{-5}$.} 
	\label{fig:intro-fig}
\end{figure}
Besides the population update, the only work that we are aware of that produces deterministic one-step predictions for complex and higher-order iterative algorithms is a recent paper by a subset of the current authors~\cite{chandrasekher2021sharp}. This paper studies a general setup of nonconvex optimization with Gaussian data with sample-splitting at each iteration, and algorithms that solve convex optimization problems at each iteration. Using the machinery of Gaussian comparison theorems,
the paper derives a deterministic one-step prediction---a so-called \emph{Gordon update}---that is shown to match the state evolution at each iteration up to fluctuations of the order $n^{-1/4}$. These guarantees are derived for a class of generalized linear models and do not apply to the problem considered in the current paper, but it is worth mentioning two other limitations of the results in~\cite{chandrasekher2021sharp}. First, non-asymptotic guarantees using the Gordon machinery are challenging to derive, and rely on the growth properties of several intermediate loss functions. Second, deviation bounds of the order $n^{-1/4}$ derived using this machinery render such results suboptimal for low-dimensional problems and are not sharp enough to provide nontrivial information at random initialization.\footnote{The analysis at random initialization in~\cite{chandrasekher2021sharp} is based on an additional, involved argument that only yields $n^{-1/2}$ deviations for the parallel component.} The techniques of the current paper overcome both of these limitations with a more direct approach.

\paragraph{AMP and first-order methods.}
Several works have focused on sharp characterizations of efficiency for first-order methods in both convex and nonconvex problems. For instance,~\citet{oymak2017sharp} analyzed the projected gradient descent algorithm on constrained least squares problems, providing both upper and lower bounds on the convergence rate and enabling sharp time-data tradeoffs.  Another line of work---studying online methods rather than large batch or full sample methods---exploits random matrix theory to precisely understand the average-case behavior of SGD and related methods in quadratic models, and demonstrates substantially faster rates than those predicted by worst-case theory~\citep{paquette2021sgd,paquette2021dynamics}. A parallel literature (dating back to the seminal papers~\citep{donoho2009message,bayati2011lasso}) allows the analysis of a particular iterative algorithm---approximate message passing or AMP---through a state evolution. More recently, this perspective has been used to derive state evolutions for other ``generalized'' first-order methods~\citep{celentano2020estimation}---revealing several appealing statistical optimality properties about the Bayes-AMP---as well as for gradient flow~\citep{celentano2021high}. In the context of nonconvex problems, these analyses proceed from a correlated initialization (i.e. they analyze local convergence), and are able to produce iterate-by-iterate predictions for first-order methods in the challenging setting without sample-splitting.  From an optimization perspective, these methods have also been analyzed in certain settings to produce (asymptotically valid) convergence rates~\cite[see, e.g.,][Theorem 15]{berthier2020state}.  The closest example from this family to our setting is the generalized-AMP for matrix sensing problems~\citep{parker2016parametric}. While these algorithms and their resulting guarantees provide a powerful machinery, they do not immediately apply beyond first-order methods.  

\paragraph{Algorithms and guarantees for matrix sensing.} In addition to the algorithm-specific results alluded to in Section~\ref{sec:intro}, there are several structural results known about the landscape of the loss function~\eqref{eq:global-loss}, and related loss functions\footnote{The matrix sensing examples in~\citet{ge2017no}, for example, deal with matrix measurement ensembles satisfying an RIP condition; however, this does not hold in our problem, see~\citet[Appendix B]{zhong2015efficient}.} in matrix sensing and completion~\citep{ge2016matrix,ge2017no,bhojanapalli2016global,park2017non,diaz2019nonsmooth,zhang2019structured}. On the one hand, these are powerful results proved without the sample-splitting assumption that show that as soon as the sample size grows above a certain threshold, there are no local minima, ensuring (for instance) that saddle-avoiding algorithms can converge from an \emph{arbitrary initialization} to global minima~\citep{jin2017escape,carmon2018accelerated}. On the other hand, such a result does not directly imply a quantitative convergence rate for natural and popularly employed algorithms such as gradient descent and alternating minimization, and so cannot characterize their efficiency especially from a random initialization.  We also note that several recent papers~\citep[][to name a few]{stoger2021small,jiang2022algorithmic} have analyzed gradient descent from a (small) random initialization and demonstrated appealing statistical properties when coupled with early stopping; these analyses, however, are specific to the trajectory of gradient descent and, to our knowledge, do not have immediate implications for complex methods such as alternating minimization.  Most closely related to our development is the recent work~\citep{lee2022randomly}, which analyzes a similar alternating minimization method from a random initialization without a sample splitting assumption.  We note that the measurements considered in that work differ from our own~\eqref{eq:model} and in particular endow the loss~\eqref{eq:global-loss} with global geometric structure through the $\ell_2$--RIP (which our measurements do not satisfy, see the discussion following Theorem~\ref{thm:global-conv-linear} for details).  Additionally, we improve upon the convergence guarantees by precisely quantifying the dependence of the convergence rate on problem-specific parameters such as the noise level and sample size. This is in line with our principal motivation, which is not to develop a state-of-the-art algorithm---although the iteration~\eqref{eq:AM} is new to our knowledge when there are nonlinearities in the model---but to pursue a sharp understanding of a natural iterative method for this problem.

\subsection{Notation and organization}
\label{subsect:notation}

\paragraph{Notation.}
We let $[d]$ denote the set of natural numbers less than or equal to $d$.  
We use boldface small letters to denote vectors and boldface capital letters to denote matrices. We let $\sign(v)$ denote the sign of a scalar $v$, with the convention that $\sign(0) = 1$. We use $\sign(\bv)$ to denote the sign function applied entrywise to a vector $\bv$. Let $\mathbbm{1}\{\cdot\}$ denote the
indicator function. 

For two sequences of non-negative reals $\{f_n\}_{n
	\geq 1}$ and $\{g_n \}_{n \geq 1}$, we use $f_n \lesssim g_n$ to indicate that
there is a universal positive constant $C$ such that $f_n \leq C g_n$ for all
$n \geq 1$. The relation $f_n \gtrsim g_n$ indicates that $g_n \lesssim f_n$,
and we say that $f_n \asymp g_n$ if both $f_n \lesssim g_n$ and $f_n \gtrsim
g_n$ hold simultaneously. We also use standard order notation $f_n = \order
(g_n)$ to indicate that $f_n \lesssim g_n$ and $f_n = \ordertil(g_n)$ to
indicate that $f_n \lesssim
g_n \log^c n$, for a universal constant $c>0$. We say that $f_n = \Omega(g_n)$ (resp. $f_n = \widetilde{\Omega}(g_n)$) if $g_n = \order(f_n)$ (resp. $g_n = \ordertil(f_n)$). The notation $f_n = o(g_n)$ is
used when $\lim_{n \to \infty} f_n / g_n = 0$, and $f_n =
\omega(g_n)$ when $g_n = o(f_n)$. Throughout, we use $c, C$ to denote universal
positive constants, and their values may change from line to line. All logarithms are to the natural base unless
otherwise stated.

We denote by $\NORMAL(\bm{\mu}, \bSig)$ a normal distribution with mean $\bm{\mu}$ and covariance matrix $\bSig$. Let $\mathsf{Unif}(S)$ denote the uniform distribution on a set $S$, where the distinction between a discrete and continuous distribution can be made from context. We say that $X \overset{(d)}{=} Y$ for two random variables $X$ and $Y$ that are equal in distribution. For $q \geq 1$ and a random variable $X$ taking values in $\real^d$, we write $\| X \|_q = (\EE[ |X|^q ])^{1/q}$ for its $L^{q}$ norm. Finally, for a real valued random variable $X$ and a strictly increasing convex function $\psi: \real_{\geq 0} \to \real_{\geq 0}$ satisfying $\psi(0) = 0$, we write $\| X \|_{\psi} = \inf\{ t > 0 \; \mid \; \EE[\psi(t^{-1} |X| )] \leq 1\}$ for its $\psi$-Orlicz norm. We make particular use of the $\psi_q$-Orlicz norm for $\psi_q(u) = \exp(|u|^q) - 1$. We say that $X$ is sub-Gaussian if $\| X \|_{\psi_2}$ is finite
and that $X$ is sub-exponential if $\| X \|_{\psi_1}$ is finite.

\paragraph{Organization.}
The rest of the paper is organized as follows. In Section~\ref{sec:one-step-main}, we provide our deterministic one-step predictions for general nonlinear $\psi$, in Theorem~\ref{thm:one-step}. In Section~\ref{sec:convergence-main}, we use our one-step updates in two special cases---the linear model in which $\psi$ is the identity function and the one-bit model in which $\psi$ is the sign function---to prove a sharp linear convergence result for the AM algorithm (in Theorems~\ref{thm:global-conv-linear} and~\ref{thm:global-conv-nonlinear}, respectively). In Section~\ref{sec:proof-one-step}, we prove Theorem~\ref{thm:one-step}, and Sections~\ref{sec:proof-conv-linear} and~\ref{sec:proof-conv-nonlinear} are dedicated to proofs of Theorems~\ref{thm:global-conv-linear} and~\ref{thm:global-conv-nonlinear}, respectively. Proofs of technical lemmas are postponed to the appendices.

\section{General deterministic one-step prediction} \label{sec:one-step-main}
We begin by deriving our one-step, deterministic updates, which are general in that they hold for a wide class of nonlinearities $\psi$.  The updates are defined in terms of a positive scalar $C(\oversamp)$, which in turn is defined implicitly as the unique\footnote{See Lemma~\ref{lem:uniqueness} for a proof that the solution is unique.} solution to the fixed point equation 
\begin{align}\label{definition-of-C}
	\frac{1}{\oversamp} = \EE\Bigl\{ \frac{G^2}{C(\oversamp) + G^2} \Bigr\} \qquad \text{ where } \qquad G \sim \mathsf{N}(0, 1).
\end{align}
We will often denote $\tau = 1/C(\oversamp)$ for convenience.  
It can be shown that $0.3\oversamp \leq C(\oversamp) \leq \oversamp$ for $\oversamp \geq 10$ (see Lemma~\ref{lemma:C(lambda)-lambda}), so that the reader should think of $C(\Lambda)$ as scaling linearly in $\Lambda$.
Given the pair $(\parcompZ_{t+1}$, $\perpcompZ_{t+1})$---see Eq.~\eqref{eq:comps}---and a function $\psi:\mathbb{R} \rightarrow \mathbb{R}$, define the random variables
\begin{align}\label{definition-X-G-V-Y}
	X,G,V \overset{\mathsf{i.i.d.}}{\sim} \mathsf{N}(0,1)\quad\text{ and }\quad Y := \psi\biggl( \frac{X}{(\parcompZ_{t+1}^{2} + \perpcompZ_{t+1}^{2})^{1/2}}\cdot \bigl(\parcompZ_{t+1}G+ \perpcompZ_{t+1}V\bigr)\biggr).
\end{align}
With this setup, we are now ready to define the following deterministic one-step state prediction for the $\bcoefX_{t + 1}$-update; we note that a symmetric update can be derived for the pair $(\parcompZ_{t+1}, \perpcompZ_{t + 1})$:
\begin{subequations}\label{eq-equivalence}
	\begin{align}
		\parcompX_{t+1}^{\mathsf{det}} &= \frac{1}{\sqrt{\parcompZ_{t+1}^{2} + \perpcompZ_{t+1}^{2}}} \cdot	\EE\left\{ \frac{G X Y}{1+\tau G^{2}} \right\} \bigg/ \EE\left\{ \frac{G^{2}}{1+\tau G^{2}} \right\},\qquad \text{ and } \label{par-equiv}\\
		(\perpcompX_{t+1}^{\mathsf{det}})^{2} &=  \frac{\EE\Big\{ \frac{G^{2} Y^{2}}{(1+\tau  G^{2})^{2}}\Big\} +  \EE\Big\{ \frac{\sigma^{2}G^{2}}{(1+\tau G^{2})^{2}} \Big\}}{ (\parcompZ_{t+1}^{2} + \perpcompZ_{t+1}^{2}) \cdot \EE\Big\{ \frac{C(\oversamp)G^{2}}{(1+\tau G^{2})^{2}} \Big\} } - \frac{2\parcompX_{t+1}^{\mathsf{det}}  \EE\Big\{ \frac{G^{3}XY}{(1+\tau G^{2})^{2}} \Big\}}{ (\parcompZ_{t+1}^{2} + \perpcompZ_{t+1}^{2})^{1/2} \cdot \EE\Big\{ \frac{C(\oversamp)G^{2}}{(1+\tau G^{2})^{2}} \Big\} } \label{perp-equiv} \\
		&\hspace{7.8cm}+ \frac{ (\parcompX_{t+1}^{\mathsf{det}})^{2} \EE\Big\{ \frac{G^{4}}{(1+\tau G^{2})^{2}} \Big\} }{ \EE\Big\{ \frac{C(\oversamp)G^{2}}{(1+\tau G^{2})^{2}} \Big\} }. \nonumber
	\end{align}
\end{subequations}
Note that the pair $(\parcompX_{t+1}^{\mathsf{det}}, \perpcompX_{t+1}^{\mathsf{det}})$ depends on the nonlinearity $\psi$ only through the random variable $Y$. \revision{ We also note that the population update can be obtained by taking $\oversamp \rightarrow \infty$ in Eq.~\eqref{eq-equivalence}. By noting $C(\oversamp) \rightarrow +\infty$ and $\tau = 1/C(\oversamp) \rightarrow 0$ as $\oversamp \rightarrow +\infty$, we obtain
\[
	\parcompX_{t+1}^{\mathsf{det}} \rightarrow \frac{\EE\left\{ GXY \right\}}{\sqrt{\parcompZ_{t+1}^{2} + \perpcompZ_{t+1}^{2}}} \quad \text{ and } \quad (\perpcompX_{t+1}^{\mathsf{det}})^{2} \rightarrow 0, \quad \text{as} \quad \oversamp \rightarrow +\infty.
\]
Thus, in the population update, the perpendicular component becomes zero in one step of alternating minimization, which is the main issue caused by taking $n \rightarrow +\infty$.
} 

While the deterministic one-step updates in Eq.~\eqref{eq-equivalence} can be defined under mild assumptions, we now state a convenient assumption under which these quantities track their empirical counterparts. 

\begin{assumption}\label{assptn:Y-psi}
	The function $\psi$ is either bounded with $\| \psi \|_{\infty} \leq C_{\psi}$ or linear, with $\| \psi' \|_{\infty} \leq C_{\psi}$.
\end{assumption}

While Assumption~\ref{assptn:Y-psi} seems strict on the face of it, we note that it can be weakened to accommodate polynomially growing functions via standard truncation arguments. It suffices for us since it covers the two canonical cases of $\psi$ that we intend to address in Section~\ref{sec:convergence-main} to follow. We also state the following assumption on the distribution of sensing vectors and noise in statistical model~\eqref{eq:model}.
\begin{assumption}\label{assptn:distribution}
\revision{
		The sensing vectors are drawn as $\{\bx_i,\bz_i\}_{i=1}^{n} \overset{\mathsf{i.i.d.}}{\sim} \mathsf{N}(0, \bI_{d}) \otimes \mathsf{N}(0, \bI_{d})$ and the noise is drawn as $\{\epsilon_{i}\}_{i=1}^{n} \overset{\mathsf{i.i.d.}}{\sim} \mathsf{N}(0, \sigma^2)$, independently of the sensing vectors. Moreover, we draw fresh samples in each step of alternating minimization~\eqref{eq:AM}.
	}
\end{assumption}
We are now poised to state our main result for this section, which is proved in Section~\ref{sec:proof-one-step}.
\begin{theorem}\label{thm:one-step}
	Suppose the data are drawn from the model~\eqref{eq:model} and that Assumptions~\ref{assptn:Y-psi} and~\ref{assptn:distribution} hold.  Consider the pairs of empirical updates $(\parcompX_{t+1}, \perpcompX_{t+1})$~\eqref{eq:left-comps} and $(\parcompZ_{t+1}, \perpcompZ_{t+1})$~\eqref{eq:right-comps} as well as the pair of deterministic one-step updates $(\parcompX_{t+1}^{\mathsf{det}},\perpcompX_{t+1}^{\mathsf{det}})$~\eqref{eq-equivalence}.  There exists a pair of universal, positive constants $(C, C')$ which depend only on the parameter $C_{\psi}$ and a universal, positive constant $C''$ such that if $\oversamp \geq C''$, the following hold with probability at least $1 - C'n^{-10}$, 
	\begin{subequations}
		\begin{itemize}
			\item[(a)] The parallel component satisfies 
			\begin{align}
				\bigl \lvert \parcompX_{t+1} - \parcompX_{t+1}^{\mathsf{det}} \bigr \rvert  &\leq \frac{C(\parsigma)}{\sqrt{\parcompZ_{t+1}^2 + \perpcompZ_{t+1}^2}} \cdot \pardevn. \label{eq:par-deviation}
			\end{align}
			\item[(b)] The perpendicular component satisfies 
			\begin{align}
				\bigl \lvert \perpcompX^2_{t+1} - (\perpcompX_{t+1}^{\mathsf{det}})^2 \bigr \rvert &\leq \frac{C (\perpsigma)}{\parcompZ_{t+1}^2 + \perpcompZ_{t+1}^2} \cdot \frac{\log^8 (n)}{\sqrt{n}} \label{eq:perp-deviation}.
			\end{align}
		\end{itemize}
	\end{subequations}
\end{theorem}
By symmetry, analogous results hold for the iterates $\{\bcoefZ_{t} \}_{t \geq 1}$. We additionally note that we have made no attempt to optimize the log factors in either deviation component.  

In light of Theorem~\ref{thm:one-step}, we deduce the key takeaway mentioned in Section~\ref{sec:intro}: Up to polylogarithmic factors, we have the optimal\footnote{To illustrate optimality of the fluctuations, consider a slightly simplified but analogous setting of the standard linear model with a standard Gaussian design $\bX\in \mathbb{R}^{n \times d}$ and observations $\by = \bX\boldsymbol{\theta} + \beps$ for standard Gaussian $\beps$. Suppose we are interested in estimating the linear functional $T(\boldsymbol{\theta}) = \langle \bv, \boldsymbol{\theta}\rangle$ for a fixed unit-norm vector $\boldsymbol{v}$.  By Le Cam's two point argument~\citep[see, e.g.,][]{wainwright2019high}, one can show that $\inf_{\hat{T}} \sup_{\boldsymbol{\theta} \in \mathbb{R}^d } \mathbb{E} \bigl[ \lvert \hat{T} - T(\boldsymbol{\theta}) \rvert \bigr] \gtrsim n^{-1/2}$.  Note that for our one-step predictions, the parallel component corresponds to estimating a fixed linear functional, and the argument sketched above illustrates that the optimal deviations are on the order $\Omega(n^{-1/2})$.} deviation bound
\begin{align}
	\max \left\{ |\parcompX_t - \parcompX^{\mathsf{det}}_{t}|, |(\perpcompX_{t})^{2} - (\perpcompX^{\mathsf{det}}_{t})^{2} | \right\} \lesssim n^{-1/2},
\end{align}
improving on the analogous bounds shown in~\cite[Theorem 3.1]{chandrasekher2021sharp}, where deviations were controlled up to $\widetilde{\mathcal{O}}(n^{-1/4})$.  This polynomial improvement is crucial to ensure that our downstream convergence results are sharp \emph{from a random initialization}.  By contrast, previous work~\citep{chandrasekher2021sharp} shows an upper bound on the convergence rate from a random initialization, only obtaining a matching lower bound upon entering a locally converging region.

Our proof of Theorem~\ref{thm:one-step} relies crucially on exploiting the fact that the iterate $\bcoefX_{t+1}$~\eqref{eq:AM} can be written as the solution to a high dimensional $M$-estimation problem.  In light of this observation, we employ and develop a leave-one-out argument due to~\cite{el2013robust}.  This technique allows us to directly analyze the optimizers, which in turn allows us to establish the sharp, non-asymptotic $\widetilde{O}(n^{-1/2})$ fluctuations.  By contrast, the Gaussian comparison inequalities leveraged by previous work~\citep{chandrasekher2021sharp} analyze the \emph{minimum} of the loss function defining each iterate and recover the iterate by utilizing strong convexity of the loss, thereby losing a factor of $n^{1/4}$ in the deviation bound.  

In spite of the fact that we leverage the technique introduced in~\cite{el2013robust}, we note that several complications arise when applying the leave-one-out argument. In particular, since our measurements do not come from the standard linear model, we are unable to remove the contribution of the signal and reduce it to the study of a pure noise model.  In order to handle this, we carefully handle the deviations of the signal component and orthogonal component separately, leveraging tools such as Warnke's typical bounded differences inequality~\citep{warnke2016method} to facilitate our analysis.  We believe that these tools may be of independent interest.  

A technical takeaway from the proof is a non-asymptotic random matrix theory result which may be of independent interest.  In particular, in Section~\ref{sec:concentration-trace-inverse}, we provide a proof that the trace inverse concentrates around the quantity $\tau$~\eqref{definition-of-C} with exponential tails:
\[
\Prob\bigl\{\bigl \lvert\trace\bigl((\bX^{\top} \bG^2 \bX)^{-1}\bigr) - \tau \bigr \rvert \geq t \bigr\} \leq Ce^{-cnt^2},
\]
where here $\bG= \diag(G_1, G_2, \dots, G_n)$ and $(G_i)_{1 \leq i \leq n} \overset{\mathsf{i.i.d.}}{\sim} \mathsf{N}(0, 1)$.  
Note that the quantity $\trace\bigl((\bX^{\top} \bG^2 \bX)^{-1}\bigr)$ corresponds to the Stieltjes transform of the empirical spectral distribution of the random matrix $\bX^{\top} \bG^2 \bX$, evaluated at $z=0$, whereby, as long as $\oversamp$ is of constant order,~\cite[][Theorem 1.1]{bai2004clt} yields $\bigl \lvert\trace\bigl((\bX^{\top} \bG^2 \bX)^{-1}\bigr) - \tau \bigr \rvert \lesssim n^{-1/2}$.  We complement this result by providing an exponential tail bound, noting that standard arguments which prove concentration of the Stieltjes transform~\cite[see, e.g,][]{el2009concentration,tao2012topics} do not apply as stated at $z = 0$ and typically assume $\oversamp = \order(1)$; note that we require $\oversamp = \Omega(\log{d})$ for our random initialization results (see Assumption~\ref{asm:random-linear} to follow).  Closer to our development are related nonasymptotic results holding at $z=0$~\cite[see, e.g.,][]{guionnet2000concentration,guntuboyina2009concentration} that bypass the direct use of the Stieltjes transform.  We layer a careful truncation argument upon these results and further characterize the typical value of the trace-inverse\footnote{In order to do so, we require the empirical spectral distribution to be bounded away from zero; we provide a proof of such a result for our setting in Section~\ref{sec:bounds-minimum-eigenvalue}, which---in this special case---forms a simple, finite-sample alternative to the proof of~\cite[Theorem 1.1]{bai1998no}.} (i.e. the quantity $\tau$) for all values of $\Lambda$. 
\begin{remark} \label{rem:poly-growth}
	\revision{
			We note that Theorem~\ref{thm:one-step} can be extended to accomodate polynomially growing link function $\psi$, i.e., $|\psi(x)| \leq C_{\psi} |x|^{D}$ for all $x \in \mathbb{R}$ (see Remark~\ref{rem:weaken-tails} in the proof). In particular, there exists a pair of constants $(C,C')$ which depend only on $C_{\psi}$ and $D$ such that with probability $1-C'n^{-10}$, 
			\begin{align}
				\bigl \lvert \parcompX_{t+1} - \parcompX_{t+1}^{\mathsf{det}} \bigr \rvert  &\leq \frac{C(\parsigma)}{\sqrt{\parcompZ_{t+1}^2 + \perpcompZ_{t+1}^2}} \cdot \frac{\log^{4D}(n)}{\sqrt{n}}. \label{eq:par-deviation-general-psi}
			\end{align}
			Thus, we still obtain the $n^{-1/2}$ deviation rate up to polylogarithmic factors for the parallel component and this enables us to analyze the convergence of the algorithm from a random initialization.
	}
\end{remark}
Let us now simplify the expressions for $\parcompX_{t+1}^{\mathsf{det}}$ and $\perpcompX_{t+1}^{\mathsf{det}}$ 
by considering two prototypical choices for the function $\psi$.

\begin{example}[Linear measurements: Identity function] \label{ex:lin}
	In the case where $\psi(w) = w$, consider the following two functions mapping $\real^2 \to \real$:
	\begin{subequations}
		\label{eq:idmaps}
		\begin{align}
			\parmapid(\alpha, \beta) &= \frac{\alpha}{\alpha^2 + \beta^2}, \qquad \text{ and } \label{eq:parmapid}\\
			\perpmapid(\alpha, \beta) &= \frac{1+\sigma^2}{C(\oversamp)} \cdot \frac{\perpcompX^2}{(\parcompX^2 + \perpcompX^2)^2} + \frac{\sigma^2}{C(\oversamp)} \cdot \frac{\parcompX^2}{(\parcompX^2 + \perpcompX^2)^2}. \label{eq:perpmapid}
		\end{align}
	\end{subequations}
	In Section~\ref{sec:proof-det-equivalence}, we show that\footnote{We note that in the special case of the identity model, one can derive the predictions~\eqref{eq:idmaps} via more direct means involving an orthogonalization argument with respect to Gaussian measure. Such a technique does not appear to extend to our general model.}
	\begin{align} \label{eq:alphabeta-FG-linear}
		\parcompX_{t+1}^{\mathsf{det}} = \parmapid\big( \parcompZ_{t+1},\perpcompZ_{t+1} \big) \quad \text{and} \quad (\perpcompX_{t+1}^{\mathsf{det}})^{2} = \perpmapid\big( \parcompZ_{t+1},\perpcompZ_{t+1} \big),
	\end{align}
	so that the deterministic one-step updates can be succinctly described using the pair of maps $(\parmapid, \perpmapid)$. Clearly, Assumption~\ref{assptn:Y-psi} holds with $C_{\psi} = 1$, so that Theorem~\ref{thm:one-step} applies to show adherence of the empirical state evolution to our deterministic one-step updates. We use this result to derive a global convergence guarantee in Theorem~\ref{thm:global-conv-linear} to follow. \hfill $\clubsuit$
\end{example}

For a second example, consider an observation model with signed measurements.
\begin{example}[One-bit measurements: Sign function] \label{ex:nonlin}
	Suppose the nonlinearity is given by $\psi(w) = \sign(w)$.
	For convenience, additionally define the scalars
	\begin{align}
		\label{eq:C2-C3}
		C_{2}(\oversamp) = \EE\left\{ \frac{W^{2}}{(C(\oversamp) + W^{2})^{2}}\right\} \qquad \text{ and } \qquad C_{3}(\oversamp) = \EE\left\{ \frac{W^{4}}{(C(\oversamp) + W^{2})^{2}} \right\},
	\end{align}
	where $W \sim \mathsf{N}(0, 1)$.  
	Using these, define the maps $\parmapbit: \mathbb{R}^2 \rightarrow \mathbb{R}$ and $\perpmapbit:\mathbb{R}^2 \rightarrow \mathbb{R}$ via
	\begin{subequations}
		\label{eq:bitmaps}
		\begin{align}
			\parmapbit(\alpha, \beta) &= \frac{2}{\pi}\frac{\oversamp}{\sqrt{\parcompX^{2} + \perpcompX^{2}}} \cdot \EE\biggl\{ \frac{|W| \phi(\frac{\parcompX}{\perpcompX} |W|)}{C(\oversamp) + W^{2}} \biggr\}, \qquad \text{ and } \label{eq:parmapbit}\\
			\perpmapbit(\alpha, \beta) &= \frac{1+\sigma^{2}}{C(\oversamp)} \cdot \frac{1}{\parcompX^{2} + \perpcompX^{2}} + \frac{\parmapbit(\parcompX, \perpcompX)^{2}}{C(\oversamp)}  \cdot \frac{C_{3}(\oversamp)}{C_{2}(\oversamp)}
			\nonumber\\
			&\qquad \qquad \qquad \qquad- \frac{4}{\pi} \cdot \frac{\parmapbit(\parcompX, \perpcompX)}{C(\oversamp)\sqrt{\parcompX^{2} + \perpcompX^{2}}} \cdot \EE\biggl\{ \frac{|W|^{3}  \phi(\frac{\parcompX}{\perpcompX} |W|)}{C_{2}(\oversamp) (C(\oversamp) + W^{2})^{2}} \biggr\},\label{eq:perpmapbit}
		\end{align}
	\end{subequations}
	where in the formulas above, $\phi(x) = \int_{0}^{x} e^{-t^2/2} \mathrm{d}t$ and $W \sim \mathsf{N}(0, 1)$.
	In Section~\ref{sec:proof-det-equivalence}, we evaluate the expectations in the formulas~\eqref{eq-equivalence} to show that
	\begin{align} \label{eq:alphabeta-FG-nonlinear}
		\parcompX_{t+1}^{\mathsf{det}} = \parmapbit\big( \parcompZ_{t+1},\perpcompZ_{t+1} \big) \quad \text{and} \quad (\perpcompX_{t+1}^{\mathsf{det}})^{2} = \perpmapbit\big( \parcompZ_{t+1},\perpcompZ_{t+1} \big),
	\end{align}
	so that the deterministic one-step updates can be succinctly described using the pair of maps $(\parmapbit, \perpmapbit)$. Assumption~\ref{assptn:Y-psi} holds with $C_{\psi} = 1$, so that Theorem~\ref{thm:one-step} then establishes adherence of the empirical state evolution to our deterministic one-step updates. This in turn enables a sharp global convergence guarantee in Theorem~\ref{thm:global-conv-nonlinear} to follow. \hfill $\clubsuit$
\end{example}

\begin{remark}[Even link function]\label{rmk:even-link-func}
\revision{
	Consider any case where the link function is even, i.e., $\psi(x) = \psi(-x)$ for all $x\in \real$. We see that Eq.~\eqref{eq-equivalence} becomes
			\begin{align*}
				\parcompX_{t+1}^{\mathsf{det}} = 0 \qquad \text{and} \qquad \big(\perpcompX_{t+1}^{\mathsf{det}}\big)^{2} = \frac{ \EE \left\{ \frac{ G^2(Y^2+\sigma^2)}{(1+\tau G^2)^2} \right\} }{(\parcompZ_{t+1}^{2} + \perpcompZ_{t+1}^{2} ) \cdot \EE\left\{ \frac{C(\oversamp) G^2}{(1+\tau G^2)^2} \right\}} \neq 0,
			\end{align*}
			where $G,X,Y$ are defined in Eq.~\eqref{definition-X-G-V-Y}. In words, one iteration of the algorithm makes the parallel component $\alpha_{t + 1}$ essentially zero (up to random fluctuations on the order $1/\sqrt{n}$), so that we no longer retain any overlap with the signal and the algorithm cannot converge to the true signal. The lack of the algorithm's convergence is related to phenomena noted in the literature~\citep[e.g.][]{plan2016generalized} and can be seen immediately by inspecting the deterministic one-step updates in Eq.~\eqref{eq-equivalence}: A necessary condition for convergence is that $\parcompX_{t+1}^{\mathsf{det}} \neq 0$, so the link function must (at the very least) satisfy $\mathbb{E}\left\{ \frac{GXY}{1 + \tau G^2} \right\} \neq 0$. 
}
\end{remark}

Having established our deterministic one-step predictions for a general class of nonlinearities $\psi$ and having showcased explicit expressions for these predictions in two canonical settings, we are now in a position to present our global convergence guarantees for these settings.

\section{Sharp global convergence guarantees} \label{sec:convergence-main}
Our results in this section rely on the following notion of \emph{sharp} linear convergence, presented in the definition below.
\begin{definition}[Sharp linear convergence]\label{def:linear-convergence}
	For parameters $0 < \rho_{1} \leq \rho_{2} <1$ and $0 \leq \varepsilon_{1} \leq \varepsilon_{2}$, the iterates $\{ \parcompX_t,\perpcompX_t \}_{t \geq 0}$ are said to exhibit $(\rho_{1}, \rho_{2}, \varepsilon_{1}, \varepsilon_{2})$--linear convergence in the squared ratio metric for $T$ iterations if for all $0 \leq t \leq T - 1$, we have
	\[
	 \rho_{1} \cdot \frac{\perpcompX_{t}^2}{\parcompX_{t}^{2}} +  \varepsilon_{1}  \leq \frac{\perpcompX_{t+1}^2}{\parcompX_{t+1}^2} \leq \rho_{2} \cdot \frac{\perpcompX_{t}^2}{\parcompX_{t}^2} + \varepsilon_{2}.
	\]
\end{definition}
In the definition, the positive scalars $\rho_{1}$ and $\rho_{2}$ should be thought of as a rate parameter, where we have $\rho_{2} < 1$ to ensure that the ratio converges. The parameters $\varepsilon_{1}$ and $\varepsilon_{2}$ are the level---or error floor---up to which linear convergence occurs. A particular feature of this definition is that if $\rho_{1} \asymp \rho_{2}$ and $\varepsilon_{1} \asymp \varepsilon_{2}$ then it postulates tight upper and lower bounds on the rate at which the error decreases and also on the eventual error floor, thus capturing the \emph{exact} rate of convergence and the \emph{exact} error floor---up to universal constants---of a linearly convergent algorithm.

Additionally, note that we define convergence in the nonstandard squared-ratio metric.  We do so as the factors $\bcoefX_{\star}, \bcoefZ_{\star}$ are identifiable only up to a scaling factor; that is, for any scalar $a \in \mathbb{R}$, the change of variables $\bcoefX_{\star} \mapsto a \bcoefX_{\star}$ and $\bcoefZ_{\star} \mapsto a^{-1} \bcoefZ_{\star}$ leaves the observations $y_i$~\eqref{eq:model} unchanged.  In other words, under the model~\eqref{eq:model}, the factor $\bcoefX_{\star}$ is identifiable only up to the one-dimensional subspace spanned by  $\bcoefX_{\star}$.  In light of this, the squared ratio provides a natural metric as it is scale-invariant and captures the distance to the subspace spanned by the factor $\bcoefX_{\star}$.  Note additionally that the convergence in the squared ratio metric implies convergence in the angular metric, as $\angle (\bcoefX_{t+1}, \bcoefX_{\star}) = \arccos\bigl((\perpcompX_{t+1}^2/\parcompX_{t+1}^2 + 1)^{-1}\bigr)$.

Having defined our notion of sharp linear convergence, we now state two assumptions on the pair $(\parcompX_{0}, \perpcompX_0)$ corresponding to the initialization $\bcoefX_{0}$. Our main theorem will hold under either of these two assumptions, provided the oversampling ratio $\Lambda$ is correspondingly controlled.
\begin{assumption} \label{asm:random-linear}
	The initialization satisfies both
	\[
	\frac{1}{50\sqrt{d}} \leq \parcompX_{0} \leq \frac{1}{30}\sqrt{\frac{1 + \sigma^2}{C(\oversamp)}}, \quad \text{ and } \quad 0.8 \leq \perpcompX_{0}^2 \leq 1.2,
	\]
	and the oversampling ratio satisfies $C_0(1+\sigma^{2}) \log(d) \leq \oversamp \leq \sqrt{d}$, for a universal, positive constant $C_0$.  
\end{assumption}
Note that this assumption encapsulates a random initialization, since if $\bcoefX_{0}$ is chosen uniformly at random from the unit ball $\mathbb{B}_{2}(1)$ with $d \geq 130$, then the conditions $\frac{1}{50\sqrt{d}} \leq \parcompX_{0} \leq \frac{1}{\sqrt{d}}$ and $\parcompX_{0}^{2} + \perpcompX_{0}^{2}=1$ are satisfied with probability at least $0.95$~\cite[see, e.g.,][Lemma 24]{chandrasekher2021sharp}.  
\begin{assumption}\label{asm:local-linear}
	The initialization satisfies 
	\[
	\frac{\perpcompX_0}{\parcompX_0} \leq 50\sqrt{\frac{C(\oversamp)}{1 + \sigma^2}}.
	\]
	Further, the oversampling ratio satisfies $C_0(1+\sigma^{2}) \leq \oversamp \leq \sqrt{d}$, for a universal, positive constant $C_0$.
\end{assumption}

Assumption~\ref{asm:local-linear} holds in a local region around the ground truth parameters, where the region gets larger as $C(\Lambda)$ (or correspondingly, $\Lambda$) gets larger. Note that from such a local initialization, we require only constant oversampling, $\oversamp \geq C_0(1+\sigma^{2})$, whereas from a random initialization, we require logarithmic oversampling, e.g. $\oversamp \geq C_0 (1+\sigma^{2})\log(d)$ (cf. Assumption~\ref{asm:random-linear}). 

Under Assumption~\ref{asm:random-linear} and Assumption~\ref{asm:local-linear}, we require the upper bound $\oversamp \leq \sqrt{d}$, which is equivalent to a sample complexity upper bound of $n \lesssim d^{3/2}$.  This should not be thought of a limiting requirement: If by contrast, $n \gtrsim d^{3/2}$, then we show in the section~\ref{sec:global-convergence-large-oversamp} that---from a random initialization---in both the linear observation model and the nonlinear observation model, two steps of AM guarantee an estimate with error 
\[
	\frac{\perpcompX_{2}^{2}}{\parcompX_{2}^{2}} \lesssim \frac{(1+\sigma^{2})d}{n} + \frac{(1+\sigma^{2})\perplogn}{\sqrt{n}}.
\]  

Having defined both assumptions, we are now in a position to state global convergence results for both the linear and one-bit models. We do so by stating a sharp convergence bound that holds when \emph{either} Assumption~\ref{asm:random-linear} or~\ref{asm:local-linear} is true, thereby providing a unified, sharp convergence claim both locally and globally around the ground truth. We state these results for the sequence $\{ \bcoefX_t \}_{t \geq 0}$, but by symmetry, identical results hold for the sequence $\{ \bcoefZ_t \}_{t \geq 0}$.

\subsection{Linear observation model}

We begin with the linear observation model. 
\begin{theorem} \label{thm:global-conv-linear}
	Suppose Assumption~\ref{assptn:distribution} holds. Consider the linear observation model~\eqref{eq:model} with $\psi(x) = x$ and the AM algorithm~\eqref{eq:AM} run for $T$ iterations. There exists a tuple of universal, positive constants $(c, c_{0}, d_0, C, C_0)$ such that the following statement holds with probability at least $1 - 2Tn^{-10}$: \\
	If  one of Assumptions~\ref{asm:random-linear} or~\ref{asm:local-linear} holds with constant $C_0$, as well as 
	\begin{align*}
		c_{0} \frac{\sigma^2}{C(\oversamp)} \geq \frac{\perplogn}{\sqrt{n}} \qquad \text{ and } \qquad d \geq d_0,
	\end{align*}
	then AM enjoys $(\rho_{1},\rho_{2},\varepsilon_{1},\varepsilon_{2})$--linear convergence in the sense of Definition~\ref{def:linear-convergence} with
	\[
		c \cdot  \Big( \frac{1+\sigma^{2}}{C(\oversamp)} \Big)^{2} \leq \rho_1 \leq \rho_{2} \leq C \cdot \Big( \frac{1+\sigma^{2}}{C(\oversamp)} \Big)^{2} \quad \text{and} \quad c \cdot   \frac{\sigma^{2}}{C(\oversamp)}  \leq \varepsilon_1 \leq \varepsilon_{2} \leq C \cdot \frac{\sigma^{2}}{C(\oversamp)} 
	\]
	Consequently, starting from an initialization satisfying Assumption~\ref{asm:random-linear} and running the algorithm for 
	\begin{align} \label{eq:lin-consq}
		\tau = \Theta\Bigl(\log_{\frac{\oversamp}{1+\sigma^{2}}}\Bigl( \frac{1}{\sigma} \cdot\frac{\perpcompX_0}{\parcompX_0}\Bigr)\Bigr) \qquad \text{  iterations, we obtain } \qquad 
		\frac{\perpcompX_{\tau}^{2}}{\parcompX_{\tau}^{2}} \lesssim \frac{\sigma^{2}}{C(\oversamp)}.
	\end{align}
\end{theorem}
Our proof shows a finer-grained characterization than the consequence~\eqref{eq:lin-consq} stated in the theorem: After running the AM algorithm for \sloppy\mbox{$\tau =  \Theta\Big(\log_{\frac{\oversamp}{1+\sigma^{2}}}(\sqrt{d}/\sigma)\Big)$} from a random initialization ($\bcoefX_{0}$ is chosen uniformly at random from the unit ball $\mathbb{B}_{2}(1)$) or $\tau = \Theta\Big(\log_{\frac{\oversamp}{1+\sigma^{2}}}(\sigma^{-1})\Big)$ iterations from a local initialization (Assumption~\ref{asm:local-linear}), we obtain $
\frac{\perpcompX_{\tau}^{2}}{\parcompX_{\tau}^{2}} \lesssim \frac{\sigma^{2}}{C(\oversamp)}
$.

A few comments on this theorem are in order, noting that $C(\oversamp) \asymp \oversamp = n/d$ (see Lemma~\ref{lemma:C(lambda)-lambda}) to aid the discussion.
First, note that eventually (i.e. after $\tau$ iterations), the AM iterates converge to a $\sigma^2 d / n$--neighborhood of the ground truth parameters, which is the typical statistical error of the problem. \revision{As a consequence, the total sample complexity for reaching estimation error $\sigma^{2}d/n$ is $\mathcal{O}\big(n \log_{\frac{n}{d(1+\sigma^{2}) }}\big(\sqrt{d}/\sigma\big) \big)$, where we additionally require that $n\gtrsim (1+\sigma^2)d\log(d)$.}

Second, Theorem~\ref{thm:global-conv-linear} implies that AM \emph{adapts} to problem difficulty in terms of the rate at which it converges.  In particular, as the number of samples $n$ increases, the rate parameter $\rho$  decreases proportionally, leading to convergence in a proportionally fewer number of steps.  To provide a particularly striking example of adapting to problem difficulty, suppose that $n \geq d^{1 + \delta}$ and $\sigma = c$ for some universal positive constants $\delta$ and $c$.  Then, $\tau = \Theta(1/\delta)$, and the algorithm converges in a constant number of iterations from a random initialization.  By contrast, if $\oversamp$ is of constant or poly-logarithmic order, the algorithm converges in $\Theta(\log{d})$ iterations.  We illustrate this difference in convergence behavior in Figure~\ref{fig:linear}, where we consider oversampling values $\oversamp \in \{15, 30, 60\}$, noting the monotonic relationship (in $\oversamp$) of the speed of convergence.

Third, we note that Theorem~\ref{thm:global-conv-linear} covers values of the noise level $\sigma$ going all the way down to near-noiseless problems---it allows for instance $\sigma^2 \asymp C(\oversamp) \cdot \perplogn/\sqrt{n} \lesssim \log^{8}(d)d^{-1/4}$ when $\oversamp\leq \sqrt{d}$. Our weak lower bound requirement on $\sigma$ arises purely due to technical reasons: The deterministic one-step updates of Theorem~\ref{thm:one-step} are accurate only up to $\widetilde{\mathcal{O}}(n^{-1/2})$--sized fluctuations, so a sharp rate can only be shown provided the error is larger than this order.  Figure~\ref{fig:linear}(b) demonstrates that our deterministic one-step updates continue to provide high-fidelity predictions even when the amount of noise is very small (i.e. with $\sigma = 10^{-5}$).  Restricting ourselves to upper bounds on the convergence rate, this assumption can be removed.  Indeed, in Section~\ref{sec:proof-conv-linear}, we show that for all $\sigma \geq 0$, if one of Assumption~\ref{asm:random-linear} or~\ref{asm:local-linear} holds, then the following holds for all $0 \leq t\leq T-1$ with probability at least $1-2Tn^{-10}$,
\begin{align}\label{linear-convergence-id-low-noise}
	\frac{\perpcompX_{t+1}^{2}}{\parcompX_{t+1}^{2}} \leq \bigg( \frac{100(1+\sigma^{2})}{C(\oversamp)} \bigg)^{2} \cdot \frac{\perpcompX_{t}^{2}}{\parcompX_{t}^{2}} + \frac{5\sigma^{2}}{C(\oversamp)} + \frac{C\log^{8}(n)}{\sqrt{n}}.
\end{align}

Fourth, our convergence guarantee differs from typical such algorithmic results in the literature, in that we exactly characterize the convergence rate of a nonconvex algorithm from a random initialization, providing sharp upper and lower bounds on both the optimization and statistical errors. Typical guarantees from a random initialization~\cite[cf.][]{chen2019gradient} show a two-stage convergence behavior, and the analysis near a random initialization is typically based only on upper bounds (see, e.g., our own previous work~\cite{chandrasekher2021sharp}).  Most closely related to our setting are the results of~\cite{zhong2015efficient}, who consider a sample-split version of alternating minimization.  They show---in a noiseless setting---that under a local initialization, after $\mathcal{O}\bigl(\log(1/\varepsilon)\bigr)$ iterations, AM achieves error $\varepsilon$.  We improve upon this guarantee along two axes.  First, our results show that the algorithm converges from a random initialization, obviating the need for a two-step procedure.  Second, we sharpen the iteration complexity from a local initialization to $\Theta\Bigl(\log_{\frac{\oversamp}{1+\sigma^{2}}}(1/\varepsilon)\Bigr)$. In more detail, when $\oversamp/(1+\sigma^{2}) = \widetilde{\mathcal{O}}(1)$, our local convergence results show that the guarantees of~\cite{zhong2015efficient} are sharp; moreover, as $\oversamp$ increases, our bounds show that the iteration complexity also decreases, a phenomenon that is not captured by results of~\cite{zhong2015efficient}. Besides this paper, we also mention the recent work of~\cite{chen2021convex} on the blind deconvolution problem, which provides local convergence guarantees of gradient descent on a ridge-regularized version of the loss~\eqref{eq:global-loss}, without a sample splitting assumption.  The work~\cite{chen2021convex} also operates in an interesting setting in which one of the design vectors is randomly sampled from the Fourier basis and the other is Gaussian.  While the results are not explicitly comparable (as the algorithms and losses are different), it would be interesting to study whether similar results to Theorem~\ref{thm:global-conv-linear} hold under Fourier design.

\begin{figure}[!hbtp]
	\centering
	\begin{subfigure}[b]{0.48\textwidth}
		\centering
		\includegraphics[scale=0.48]{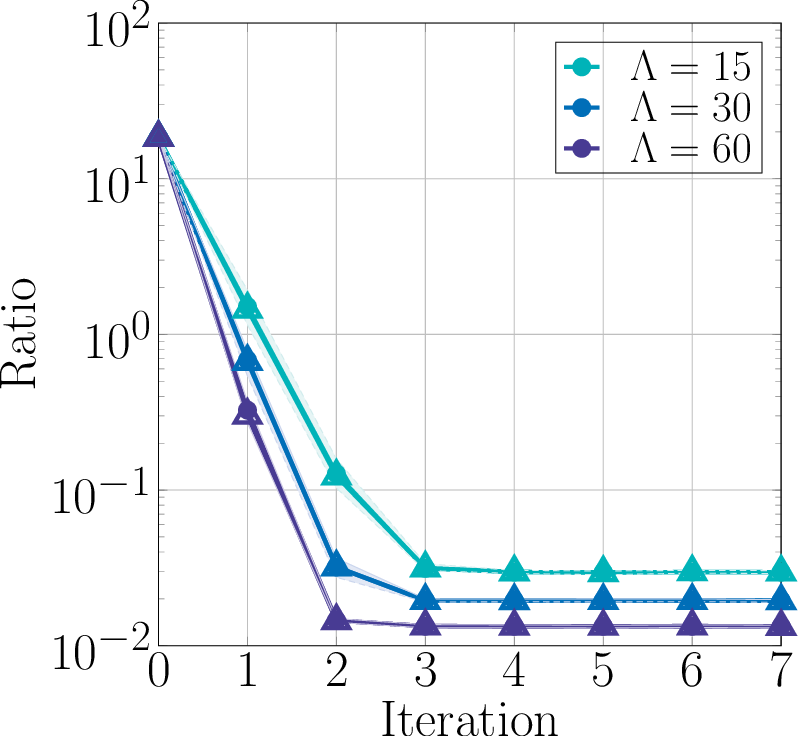} 
		\caption{$\sigma = 0.1, d = 400$.}    
		\label{subfig:linear-a}
	\end{subfigure}
	\hfill
	\begin{subfigure}[b]{0.48\textwidth}  
		\centering 
		\includegraphics[scale=0.48]{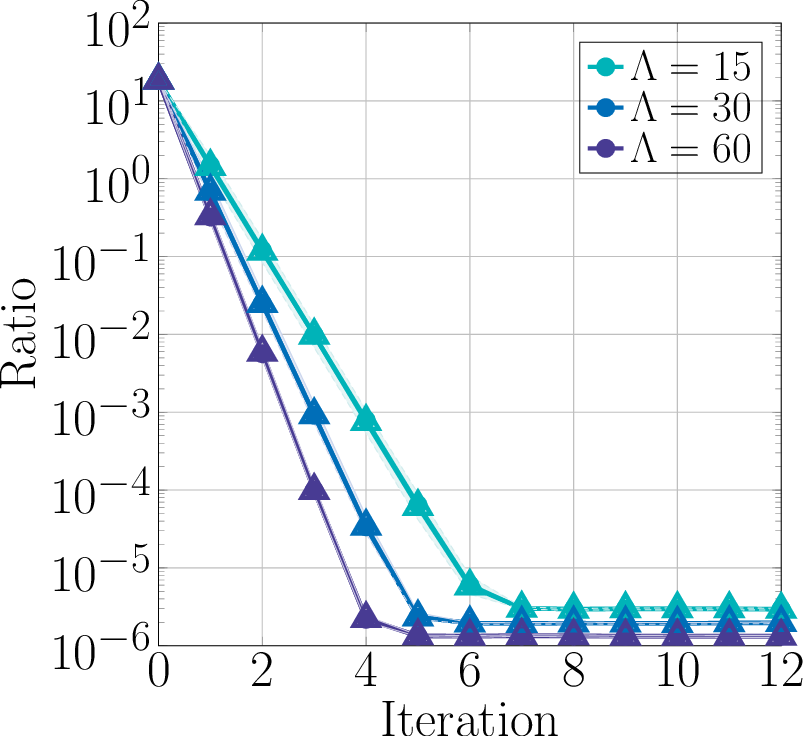} 
		\caption{$\sigma = 10^{-5},d=400$.}
		\label{subfig:linear-b}
	\end{subfigure}
	\caption{High noise (panel (a)) and low noise (panel (b)) behavior of AM for the linear observation model~\eqref{eq:model} with $\psi(w) = w$.  Each experiment consists of $50$ independent trials of AM, started from a random initialization with dimension $d = 400$ and run to convergence.
		Hollow triangular marks denote the median over the $50$ independent trials and filled-in circular marks (barely visible) denote the deterministic one-step predictions.  Shaded envelopes denote the interquartile range over the $50$ independent trials.} 
	\label{fig:linear}
\end{figure}

Finally, it is instructive to consider what guarantees are known (for other algorithms) under deterministic assumptions on the data.  On the one hand, we note that the $\ell_2$--RIP does not hold for our measurement operator~\cite[Claim 4.2]{zhong2015efficient}, whereby we are operating in a setting in which there could exist spurious local minima~\citep{bhojanapalli2016global} in the loss~\eqref{eq:global-loss} even when $\psi = \mathsf{id}$. If instead of minimizing the loss~\eqref{eq:global-loss}, one chooses to minimize a non-smooth, $\ell_1$ variant, then it is known that an $\ell_1$--RIP is satisfied and guarantees sharp growth, whereby---locally---methods such as the prox-linear method and subgradient descent enjoy quadratic (resp. linear) convergence~\citep{charisopoulos2021composite}. As mentioned before, these types of landscape-based results are not directly comparable with the probabilistic analysis that leads to Theorem~\ref{thm:global-conv-linear}; while they are deterministic and can handle a wide range of measurement ensembles, it is typically difficult to establish sharp convergence rates (in the sense of Definition~\ref{def:linear-convergence}) using only these properties.

Let us make a brief comment on the proof technique. We begin by applying the one-step updates from Theorem~\ref{thm:one-step} to reduce the complexity from studying a high-dimensional, random iteration to studying a two-dimensional, deterministic recursion.  
We note that while the deterministic recursion is---in general---straightforward to analyze, the fluctuations around the deterministic one-step predictions complicate matters.  This is especially so from a random initialization in which case the fluctuations can be at nearly the same scale as the predictions.  Nonetheless, by carefully accounting for the growth of the parallel component and perpendicular component at each iteration, we show that the convergence properties suggested by the deterministic one-step updates~\eqref{eq:idmaps} do indeed hold.

\subsection{One-bit observation model}
We turn now to our global convergence guarantees with one-bit measurements.

\begin{theorem}\label{thm:global-conv-nonlinear}
	Suppose Assumption~\ref{assptn:distribution} holds. Consider the nonlinear observation model~\eqref{eq:model} with $\psi(x) = \sign(x)$ and the AM algorithm~\eqref{eq:AM} run for $T$ iterations. There exists a tuple of universal, positive constants $(c,d_0,C_0,C)$ such that the following statement holds with probability at least $1 - 2Tn^{-10}$: \\
	If one of Assumptions~\ref{asm:random-linear} or~\ref{asm:local-linear} holds with constant $C_0$ and $d\geq d_{0}$, then AM enjoys $(\rho_1,\rho_2,\varepsilon_1,\varepsilon_2)$--linear convergence in the sense of Definition~\ref{def:linear-convergence} with
	\[
		c \cdot  \Big( \frac{1+\sigma^{2}}{C(\oversamp)} \Big)^{2} \leq \rho_1 \leq \rho_{2} \leq C \cdot \Big( \frac{1+\sigma^{2}}{C(\oversamp)} \Big)^{2} \quad \text{and} \quad c \cdot   \frac{1+\sigma^{2}}{C(\oversamp)}  \leq \varepsilon_1 \leq \varepsilon_{2} \leq C \cdot \frac{1+\sigma^{2}}{C(\oversamp)} 
	\]
	Consequently, starting at an initialization satisfying Assumption~\ref{asm:random-linear} and running the algorithm for 
	\begin{align} \label{eq:nonlin-consq}
		\tau = \Theta\Bigl(\log_{\frac{\oversamp}{1+\sigma^{2}}}\Bigl(\frac{\perpcompX_{0}}{\parcompX_{0}}\Bigr)\Bigr) \qquad \text{ iterations, we obtain } \qquad
		\frac{\perpcompX_{\tau}^{2}}{\parcompX_{\tau}^{2}} \lesssim \frac{1 + \sigma^{2}}{C(\oversamp)}.
	\end{align}
\end{theorem}
As in Theorem~\ref{thm:global-conv-linear}, the consequence~\eqref{eq:nonlin-consq} can be stated in more detail as follows: After running the AM algorithm for $\tau =  \Theta\Big(\log_{\frac{\oversamp}{1+\sigma^{2}}}(d)\Big)$ from a random initialization ($\bcoefX_{0}$ is chosen uniformly at random from the unit ball $\mathbb{B}_{2}(1)$) or $\tau = \Theta(1)$ iterations from a local initialization (Assumption~\ref{asm:local-linear}), we reach the error floor. \revision{As a consequence, the total sample complexity for reaching estimation error $(1+\sigma^2)d/n$ is $\mathcal{O}\big(n \log_{\frac{n}{d(1+\sigma^2)}}(d) \big)$, where we additionally require that $n\gtrsim (1+\sigma^2)d\log(d)$.}

\begin{figure}[ht!]
	\centering
	\begin{subfigure}[b]{0.48\textwidth}
		\centering
		\includegraphics[scale=0.48]{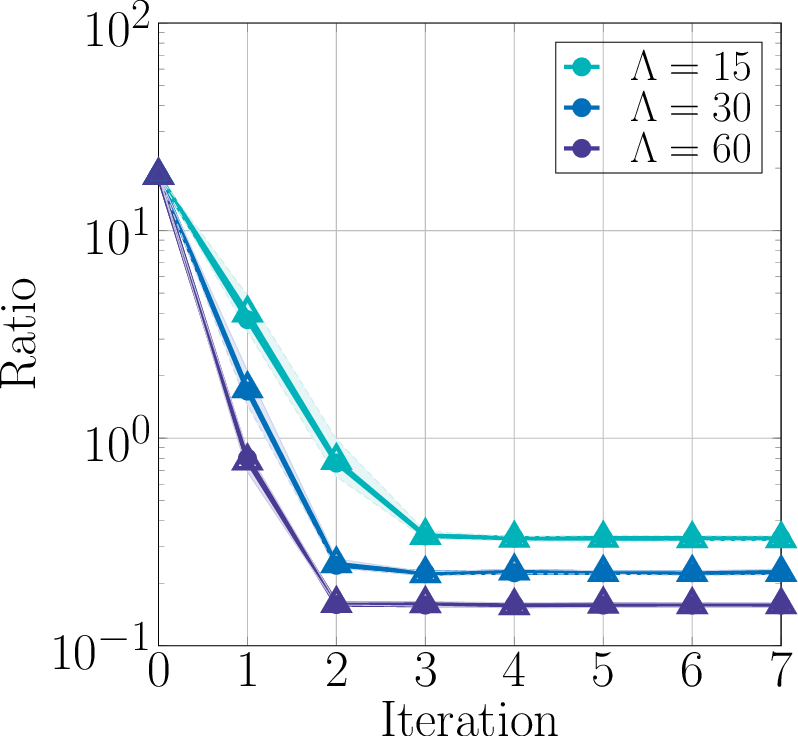} 
		\caption{$\sigma = 0.1, d = 400$.}    
		\label{subfig:bit-a}
	\end{subfigure}
	\hfill
	\begin{subfigure}[b]{0.48\textwidth}  
		\centering 
		\includegraphics[scale=0.48]{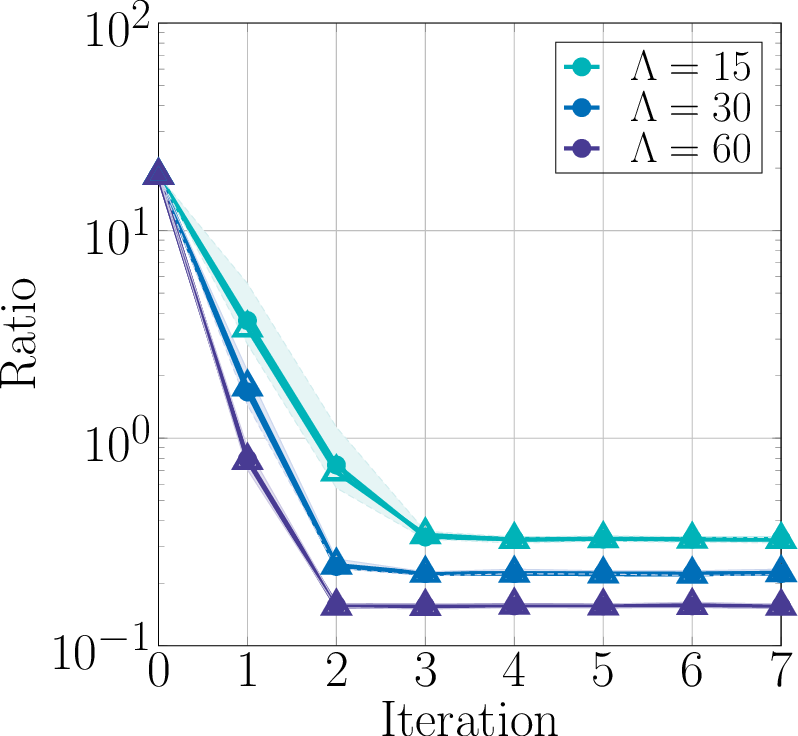}
		\caption{$\sigma = 10^{-5},d=400$.}
		\label{subfig:bit-b}
	\end{subfigure}
	\caption{High noise (panel (a)) and low noise (panel (b)) behavior of AM for the linear observation model~\eqref{eq:model} with $\psi(w) = \sign(w)$.  Each experiment consists of $50$ independent trials of AM, started from a random initialization with dimension $d = 400$ and run to convergence.
		Hollow triangular marks denote the median over the $50$ independent trials and filled-in circular marks (barely visible) denote the deterministic one-step predictions.  Shaded envelopes denote the interquartile range over the $50$ independent trials.} 
	\label{fig:bit}
\end{figure}
Some comments on specific aspects of the theorem are in order.  First, note that the AM algorithm~\eqref{eq:AM} is tailored to minimize the negative log-likelihood of the linear observation model~\eqref{eq:model} with $\psi(w) = w$.  Nonetheless, when the noise level $\sigma$ is constant, Theorem~\ref{thm:global-conv-nonlinear} demonstrates that there is no loss in running the misspecified AM algorithm as the eventual error floor is on the order $\sigma^2 d/n$, which is the typical statistical rate of the problem. This general phenomenon has been observed before in the literature on misspecified linear models; for instance,~\cite{plan2016generalized} demonstrate that the misspecified Lasso with one-bit measurements attains the same error $\sigma^2 d/n$.  Our work is complementary, providing two additional insights.  First, and in contrast to the one-shot least-squares problem, our objective remains nonconvex even upon ignoring the nonlinearity. In spite of this, we show that AM converges to a similar neighborhood of the ground truth, but must now be run for logarithmically many iterations from a random initialization instead of just for one iteration.  In the low-noise $\sigma \downarrow 0$ regime, the eventual error attained by the algorithm is likely statistically suboptimal. See~\citet{pananjady2021single} for instances of nonlinearities in single-index models where the low-noise nature of the problem can be exploited to reduce the statistical error.

Second, our proof comes with a lower bound, demonstrating that even in low noise problems, the statistical error remains bottlenecked at the rate $d/n$.  In Figure~\ref{fig:bit}, we demonstrate the convergence behavior in both high noise (Figure~\ref{subfig:bit-a}) and low noise (Figure~\ref{subfig:bit-b}) problems. Unlike the linear case, the eventual error floor is of the same order in both these examples.

Next, under the local initialization Assumption~\ref{asm:local-linear}, we note that the AM iteration converges in exactly one step, regardless of the noise level.  This is because under Assumption~\ref{asm:local-linear}, the ratio $\perpcompX_0^2/\parcompX_0^2 \lesssim C(\oversamp)/(1 + \sigma^2)$ and the algorithm converges linearly with rate $\rho = (1+\sigma^2)^2/C(\oversamp)^2$.  By contrast, when run on the linear model in the low noise regime, the AM algorithm will take several steps to converge.  We illustrate this difference---along with the general mode of convergence---through a schematic diagram in Figure~\ref{fig:conv-schematic}. See the caption of this figure for a detailed discussion of this distinction.
\begin{figure}[!ht]
	\centering
	\begin{subfigure}[b]{0.3\textwidth}
		\centering
		\includegraphics[scale=0.55]{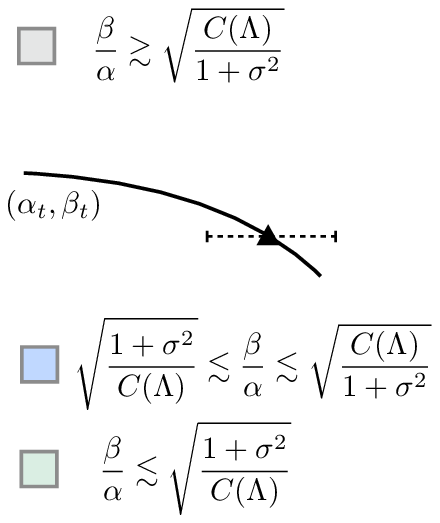} 
	\end{subfigure}
	\hfill
	\begin{subfigure}[b]{0.34\textwidth}  
		\centering 
		\includegraphics[scale=0.69]{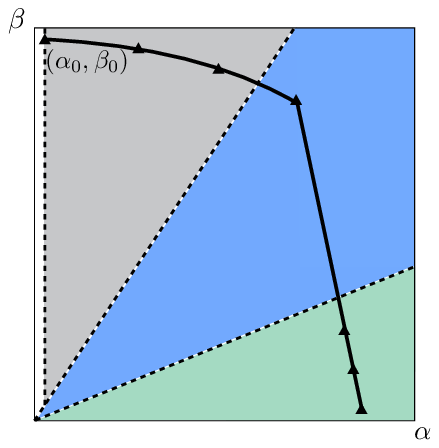}
		\caption{Linear model.}
	\end{subfigure}
	\hfill
	\begin{subfigure}[b]{0.34\textwidth}  
		\centering 
		\includegraphics[scale=0.69]{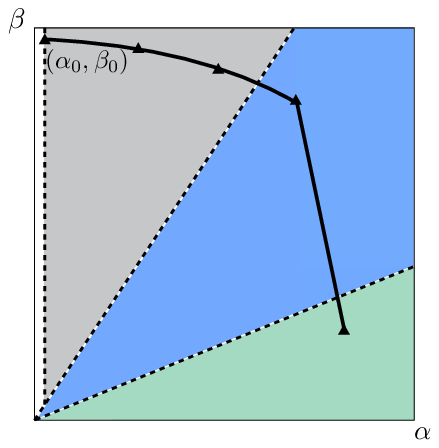}
		\caption{Nonlinear model.}
	\end{subfigure}
	\caption{\textbf{Schematic of Theorems~\ref{thm:global-conv-linear} and~\ref{thm:global-conv-nonlinear}.} The first stage (shaded gray) is the random initialization phase (under Assumption~\ref{asm:random-linear}).  After $\log_{\oversamp}(d)$ iterations, the algorithm reaches the shaded blue region, indicating the local phase (Assumption~\ref{asm:local-linear}).  Panel (b) shows convergence in the nonlinear model, where one step propels the iterates into the shaded green region, and the iterates reach the final error floor.  By contrast, panel (a) demonstrates convergence in the linear model, where one step propels the iterates into the shaded green region, but the algorithm continues to converge linearly.}
	\label{fig:conv-schematic}
\end{figure}

Finally, we note that while we consider the specific one-bit observation model, our technique is broadly applicable to other nonlinearities.  In particular, the one step guarantees provided by Theorem~\ref{thm:one-step} hold under only Assumption~\ref{assptn:Y-psi}.  Moreover, the proof of Theorem~\ref{thm:global-conv-nonlinear} relies on regularity properties of the maps~\eqref{eq:bitmaps}, which can be straightforwardly established for other choices of the nonlinearity $\psi$.

\section{Proof of Theorem~\ref{thm:one-step}: One-step predictions} \label{sec:proof-one-step}
This section is dedicated to the proof of the one-step updates, and is organized as follows.  We begin by outlining the proof strategy and establishing some preliminary notions.  Then we use these to prove Theorem~\ref{thm:one-step}(a) in Section~\ref{sec:proof-parallel} and Theorem~\ref{thm:one-step}(b) in Section~\ref{sec:proof-orthogonal}.  

Now, consider one step of the AM algorithm and recall that the iterates admit the closed form characterization~\eqref{eq:equiv-step}
\begin{align*}
	\bcoefX_{t+1} = \bigl(\bX^{\top} \bW_{t+1}^2 \bX \bigr)^{-1} \bX^{\top} \bW_{t+1} \by, \quad \text{ where } \quad \bW_{t+1} = \diag(\bZ \bcoefZ_{t+1}).
\end{align*}
Since the Gaussian distribution is rotationally invariant, we assume without loss of generality that $\bcoefX_{\star} = \boldsymbol{e}_{1}$, whence $\parcompX_{t+1} = \coefX_{t+1, 1}$ and $\perpcompX_{t+1} = \| \bcoefX_{t+1, \backslash 1} \|_2$. Here we use the notation $\bcoefX_{t+1} = [\coefX_{t+1, 1} \; \mid \; \bcoefX_{t+1, \backslash 1}]$.  It thus suffices to understand $\bcoefX_{t+1}$ on a coordinate-by-coordinate basis.  Towards this goal, we consider leaving the $k$-th column out, for $k \in \{1, 2, \dots, d\}$ and introduce the notation $\Xcol{k} \in \mathbb{R}^n$ to denote the $k$-th column of the data matrix $\bX$ and $\bX_{\backslash k} \in \mathbb{R}^{n \times (d - 1)}$ to denote the data matrix without the $k$-th column. We use $\bx_{j\backslash k} \in \mathbb{R}^{d-1}$ to denote the $j$-th column of $\bX_{\backslash k}^{\top} \in \mathbb{R}^{(d-1)\times n}$. We then define the projection matrices $\bP_{\backslash k}$ and $\bS_{\backslash k}$ as
\begin{align}
	\label{eq:def-loo-proj}
	\bP_{\backslash k} = \bW_{t+1} \bX_{\backslash k} \cdot \bigl( \bX_{\backslash k}^{\top} \bW_{t+1}^2 \bX_{\backslash k} \bigr)^{-1} \cdot \bX_{\backslash k}^{\top} \bW_{t+1} \quad \text{ and } \quad \bS_{\backslash k} = \bI - \bP_{\backslash k},
\end{align}
where we note that $\bP_{\backslash k}$ denotes the orthogonal projector onto the subspace spanned by the collection of vectors $\{\bW_{t+1} \Xcol{j}\}_{j \neq k}$ and $\bS_{\backslash k}$ the projector onto the orthogonal subspace.  With this notation in hand, we claim the following per-coordinate characterization of the iterate $\bcoefX_{t+1}$:
\begin{align} \label{eq:k-coord}
	\coefX_{t+1, k} = \frac{\langle \Xcol{k}, \bW_{t+1} \bS_{\backslash k} \by \rangle}{\langle \Xcol{k}, \bW_{t+1} \bS_{\backslash k} \bW_{t+1} \Xcol{k} \rangle}.
\end{align}
Claim~\eqref{eq:k-coord} is proved in Section~\ref{app:per-coord}.
We now turn to the proof of Theorem~\ref{thm:one-step}(a).

\subsection{Parallel component: Proof of Theorem~\ref{thm:one-step}(a)} \label{sec:proof-parallel}
We begin by stating two lemmas.  The first demonstrates that the denominator concentrates around a fixed quantity (which is zero for $k \neq 1$), with fluctuations on the order $\widetilde{\order}(n^{-1/2})$.  We provide the proof of Lemma~\ref{lem:k-coord-denom} in Section~\ref{sec:proof-lem-k-coord-denom}.
\begin{lemma} \label{lem:k-coord-denom}
	Let $X,G,V$ and $Y$ be random variables defined in equation~\eqref{definition-X-G-V-Y} and let $\tau = 1/C(\oversamp)$. Under the assumptions of Theorem~\ref{thm:one-step}, there exists a positive constant $C$ such that for each $1 \leq k \leq d$, with probability at least $1 - n^{-25}$, we have
	\begin{align*}
		\frac{1}{n}\Bigl \lvert  \langle \Xcol{k}, \bW_{t+1} \bS_{\backslash k} \bW_{t+1} \Xcol{k} \rangle - n(\parcompZ_{t+1}^2 + \perpcompZ_{t+1}^2) \cdot \EE\Bigl\{ \frac{G^{2}}{1+\tau G^{2}} \Bigr\}  \Bigr \rvert \leq C (\parcompZ_{t+1}^2 + \perpcompZ_{t+1}^2)\sqrt{\frac{\log{n}}{n}}.
	\end{align*}
\end{lemma}
While Lemma~\ref{lem:k-coord-denom} controls the denominator on the RHS of Eq.~\eqref{eq:k-coord} for all $k \in [d]$, the following lemma demonstrates that the numerator in Eq.~\eqref{eq:k-coord} also concentrates around a deterministic quantity with fluctuations on the order $\widetilde{O}(n^{-1/2})$.  We provide the proof of Lemma~\ref{lem:k-coord-num} in Section~\ref{sec:proof-lem-k-coord-num}.
\begin{lemma} \label{lem:k-coord-num}
	Let the random variables $X,G,V$ and $Y$ be as in equation~\eqref{definition-X-G-V-Y} and let $\tau = 1/C(\oversamp)$. Under the assumptions of Theorem~\ref{thm:one-step}, there exist a pair of universal, positive constants $(C_1, C_2)$ which depend only on $C_{\psi}$ and a universal, positive constant $C'$ such that the following hold for all $n \geq C'$.
	\begin{itemize}
		\item[(a)] With probability at least $1 - n^{-25}$,
		\[
		\frac{1}{n} \Bigl \lvert \langle \Xcol{1}, \bW_{t+1} \bS_{\backslash 1} \by \rangle- n \big(\parcompZ_{t+1}^{2} + \perpcompZ_{t+1}^{2}\big)^{1/2} \cdot \EE\Bigl\{ \frac{G X Y}{1+\tau G^{2}} \Bigr\}  \Bigr \rvert \leq  C_1(\parsigma)\sqrt{\frac{(\parcompZ_{t+1}^2 + \perpcompZ_{t+1}^2)\cdot \log{n}}{n}} .
		\]
		\item[(b)] If $k \in \{2, 3, \dots, d\}$, then with probability at least $1 - n^{-25}$,
		\[	\frac{1}{n} \bigl \lvert \langle \Xcol{k}, \bW_{t+1} \bS_{\backslash k} \by \rangle \bigr \rvert \leq C_2(\parsigma)\sqrt{\frac{(\parcompZ_{t+1}^2 + \perpcompZ_{t+1}^2)}{n}} \cdot \log{n}.
		\]
	\end{itemize}
\end{lemma}
Taking these lemmas as given, we turn to the proof of part (a) of the theorem.  Recall that by rotational invariance of the Gaussian distribution, we may assume that $\parcompX_{t+1} = \coefX_{t+1, 1}$.  Consequently, by definition of the deterministic update~\eqref{eq-equivalence}, we have
\begin{align*}
	\bigl \lvert \coefX_{t+1, 1} - \parcompdetX_{t+1} \bigr \rvert &= \biggl \lvert \frac{ \frac{1}{n}\langle \Xcol{1}, \bW_{t+1} \bS_{\backslash 1} \by \rangle}{ \frac{1}{n}\langle \Xcol{1}, \bW_{t+1} \bS_{\backslash 1} \bW_{t+1} \Xcol{1} \rangle } - \frac{\sqrt{\parcompZ_{t+1}^{2} + \perpcompZ_{t+1}^{2}}}{\parcompZ_{t+1}^{2} + \perpcompZ_{t+1}^{2}} \cdot \frac{\EE\bigl\{ \frac{GX  Y}{1+\tau G^{2}} \bigr\}}{ \EE\bigl\{ \frac{G^{2}}{1+\tau G^{2}} \bigr\} }  \biggr \rvert \\
	&\leq \frac{1}{ \frac{1}{n}\langle \Xcol{1}, \bW_{t+1} \bS_{\backslash 1} \bW_{t+1} \Xcol{1} \rangle }\cdot \bigl[A + \parcompdetX_{t+1} B \bigr],
\end{align*}
where
\begin{align*}
	A &= \Bigl \lvert \frac{1}{n} \langle \Xcol{1}, \bW_{t+1} \bS_{\backslash 1} \by \rangle - \sqrt{\parcompZ_{t+1}^{2} + \perpcompZ_{t+1}^{2}} \cdot \EE\Bigl\{ \frac{GX Y}{1+\tau G^{2}} \Bigr\} \Bigr \rvert \qquad \text{ and } \\
	B &=  \Bigl \lvert \frac{1}{n} \langle \Xcol{1}, \bW_{t+1} \bS_{\backslash 1} \bW_{t+1} \Xcol{1} \rangle - (\parcompZ_{t+1}^{2} + \perpcompZ_{t+1}^{2}) \cdot  \EE\Bigl\{ \frac{G^{2}}{1+\tau G^{2}} \Bigr\}  \Bigr \rvert.
\end{align*}
Applying Lemma~\ref{lem:k-coord-num}(a) to bound term $A$ and Lemma~\ref{lem:k-coord-denom} to bound term $B$ yields the pair of inequalities 
\[
A \leq C_1(\parsigma)\sqrt{\frac{(\parcompZ_{t+1}^2 + \perpcompZ_{t+1}^2)\log{n}}{n}}  \text{ and } B \leq C_{1}(\parcompZ_{t+1}^2 + \perpcompZ_{t+1}^2)\sqrt{\frac{\log{n}}{n}}, \text{ with probability } \geq 1 - n^{-25}.
\]
Applying Lemma~\ref{lem:k-coord-denom} once more yields the lower bound 
\[
\langle \Xcol{1}, \bW_{t+1} \bS_{\backslash 1} \bW_{t+1} \Xcol{1} \rangle  \geq c (\parcompZ_{t+1}^2 + \perpcompZ_{t+1}^2) n, \quad \text{ with probability } 1 - n^{-25},
\]
for a positive constant $c$. Finally, we use the estimate $\parcompdetX_{t+1} \lesssim (\parcompZ_{t+1}^2 + \perpcompZ_{t+1}^2)^{-1/2}$ (see Lemma~\ref{lemma:upper-bound-para-det}) and assemble the pieces to complete the proof. \qed

It remains to prove Lemmas~\ref{lem:k-coord-denom} and~\ref{lem:k-coord-num}.  For the proofs of both, it is useful to introduce the change of variables 
\begin{align} \label{eq:change-of-var-W}
	\bW_{t+1} = \sqrt{\parcompZ_{t+1}^2 + \perpcompZ_{t+1}^2} \bG, \quad \text{ where } \quad \bG = \diag(G_1, G_2, \dots, G_n),
\end{align}   
and $(G_i)_{1 \leq i \leq n} \overset{\mathsf{i.i.d.}}{\sim} \mathsf{N}(0, 1)$.  

\begin{remark} \label{rem:weaken-tails}
	\revision{
			As alluded to in Remark~\ref{rem:poly-growth}, we note that the proof can be extended to accommodate polynomially growing link function $\psi$, i.e., $|\psi(x)| \leq C_{\psi} |x|^{D}$ for all $x\in \mathbb{R}$.
			In particular, by employing a general Hanson--Wright inequality~\citep[Corollary 1.4]{gotze2021concentration}, one can show that with probability at least $1-n^{-25}$, 
			\begin{align*}
					\frac{1}{n} \Bigl \lvert \langle \Xcol{1}, \bW_{t+1} \bS_{\backslash 1} \by \rangle- n \big(\parcompZ_{t+1}^{2} + \perpcompZ_{t+1}^{2}\big)^{1/2} \cdot \EE\Bigl\{ \frac{G X Y}{1+\tau G^{2}} \Bigr\}  \Bigr \rvert \leq  C_1(\parsigma)\sqrt{\frac{(\parcompZ_{t+1}^2 + \perpcompZ_{t+1}^2)}{n}} \cdot \log^{4D}(n),
			\end{align*}
			where $C_{1}$ is a constant depends only on $C_{\psi}$ and $D$. Moreover, Lemma~\ref{lem:k-coord-denom} and Lemma~\ref{lem:k-coord-num}(b) still hold. Then by following the same argument above, one can prove inequality~\eqref{eq:par-deviation-general-psi}.
	}
\end{remark}

\subsubsection{Proof of Lemma~\ref{lem:k-coord-denom}} \label{sec:proof-lem-k-coord-denom}
Recalling the change of variables~\eqref{eq:change-of-var-W}, we first re-scale, writing
\[
\langle \Xcol{k}, \bW_{t+1} \bSk{k} \bW_{t+1} \Xcol{k}\rangle = (\parcompZ_{t+1}^2 + \perpcompZ_{t+1}^2) \cdot \langle \Xcol{k}, \bG \bSk{k} \bG \Xcol{k} \rangle.
\]
Also recall that $\tau = C(\Lambda)^{-1}$.
We next decompose
\begin{align} \label{ineq:decomp-k-coord}
	\Bigl \lvert \frac{1}{n}\bigl \langle \Xcol{k}, \bG \bS_{\backslash k} \bG \Xcol{k} \bigr \rangle - \EE \Bigl\{ \frac{G^2}{1 + \tau G^2} \Bigr\} \Bigr \rvert \leq T_1 + T_2,
\end{align}
where 
\begin{align*}
	T_1 &= \Bigl \lvert \frac{1}{n}\bigl \langle \Xcol{k}, \bG \bS_{\backslash k} \bG \Xcol{k} \bigr \rangle - \frac{1}{n} \trace(\bG \bS_{\backslash k} \bG) \Bigr \rvert, \quad \text{ and } \quad T_2 &=  \Bigl \lvert \frac{1}{n} \trace(\bG \bS_{\backslash k} \bG) -  \EE \Bigl\{ \frac{G^2}{1 + \tau G^2} \Bigr\} \Bigr \rvert.
\end{align*}
We claim that
\begin{align} \label{ineq:T1T2-bound-kcoord-denom}
	T_1 \vee T_2 \leq C\sqrt{\log{n}/n}, \qquad \text{ with probability } \geq 1 - n^{-25},
\end{align}
which we prove momentarily.
Substituting this bound into the RHS of the inequality~\eqref{ineq:decomp-k-coord} and applying the rescaling~\eqref{eq:change-of-var-W} yields the desired result.  It remains to bound the terms $T_1$ and $T_2$.  

\paragraph{Bounding $T_1$~\eqref{ineq:T1T2-bound-kcoord-denom}.}
We begin by claiming the following pair of inequalities: 
\begin{align}\label{ineq:GSG-F-GSG-op}
	\| \bG \bS_{\backslash k} \bG\|_F^2 \leq C n, \quad \text{ and } \quad \| \bG \bS_{\backslash k} \bG\|_{\op} \leq C \log{n}, \quad \text{ with probability } \geq 1 - n^{-25}.
\end{align}
Applying these in conjunction with the Hanson--Wright inequality, we obtain
\[
\Prob \Bigl\{ \frac{1}{n} \bigl \lvert \bigl \langle \bX^{(k)}, \bG \bS_{\backslash k} \bG\bX^{(k)} \bigr \rangle - \frac{1}{n}\trace(\bG\bS_{\backslash k} \bG) \bigr \rvert \geq t \Bigr\} \leq 2 \exp\Bigl\{ -c \min \Bigl(nt^2, \frac{nt}{\log{n}}\Bigr)\Bigr\} + \frac{1}{n^{25}},
\]
It remains to prove the inequality~\eqref{ineq:GSG-F-GSG-op}.

\bigskip
\noindent\underline{Bounding $\| \bG \bS_{\backslash k} \bG\|_F^2$~\eqref{ineq:GSG-F-GSG-op}.}
Note that
\begin{align*}
	\| \bG \bS_{\backslash k} \bG\|_F^2 \overset{\1}{=} \bigl\| (\bS_{\backslash k} \bG)^{\top} (\bS_{\backslash k} \bG )\bigr\|_{F}^2 &= \mathsf{tr}\Bigl( \bigl[(\bS_{\backslash k} \bG)^{\top} (\bS_{\backslash k} \bG)\bigr]^2\Bigr) = \sum_{i=1}^{n} \sigma_{i}^4 ( \bS_{\backslash k} \bG ),
\end{align*}
where in step $\1$ we have used idempotence of the projection matrix $\bS_{\backslash k}$.  
Applying the Courant--Fischer theorem in conjunction with the fact that $\sigma_{\max}(\bS_{\backslash k}) = 1$ yields the bound
\[
\sum_{i=1}^{n} \sigma_{i}^4 ( \bS_{\backslash k} \bG ) \leq \sum_{i=1}^{n} \sigma_{\max}^4(\bS_{\backslash k}) \cdot \sigma_{i}^4 (\bG ) \leq \sum_{i=1}^{n} G_i^4,
\]
Now, by assumption, $n \geq C'$, and $G_i^4$ has bounded moments of a constant order. Applying~\citet[Ex. 2.20]{wainwright2019high}, we obtain that $\sum_{i=1}^{n} G_i^4 \leq Cn$ with probability at least $1 - n^{-25}$.  Combining the pieces then yields
\begin{align*}
	\| \bG \bS_{\backslash k} \bG\|_F^2 \leq C n \qquad \text{ with probability } \geq 1 - n^{-25}.
\end{align*}

\noindent\underline{Bounding $\| \bG \bS_{\backslash k} \bG\|_{\op}$~\eqref{ineq:GSG-F-GSG-op}.}  We have 
\[
\| \bG \bS_{\backslash k} \bG\|_{\op} \leq \| \bG \|_{\op}^2 = \bigl(\max\bigl\{\lvert G_1 \rvert, \dots , \lvert G_n \rvert \bigr\}\bigr)^2 \leq C\log{n},
\]
where the final inequality holds with probability $\geq 1 - n^{-25}$.

\paragraph{Bounding $T_2$~\eqref{ineq:T1T2-bound-kcoord-denom}.}
We expand 
\[
\trace(\bG \bS_{\backslash k} \bG) = \sum_{i=1}^{n} G_i^2 - G_i^2  \cdot \bP_{\backslash k} (i, i),
\]
where we recall the projection matrix $\bP_{\backslash k}$~\eqref{eq:def-loo-proj} and use the notation $\bP_{\backslash k}(i, i)$ to denote the entry $i$-th diagonal entry of $\bP_{\backslash k}$. Recall that $\bx_{j\backslash k}$ is the $j$-th column of $\bX_{\backslash k}^{\top}$. In order to compactly represent these entries, we introduce the notation
\[
\bSig = \frac{1}{n}\bX_{\backslash k}^{\top} \bG^2 \bX_{\backslash k}, \quad \bSig_i = \frac{1}{n}\sum_{j \neq i} G_j^2 \bx_{j \backslash k} \bx_{j \backslash k}^{\top}, \quad \text{ and } \quad \tau_i = \frac{1}{n}\langle \bx_{i \backslash k}, \bSig_{i}^{-1} \bx_{i \backslash k}\rangle, 
\]
whence straightforward computation yields
\begin{align} \label{eq:P-entries}
	\bP_{\backslash k}(i, i) =\frac{1}{n}G_i^2 \langle \bx_{i \backslash k}, \bSig^{-1} \bx_{i \backslash k} \rangle \overset{\1}{=} \frac{G_i^2 \tau_i}{1 + G_i^2 \tau_i}.
\end{align}
Here in step $\1$ we have used the Sherman--Morrison formula so that
\[
	\bSig^{-1} = \Big(\bSig_{i} + \frac{1}{n} G_{i}^{2} \bx_{i\backslash k} \bx_{i \backslash k}^{\top}\Big)^{-1} = \bSig_{i}^{-1} - \frac{1}{n} G_{i}^{2} \cdot\frac{\bSig_{i}^{-1}  \bx_{i\backslash k} \bx_{i \backslash k}^{\top} \bSig_{i}^{-1} }{1+G_{i}^{2} \tau_{i}}.
\] 
Taking stock, we obtain the identity
\[
\trace(\bG \bS_{\backslash k} \bG) = \sum_{i=1}^{n}\frac{G_i^2}{1 +G_i^2 \tau_i},
\]
and we bound $T_2$ in turn by 
\begin{align}\label{ineq:T2-decomp}
	T_2 \leq \underbrace{\Bigl \lvert \frac{1}{n}\sum_{i=1}^{n}\frac{G_i^2}{1 + G_i^2 \tau_i} - \frac{1}{n}\sum_{i=1}^{n}\frac{G_i^2}{1 + G_i^2 \tau} \Bigr \rvert}_{A} + \underbrace{\Bigl \lvert \frac{1}{n}\sum_{i=1}^{n}\frac{G_i^2}{1 +G_i^2 \tau} - \EE \Bigl\{ \frac{G^2}{1 + \tau G^2}\Bigr\} \Bigr \rvert}_{B}.
\end{align}

\bigskip
\noindent\underline{Bounding term $A$~\eqref{ineq:T2-decomp}.}
Re-arranging yields
\[
A = \frac{1}{n}\Bigl \lvert \sum_{i=1}^{n} \frac{G_i^4 \cdot (\tau- \tau_i)}{(1 + G_i^2 \tau_i)(1 + G_i^2 \tau)} \Bigr \rvert \leq \frac{1}{n} \sum_{i=1}^{n}G_i^4\cdot \lvert \tau_i - \tau \lvert \leq \frac{1}{n} \max_{i \in [n]} \bigl\{\lvert \tau_i - \tau \rvert\bigr\} \cdot \sum_{i=1}^{n} G_i^4.
\]
We first handle the maximum deviation of $\tau_i$ from $\tau$. Applying the triangle inequality yields
	\begin{align}\label{ineq1-max-dev-tau}
		\big \lvert \bx_{i,\setminus k}^{\top} \bSig_{i}^{-1} \bx_{i, \setminus k} - \tau \big \rvert \leq 
		\big \lvert \bx_{i,\setminus k}^{\top} \bSig_{i}^{-1} \bx_{i, \setminus k} - \trace(\bSig_{i}^{-1}) \big \rvert + 
		\big \lvert \trace(\bSig_{i}^{-1}) - \tau \big \rvert.
	\end{align}
To bound the first term in the above decomposition, we apply Lemma~\ref{lem:general-trace-concentration} and obtain that there exists a universal, positive constant $C$ such that 
	\begin{align}\label{ineq2-max-dev-tau}
		\Pr \Big\{ \lvert \bx_{i,\setminus k}^{\top} \bSig_{i}^{-1} \bx_{i, \setminus k} - \trace(\bSig_{i}^{-1}) \big \rvert \geq C \sqrt{\log(n)/n}  \Big\} \leq n^{-30}.
	\end{align}
To bound the second term, note that $\tau = C(\oversamp)^{-1}$ by definition. Consequently, by applying Lemma~\ref{lem:tau-concentration}, we obtain that there exists another universal, positive constant $C$ such that 
	\begin{align}\label{ineq3-max-dev-tau}
		\Pr \Big\{ \lvert \trace(\bSig_{i}^{-1}) - \tau \big \rvert \geq C \sqrt{\log(n)/n}  \Big\} \leq n^{-30}.
	\end{align}
Substituting inequality~\eqref{ineq2-max-dev-tau} and inequality~\eqref{ineq3-max-dev-tau} into inequality~\eqref{ineq1-max-dev-tau} yields
	\[
		\big \lvert \bx_{i,\setminus k}^{\top} \bSig_{i}^{-1} \bx_{i, \setminus k} - \tau \big \rvert \leq C\sqrt{\frac{\log(n)}{n}}, \quad \text{with probability} \geq 1 - 2n^{-30}.
	\] 
Applying the union bound then yields
\begin{align} \label{eq:tau-stuff}
	\max_{i \in [n]} \; \lvert \tau_i - \tau \rvert = \max_{i \in [n]} \; \bigl\lvert \bx_{i \backslash k}^{\top} \bSig_{i}^{-1} \bx_{i\backslash k} - \tau \bigr\rvert \leq C\sqrt{\frac{\log{n}}{n}}, \qquad \text{ with probability } \geq 1 - n^{-25}.
\end{align}
Once again applying~\citet[Ex. 2.20]{wainwright2019high} yields $\sum_{i=1}^{n} G_i^4 \leq Cn$ with probability at least $1 - n^{-25}$, whence we obtain the bound
\begin{align} \label{eq:A-bound}
	A \leq C\sqrt{\frac{\log{n}}{n}}, \qquad \text{ with probability } \geq 1 - n^{-25}.
\end{align}

\bigskip
\noindent\underline{Bounding $B$~\eqref{ineq:T2-decomp}.}
Note that
\[
\Bigl \| \frac{G^2}{1 + G^2 \tau} \Bigr \|_{\psi_1} \leq \| G^2 \|_{\psi_1} \leq C.
\]
Consequently, an application of Bernstein's inequality yields the bound
\begin{align} \label{eq:B-bound}
	B\leq C \sqrt{\frac{\log{n}}{n}} \qquad \text{ with probability } \geq 1 - n^{-25}.
\end{align}

\medskip

Combining bounds~\eqref{eq:A-bound} and~\eqref{eq:B-bound} yields the claimed result.
\qed

\subsubsection{Proof of Lemma~\ref{lem:k-coord-num}} \label{sec:proof-lem-k-coord-num}
Recalling the change of variables~\eqref{eq:change-of-var-W}, we write the projection matrices $\bP_{\backslash k}$ and $\bS_{\backslash k}$ as
\begin{align}
	\label{eq:def-loo-proj-normalized}
	\bP_{\backslash k} = \bG \bX_{\backslash k} \cdot \bigl( \bX_{\backslash k}^{\top} \bG^2 \bX_{\backslash k} \bigr)^{-1} \cdot \bX_{\backslash k}^{\top} \bG \quad \text{ and } \quad \bS_{\backslash k} = \bI - \bP_{\backslash k}.
\end{align}
Also recall from equation~\eqref{definition-X-G-V-Y} that
\[
X,G,V \overset{\mathsf{i.i.d.}}{\sim} \mathcal{N}(0,1)\;\;\text{and}\;\;  Y = \psi\left( \frac{\parcompZ_{t+1}}{\sqrt{\parcompZ_{t+1}^{2} + \perpcompZ_{t+1}^{2}}} \cdot G\cdot X + \frac{\perpcompZ_{t+1}}{\sqrt{\parcompZ_{t+1}^{2} + \perpcompZ_{t+1}^{2}}} \cdot V\cdot X \right).
\]
\paragraph{Proof of Lemma~\ref{lem:k-coord-num}(a):}
It suffices to analyze the quadratic form $\langle \Xcol{1}, \bG \bS_{\backslash 1} \by \rangle$.  
Write the vector $\by$ as
\[
\by = \psi\bigl((\bZ \nu_{\star}) \odot \Xcol{1}\bigr) + \beps \overset{(d)}{=} \psi\Bigl(\frac{\parcompZ_{t+1}}{\sqrt{\parcompZ_{t+1}^2 + \perpcompZ_{t+1}^2}} \bG \Xcol{1} +  \frac{\perpcompZ_{t+1}}{\sqrt{\parcompZ_{t+1}^2 + \perpcompZ_{t+1}^2}} \bV \Xcol{1}\Bigr) + \beps,
\]
where $\bV = \diag(V_1, V_2, \dots, V_n)$ is a diagonal matrix consisting of i.i.d. standard Gaussian random variables independent of everything else,
and the distributional equivalence comes from decomposing $\bcoefZ_{\star}$ as $\bcoefZ_{\star} = \bP_{\bcoefZ_{t+1}} \bcoefZ_{\star} + \bP_{\bcoefZ_{t+1}}^{\perp} \bcoefZ_{\star}$.  
We now make two complementary claims:
\small
\begin{subequations}
	\begin{align}
		&\Prob_{\Xcol{1}, \beps} \left\{  \Big\vert \langle \Xcol{1}, \bG \bS_{\backslash 1} \by \rangle -  \EE[\langle \Xcol{1}, \bG \bS_{\backslash 1} \by \rangle \;|\; \bG, \bX_{\backslash 1}, \bV] \Big \vert \gtrsim  (\parsigma)\Big(\sqrt{\log n \cdot \trace(\bG^2)} + \log(n) \| \bG \|_{\op}  \Big) \right\} \lesssim n^{-25}, \label{eq:first-prob} \\
		&\Prob_{\bG, \bX_{\backslash 1}, \bV} \left\{ \Big\vert  \EE[\langle \Xcol{1}, \bG \bS_{\backslash 1} \by \rangle \;|\; \bG, \bX_{\backslash 1}, \bV] - n \cdot \EE\bigg\{ \frac{GXY}{1+\tau G^{2}}\bigg\} \Big\vert \gtrsim \sqrt{n \log n} \right\} \lesssim n^{-25}. \label{eq:second-prob}
	\end{align}
\end{subequations}
\normalsize

We prove these claims momentarily, but let us first use them to prove the desired result.
Hoeffding's inequality yields that with probability greater than $1 - n^{-25}$, we have
\begin{align*}
	\| \bG \|_{\op} \lesssim \sqrt{\log n} \qquad \text{ and } \trace(\bG^2) \lesssim n.
\end{align*}
Using this result in conjunction with the law of total probability yields
\begin{align*}
	\Prob \left\{  \Big\vert \langle \Xcol{1}, \bG \bS_{\backslash 1} \by \rangle -  n \cdot \EE\bigg\{ \frac{GXY}{1+\tau G^{2}}\bigg\} \Big\vert \gtrsim (\parsigma) \sqrt{n\log(n)}  \right\} &\lesssim n^{-10}.
\end{align*}
The result follows upon changing variables back to $\bW_{t+1} = \sqrt{\parcompZ_{t+1}^2 + \perpcompZ_{t+1}^2} \bG$.

\medskip

\noindent \underline{Proof of claim~\eqref{eq:first-prob}:}
Note that the matrix $\bG\bS_{\backslash 1}$ is measurable with respect to the triplet $(\bG, \bX_{\backslash 1}, \bV)$, and that
\begin{align*}
	\| \bG\bS_{\backslash 1} \|_{\op} \leq \| \bG \|_{\op} \quad \text{ and } \| \bG \bS_{\backslash 1} \|_F^2 = \sum_{i=1}^{n} \sigma_i^2\bigl(\bG^2\bS_{\backslash 1}\bigr) \leq \sum_{i=1}^{n} \sigma_i^2 \bigl(\bG^2\bigr) = \trace (\bG^2).
\end{align*}
Letting $\by' = \by - \boldsymbol{\epsilon}$, note the decomposition
\begin{align*}
\langle \Xcol{1}, \bG\bS_{\backslash 1} \by \rangle &= \underbrace{ \langle \Xcol{1}, \bG\bS_{\backslash 1} \boldsymbol{\epsilon} \rangle }_{T_{1}} + \underbrace{ \langle \Xcol{1}, \bG\bS_{\backslash 1} \by' \rangle }_{T_{2}}
\end{align*}
Letting $\boldsymbol{\epsilon}' = \boldsymbol{\epsilon}/\sigma$, we obtain that
\[
	T_{1} = \frac{\sigma}{2} \cdot \Bigl[\langle \Xcol{1} + \bepsilon' , \bG\bS_{\backslash 1} (\Xcol{1} + \bepsilon') \rangle - \langle \Xcol{1}, \bG\bS_{\backslash 1} \Xcol{1} \rangle - \langle \bepsilon', \bG\bS_{\backslash 1} \bepsilon' \rangle \Bigr]
\]
Towards bounding $T_{1}$, note that entries of $\bepsilon'$ and $\Xcol{1}$ are standard normal random variables. Applying the Hanson--Wright inequality to each term on the RHS (treating the matrix $\bG\bS_{\backslash 1}$ as deterministic) and noting
$\EE[ T_{1} | \bG, \bX_{\backslash 1}, \bV ] = 0$ (since $\Xcol{1}$ and $\bepsilon$ are independent) yields
\[
	\Pr\Big\{ \big | T_{1} \big| \geq C \cdot \sigma \cdot \Big(\sqrt{\log n \cdot \trace(\bG^2)} + \log n \cdot \| \bG \|_{\op}  \Big) \Big\} \leq n^{-30}.
\]
Towards bounding $T_{2}$, we obtain
\[
	T_{2} = \frac{1}{2} \cdot \Bigl[\langle \Xcol{1} + \by' , \bG\bS_{\backslash 1} (\Xcol{1} + \by') \rangle - \langle \Xcol{1}, \bG\bS_{\backslash 1} \Xcol{1} \rangle - \langle \by', \bG\bS_{\backslash 1} \by' \rangle \Bigr].
\]
Continuing, note that entries of $\by'$ and $\Xcol{1}$ are zero mean Gaussian random variables whose variances are bounded by some universal constant $C$. Applying the Hanson--Wright inequality to each term on the RHS (treating the matrix $\bG\bS_{\backslash 1}$ as deterministic) then yields 
\[
	\Pr\Big\{ \big | T_{2} - \EE \big[ \langle \Xcol{1}, \bG\bS_{\backslash 1} \by' \rangle | \bG, \bX_{\backslash 1}, \bV \big] \big| \geq C  \Big(\sqrt{\log n \cdot \trace(\bG^2)} + \log n \cdot \| \bG \|_{\op}  \Big) \Big\} \leq n^{-30}.
\]
Putting the bounds on $T_{1}$ and $T_{2}$ together yields the desired result.

\noindent \underline{Proof of claim~\eqref{eq:second-prob}:}
Note that
\begin{align}\label{eq:expansion-cond-expectation}
	\EE \bigl\{\langle \Xcol{1}, \bG\bS_{\backslash 1} \by \rangle \mid \bG, \bX_{\backslash 1}, \bV \bigr\} = \sum_{i=1}^{n} \underbrace{\EE\bigl\{X_{i1} y_i \mid G_{i},V_{i} \bigr\}}_{=:\xi_i} G_i \bS_{\backslash k}(i, i),
\end{align}
where $\bS_{\backslash k}(i, i)$ denotes the $(i, i)$-th entry of the matrix $\bS_{\backslash k}$. Recall that as in the proof of Lemma~\ref{lem:k-coord-denom} (see Eq.~\eqref{eq:P-entries}) we have the representation
\[
\bigl(\bSk{k}\bigr)_{ii} = \frac{1}{1 +G_i^2 \tau_i}, 
\]
where $\tau_i = \frac{1}{n}\langle \bx_{i \backslash k}, \bSig_i^{-1} \bx_{i \backslash k} \rangle$ and $\bSig_i = \frac{1}{n}\sum_{j \neq i} G_j^2 \bx_{j \backslash k} \bx_{j \backslash k}^{\top}$.  Substituting this into the expansion~\eqref{eq:expansion-cond-expectation} and re-arranging yields the equivalent relation
\[
\EE \bigl\{\langle \Xcol{1}, \bG\bS_{\backslash 1} \by \rangle \mid \bG, \bX_{\backslash 1}, \bV \bigr\} = \sum_{i=1}^{n} \frac{\xi_i G_i}{1 + G_i^2 \tau} + \sum_{i=1}^{n}\frac{\xi_i G_i^3 (\tau - \tau_i)}{(1 + G_i^2 \tau)(1 + G_i^2 \tau_{i})}.
\]
Note that this is the random variable we wish to control, where the remaining randomness resides in the triplet $(\bG, \bX_{\backslash 1}, \bV)$. By Assumption~\ref{assptn:Y-psi}, we have $\| \xi_i \|_{\psi_2} \leq C_{\psi}$, whence $\| (\xi_i G_i)/(1 + G_i^2 \tau)\|_{\psi_1} \leq \| \xi_i G_i \|_{\psi_1} \leq C'_{\psi}$.  Thus, by Bernstein's inequality
\begin{align*}
	\Bigl \lvert \frac{1}{n} \sum_{i=1}^{n} \frac{\xi_iG_i}{1 + G_i^2 \tau} - \EE \Bigl\{\frac{\xi_iG_i}{1 + G_i^2 \tau}\Bigr\}\Bigr \rvert  \lesssim  \sqrt{\frac{\log{n}}{n}} \cdot, \quad \text{ with probability } \geq 1 - n^{-25}.
\end{align*}
Moreover, 
\begin{align*}
	\sum_{i=1}^{n}\frac{\xi_iG_i^3 (\tau - \tau_i)}{(1 + G_i^2 \tau)(1 + G_i^2 \tau_{i})} &\leq \max_{i \in [n]} \; \lvert \tau_i - \tau \rvert \cdot  \sum_{i=1}^{n}\xi_i G_i^3.
\end{align*}
Note (as in the proof of Lemma~\ref{lem:k-coord-denom}, see Eq.~\eqref{eq:tau-stuff}) that with probability at least $1 - n^{-25}$, 
\[
\max_{i \in [n]} \; \lvert \tau_i - \tau \rvert \lesssim  \sqrt{\frac{\log{n}}{n}}, \quad \text{ with probability } \geq 1 - n^{-25}.
\]
Next, since $\xi_i G_i^3$ has bounded moments of constant order, we may apply~\citet[Ex. 2.20]{wainwright2019high} to obtain $\sum_{i=1}^{n} \xi_i G_i^3 \leq Cn$. Finally, a straightforward calculation yields that
\[
\EE \Bigl\{\frac{\xi_iG_i}{1 + G_i^2 \tau}\Bigr\} = \EE \Bigl\{\frac{GXY}{1 + G^2 \tau}\Bigr\}.
\]
Combining the pieces, we see that with probability at least $1 - C n^{-25}$,
\begin{align}\label{ineq:second-term}
	\Bigl \lvert \frac{1}{n}\EE  \bigl\{\langle \Xcol{1}, \bG \bS_{\backslash 1} \by \rangle \mid \bG, \bX_{\backslash 1}, \bV \bigr\} - \EE \Bigl\{\frac{GXY}{1 + G^2 \tau}\Bigr\} \Bigr \rvert \leq C \sqrt{\frac{\log{n}}{n}},
\end{align}
as claimed.
This completes the proof of the claims and hence part (a) of the lemma. \qed

\paragraph{Proof of Lemma~\ref{lem:k-coord-num}(b):} 
Recall the change of variables~\eqref{eq:change-of-var-W} and note the bound
\[
\| \bG \bSk{k} \by \|_2^2 \leq C(1+\sigma^{2})n \log{n}, \quad \text{ with probability } \geq 1 - n^{-25},
\]
where we have used sub-multiplicativity of the operator norm, a standard bound on maxima of Gaussian random variables to bound $\| \bG \|_{\op}$, and~\citet[Ex. 2.20]{wainwright2019high} to upper bound $\| \by \|_2^2$.  Since the random vectors $\Xcol{k}$ and $\bG \bSk{k} \by$ are independent, we condition on the inequality of the above display and apply Hoeffding's inequality to obtain the bound
\[
\frac{1}{n} \bigl \lvert \langle \Xcol{k}, \bG \bSk{k} \by \rangle \bigr \rvert \leq \frac{C (\parsigma) \log{n}}{\sqrt{n}}, \quad \text{ with probability } \geq 1 - n^{-25}.
\]
The result is obtained upon using the change of variables $\sqrt{\parcompZ_{t+1}^2 + \perpcompZ_{t+1}^2} \bG = \bW_{t+1}$ once more.  \qed

\subsection{Orthogonal component: Proof of Theorem~\ref{thm:one-step}(b)} \label{sec:proof-orthogonal}
Recall the change of variables~\eqref{eq:change-of-var-W} and write
\[
\bP = \bG \bX \cdot \bigl( \bX^{\top} \bG^2 \bX \bigr)^{-1} \cdot \bX^{\top} \bG \quad \text{ and } \quad \bS = \bI - \bP.
\]
Also define the following quantities obtained by leaving the $i$-th \emph{sample} out:
\begin{align*}
	\by^{(-i)} &:= [y_1 \; |\; \dots \; | \; y_{i-1} \; | \; y_{i + 1} \; | \; \dots \; | \; y_n] \\
	\bX^{(-i)} &:= [\bx_1 \; |\; \dots \; | \; \bx_{i-1} \; | \; \bx_{i + 1} \; | \; \dots \; | \; \bx_n]^{\top}, \text{ and } \\
	\bG^{(-i)} &:= \diag(G_1, \dots, G_{i -1}, G_{i+1}, \dots, G_n).
\end{align*}
Using these, define the estimators
\begin{align} \label{eq:def-mu-i-mu-bar}
	\bmuG^{(i)} := \argmin_{\bcoefX \in \mathbb{R}^{d}}\; \bigl \| \by^{(-i)} - \bG^{(-i)} \bX^{(-i)} \bcoefX \bigr \|_2^2, \quad \text{ and } \quad  \bmuG := \argmin_{\bcoefX \in \mathbb{R}^{d}}\; \bigl \| \by - \bG \bX \bcoefX \bigr \|_2^2.
\end{align}
Note that $\bmuG$ is a scaled version of the true estimate $\bcoefX_{t + 1}$.

We are now ready to state two key lemmas.  The first demonstrates that the sum of squares of the numerator~\eqref{eq:k-coord} concentrates, for all $k \neq 1$, around a deterministic quantity with fluctuations on the order $\widetilde{O}(\sqrt{d}/n)$.  We provide the proof of Lemma~\ref{lem:orthogonal-concentration} in Section~\ref{sec:proof-lem-orthogonal-concentration}.
\begin{lemma}
	\label{lem:orthogonal-concentration}
	Under the assumptions of Theorem~\ref{thm:one-step}, 
	there exists a universal, positive constant $C_1$ such that with probability at least $1 - n^{-15}$,
	\begin{align*}
		&\Bigl \lvert \frac{1}{n^2}  \sum_{k=2}^{d}  \bigl\langle \Xcol{k}, \bG \bS_{\backslash k} \by \bigr\rangle^2 - \frac{d-1}{n^2} \EE \bigl\{\| \bG \bS \by \|_2^2\bigr\}\Bigr \rvert \leq  \frac{C_1(1 + \sigma^2)  \log^8(n) }{\sqrt{n \oversamp}}.
	\end{align*}
\end{lemma}
Note that in the above expression, $\| \bG \bS\by \|_2^2$ can be interpreted as a weighted sum of squared residuals since, by definition, $\bS \by = \by - \bP \by = \by - \bG \bX \bmuG$. 

The next lemma shows that the norm of the predictor concentrates around its expectation with fluctuations on the order $n^{-1/2}$.  We provide its proof in Section~\ref{sec:proof-lem-norm-concentration}.
\begin{lemma}
	\label{lem:norm-concentration}
	There exists a universal, positive constant $C_1$ such that the following holds
	\[
	\bigl \lvert \| \bmuG^{(i)} \|_2^2 - \EE \| \bmuG \|_2^2 \bigr \rvert \leq C_1\frac{(1 + \sigma^2)\log^{3}{n}}{\sqrt{n}}, \qquad \text{ with probability } \qquad \geq 1 - n^{-15}.
	\]
\end{lemma}
With these lemmas in hand, we proceed to the proof, which consists of three steps.  First, we show that---up to fluctuations on the order $n^{-1/2}$---it suffices to ignore the denominator and understand the sum of squares of the numerator~\eqref{eq:k-coord}, whence we apply Lemma~\ref{lem:orthogonal-concentration}.  Second, we execute a leave one sample out argument to compute $\EE\{ \| \bG \bS \by \|_2^2\}$.  Finally, we apply Lemma~\ref{lem:norm-concentration} in conjunction with the previous step to compute a deterministic approximation for the orthogonal component, around which the empirical concentrates with fluctuations on the order $n^{-1/2}$.  

\paragraph{Step 1: Reducing to studying the numerator.}
Applying the characterization~\eqref{eq:k-coord} in conjunction with the rotational invariance of the Gaussian distribution, we obtain
\begin{align}\label{eq:perpcomp-step1}
	\perpcompX_{t+1}^2 = \| \bcoefX_{t+1, \backslash 1} \|_2^2 = \sum_{k \neq 1} \frac{\langle \Xcol{k}, \bW_{t+1} \cdot \bS_{\backslash k} \by \rangle^2}{\langle \Xcol{k}, \bW_{t+1} \bS_{\backslash k} \bW_{t+1} \Xcol{k} \rangle^2} = \frac{1}{\parcompZ_{t+1}^2 + \perpcompZ_{t+1}^2}\sum_{k\neq 1} \frac{A_k^2}{B_k^2},
\end{align}
where in the last equality we have set
\[
A_k = \langle \Xcol{k}, \bG \bS_{\backslash k} \by \rangle \quad \text{ and } \quad B_k = \langle \Xcol{k}, \bG \bS_{\backslash k} \bG \Xcol{k} \rangle.
\]
Going forward, we will write $\perpcompbarX_{t+1}^2 = (\parcompZ_{t+1}^2 + \perpcompZ_{t+1}^2) \cdot  \perpcompX_{t+1}^2$.
By Lemma~\ref{lem:k-coord-denom}, \sloppy\mbox{$B_k/n = \EE\{G^2/(1 + \tau G^2)\} + \widetilde{\order}(n^{-1/2})$}, whence we decompose
\begin{align}\label{ineq:perpcomp-step2}
	\sum_{k\neq 1} \frac{A_k^2}{B_k^2}- \frac{A_k^2}{n^2 \EE\{ \frac{G^2}{1 + \tau G^2}\}^2} &= \frac{1}{n^2}\sum_{k \neq 1} A_k^2 \cdot \biggl(\frac{ \EE\{ \frac{G^2}{1 + \tau G^2}\}^2 - B_k^2/n^2}{ \EE\{ \frac{G^2}{1 + \tau G^2}\}^2 B_k^2/n^2}\biggr) \nonumber\\
	&= \frac{1}{n^2}\sum_{k \neq 1} A_k^2 \cdot \biggl(\frac{( \EE\{ \frac{G^2}{1 + \tau G^2}\} - B_k/n)( \EE\{ \frac{G^2}{1 + \tau G^2}\} + B_k/n)}{ \EE\{ \frac{G^2}{1 + \tau G^2}\}^2 B_k^2/n^2}\biggr)\nonumber\\
	&\overset{\1}{\leq} \frac{C (1 + \sigma^2) \log^{5/2}{n}}{\sqrt{n}},
\end{align}
where step $\1$ follows with probability at least $1 - n^{-15}$ upon applying Lemmas~\ref{lem:k-coord-denom} and~\ref{lem:k-coord-num}(b) in conjunction with the union bound and the sandwich relation  
\sloppy\mbox{$c \leq \EE\{G^2/(1 + \tau G^2)\} \leq 1$} (where the lower bound holds since by assumption $\oversamp \geq C_0$).  Combining inequalities~\eqref{eq:perpcomp-step1} and~\eqref{ineq:perpcomp-step2} yields
\[
\Bigl \lvert \perpcompbarX_{t+1}^2 -  \EE\Bigl\{ \frac{G^2}{1 + \tau G^2}\Bigr\}^{-2} \cdot \frac{1}{n^2}\sum_{k \neq 1} A_k^2 \Bigr \rvert \leq \frac{C (1 + \sigma^2) \log^{5/2}{n}}{\sqrt{n}}.
\]
Moreover, applying the triangle inequality yields
\begin{align*}
	\Bigl \lvert \frac{1}{n^2}\sum_{k \neq 1} A_k^2 - \frac{d}{n^2}\EE\bigl\{ \| \bG\bS \by \|_2^2 \bigr\} \Bigr \rvert &\leq \Bigl \lvert \frac{1}{n^2}\sum_{k \neq 1} A_k^2 - \frac{d-1}{n^2}\EE\bigl\{ \| \bG\bS \by \|_2^2 \bigr\} \Bigr \rvert + \frac{1}{n^{2}} \EE\bigl\{ \| \bG\bS \by \|_2^2 \bigr\} \\
	&\leq \frac{C (1 + \sigma^2)\log^{8}{n}}{\sqrt{n}} + \frac{C(1+\sigma^2) \log(n)}{n} \leq \frac{2C (1 + \sigma^2)\log^{8}{n}}{\sqrt{n}},
\end{align*}
where we have applied Lemma~\ref{lem:orthogonal-concentration} and used $\| \bG\bS \by \|_2 \leq \|\bG\|_{\op} \cdot \|\by\|_{2}$. Combining the pieces together and noting $c \leq \EE\{G^2/(1 + \tau G^2)\}$ yields
\begin{align}\label{ineq:deviation-numerator-step1}
	\Bigl \lvert \perpcompbarX_{t+1}^2 - \EE\Bigl\{ \frac{G^2}{1 + \tau G^2}\Bigr\}^{-2} \cdot \frac{d}{n^2}\EE\bigl\{ \| \bG\bS \by \|_2^2 \bigr\} \Bigr \rvert \leq \frac{C (1 + \sigma^2)\log^{8}{n}}{\sqrt{n}}.
\end{align}

\paragraph{Step 2: Computing $\EE\{ \| GSy \|_2^2 \}$.} We execute a leave-one-sample-out argument. Mirroring the notation used by~\cite{el2013robust}, define the weighted residuals 
\begin{align} \label{eq:def-residual}
	R_{i} = G_i^2 \cdot \langle \bx_{i}, \bmuG \rangle - G_i y_i \qquad \text{ and } \qquad r_{i,(i)} = G_i^2 \cdot \langle \bx_{i }, \bmuG^{(i)} \rangle -G_i y_i,
\end{align}
where we recall $\bmuG^{(i)}$ and $\bmuG$ from Eq.~\eqref{eq:def-mu-i-mu-bar}.
Note that $R_{i}$ is the weighted residual with respect to the predictor $\bmuG$, whereas $r_{i, (i)}$ is the weighted residual with respect to the predictor $\bmuG^{(i)}$.  These two notions are related by a simple formula 
\begin{align}
	\label{eq:loo-residual-tool}
	R_{i} = \frac{r_{i, (i)}}{1 + G_i^2 \langle \bx_{i}, \bSig_{i}^{-1} \bx_{i} \rangle},
\end{align}
where we have additionally used the notation $\bSig_{i} = \sum_{j \neq i} G_j^2 \bx_{j} \bx_{j}^{\top}$.  We take this relation for granted for the time being, providing its proof in Section~\ref{sec:proof-loo-tools}.  Continuing, note that $(\bG \bS \by)_{i} = -R_i$, whence we apply the relation~\eqref{eq:loo-residual-tool} to obtain
\[
\frac{d}{n^2}\EE \bigl\{ \| \bG \bS \by \|_2^2 \bigr\} = \frac{1}{\Lambda} \cdot \frac{1}{n}\sum_{i=1}^{n} \EE \Bigl\{ \frac{r_{i, (i)}^2}{(1 + G_i^2\bx_{i}^{\top} \bSig_{i}^{-1} \bx_{i})^2} \Bigr\}.
\]
Next, we replace the quadratic form $\bx_{i}^{\top} \bSig_{i}^{-1} \bx_i$ by a deterministic approximation. A straightforward calculation yields that
\begin{align*}
	(1+G_{i}^{2}\bx_{i}^{\top} \bSig_{i}^{-1} \bx_{i})^{-2} - (1+G_{i}^{2}\tau)^{-2} &= \frac{G_{i}^{2}(\tau - \bx_{i}^{\top} \bSig_{i}^{-1} \bx_{i}) \cdot ((1+G_{i}^{2}\tau)^{-1} + (1+G_{i}^{2}\bx_{i}^{\top} \bSig_{i}^{-1} \bx_{i})^{-1}) }{(1+G_{i}^{2}\tau)(1+G_{i}^{2}\bx_{i}^{\top} \bSig_{i}^{-1} \bx_{i})} \\&\leq 2G_{i}^{2}(\tau - \bx_{i}^{\top} \bSig_{i}^{-1} \bx_{i}).
\end{align*} 
Consequently, we obtain the bound
\begin{align*}
	\Bigl \lvert \EE \Bigl\{ \frac{r_{i, (i)}^2}{(1 + G_i^2\bx_{i}^{\top} \bSig_{i}^{-1} \bx_{i})^2} \Bigr\} - \EE \Bigl\{ \frac{r_{i, (i)}^2}{(1 + G_i^2\tau)^2} \Bigr\}\Bigr \rvert &\leq 2\EE\bigl\{ r_{i, (i)}^2 G_i^2  \cdot \lvert \tau - \bx_i^{\top} \bSig_{i}^{-1} \bx_i \rvert \bigr\}\\
	&\overset{\1}{\leq} 2\sqrt{\EE\bigl\{ (\tau - \bx_i^{\top} \bSig_{i}^{-1} \bx_i) ^2\bigr\}} \sqrt{\EE\bigl\{ r_{i, (i)}^4 G_i^4\bigr\}},
\end{align*}
where step $\1$ follows from the Cauchy--Schwarz inequality. Applying Lemmas~\ref{lem:general-trace-concentration} and~\ref{lem:tau-concentration} yields the bound
\[
\EE\bigl\{ (\tau - \bx_i^{\top} \bSig_{i}^{-1} \bx_i) ^2\bigr\} = \int_{0}^{\infty} \Prob\bigl\{ (\tau - \bx_i^{\top} \bSig_{i}^{-1} \bx_i) ^2 \geq t \bigr\} \mathrm{d}t \leq \frac{C\log{n}}{n},
\]
which bounds the first term in the RHS of the previous display.
To bound the second term, we apply the numeric inequality $(A + B)^{4} \leq 2^{3}\cdot (A^4 + B^{4})$ to obtain the estimate
\begin{align*}
	\EE\bigl\{ r_{i, (i)}^4 G_i^2 \bigr\} \leq 8 \EE\{ G_i^{10} \langle \bx_i, \bmuG^{(i)}\rangle^{4} + G_i^{6} y_i^4\} \leq C(1 + \sigma^2)^2.  
\end{align*}
To prove the final inequality, note that $\langle \bx_i, \bmuG^{(i)} \rangle \overset{(d)}{=} \| \bmuG^{(i)} \|_2 Z$, where $Z \sim \mathsf{N}(0, 1)$. Then apply the Cauchy--Schwarz inequality in conjunction with 
\begin{align*}
	\| \bmuG^{(i)} \|_2^{8} = \big\| \big( ( \bX^{(-i)} )^{\top} (\bG^{(-i)})^{2} \bX^{(-i)} \big)^{-1} ( \bX^{(-i)} )^{\top} \bG^{(-i)} \by^{(-i)} \big\|_{2}^{8} \overset{\1}{\lesssim} \frac{1}{n} \|\bG^{(-i)} \by^{(-i)}\|_{2}^{4} \lesssim (1 + \sigma^2)^4,
\end{align*}
where in step $\1$ we apply Lemma~\ref{lem:eigenvalues-G} so that $\big\| \big( ( \bX^{(-i)} )^{\top} (\bG^{(-i)})^{2} \bX^{(-i)} \big)^{-1} \big\|_{\op} \lesssim n^{-1}$ and $\|\bX^{(-i)}\|_{\op} \lesssim \sqrt{n}$.
Now, expand the squared residual $r_{i, (i)}^2$ as
\[
r_{i, (i)}^2 = G_i^2 y_i^2 - 2G_i^3 y_i X_{i1} \bmuG^{(i)}(1) - 2G_i^3 y_i \sum_{j \neq 1}X_{ij} \bmuG^{(i)}(j)  + G_i^4 \langle \bx_{i}, \bmuG^{(i)} \rangle^2,
\]
and note that the third term is zero-mean.  By Lemma~\ref{lem:norm-concentration}, we have \sloppy\mbox{$\lvert \EE \langle \bx_i, \bmuG^{(i)} \rangle^2 - \EE \| \bmuG \|_2^2 \rvert = \widetilde{\order}((1 + \sigma^2)n^{-1/2})$}.  Thus, letting $\parcompdetbarX = (\parcompZ_{t+1}^2 + \perpcompZ_{t+1}^2)^{1/2} \parcompdetX_{t+1}$ and applying Theorem~\ref{thm:one-step}(a), we obtain
\[
\Bigl \lvert \EE \Bigl\{ \frac{r_{i, (i)}^2}{(1 + W_i^2\tau)^2} \Bigr\} - \EE \Bigl\{ \frac{G_i^2 y_i^2 - 2G_i^3 y_i X_{i1} \parcompdetbarX + G_i^4  \EE \| \bmuG \|_2^2}{(1 + G_i^2\tau)^2} \Bigr\} \biggr \rvert \leq \frac{C(1 + \sigma^2)\log^{3}{n}}{\sqrt{n}}.
\]
Putting together the pieces yields the bound
\begin{align} \label{ineq:deviation-expectation-step2}
	\biggl \lvert \frac{d}{n^2}\EE \bigl\{ \| \bG \bS \by \|_2^2 \bigr\}  - \frac{1}{\oversamp} \cdot \EE \Bigl\{ \frac{G_i^2 y_i^2 - 2G_i^3 y_i X_{i1} \parcompdetbarX_{t+1} + G_i^4  \EE \| \bmuG \|_2^2}{(1 + G_i^2\tau)^2} \Bigr\} \biggr \rvert \leq \frac{C(1 + \sigma^2)\log^{3}{n}}{\sqrt{n}}.
\end{align}

\paragraph{Step 3: Solving for $(\perpcompdetX_{t+1})^2$.}  
Combining the inequalities~\eqref{ineq:deviation-numerator-step1} and~\eqref{ineq:deviation-expectation-step2} yields 
\begin{align}\label{ineq:orthogonal-step3-ineq1}
	\biggl \lvert \perpcompbarX_{t+1}^2 - \EE\Bigl\{ \frac{G^2}{1 + \tau G^2}\Bigr\}^{-2} \cdot \frac{1}{\oversamp} \cdot \EE \Bigl\{ \frac{G_i^2 y_i^2 - 2G_i^3 y_i X_{i1} \parcompdetbarX_{t+1} + G_i^4  \EE \| \bmuG \|_2^2}{(1 + G_i^2\tau)^2} \Bigr\}\biggr \rvert  \leq \frac{C(1 + \sigma^2)\log^{8}{n}}{\sqrt{n}}.
\end{align}
Define the deterministic proxy $\breve{\perpcompX}_{t+1}$ as
\[
\breve{\perpcompX}_{t+1}^2 := \EE \| \bmuG \|_2^2 - (\parcompdetbarX_{t+1})^2,
\]
and note that applying Theorem~\ref{thm:one-step}(a) in conjunction with Lemma~\ref{lem:norm-concentration} yields
\begin{align}\label{ineq:deterministic-orthogonal}
	\Bigl \lvert \perpcompbarX_{t+1}^2 - \breve{\perpcompX}_{t+1}^2 \Bigr \rvert \leq \frac{C(1 + \sigma^2) \log^{3}{n}}{\sqrt{n}}, \qquad \text{ with probability } \geq 1 - n^{-10}.
\end{align}
We emphasize that in the above display, $\perpcompbarX_{t+1}^2$ is a random variable, whereas $\breve{\perpcompX}_{t+1}^2$ is deterministic.  Subsequently, we substitute inequality~\eqref{ineq:deterministic-orthogonal} into inequality~\eqref{ineq:orthogonal-step3-ineq1} to obtain
\begin{align} \label{ineq:T1-T2-bound-orthogonal}
	\bigl \lvert T_1 \cdot \perpcompbarX_{t+1}^2 - T_2 \bigr \rvert \leq \frac{C(1 + \sigma^2)\log^{8}{n}}{\sqrt{n}},
\end{align}
where 
\begin{align*}
	T_1 &= 1 - \EE\Bigl\{ \frac{G^2}{1 + \tau G^2}\Bigr\}^{-2} \cdot \frac{1}{\oversamp} \EE \Bigl \{\frac{ G^4}{(1 + \tau G^2)^2} \Bigr\}, \qquad \text{ and }\\
	T_2 &= \EE\Bigl\{ \frac{G^2}{1 + \tau G^2}\Bigr\}^{-2} \cdot \frac{1}{\oversamp} \EE \Bigl \{ \frac{G^2 Y^2 - 2\parcompdetbarX_{t+1} G^3 X Y + G^4 (\parcompdetbarX_{t+1})^2}{(1 + \tau G^2)^2} \Bigr\}.
\end{align*}
Note that, since $\tau = 1/C(\oversamp)$ and $C(\oversamp) \asymp \oversamp \geq C$ (see Lemma~\ref{lemma:C(lambda)-lambda}), we obtain the lower bound $T_1 \geq c > 0$.  
Consequently, inequality~\eqref{ineq:T1-T2-bound-orthogonal} implies that
\begin{align}\label{ineq:T1-T2-bound-orthogonal-2}
	\bigl \lvert \perpcompbarX_{t+1}^2 - T_2/T_{1} \bigr \rvert \leq \frac{C(1 + \sigma^2)\log^{8}{n}}{\sqrt{n}}.
\end{align}
At this juncture, we recall the definition of $C(\oversamp)$ in equation~\eqref{definition-of-C} and $\tau = C(\oversamp)^{-1}$, which implies that
\[
\oversamp \cdot \EE\Bigl\{ \frac{G^2}{1 + \tau G^2}\Bigr\}^{2} - \EE \Bigl \{\frac{ G^4}{(1 + \tau G^2)^2} \Bigr\} = C(\oversamp) \cdot \EE \Bigl \{\frac{ G^2}{(1 + \tau G^2)^2} \Bigr\}.
\]
Using the equality in the display above, a straightforward calculation yields that 
\begin{align}\label{eq:equivalent-T1/T2}
	(\perpcompdetX_{t+1})^2 = (\parcompZ_{t+1}^2 + \perpcompZ_{t+1}^2)^{-1} \cdot T_2/T_1.
\end{align}
Combining the inequality~\eqref{ineq:T1-T2-bound-orthogonal-2} with the identity~\eqref{eq:equivalent-T1/T2} yields the desired result. \qed

\subsubsection{Proof of Lemma~\ref{lem:orthogonal-concentration}} \label{sec:proof-lem-orthogonal-concentration}
We begin by defining some notation for quantities when predictors are left out. For $k \neq \ell \neq m$, define:
\begin{align} \label{eq:u-notation}
	\bu_k = \bS_{\backslash k} \bG \Xcol{k}, \quad \bu_{k, \ell} = \bS_{\backslash k, \ell} \bG \Xcol{\ell} \quad \text{ and } \quad   \bu_{k, \ell, m} = \bS_{\backslash k, \ell, m} \bG \Xcol{m}. 
\end{align}
Note that $\bu_{k, \ell, m}$ and $\bu_{k, \ell}$ are still in boldface, to distinguish them from any particular coordinate of $\bu$.
The following set of rank one update formulae, whose proofs we provide in Section~\ref{sec:proof-loo-tools}, are useful:
\begin{align}
	\label{eq:recursive-S}
	\bS = \bS_{\backslash k} - \frac{1}{\| \bu_k \|_2^2} \bu_k &\bu_k^{\top}, \qquad \bS_{\backslash k} = \bS_{\backslash k, \ell} - \frac{1}{\| \bu_{k, \ell} \|_2^2} \bu_{k, \ell} \bu_{k, \ell}^{\top}, \qquad \text{ and }\nonumber\\
	&  \bS_{\backslash k, \ell} = \bS_{\backslash k, \ell, m} - \frac{1}{\| \bu_{k, \ell, m} \|_2^2} \bu_{k, \ell, m} \bu_{k, \ell, m}^{\top}.
\end{align}
Note that the matrix $\bS$ is a projection matrix onto the null space of the matrix $\bG \bX$, which is of dimension $n - d$.  The rank one update splits this projection into a projection onto the null space of the matrix $\bG \bX_{\backslash k}$---which is of dimension $n - d + 1$---and the subspace spanned by the vector $\bu_k$.  In the sequel, we will additionally use the shorthand 
\begin{align} \label{eq:v-notation}
	\bv_k = \bG \bS_{\backslash k} \by \qquad \text{ and } \qquad \bv = \bG \bS \by.
\end{align}
Another convenient abstraction is to view $\bu_k$, $\bu_{k, \ell}$, $\bS$, $\bS_k$ $\bv$, $\bv_k$ as \emph{functions} mapping the tuple $(\Xcol{1}, \Xcol{2}, \ldots, \Xcol{d}; \bG)$ to its respective space.  For instance, letting $\mathcal{D}^n$ denote the space of $n \times n$ diagonal matrices for convenience, we define the function $\bu_{k}$ as 
\begin{align*}
	\bu_{k}: \underbrace{\mathbb{R}^{n} \times \mathbb{R}^n \times \cdots \times \mathbb{R}^n}_{d \text{ times}} \times \mathcal{D}^n &\rightarrow \mathbb{R}^{n}\\
	(\Xcol{1}, \Xcol{2}, \ldots, \Xcol{d}; \bG) &\mapsto \bS_{\backslash k} \bG \Xcol{k}.
\end{align*}
Equipped with this viewpoint, define the function  $f: \underbrace{\mathbb{R}^{n} \times \mathbb{R}^n \times \cdots \times \mathbb{R}^n}_{d \text{ times}} \times \mathcal{D}^n \rightarrow \mathbb{R}$ as 
\begin{align} \label{eq:def-f}
	f\bigl(\bp_1, \bp_2, \dots, \bp_d; \bM \bigr) &:= 
	\frac{1}{n^2}\cdot \sum_{k=2}^{d}  \bigl\langle \bp_k, \bv_k(\bp_1, \ldots, \bp_d; \bM) \bigr \rangle^2,
\end{align}
where each $\bp_i \in \real^n$ and $\bM \in \mathcal{D}^n$. When the context is clear, we will often abuse notation and omit the arguments of (for instance) the function $\bu_k$ and write $\bu_k(\bp_1, \bp_2, \dots, \bp_d; \bM) = \bu_k$.  Using this notation, our goal is to prove that $f(\Xcol{1}, \ldots, \Xcol{d}; \bG)$ concentrates within $\ordertil(n^{-1/2})$ around the deterministic quantity $\frac{d - 1}{n^2} \EE\{ \| \bG \bS \by \|_2^2 \}$.  In order to do this, we will employ Warnke's typical bounded differences inequality~\citep[Theorem 2]{warnke2016method}, which requires two preliminary ingredients: (i) the existence of a high probability regularity set $\mathcal{S}$ and (ii) the construction of a truncated function $\fdown$ which preserves bounded differences on the regularity set.  

\paragraph{Step 1: Defining the regularity set $\mathcal{S}$.} As mentioned above, we first define a regularity set $\mathcal{S} \subseteq \underbrace{\mathbb{R}^{n} \times \mathbb{R}^n \times \cdots \times \mathbb{R}^n}_{d \text{ times}} \times \mathcal{D}^n$ in the following way.  Let
\begin{subequations}\label{eq:events-E}
	\begin{align}\label{eq:def-E1}
		\mathcal{S}_1 &= \Bigl\{(\bp_1, \ldots, \bp_d; \bM): c_1n \leq \| \bu_{\ell_1} \|_2^2 \leq C_1n, \quad c_1 n \leq \| \bu_{\ell_1, \ell_2}\|_2^2 \leq C_1n, \notag \\
		&\quad \qquad \text{ and } \quad c_1n \leq \| \bu_{\ell_1, \ell_2, \ell_3}  \|_2^2 \leq C_1 n, \qquad \text{ for all } \quad \ell_1 \neq \ell_2 \neq \ell_3 \subseteq [d]^{3}\Bigr\}.
	\end{align}
	Similarly viewing the quantities $\bu_{k, \ell}$ and $\bu_{k, \ell}$ as functions, let
	\begin{align} \label{eq:def-E2}
		\mathcal{S}_2 = \Bigl \{(\bp_1, \ldots, \bp_d; \bM): \langle \bu_{\ell_1}, \bu_{\ell_2} \rangle \vee \langle \bu_{\ell_1, \ell_2}, \bu_{\ell_2, \ell_1}& \rangle  \vee  \langle \bu_{\ell_1, \ell_2, \ell_3}, \bu_{\ell_3, \ell_1, \ell_2} \rangle \leq C_2\sqrt{n}\log^{3/2}{n}, \nonumber\\
		&\;\; \text{ for all } \;\; \ell_1 \neq \ell_2 \neq \ell_3 \subseteq [d]^{3}\Bigr\}.
	\end{align}
	Next, similarly viewing $\by$ as a function, let 
	\begin{align}\label{eq:def-E3}
		\mathcal{S}_3 &= \Bigl\{(\bp_1, \ldots, \bp_d; \bM): \langle \by, \bu_{\ell_1} \rangle \vee \langle \by, \bu_{\ell_1, \ell_2} \rangle \vee \langle \by, \bu_{\ell_1, \ell_2, \ell_3} \rangle \leq C(1 + \sigma)\sqrt{n}\log^{3/2}{n}, \qquad \text{ and } \nonumber\\
		& \langle \by, \bu_{1} \rangle \vee \langle \by, \bu_{\ell_1, 1} \rangle \vee \langle \by, \bu_{\ell_1, \ell_2, 1} \rangle \leq C(1 + \sigma)n\log^{3/2}{n}, \;  \text{ for all } \; 1 \neq \ell_1 \neq \ell_2 \neq \ell_3 \subseteq [d]^{3} \Bigr\}.
	\end{align}
	The final such set allows us to bound quadratic forms appearing in the proof:
	\begin{align}\label{eq:def-E4}
		\mathcal{S}_4 = \Bigl\{ (\bp_1, \ldots, \bp_d; \bM): \bp_1^{\top} \bM \sum_{k \neq \{1, \ell\}} \bu_{1, \ell, k}\bu_{1, \ell,k}^{\top} \bM \bp_\ell \vee \by^{\top} &\sum_{k \neq \{1, \ell\}} \bu_{1, \ell, k} \bu_{1, \ell,k}^{\top} \bM \bp_{\ell} \leq (1 + \sigma)n^{3/2}\log^{7/2}(n) \nonumber \\
		& \text{ for all } \ell \in \{2, 3, \dots, d\} \Bigr\}.
	\end{align}
	Finally, define $\mathcal{S}$ as the intersection of the above sets
	\begin{align} \label{eq:S-defn}
		\mathcal{S} := \mathcal{S}_1 \cap \mathcal{S}_2 \cap \mathcal{S}_3 \cap \mathcal{S}_4.
	\end{align}
\end{subequations}
Next, define the Hamming metric $\rho: \mathbb{R}^{n \times d} \times \mathbb{R}^{n \times d} \rightarrow \{0, 1, 2, \dots, d\}$ such that if $\bP, \bP'$ differ in at most $k \leq d$ columns, then $\rho(\bP, \bP') = k$.  The next lemma demonstrates a key stability property of the regularity set $\mathcal{S}$.  We provide its proof in Section~\ref{sec:proof-lem-lower-truncation}.
\begin{lemma}
	\label{lem:lower-truncation}
	Let $\bG\in \mathcal{D}^n$, $\bX$ and $\bX'$ be such that $(\bX, \bG), (\bX', \bG) \in \mathcal{S}$ and $\rho(\bX, \bX') \leq 2$ and consider the function $f$~\eqref{eq:def-f}.  There exists a universal positive constant $C_0$ such that the following holds
	\[
	\bigl \lvert f(\bX) - f(\bX') \bigr \rvert \leq C_0(1 + \sigma^2)\frac{\log^{15/2}(n)}{n}. 
	\]
\end{lemma}
In words, Lemma~\ref{lem:lower-truncation} shows that if two collections of $d$ vectors $\bX$ and $\bX'$ differ in at most two columns, then 
the function evaluations $f(\bX)$ and $f(\bX')$ are extremely close.  

\paragraph{Step 2: Truncation.} In order to exploit Lemma~\ref{lem:lower-truncation}, we define the function $\fdown: \mathbb{R}^{n \times d} \times \mathcal{D}^n$ as 
\begin{align}\label{eq:fdown}
	f^{\downarrow}(\bP; \bM) = \inf_{\bP': (\bP', \bM ) \in \mathcal{S}}\; \Bigl\{ f(\bP'; \bM) + 2C_0(1 + \sigma^2)\frac{\log^{15/2}(n)}{n} \cdot \rho(\bP, \bP') + D \cdot \mathbbm{1}\{\rho(\bP, \bP') > 1\}\Bigr\},
\end{align}
where $C_0$ is as in Lemma~\ref{lem:lower-truncation} and $D$ verifies the inequality \sloppy\mbox{$D \geq \sup_{(\bP, \bM) \in \mathcal{S}}\; f(\bP; \bM)$}.  By construction, $\fdown$ enjoys the following properties.
\begin{lemma}
	\label{lem:properties-fdown}
	Consider the functions $f$~\eqref{eq:def-f} and $\fdown$~\eqref{eq:fdown}.  The following hold.
	\begin{itemize}
		\item[(a)] If $(\bP, \bM) \in \mathcal{S}$, then $\fdown(\bP; \bM) = f(\bP; \bM)$.
		\item[(b)] If $(\bP, \bM) \in \mathcal{S}$ and $\rho(\bP, \bP') \leq 1$, then there exists a universal, positive constant $C$ such that 
		\[
		\lvert \fdown(\bP; \bM) - \fdown(\bP'; \bM) \rvert \leq C(1 + \sigma^2) \frac{\log^{15/2}(n)}{n}.
		\]
	\end{itemize}
\end{lemma}
We provide the proof of Lemma~\ref{lem:properties-fdown} at the end of the subsection.
Finally, since our goal is to understand $f(\Xcol{1}, \ldots, \Xcol{d}; \bG)$, we will first show that the random variable $\ind{ (\bp_1, \ldots, \bp_d; \bM) \in \mathcal{S} }$ is equal to $1$ with high probability when $\bp_i \overset{\mathsf{i.i.d.}}{\sim} \mathsf{N}(0, \bI_n)$ and $\diag(\bM) \sim \mathsf{N}(0, \bI_n)$. This lemma, whose proof we provide in Section~\ref{sec:proof-lem-truncation-probability}, is presented below.
\begin{lemma}
	\label{lem:truncation-probability}
	Suppose the assumptions of Lemma~\ref{lem:orthogonal-concentration} hold, and let $\mathcal{S}$ be defined according to Eq.~\eqref{eq:S-defn}. Then
	\[
	\Prob\{ (\Xcol{1}, \ldots, \Xcol{d}; \bG) \in \mathcal{S} \} \geq 1 - n^{-20}.
	\]
	Consequently, for all $t \leq n^{17}$, we have
	\[
	\Prob\bigl\{ \lvert f(\bX; \bG) - \EE f(\bX; \bG) \rvert \geq t \bigr\} \leq \Prob\bigl\{ \lvert f^{\downarrow}(\bX; \bG) - \EE f^{\downarrow}(\bX; \bG) \rvert \geq t/2 \bigr\} + C n^{-15}. 
	\]
\end{lemma}
In words, Lemma~\ref{lem:truncation-probability} reduces the problem to understanding the fluctuations of the function $f^{\downarrow}$. We will use the typical bounded differences inequality to do so.  Equipped with these preliminaries, we turn to the proof of Lemma~\ref{lem:orthogonal-concentration}.

\paragraph{Step 3: Putting together the pieces.} We require one additional lemma, which shows that leaving any column $k \neq 1$ out does not have a large effect on the sum of expected squared residuals.   We provide its proof in Section~\ref{sec:proof-lem-expectation-leave-out-first-col}.
\begin{lemma}
	\label{lem:expectation-leave-out-first-col}
	Under the setting of Lemma~\ref{lem:orthogonal-concentration}, there exists a universal, positive constant $C$ such that the following holds for any $k \neq 1$:
	\[
	\Bigl \lvert \EE\Bigl\{ \| \bG \bS_{\backslash k} \by \|_2^2 \Bigr\} - \EE \Bigl\{ \| \bG \bS \by \|_2^2 \Bigr\} \Bigr \rvert \leq C (1 + \sigma)\sqrt{n} \log^{5/2}(n).
	\]
\end{lemma}
Equipped with this lemma, we complete the proof of Lemma~\ref{lem:orthogonal-concentration}.  First, we apply~\citet[Theorem 2]{warnke2016method}---taking the parameter $\gamma_k$ small enough---in conjunction with Lemma~\ref{lem:properties-fdown} and the definition of $\mathcal{S}$ to obtain the inequality
\begin{align} \label{ineq:fdown-deviation}
	\Prob\bigl\{ \lvert f^{\downarrow} - \EE f^{\downarrow} \rvert \geq t/2 \bigr\} \leq 2\exp\Bigl\{-c\frac{t^2n^2}{d(1 + \sigma^2)^2\log^{16}(n)}\Bigr\}.
\end{align}
Consequently, by Lemma~\ref{lem:truncation-probability}, 
\[
\Prob\bigl\{ \lvert f - \EE f \rvert \geq t \bigr\} \leq 2\exp\Bigl\{-c\frac{t^2n^2}{d(1 + \sigma^2)^2\log^{16}(n)}\Bigr\} + Cn^{-15}.
\]
The result then follows by setting
\[
t = C(1 + \sigma^2)\frac{\sqrt{d}\log^{8}(n)}{n},
\]
and applying Lemma~\ref{lem:expectation-leave-out-first-col}.
\qed

\paragraph{Proof of Lemma~\ref{lem:properties-fdown}:} We prove each part in turn.  

\medskip
\noindent \underline{Proof of part (a)}:
We first note the trivial inequality that if $(\bP, \bM) \in \mathcal{S}$, then $\fdown(\bP; \bM) \leq f(\bP; \bM)$ by definition.  Towards showing the reverse inequality, first note that it suffices to consider the constrained infimum over all $\bP'$ with $\rho(\bP, \bP') \leq 1$.  Moreover, note that by Lemma~\ref{lem:lower-truncation}, if $(\bP, \bM), (\bP', \bM) \in \mathcal{S}$ and $\rho(\bP, \bP') \leq 2$, then $\lvert f(\bP') - f(\bP) \rvert \leq C_0  \frac{\log^{15/2}(n)}{n}$.  Consequently, if $\rho(\bP, \bP') \leq 1$, $f(\bP) \leq f(\bP') + 2C_0  \frac{\log^{15/2}(n)}{n} \cdot \rho(\bP, \bP')$.  Taking the infimum of the right hand side over $\bP' \in \mathcal{S}$ with $\rho(\bP, \bP') \leq 2$ yields the desired inequality $f(\bP) \leq \fdown(\bP)$.  

\medskip
\noindent \underline{Proof of part (b):} Note that since $(\bP, \bM) \in \mathcal{S}$, and $\rho(\bP, \bP') \leq 1$, there exists at least one $\widetilde{\bP}$ with $\rho(\widetilde{\bP}, \bP')$ such that $(\widetilde{\bP}, \bM) \in \mathcal{S}$.  Moreover, by the triangle inequality, $\rho(\widetilde{\bP}, \bP) \leq 2$.  Thus, by Lemma~\ref{lem:lower-truncation}, 
\[
\lvert \fdown(\widetilde{\bP}; \bM) - \fdown(\bP; \bM) \rvert \leq C_0 \frac{\log^{15/2}{n}}{n}.
\]
Combining the elements yields the result. \qed

\subsubsection{Proof of Lemma~\ref{lem:norm-concentration}} \label{sec:proof-lem-norm-concentration}
We follow a similar idea as in the proof of Lemma~\ref{lem:orthogonal-concentration}; the main difference is that here we define our functions with arguments as the rows instead of the columns.  That is, we define
\begin{subequations}\label{eq:exchangeable-rows}
	\begin{align}
		\bX &= [\bx_1 \; | \; \dots \; | \; \bx_n]^{\top},
	\end{align}
	with $(\bx_i)_{1 \leq i \leq n} \overset{\mathsf{i.i.d.}}{\sim} \mathsf{N}(0, \bI_d)$.  We then define our metric $\rho: \mathbb{R}^{n \times d} \times \mathbb{R}^{n \times d} \rightarrow \mathbb{N}$ such that $\rho(\bX, \bX')$ counts the number of \emph{rows} in which $\bX$ and $\bX'$ differ.  
\end{subequations} 

Next, viewing the random variables $\bX, R_i, \bmuG^{(i)}$ as functions mapping from $\underbrace{\mathbb{R}^{d} \times \dots \times \mathbb{R}^d}_{n \text{ times}} \times \mathcal{D}^n$ to their respective spaces, we define the set 
\begin{align*}
	\mathcal{S} = \biggl\{ (\bp_1, \bp_2, \dots \bp_n, \bM) \in &\mathbb{R}^d \times \dots \times \mathbb{R}^d \times \mathcal{D}^n: \text{ for all } i \in [n], \quad \lvert R_i(\bp_1, \bp_2, \dots \bp_n, \bM))\rvert \leq C_1(1 + \sigma) \log^{3/2}{n},\\
	&\| \bSig_{i}(\bp_1, \bp_2, \dots \bp_n, \bM)^{-1} \bx_i(\bp_1, \bp_2, \dots \bp_n, \bM) \|_2^2 \leq \frac{C_1 \log{n}}{n},\\ &\text{ and } \; \langle \bcoefX_{t+1}^{(i)}(\bp_1, \bp_2, \dots \bp_n, \bM), \bSig_{i}^{-1} \bx_i(\bp_1, \bp_2, \dots \bp_n, \bM) \rangle \leq \frac{C_1(1 + \sigma)  \log{n}}{n} \biggr\},
\end{align*}
where $\bSig_{i}(\bp_1, \bp_2, \dots \bp_n, \bM) = \sum_{j \neq i} M_{j}^2 \bp_{j} \bp_{j}^{\top}$.  We claim the following bound (deferring the proof to the end of the section) \begin{align}\label{ineq:truncation-norm-concentration}
	\Prob\{ (\bx_1, \bx_2, \dots, \bx_n; \bG) \in \mathcal{S} \} \geq 1 - n^{-15}.
\end{align}
With this in hand, we define the function $f: \mathbb{R}^{n \times d} \times \mathcal{D}^n \rightarrow \mathbb{R}$ as 
\[
f(\bP; \bM) = \| \bcoefX_{t+1}(\bP; \bM) \|_2^2.
\]
We note that if $(\bX, \bG) \in \mathcal{S}$, $(\bX', \bG) \in \mathcal{S}$ and $\rho(\bX, \bX') \leq 1$, then
\begin{align}\label{ineq:f-bound-rows}
	\bigl \lvert f(\bX) - f(\bX') \bigr \rvert \leq \frac{C(1 + \sigma^2) \log^{5}{n}}{n}.
\end{align}
To see this, note that $\bmuG = \bmuG^{(i)} - R_i \bSig_{i}^{-1} \bx_i$,
and similarly for the predictor formed using the data $\bX'$.  Thus, 
\begin{align*}
	\bigl \lvert f(\bX) - f(\bX') \bigr \rvert &= \bigl \lvert R_I^2 \| \bSig_{I}^{-1} \bx_I \|_2^2 - R_I'^2 \| \bSig_{I}^{-1} \bx_I' \|_2^2 - 2 R_I \langle \bx_I, \bSig_{I}^{-1} \bmuG^{(I)} \rangle + 2 R_I' \langle \bx_I', \bSig_I^{-1} \bmuG^{(I)} \rangle \bigr \rvert,
\end{align*}
where we emphasize that $I \in [n]$ denotes the row index in which $\bX$ and $\bX'$ differ.  The inequality~\eqref{ineq:f-bound-rows} then follows from the properties of the set $\mathcal{S}$.  Defining $\fdown$ as in the proof of Lemma~\ref{lem:orthogonal-concentration} and following identical steps yields the result.  We omit the details for brevity.  It remains to prove the inequality~\eqref{ineq:truncation-norm-concentration}.

\paragraph{Proof of the inequality~\eqref{ineq:truncation-norm-concentration}.}
We tackle each in turn, beginning with the residual $R_i$.  

\medskip
\noindent \underline{Bounding the residual $R_i$:}
First, apply the leave one sample out update~\eqref{eq:loo-residual-tool} to see that
\[
R_i = \frac{r_{i, (i)}}{1 +G_i^2 \langle \bx_i, \bSig_{i}^{-1} \bx_i \rangle} = \frac{G_i^2 \langle \bx_i, \bmuG^{(i)}\rangle - G_i y_i}{1 + G_i^2 \langle \bx_i, \bSig_i^{-1} \bx_i \rangle} \leq \bigl \lvert G_i^2 \langle \bx_i, \bmuG^{(i)}\rangle - G_i y_i \bigr \rvert,
\]
where the final inequality follows since $\bSig_i^{-1}$ is PSD so that $G_i^2 \langle \bx_i, \bSig_i^{-1} \bx_i \rangle \geq 0$.
Now, by definition~\eqref{eq:def-mu-i-mu-bar}
\[
\bmuG^{(i)} = \bSig_{i}^{-1} \bX_{(-i)}^{\top} \bG_{(-i)} \by_{(-i)}.
\]
We next invoke Lemma~\ref{lem:eigenvalues-G}(b) to obtain the operator norm bound $\| \bSig_{i}^{-1} \|_{\mathsf{op}} \leq C/n$ with probability at least $1 - e^{-cn}$.  In addition, we apply~\citet[Theorem 6.1]{wainwright2019high} to upper bound the operator norm $\| \bX_{-i} \|_{\mathsf{op}} \leq C\sqrt{n}$ with probability at least $1 - e^{-cn}$.  Moreover, we note that with probability at least $1 - n^{-15}$, $\| \bG\|_{\op} \leq C\sqrt{\log{n}}$.  
Finally, using the high probability operator norm bounds $\| \bX \|_{\mathsf{op}} \leq C\sqrt{n}$ and $\| \bZ \|_{\mathsf{op}} \leq C \sqrt{n}$ in conjunction with the definition of the responses $\by$, we note that $\| \by \|_2 \leq C(1 + \sigma)\sqrt{n}$.  Putting these pieces together, we obtain the high probability bound $\| \bmuG^{(i)} \|_2 \leq (1 + \sigma)\sqrt{\log{n}}$.  Additionally, the coefficient vector $\bmuG^{(i)}$ is independent of the $i$th sample, whence we apply Hoeffding's inequality to obtain 
\[
\max_{i \in [n]} \;\langle \bx_i, \bmuG^{(i)} \rangle \leq C(1 + \sigma) \log{n}, \qquad \text{ with probability at least } \qquad 1 - n^{-15}.
\]
A similar argument implies that $y_i \leq (1 + \sigma)\sqrt{\log{n}}$ with probability at least $1 - n^{-15}$.  Combining the elements, we obtain the bound
\[
R_i \leq C (1 + \sigma)\log^{3/2}{n}, \qquad \text{ with probability at least } \qquad 1 - n^{-15}.
\]

\medskip
\noindent \underline{Bounding $\| \bSig_{i}^{-1} \bx_i \|_2^2$.}
Note that
\[
\| \bSig_{i}^{-1} \bx_i \|_2^2 \leq \| \bSig_{i}^{-1} \|_{\mathsf{op}}^2 \cdot \| \bx_i \|_2^2 \overset{\1}{\leq} \frac{C}{n^2} \cdot \| \bx_i \|_2^2 \leq \frac{C\log{n}}{n},
\]
where step $\1$ follows with probability at least $1 - e^{-cn}$ by Lemma~\ref{lem:eigenvalues-G} and the final step follows with probability at least $1 - n^{-15}$ upon applying Hoeffding's inequality.

\medskip
\noindent \underline{Bounding $\langle \bmuG^{(i)}, \bSig_i^{-1} \bx_i \rangle$.}
Note that $\bx_i$ is independent of both $\bSig_{i}^{-1}$ as well as $\bmuG^{(i)}$.  We thus apply Hoeffding's inequality to obtain the bound
\[
\max_{i \in [n]} \; \langle \bx_i, \bSig_{i}^{-1} \bmuG^{(i)} \rangle \leq C \| \bSig_{i}^{-1} \bmuG^{(i)} \|_2 \sqrt{\log{n}}, \qquad \text{ with probability at least } 1 - n^{-15}.
\]
Once more invoking Lemma~\ref{lem:eigenvalues-G} and recalling the norm bound on the coefficients $\bmuG^{(i)}$ yields the desired result. \qed
\section{Proof of Theorems~\ref{thm:global-conv-linear} and~\ref{thm:global-conv-nonlinear}: Global convergence guarantees}
This section is organized as follows.  First, we outline the common proof strategy.  Then, in Section~\ref{sec:proof-conv-linear}, we employ this strategy to prove Theorem~\ref{thm:global-conv-linear} and in Section~\ref{sec:proof-conv-nonlinear}, we prove Theorem~\ref{thm:global-conv-nonlinear}.  

We begin with some preliminaries.  First, define the filtrations $\mathcal{F}_t$ and $\widetilde{\mathcal{F}}_t$ as
\[
\mathcal{F}_t = \sigma\bigl(\{\parcompX_{s}, \perpcompX_{s}\}_{1 \leq s \leq t}, \{\parcompZ_{s}, \perpcompZ_{s}\}_{1 \leq s \leq t}\bigr) \quad \text{ and } \quad \widetilde{\mathcal{F}}_t = \sigma\bigl(\{\parcompX_{s}, \perpcompX_{s}\}_{1 \leq s \leq t-1}, \{\parcompZ_{s}, \perpcompZ_{s}\}_{1 \leq s \leq t}\bigr),
\]
and define the events $\mathcal{A}_t \in \mathcal{\widetilde{F}}_{t}$ and $\mathcal{B}_t \in \mathcal{F}_t$ as
\begin{align}\label{good-event-A-B}
	\begin{split}
		\mathcal{A}_t &= \Bigl \{ |\parcompZ_{t} - \parcompZ_{t}^{\mathsf{det}} | \leq \frac{C_{1}(\parsigma)}{(\parcompX_{t-1}^{2} + \perpcompX_{t-1}^{2})^{1/2}} \cdot \pardevn, \;\; |\perpcompZ_{t}^{2} - (\perpcompZ_{t}^{\mathsf{det}})^{2} | \leq \frac{C_{1}(1+\sigma^{2})}{\parcompX_{t-1}^{2} + \perpcompX_{t-1}^{2}} \cdot \frac{\perplogn}{\sqrt{n}} \Bigr \}, \quad \text{ and }\\
		\mathcal{B}_t &= \Bigl \{ |\parcompX_{t} - \parcompX_{t}^{\mathsf{det}} | \leq \frac{C_{1}(\parsigma)}{(\parcompZ_{t}^{2} + \perpcompZ_{t}^{2})^{1/2}} \cdot \pardevn, \;\; |\perpcompX_{t}^{2} - (\perpcompX_{t}^{\mathsf{det}})^{2} | \leq \frac{C_{1}(1+\sigma^{2})}{\parcompZ_{t}^{2} + \perpcompZ_{t}^{2}} \cdot \frac{\perplogn}{\sqrt{n}} 
		\Bigr \}.
	\end{split}
\end{align}
Applying Theorem~\ref{thm:one-step} in conjunction with the union bound yields
\begin{align}\label{eq:good-event}
	\Pr\Bigl \{ \bigcap_{t=1}^{T} \mathcal{A}_t \cap \mathcal{B}_t \Bigr\} \geq 1 - 2Tn^{-10}.
\end{align}
Henceforth, we work on this event.  Since we are interested in the ratio $\perpcompX^2/\parcompX^2$, we define the update functions $\ratiomapid: \mathbb{R} \rightarrow \mathbb{R}$ and $\ratiomapbit: \mathbb{R} \rightarrow \mathbb{R}$ as
\begin{subequations}
	\label{eq:ratiomaps}
	\begin{align}
		\ratiomapid(x) &= \frac{1 + \sigma^2}{C(\oversamp)} x + \frac{\sigma^2}{C(\oversamp)}, \quad \text{ and } \\
		\ratiomapbit(x) &= \frac{\pi^{2}}{4}\frac{1+\sigma^{2}}{C(\oversamp) \oversamp^2} \EE \Bigl\{ \frac{\lvert W \rvert \phi(\frac{\lvert W\rvert}{\sqrt{x}})}{C(\oversamp) + W^2}\Bigr\}^{-2} + \frac{1}{C(\oversamp)}  \frac{ C_{3}(\oversamp)}{C_{2}(\oversamp)} -
		\frac{2\oversamp^{-1}}{C(\oversamp)C_{2}(\oversamp)}  \frac{ \EE\Bigl\{ \frac{|W|^{3} \cdot \phi\left(\frac{|W|}{\sqrt{x}}\right) }{(C(\oversamp) + W^{2})^{2}} \Bigr\} }{  \EE\Bigl\{ \frac{|W| \cdot \phi\left(\frac{|W|}{\sqrt{x}}\right)}{C(\oversamp) + W^{2}} \Bigr\} },
	\end{align}
\end{subequations}
where we define $\phi(x) := \int_{0}^{x}e^{-t^2/2}\mathrm{d}t$.  Recalling the updates $\parcompdetX_{t+1}$ and $\perpcompdetX_{t+1}$ from Examples~\ref{ex:lin} and~\ref{ex:nonlin}, straightforward calculation yields the identities
\begin{align}\label{eq:iterated-h}
	\frac{(\perpcompdetX_{t+1})^2}{(\parcompdetX_{t+1})^2} = h_{\psi} \circ h_{\psi} \Bigl(\frac{\perpcompX_{t}^2}{\parcompX_{t}^2}\Bigr), \quad \text{ for } \psi \in \{\mathsf{id}, \mathsf{sgn}\}.
\end{align}
Finally, we define the shorthand
\[\rho = \frac{C(\oversamp)}{1 + \sigma^2}\]
for convenience.

With this notation in hand, we now state two key technical lemmas.  The first shows that---from a random initialization---the parallel component increases geometrically with a rate $C(\oversamp)/(1 + \sigma^2)$, while the perpendicular component remains bounded.   We provide the proof of the following lemma in Section~\ref{sec:proof-lem-initialization}.
\begin{lemma}\label{lem:initialization}
	(a) Suppose the following bounds hold
	\[
	\parcompX_{t} \in \Bigl[\frac{1}{50\sqrt{d}}, 1\Bigr], \quad \perpcompX_{t}^{2} \in [0.5, 2], \quad \text{ and } \quad \frac{\perpcompX_{t}^2}{\parcompX_{t}^2} \geq 20 \frac{C(\oversamp)}{1 + \sigma^2}.
	\]
	Then for both linear and one-bit measurements, there exists a pair of universal, positive constants $(C_0, C_1)$ such that if $C (1+\sigma^{2}) \log^{2}(d) \leq \oversamp \leq \sqrt{d}$, then on the event $\mathcal{A}_{t+1} \cap \mathcal{B}_{t+1}$, we have
	\begin{subequations} \label{eq:strange-dream}
		\begin{align} 
			\frac{\rho}{60} \cdot \parcompX_{t} \leq \parcompX_{t+1} &\leq 5\rho \cdot \parcompX_{t} \quad \text{and} \label{eq:strange-dream-a} \\
			\perpcompX_{t}^{2} - 100\rho \parcompX_{t}^{2}-\frac{C_{1}\perplogd \sqrt{\oversamp}}{\sqrt{d}} \leq \perpcompX_{t+1}^{2} &\leq \perpcompX_{t}^{2} + 40\rho\parcompX_{t}^{2} + \frac{C_{1}\perplogd \sqrt{\oversamp} }{\sqrt{d}}. \label{eq:strange-dream-b}
		\end{align}
		(b) Consequently, if Assumption~\ref{asm:random-linear} holds and we define $T_{\star} := \min\{ t: \alpha_{t + 1} > 1/(30\sqrt{\rho})\}$, then for all $0 \leq t \leq T_{\star}$, we have that Equations~\eqref{eq:strange-dream} hold, and also that 
		\begin{align}
			0.5 \leq \perpcompX_{t}^{2} \leq 1.4. \label{eq:strange-dream-c}
		\end{align}
	\end{subequations}
\end{lemma}
The next lemma controls the deviations of the ratio $\perpcompX_{t+1}^2/\parcompX_{t+1}^2$ around the RHS of~\eqref{eq:iterated-h} when the iterates are in the intermediate region (see Figure~\ref{fig:conv-schematic}). We prove the lemma in Section~\ref{sec:proof-lemma-good-region}.
\begin{lemma}\label{lemma2-good-region}
	Consider $\psi \in \{\mathsf{id}, \mathsf{sgn}\}$ and let $h_{\psi}$ be as in the equation~\eqref{eq:ratiomaps}.  Suppose that $\perpcompX_{t}/\parcompX_{t} \leq 50\sqrt{\rho}$, where $\rho = C(\oversamp)/(1 + \sigma^2)$.  There exists a pair of universal positive constants $(C, C')$ such that if $C(1+\sigma^{2}) \leq \oversamp \leq \sqrt{d}$, then on the event $\mathcal{A}_{t+1} \cap \mathcal{B}_{t+1}$, we have
	\[
	\left| \frac{\perpcompX_{t+1}^{2}}{\parcompX_{t+1}^{2}} - \ratiomappsi \circ \ratiomappsi\left(  \frac{\perpcompX_{t}^{2}}{\parcompX_{t}^{2}} \right)\right| \leq C'\frac{\perplogn(1+\sigma^{2})}{\sqrt{n}}.
	\]
\end{lemma}
Equipped with these lemmas, we are now ready to prove the two theorems. 

\subsection{Convergence in the linear model: Proof of Theorem~\ref{thm:global-conv-linear}}\label{sec:proof-conv-linear}
We execute the proof 
in two stages: (i) from random initialization to a local initialization, under Assumption~\ref{asm:random-linear} and (ii) local refinement under Assumption~\ref{asm:local-linear}. 
We split the analysis into the random initialization region in which $t \leq T_{\star}$ and the local convergence region in which $t > T_{\star}$.  
\paragraph{Stage 1: Random initialization, $t \leq T_\star$.}
Take $d_{0}$ large enough that $0.25C_{1} \geq \log^{8}(d_{0}) /d_{0}^{1/4}$. Then for $d\geq d_{0}$, combining Lemma~\ref{lem:initialization}(b)---in particular, Eqs.~\eqref{eq:strange-dream-b} and~\eqref{eq:strange-dream-c}---with the assumption $\oversamp\leq \sqrt{d}$ yields the lower bound $0.5\perpcompX_{t}^2 \geq C_1\perplogd\sqrt{\oversamp}/\sqrt{d}$, whence we deduce the sandwich relation
\begin{align*}
	0.5\perpcompX_{t}^{2} - 100\rho \parcompX_{t}^{2} \leq \perpcompX_{t+1}^{2} &\leq 1.5\perpcompX_{t}^{2} + 40\rho\parcompX_{t}^{2}.
\end{align*}
We now upper bound the squared ratio as
\begin{align}
	\label{ineq:ub-squared-ratio-lin}
	\frac{\perpcompX_{t+1}^{2}}{\parcompX_{t+1}^{2}} \overset{\1}{\leq} \frac{1.5\perpcompX_{t}^{2} + 40\rho\parcompX_{t}^{2}}{(\rho/60)^{2} \cdot \parcompX_{t}^{2}} \leq \frac{90^{2}}{\rho^{2}}\cdot \frac{\perpcompX_{t}^{2}}{\parcompX_{t}^{2}} + \frac{40\cdot 60^{2}}{\rho} \leq \frac{100^{2}}{\rho^{2}}\cdot \frac{\perpcompX_{t}^{2}}{\parcompX_{t}^{2}},
\end{align}
where step $\1$ follows by applying Lemma~\ref{lem:initialization}(b) and the final inequality follows from the bound $\frac{\perpcompX_{t}^{2}}{\parcompX_{t}^{2}} \geq 0.8\cdot30^{2}\cdot\rho$.  We similarly lower bound the squared ratio as 
\begin{align}
	\label{ineq:lb-squared-ratio-lin}
	\frac{\perpcompX_{t+1}^{2}}{\parcompX_{t+1}^{2}} \geq \frac{0.5\perpcompX_{t}^{2} - 100\rho \parcompX_{t}^{2}}{(5\rho)^{2} \cdot \parcompX_{t}^{2}} = \frac{1}{50\rho^{2}} \cdot \frac{\perpcompX_{t}^{2}}{\parcompX_{t}^{2}} - \frac{4}{\rho} \geq  \frac{1}{100\rho^{2}} \cdot \frac{\perpcompX_{t}^{2}}{\parcompX_{t}^{2}} + \frac{1}{\rho},
\end{align}
where the first inequality follows from Lemma~\ref{lem:initialization}(b) and the final inequality follows from the bound $\frac{\perpcompX_{t}^{2}}{\parcompX_{t}^{2}} \geq 0.8\cdot30^{2}\cdot\rho$.  Putting together inequalities~\eqref{ineq:ub-squared-ratio-lin} and~\eqref{ineq:lb-squared-ratio-lin} yields
\begin{align}\label{ineq:case1-lin}
	\frac{\sigma^{2}}{C(\oversamp)} + \frac{1}{100\rho^{2}} \cdot \frac{\perpcompX_{t}^{2}}{\parcompX_{t}^{2}} \leq \frac{\perpcompX_{t+1}^{2}}{\parcompX_{t+1}^{2}} \leq \frac{100^{2}}{\rho^{2}}\cdot \frac{\perpcompX_{t}^{2}}{\parcompX_{t}^{2}}, \quad \text{ as long as } t \leq T_{\star}.
\end{align}

\paragraph{Stage 2: Local refinement, $t > T_\star$.}
We first claim that 
\begin{align}\label{claim2-proof-thm-linear}
	\frac{\perpcompX_{t}}{\parcompX_{t}} \leq 50\sqrt{\rho}, \quad \text{ as long as } \quad t > T_\star.
\end{align}
Assuming that inequality~\eqref{claim2-proof-thm-linear} holds, we apply Lemma~\ref{lemma2-good-region}---specializing to the linear observation model $\psi = \mathsf{id}$---whence we obtain the inequality 
\begin{align}\label{ineq-id-ratio-error-sharp-bounds}
\Bigl \lvert  \frac{\perpcompX_{t+1}^{2}}{\parcompX_{t+1}^{2}} - \Bigl[\rho^{-2} \cdot \frac{\perpcompX_{t}^{2}}{\parcompX_{t}^{2}} + \frac{\sigma^{2}}{C(\oversamp)} \cdot \bigl(1+\rho^{-1}\bigr) \Bigr] \Bigr\rvert \leq \frac{C'\perplogn(1+\sigma^{2})}{\sqrt{n}}.
\end{align}
Using the assumption that $\sigma^{2}/C(\oversamp) \geq 10C'\perplogn/\sqrt{n}$, $\oversamp\leq \sqrt{d}$, and unrolling inequality~\eqref{ineq-id-ratio-error-sharp-bounds} above yields
\begin{align}\label{ineq:case2-lin}
	\rho^{-2} \cdot \frac{\perpcompX_{t}^{2}}{\parcompX_{t}^{2}} + \frac{\sigma^{2}}{10C(\oversamp)} \leq \frac{\perpcompX_{t+1}^{2}}{\parcompX_{t+1}^{2}} \leq \rho^{-2} \cdot \frac{\perpcompX_{t}^{2}}{\parcompX_{t}^{2}} + \frac{5\sigma^{2}}{C(\oversamp)}.
\end{align}
Finally, putting the inequalities~\eqref{ineq:case1-lin} and~\eqref{ineq:case2-lin} together yields the desired result.  It remains to prove the inequality~\eqref{claim2-proof-thm-linear}. \\

\bigskip
\noindent \underline{Proof of the inequality~\eqref{claim2-proof-thm-linear}.}
We proceed by induction.  For the base case $t = T_\star + 1$, we apply Lemma~\ref{lem:initialization}(b)---in particular, inequality~\eqref{eq:strange-dream-b} with $t = T_\star$---to obtain 
\begin{align*}
	\perpcompX_{T_{\star}+1}^{2} \leq \perpcompX_{T_{\star}}^{2} + 40\rho\parcompX_{T_{\star}}^{2} +\frac{C_{1}\perplogd\sqrt{\oversamp}}{\sqrt{d}} \leq 1.4 + \frac{40}{900} + \frac{C_{1}\perplogd\sqrt{\oversamp}}{\sqrt{d}} \leq 1.6.
\end{align*}
Thus, by definition of $T_\star$, we have $\perpcompX_{T_\star + 1}/\parcompX_{T_\star + 1} \leq 1.6 \cdot 30 \sqrt{\rho} \leq 50\sqrt{\rho}$.  Proceeding to the inductive step, suppose that the claim~\eqref{claim2-proof-thm-linear} holds for some $k > T_\star + 1$.  By the induction hypothesis, $\perpcompX_{k}/\parcompX_{k} \leq 50\sqrt{\rho}$, and we apply Lemma~\ref{lemma2-good-region} to obtain 
\[
\frac{\perpcompX_{k+1}^{2}}{\parcompX_{k+1}^{2}} \leq \frac{1}{\rho^{2}} \cdot \frac{\perpcompX_{k}^{2}}{\parcompX_{k}^{2}} + \frac{2\sigma^{2}}{C(\oversamp)} + \frac{C'\perplogn(1+\sigma^{2})}{\sqrt{n}} \leq (50\sqrt{\rho})^{2},
\]
thereby proving that the induction hypothesis is true at iteration $k + 1$.

\paragraph{Putting together the pieces to prove consequence:} Using inequalities~\eqref{ineq:case1-lin} and~\eqref{ineq:case2-lin}, we obtain that
\[
\frac{1}{100\rho^{2}}\cdot \frac{\perpcompX_{t}^{2}}{\parcompX_{t}^{2}} + \frac{\sigma^{2}}{10C(\oversamp)} \leq \frac{\perpcompX_{t+1}^{2}}{\parcompX_{t+1}^{2}} \leq \frac{100^{2}}{\rho^{2}}\cdot \frac{\perpcompX_{t}^{2}}{\parcompX_{t}^{2}} + \frac{5\sigma^{2}}{C(\oversamp)},\qquad \forall\; 0 \leq t\leq T.
\] 
Applying the above inequality recursively yields that
\[
\bigg(\frac{1}{100\rho^{2}} \bigg)^{t}\cdot \frac{\perpcompX_{0}^{2}}{\parcompX_{0}^{2}} + \frac{\sigma^{2}}{10C(\oversamp)} \leq \frac{\perpcompX_{t}^{2}}{\parcompX_{t}^{2}} \leq \bigg(\frac{100^{2}}{\rho^{2}}\bigg)^{t} \cdot \frac{\perpcompX_{0}^{2}}{\parcompX_{0}^{2}} + \frac{10\sigma^{2}}{C(\oversamp)},\qquad \forall\; 1 \leq t\leq T.
\]
Consequently, for $t \geq \log_{\rho/100}(\frac{\perpcompX_{0}}{\parcompX_{0}\sigma})+1$, we have that $\frac{\perpcompX_{t}^{2}}{\parcompX_{t}^{2}} \lesssim \frac{\sigma^{2}}{C(\oversamp)}$. And for $t \leq  \log_{10\rho}(\frac{\perpcompX_{0}}{\parcompX_{0}\sigma})-1$, we have that $\frac{\perpcompX_{t}^{2}}{\parcompX_{t}^{2}} \geq \sigma^{2} \cdot (10\rho)^{2}$. Since $\log_{\rho/100}(\frac{\perpcompX_{0}}{\parcompX_{0}\sigma}) \asymp \log_{10\rho}(\frac{\perpcompX_{0}}{\parcompX_{0}\sigma}) \asymp \log_{\frac{\oversamp}{1+\sigma^{2}}}(\frac{\perpcompX_{0}}{\parcompX_{0}\sigma})$, we conclude that it takes $\tau = \Theta( \log_{\frac{\oversamp}{1+\sigma^{2}}}(\frac{\perpcompX_{0}}{\parcompX_{0}\sigma}) )$ iterations to get $\frac{\perpcompX_{t}^{2}}{\parcompX_{t}^{2}} \lesssim \frac{\sigma^{2}}{C(\oversamp)}$.
\qed

\paragraph{Proof of the case $\sigma \geq 0$:} Note that using the assumption $\oversamp \leq \sqrt{d}$, unrolling inequality~\eqref{ineq-id-ratio-error-sharp-bounds} above yields
\[
	\frac{\perpcompX_{t+1}^{2}}{\parcompX_{t+1}^{2}} \leq \rho^{-2} \cdot \frac{\perpcompX_{t}^{2}}{\parcompX_{t}^{2}} + \frac{5\sigma^{2}}{C(\oversamp)} + \frac{C\log^{8}(n)}{\sqrt{n}}.
\]
Combining the inequality in the display above with inequality~\eqref{ineq:case1-lin} yields for all $t\geq 0$
\[
\frac{\perpcompX_{t+1}^{2}}{\parcompX_{t+1}^{2}} \leq \frac{100^{2}}{\rho^{2}} \cdot \frac{\perpcompX_{t}^{2}}{\parcompX_{t}^{2}} + \frac{5\sigma^{2}}{C(\oversamp)} + \frac{C\log^{8}(n)}{\sqrt{n}}.
\]
This proves inequality~\eqref{linear-convergence-id-low-noise}.

\subsection{Convergence in the non-linear model: Proof of Theorem~\ref{thm:global-conv-nonlinear}}\label{sec:proof-conv-nonlinear}
We pursue a nearly identical strategy to the proof of Theorem~\ref{thm:global-conv-linear}, splitting the analysis into two stages (i) from random initialization and (ii) local refinement. Here, however, the local refinement stage consists of two substages.  We require the following technical lemma that we prove in Section~\ref{sec:proof-lem-h-sgn-properties}.
\begin{lemma}\label{lem:h-sgn-properties}
	Consider the function $\ratiomapbit$~\eqref{eq:ratiomaps}.  The following hold:
	\begin{itemize}
		\item[(a)] If $x \geq 100$, then 
		\[
		\frac{\pi^{2}}{8} \cdot \frac{1+\sigma^{2}}{C(\oversamp)} \cdot x + \frac{20(1+\sigma^{2})}{C(\oversamp)} \leq \ratiomapbit(x) \leq \frac{\pi^{2}}{2} \cdot \frac{1+\sigma^{2}}{C(\oversamp)} \cdot x + \frac{20(1+\sigma^{2})}{C(\oversamp)}.
		\]
		\item[(b)] If $0 \leq x < 100$, then 
		\[
		\ratiomapbit(x) \leq 50 \pi^{2} \cdot \frac{1+\sigma^{2}}{C(\oversamp)}.
		\]
		\item[(c)] There exists a universal constant $c>0$ such that $\ratiomapbit(x) \geq \frac{c(1+\sigma^{2})}{C(\oversamp)}$ for all $x > 0$.
	\end{itemize}
\end{lemma}

Recall the iteration $T_\star$ defined in Lemma~\ref{lem:initialization}.

\paragraph{Stage 1: Random initialization, $t \leq T_\star$.} Proceeding the identical proof of inequalities~\eqref{ineq:ub-squared-ratio-lin} and~\eqref{ineq:lb-squared-ratio-lin}, we obtain that
\begin{align}\label{ineq:random-init-nonlinear-ub}
	\frac{1}{100\rho^{2}} \cdot \frac{\perpcompX_{t}^{2}}{\parcompX_{t}^{2}} + \frac{1}{\rho} \leq \frac{\perpcompX_{t+1}^{2}}{\parcompX_{t+1}^{2}} \leq \frac{90^{2}}{\rho^{2}}\cdot \frac{\perpcompX_{t}^{2}}{\parcompX_{t}^{2}} + \frac{40\cdot 60^{2}}{\rho},
\end{align}
which immediately proves Theorem~\ref{thm:global-conv-nonlinear} for $t \leq T_{\star}$.  

\paragraph{Stage 2: Local refinement, $t > T_\star$.} Unlike the proof of Theorem~\ref{thm:global-conv-linear}, this stage splits into two further sub-stages.  Our first claim---whose proof is identical to that of the inequality~\eqref{claim2-proof-thm-linear}---is that
\begin{align} \label{claim2-proof-thm-nonlinear}
	\frac{\perpcompX_{t}}{\parcompX_{t}} \leq 50\sqrt{\rho}, \qquad \text{ as long as } \qquad t > T_\star.
\end{align}
Lemma~\ref{lemma2-good-region} then yields
\begin{align}
	\label{ineq:local-region-nonlinear}
	\Bigl \lvert \frac{\perpcompX_{t+1}^2}{\parcompX_{t+1}^2} - \ratiomapbit \circ \ratiomapbit \Bigl(\frac{\perpcompX_{t}}{\parcompX_{t}}\Bigr) \Bigr \rvert \leq \frac{C_1 \perplogn(1+\sigma^{2})}{\sqrt{n}} = o_{d}\Big( \frac{1+\sigma^{2}}{C(\oversamp)}\Big),
\end{align}
where the last step follows from $ \oversamp \leq \sqrt{d}$ and $C(\oversamp) \asymp \oversamp = n/d$. We split the rest of the proof into three cases.

\bigskip
\noindent\underline{Case 1: $t > T_\star$, $\perpcompX_{t}/\parcompX_{t} \geq 10$, and $\ratiomapbit(\perpcompX_t^2/\parcompX_t^2) > 100$.}  
We apply Lemma~\ref{lem:h-sgn-properties}(a) to obtain the sandwich relation
\begin{align*}
	\frac{\pi^{4}}{64} \cdot \frac{1}{\rho^{2}} \cdot \frac{\perpcompX_{t}^{2}}{\parcompX_{t}^{2}} + \frac{20}{\rho} \leq \ratiomapbit \circ \ratiomapbit \bigg(\frac{\perpcompX_{t}^{2}}{\parcompX_{t}^{2}}\bigg) \leq \frac{\pi^{4}}{4} \cdot \frac{1}{\rho^{2}} \cdot \frac{\perpcompX_{t}^{2}}{\parcompX_{t}^{2}} + \frac{30}{\rho}.
\end{align*}
Combining with inequality~\eqref{ineq:local-region-nonlinear} and the assumption $\oversamp \leq \sqrt{d}$, we obtain that
\begin{align*}
	\frac{\pi^{4}}{64} \cdot \frac{1}{\rho^{2}} \cdot \frac{\perpcompX_{t}^{2}}{\parcompX_{t}^{2}} + \frac{10}{\rho} \leq \frac{\perpcompX_{t+1}^{2}}{\parcompX_{t+1}^{2}} \leq \frac{\pi^{4}}{4} \cdot \frac{1}{\rho^{2}} \cdot \frac{\perpcompX_{t}^{2}}{\parcompX_{t}^{2}} + \frac{35}{\rho},
\end{align*}
which completes the proof for this case.  

\medskip
\noindent\underline{Case 2: $t > T_\star$, $\perpcompX_{t}/\parcompX_{t} \geq 10$, and $\ratiomapbit(\perpcompX_t^2/\parcompX_t^2) \leq 100$.} Applying Lemma~\ref{lem:h-sgn-properties}(b) in conjunction with the inequality~\eqref{ineq:local-region-nonlinear} and the assumption $\oversamp \leq \sqrt{d}$ yields the desired upper bound
\[
\frac{\perpcompX_{t+1}^{2}}{\parcompX_{t+1}^{2}} \leq \frac{55\pi^{2}}{\rho}.
\]
Similarly, applying Lemma~\ref{lem:h-sgn-properties}(c) yields that there exists a universal constant $c$ such that
\[
\frac{\perpcompX_{t+1}^{2}}{\parcompX_{t+1}^{2}} \geq c \cdot \frac{1}{\rho} \geq \frac{c}{2} \cdot \frac{1}{50^{2}} \cdot \frac{1}{\rho^{2}} \cdot \frac{\perpcompX_{t}^{2}}{\parcompX_{t}^{2}} + \frac{c}{2} \cdot \frac{1}{\rho},
\]
where in the last step we use $\perpcompX_{t}/\parcompX_{t} \leq 50\sqrt{\rho}$. Putting the two pieces together concludes the proof for this case.

\medskip
\noindent\underline{Case 3: $t > T_\star$ and $\perpcompX_{t}/\parcompX_{t} \leq 10$.}
Following the exact same logic as in Case 2, we obtain the same result in Case 2.
\medskip

\paragraph{Putting together the pieces to prove consequence:} Putting the result of Stage 1 and Stage 2 together, we obtain that there exists a universal constant $c$ such that
\[
\frac{c}{500} \cdot \frac{1}{\rho^{2}} \cdot \frac{\perpcompX_{t}^{2}}{\parcompX_{t}^{2}} + \frac{c}{2} \cdot \frac{1}{\rho} \leq \frac{\perpcompX_{t+1}^{2}}{\parcompX_{t+1}^{2}} \leq \frac{90^{2}}{\rho^{2}} \cdot \frac{\perpcompX_{t}^{2}}{\parcompX_{t}^{2}} + \frac{40\cdot 60^{2}}{\rho},\qquad \forall\; 0 \leq t\leq T.
\]
Applying the above inequality recursively yields that
\[
\bigg(\frac{c}{500\rho^{2}}\bigg)^{t}  \cdot \frac{\perpcompX_{0}^{2}}{\parcompX_{0}^{2}} + \frac{c}{2} \cdot \frac{1}{\rho} \leq \frac{\perpcompX_{t}^{2}}{\parcompX_{t}^{2}} \leq \bigg( \frac{90^{2}}{\rho^{2}} \bigg)^{t} \cdot \frac{\perpcompX_{0}^{2}}{\parcompX_{0}^{2}} + \frac{80\cdot 60^{2}}{\rho} ,\qquad \forall\; 1 \leq t\leq T.
\]
Consequently, for $t \geq \log_{\rho/90}(\perpcompX_{0}/\parcompX_{0}) + 1$, we obtain that $\frac{\perpcompX_{t}^{2}}{\parcompX_{t}^{2}} \lesssim 1/\rho = \frac{1+\sigma^{2}}{C(\oversamp)}$. And for $t \leq \log_{(500/c)^{1/2}\rho}(\perpcompX_{0}/\parcompX_{0})$, we obtain that $\frac{\perpcompX_{t}^{2}}{\parcompX_{t}^{2}} \geq 1$. Since $\log_{(500/c)^{1/2}\rho}(\perpcompX_{0}/\parcompX_{0}) \asymp \log_{\rho/90}(\perpcompX_{0}/\parcompX_{0}) \asymp \log_{\frac{\oversamp}{1+\sigma^{2}}}(\perpcompX_{0}/\parcompX_{0})$, we conclude that it takes $\tau = \Theta(\log_{\frac{\oversamp}{1+\sigma^{2}}}(\perpcompX_{0}/\parcompX_{0}))$ iterations to get $\perpcompX_{\tau}^{2}/\parcompX_{\tau}^{2} \lesssim \frac{1+\sigma^{2}}{C(\oversamp)}$.
\qed

\section{Discussion}
We analyzed the alternating minimization algorithm for generalized rank one matrix estimation with Gaussian measurements, demonstrating sharp linear convergence from a random initialization.  The bedrock of our analysis is a two-dimensional deterministic one-step update which provides a good approximation for natural error metrics---such as the angular estimation error---of the high-dimensional iterates.  Crucially, we showed that our deterministic one-step approximations are accurate up to fluctuations of the order $\widetilde{\order}(n^{-1/2})$, thus enabling our analysis through the entire trajectory. \revision{We note that our techniques for proving non-asymptotic concentration bounds around the deterministic one-step predictions can be applied to the models and algorithms in~\citet{chandrasekher2021sharp}, as the problems considered in that paper have polynomially growing link functions, and the algorithms are either alternating minimization or subgradient descent. The techniques introduced in this paper would improve the deviation rates from $n^{-1/4}$~\citep[Theorem 1]{chandrasekher2021sharp} to $n^{-1/2}$.} Let us conclude with a few intriguing open questions and extensions.  

First, our analysis crucially required a sample-splitting assumption.  This assumption can be interpreted as a large-batch assumption in which $N$ samples are split into batches of size $d \oversamp$, with $\oversamp > 1$.  As noted in the discussion following Theorem~\ref{thm:global-conv-linear}, such an assumption imposes a $\log{d}$ overhead in sample complexity.  It would be interesting to understand whether the alternating minimization algorithm can be analyzed without this sample-splitting assumption---thereby using all of the data at each iteration---using either a further leave-one-out mechanism~\citep{chen2019gradient} or using tools derived in the context of approximate message passing~\citep{berthier2020state}. \revision{We note that as shown in Figure~\ref{figure-no-sample-split}, the deterministic one-step predictions in Eq.~\eqref{eq-equivalence} are no longer exact when the algorithm is run without sample-splitting. The eventual statistical error, however, is exactly predicted, which is a phenomenon worth exploring in future work.}

\begin{figure}[!h]
	\centering
	\begin{subfigure}[b]{0.48\textwidth}
		\centering
		\includegraphics[scale=0.5]{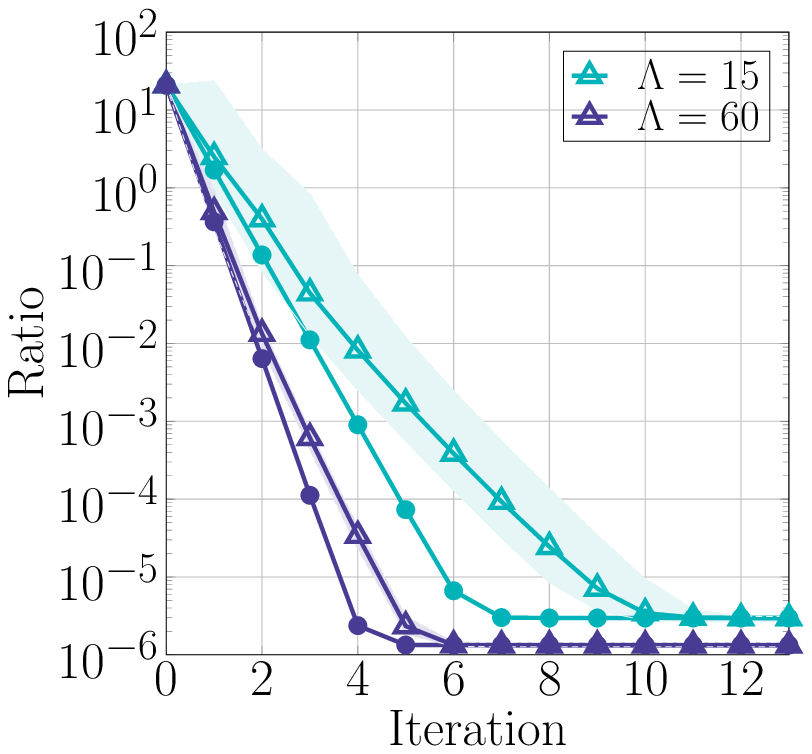}
		\caption{$\psi(w) = w$, $d = 400$, $\sigma = 10^{-5}$.} 
	\end{subfigure}
	\caption{Comparisons of the empirical behavior of the AM algorithm without sample-splitting and the deterministic one-step predictions with sample-splitting. Hollow triangular marks denote the median of the ratio over the 30 independent trials and filled-in circular marks denote the deterministic one-step predictions. Shaded envelopes denote the range over the 30 independent trials.} 
	\label{figure-no-sample-split}
\end{figure}

Second, much of our analysis was enabled by the Gaussian assumption on the sensing vectors $\bx_i$ and $\bz_i$.  While this is a very specific assumption, it has been shown in many cases that predictions such as those in Section~\ref{sec:one-step-main} enjoy a universality property~\citep[to name a few]{chatterjee2006generalization,bayati2015universality,montanari2022universality}.  Do similar guarantees hold for the iterative algorithm considered here? Moreover, are the fast convergence rates derived in Section~\ref{sec:convergence-main} robust to a larger class of random ensembles?

Third, we believe that the leave-one-out method developed here (which is in turn based on that of~\citet{el2013robust}) is broadly applicable to other models and algorithms, going beyond rank one estimation and alternating minimization. A candidate algorithm to analyze is a related alternating iteration developed in the context of semi-parametric single-index models to improve upon the estimation performance in near-noiseless problems~\citep{pananjady2021single}.
More broadly, an interesting direction is to provide a mechanism by fine-grained comparisons between different algorithms that can be carried out for any statistical model. Can the deterministic one-step updates of Section~\ref{sec:one-step-main} be used to guide the choice of practical algorithms for a given data analysis task?

Fourth, in a more speculative direction, Theorems~\ref{thm:global-conv-linear} and~\ref{thm:global-conv-nonlinear}---taken together with the deterministic one-step updates---open the door to a type of algorithmic model selection in the rank one matrix recovery problem.  In particular, although the convergence rates of AM (in the sense of optimization error) are of the same order for both the identity and one-bit models, the precise constants defining the rate of convergence differ between the two models. We illustrate this phenomenon in Figure~\ref{fig:linear-bit}. While we do not capture these constants exactly in Theorems~\ref{thm:global-conv-linear} and~\ref{thm:global-conv-nonlinear}, our deterministic one-step predictions of Theorem~\ref{thm:one-step} do capture the precise constants.  Thus, when the statistician receives data from the generalized rank one model but does not know the non-linearity $\psi$, a possible strategy for model selection consists of \mbox{(i.)} running the AM algorithm to convergence, \mbox{(ii.)} computing the empirical convergence rate and \mbox{(iii.)} selecting the model by comparing the empirical convergence rate to the convergence rate predicted by the deterministic one-step updates. We leave a detailed investigation of such a procedure to future work.
\begin{figure}[ht!]
	\centering
	\includegraphics[scale=0.5]{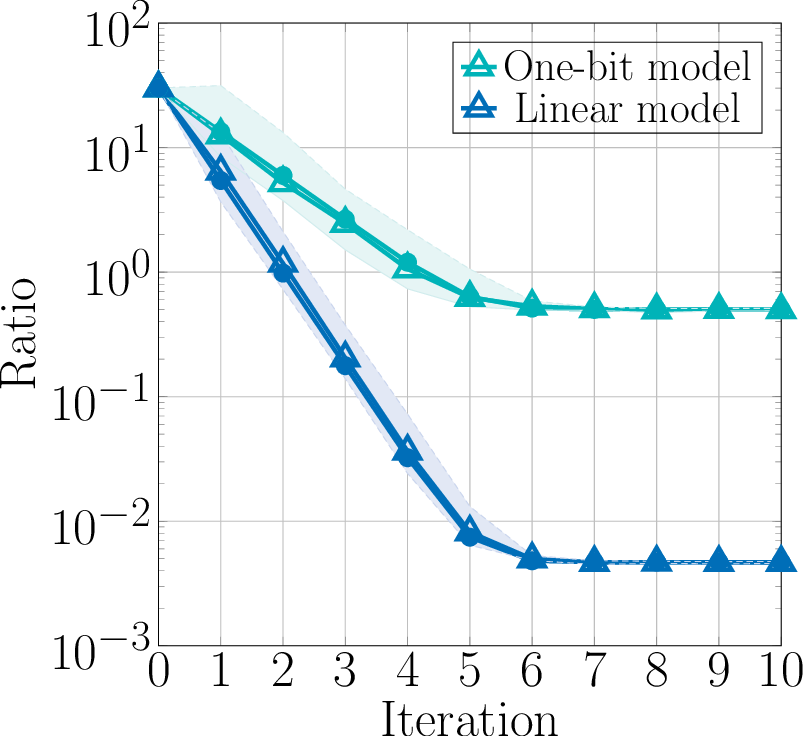}
	\caption{Comparison of empirical behavior of AM for linear observation model~\eqref{eq:model} with $\psi(w) = w$ and for the one-bit observation model~\eqref{eq:model} with $\psi(w) = \mathsf{sign}(w)$.  Note that the convergence rates---while equivalent in order---differ in the exact constants.  Moreover, the deterministic one-step predictions (circular marks) are able to distinguish the rates up to these constants, opening the door to algorithmic model selection.}
	\label{fig:linear-bit}
\end{figure}

Finally, we believe that the general low-rank case can still in principle be analyzed via state evolution updates that can be derived in a similar spirit to those showcased here, but it is important to track a larger number of states, i.e., $\mathcal{O}(k^2)$ state variables if the underlying matrix is rank-$k$. In addition, if the rank is greater than $1$, then one needs to carefully handle identifiability issues when defining the state evolution. 

Let us illustrate using the rank-2 case for simplicity and consider the ground truth matrix being
\[
	\boldsymbol{M}_{\star} = \bcoefX_{\star}^{1} (\bcoefZ_{\star}^{1})^{\top} + \bcoefX_{\star}^{2} (\bcoefZ_{\star}^{2})^{\top}  = \begin{bmatrix} \bcoefX_{\star}^{1} \;\vert\; \bcoefX_{\star}^{2} \end{bmatrix} 
	\begin{bmatrix} \bcoefZ_{\star}^{1} \; \vert \; \bcoefZ_{\star}^{2} \end{bmatrix}^{\top}.
\]
Note that in this case, the matrix $\begin{bmatrix} \bcoefX_{\star}^{1} \; \vert \; \bcoefX_{\star}^{2} \end{bmatrix}$ is not identifiable since $\begin{bmatrix} \bcoefX_{\star}^{1} \; \vert \;  \bcoefX_{\star}^{2} \end{bmatrix} \boldsymbol{U}$ can also be a true left singular matrix for any unitary matrix $\boldsymbol{U} \in \mathbb{R}^{2\times 2}$. 

Nevertheless, if there is no nonlinearity, then the matrix estimate $\boldsymbol{M}_{\star}$ is identifiable given a large enough sample size. To provide a sketch of how one might analyze the AM algorithm for this case, suppose we have singular vector estimates $\{\bcoefX_{t}^{\ell}, \bcoefZ_{t}^{\ell}\}_{\ell = 1,2}$ at the $t$-th iteration. Now define for each $\ell,k \in\{1,2\}$ a new collection of state variables
\begin{align*}
		\parcompX_{t}^{\ell,k} &= \langle \bcoefX_{t}^{\ell}, \bcoefX_{\star}^{k} \rangle,\quad \perpcompX_{t}^{\ell} = \big\| \bP_{\mathsf{span}\{\bcoefX_{\star}^{1}, \bcoefX_{\star}^2\}}^{\perp} \bcoefX_{t}^{\ell} \big\|_{2}, \quad \eta_{t} = \big\| \bP_{\mathsf{span}\{\bcoefX_{\star}^{1}, \bcoefX_{\star}^2\}}^{\perp} (\bcoefX_{t}^{1} + \bcoefX_{t}^{2}) \big\|_{2},\\
		\parcompZ_{t}^{\ell,k} &= \langle \bcoefZ_{t}^{\ell}, \bcoefZ_{\star}^{k} \rangle,\quad\; \perpcompZ_{t}^{\ell} = \big\| \bP_{\mathsf{span}\{\bcoefZ_{\star}^{1}, \bcoefZ_{\star}^2\}}^{\perp} \bcoefZ_{t}^{\ell} \big\|_{2},\quad \;\; \widetilde{\eta}_{t} = \big\| \bP_{\mathsf{span}\{\bcoefZ_{\star}^{1}, \bcoefZ_{\star}^2\}}^{\perp} (\bcoefZ_{t}^{1} + \bcoefZ_{t}^{2}) \big\|_{2}.
\end{align*}
The Frobenius norm distance $\| \boldsymbol{M}_* - \boldsymbol{M}_t \|_F^2$ can be expressed purely using the state variables defined in the display above. Moreover, using the techniques of this paper\footnote{To see why, note that since the distribution of sensing vectors is unitarily invariant, we can assume $\bcoefX_{\star}^{k} = \boldsymbol{e}_{k}$, the $k$-th standard basis vector. Then each parallel component of the state corresponds to an entry, e.g., $\parcompX_{t}^{\ell,k} = \bcoefX_{t}^{\ell}(k)$, while each perpendicular component corresponds to a sum of squares of entries, e.g., $(\perpcompX_{t}^{1})^2 = \sum_{k=3}^{d} (\bcoefX_{t}^{1}(k))^2$. Since each step of the AM algorithm solves a linear least squares problem, one can use the leave-one-out technique to obtain a closed-form expression for each entry of the next iterate $\bcoefX_{t+1}^{\ell}(k)$ (see, e.g., Eq.~\eqref{eq:k-coord}). From these quantities, the deterministic predictions of the state in the next iteration $(\parcompX_{t+1}^{\ell,k},\perpcompX_{t+1}^{\ell},\eta_{t+1})_{k, \ell = 1,2}$ can be derived.}, the state variables $(\parcompX_{t+1}^{\ell,k},\perpcompX_{t+1}^{\ell},\eta_{t+1})_{k, \ell = 1,2}$ at the next iteration can be predicted from the above state variables at iteration $t$, so that $\| \boldsymbol{M}_* - \boldsymbol{M}_{t+1} \|_F^2$ can be expressed in terms of $\| \boldsymbol{M}_* - \boldsymbol{M}_t \|_F^2$. Having said that, the analysis of the algorithm through the state evolution update would be significantly more complex than in the rank-1 case considered in this paper, and presents an important future direction.

\section*{Acknowledgments}
We thank the anonymous reviewers for their insightful comments, which improved the scope and presentation of the paper.

\section*{Funding}
KAC was at Stanford University when part of this work was performed, where he was supported in part by a National Science Foundation Graduate Research Fellowship and the Sony Stanford Graduate Fellowship. ML and AP were supported in part by the National Science Foundation through grants CCF-2107455 and DMS-2210734, and by research awards from Adobe, Amazon, and Mathworks.

\bibliographystyle{abbrvnat}
\bibliography{refs-blind-deconvolution,refs-intro}

\normalsize

\appendix
\section{Auxiliary proofs for the one-step updates} \label{sec:deviation-aux}
This section is organized as follows: In Section~\ref{app:per-coord}, we provide the proof of Claim~\eqref{eq:k-coord}; in Section~\ref{sec:proof-orthogonal-concentration-deferred-lemmas}, we provide the proofs of supplementary lemmas for the concentration of the orthogonal component; in Section~\ref{sec:proof-loo-tools}, we provide the derivations of the leave one out tools used in the proof of the orthogonal component in Section~\ref{sec:proof-orthogonal}; and finally in Section~\ref{sec:proof-det-equivalence}, we prove the equivalent representations of the deterministic updates $\parcompdetX_{t+1}$ and $\perpcompdetX_{t+1}$.

\subsection{Proof of Claim~\eqref{eq:k-coord}} \label{app:per-coord}

Recall that the update $\bcoefX_{t+1}$~\eqref{eq:AM} is defined as the minimizer of a least-squares cost.  Thus, it satisfies the KKT condition
\[
\bX^{\top} \bW_{t+1}^2 \bX\bcoefX_{t+1} - \bX^{\top} \bW_{t+1} \by = 0.
\]
Note that $\bX \bcoefX_{t+1} = \Xcol{k} \coefX_{t+1, k} + \Xloco{k} \bcoefX_{t+1, \backslash k}$, so that the above display can be written as
\begin{align} \label{eq:aladdin-system}
\coefX_{t+1, k} \cdot \bX^{\top} \bW_{t+1}^2 \Xcol{k} + \bX^{\top} \bW_{t+1}^2 \Xloco{k}\bcoefX_{t+1, \backslash k} - \bX^{\top} \bW_{t+1} \by = 0.
\end{align}
Eq.~\eqref{eq:aladdin-system} is a $d$-dimensional linear system in $d$ variables. Separating the $k$-th equation from the remaining $(d-1)$ equations, we have
\begin{subequations}
	\begin{align}
		\coefX_{t+1, k} \cdot \langle \Xcol{k}, \bW_{t+1}^2 \Xcol{k} \rangle - \langle \Xcol{k}, \bW_{t+1} \by \rangle + \langle \Xcol{k}, \bW_{t+1}^2 \Xloco{k} \bcoefX_{t+1, \backslash k} \rangle &= 0 \label{eq:KKT-1dim}\\
		\coefX_{t+1, k} \cdot \Xloco{k}^{\top} \bW_{t+1}^2 \Xcol{k} + \Xloco{k}^{\top} \bW_{t+1}^2 \Xloco{k}\bcoefX_{t+1, \backslash k} - \Xloco{k}^{\top} \bW_{t+1} \by &= 0.\label{eq:KKT-rest}
	\end{align}
\end{subequations}
Multiplying $\bigl( \Xloco{k}^{\top} \bW_{t+1}^2 \Xloco{k}\bigr)^{-1}$ on both sides of equation~\eqref{eq:KKT-rest} and re-arranging terms yields
\begin{align} \label{eq:theta-rest-1}
	\bcoefX_{t+1, \backslash k} = \bigl( \Xloco{k}^{\top} \bW_{t+1}^2 \Xloco{k}\bigr)^{-1} \Xloco{k}^{\top} \bW_{t+1} \by - \coefX_{t+1, k} \cdot \bigl( \Xloco{k}^{\top} \bW_{t+1}^2 \Xloco{k}\bigr)^{-1} \Xloco{k}^{\top} \bW_{t+1}^2 \Xcol{k}.
\end{align}
Substituting the characterization of $\bcoefX_{t+1, \backslash k}$~\eqref{eq:theta-rest-1} into the one-dimensional equation~\eqref{eq:KKT-1dim} and re-arranging terms yields the desired result. \qed

\subsection{Deferred proofs for the orthogonal component} \label{sec:proof-orthogonal-concentration-deferred-lemmas}
This section is dedicated to the proofs of Lemma~\ref{lem:truncation-probability}, provided in Section~\ref{sec:proof-lem-truncation-probability}; Lemma~\ref{lem:lower-truncation}, provided in Section~\ref{sec:proof-lem-lower-truncation}; and Lemma~\ref{lem:expectation-leave-out-first-col}, provided in Section~\ref{sec:proof-lem-expectation-leave-out-first-col}.

\subsubsection{Proof of Lemma~\ref{lem:truncation-probability}} \label{sec:proof-lem-truncation-probability}  
We bound the probability of each $(\Xcol{1}, \Xcol{2}, \dots, \Xcol{d}; \bG) \in \mathcal{S}_i$ in turn, and then complete the proof by applying the union bound.

We first turn bound
\[
	\mathbb{P}\{(\Xcol{1}, \Xcol{2}, \dots, \Xcol{d}; \bG) \in \mathcal{S}_1\}.
\]
We provide the proof for a given index $\ell$, noting that an identical argument yields the same guarantee for $\bu_{\ell_1, \ell_2}$ and $\bu_{\ell_1, \ell_2, \ell_3}$.  Thus, we exploit idempotence of the projection matrix $\bS_{\ell}$ to expand
\[
\| \bu_{\ell} \|_2^2 = \bigl \langle \Xcol{\ell}, \bG \bS_{\backslash \ell} \bG \Xcol{\ell} \bigr \rangle.  
\]
Normalizing by the sum of squares $\parcompZ_{t+1}^2 + \perpcompZ_{t+1}^2$ and applying Lemma~\ref{lem:k-coord-denom} followed by a union bound yields the result.\\\\ 

We then turn to bound 
\[
	\Prob\{ (\Xcol{1}, \Xcol{2}, \dots, \Xcol{d}; \bG) \in \mathcal{S}_2 \}.
\]
We again provide the proof for the single-index version $\langle \bu_{\ell_1}, \bu_{\ell_2} \rangle$, noting that the multi-index statements are shown in an identical manner.  Using the definition of $\bu_{\ell}$~\eqref{eq:u-notation} and applying the rank one update~\eqref{eq:recursive-S}, we write
\begin{align*}
	\langle \bu_{\ell_1}, \bu_{\ell_2} \rangle &= (\Xcol{\ell_1})^{\top} \bG\bS_{\backslash \ell_1} \bS_{\backslash \ell_2} \bG \Xcol{\ell_2} \\
	&= (\Xcol{\ell_1})^{\top} \bG\bS_{\backslash \ell_1, \ell_2} \bS_{\backslash \ell_2} \bG \Xcol{\ell_2} - \frac{1}{\| \bu_{\ell_1, \ell_2}\|_2^2} (\Xcol{\ell_1})^{\top} \bG \bu_{\backslash \ell_1, \ell_2} \bu_{\backslash \ell_1, \ell_2}^{\top}\bS_{\backslash \ell_2} \bG \Xcol{\ell_2}
\end{align*}
Expanding once more, we write the second term as the product of 
\[
\frac{1}{\| \bu_{\ell_1, \ell_2}\|_2^2}, \qquad (\Xcol{\ell_1})^{\top} \bG\bS_{\backslash \ell_1, \ell_2}  \bG \Xcol{\ell_2}, \qquad \text{ and } \qquad (\Xcol{\ell_2})^{\top} \bG\bS_{\backslash \ell_1, \ell_2} \bS_{\backslash \ell_2} \bG \Xcol{\ell_2}
\]
Applying Hoeffding's inequality in conjunction with the operator norm bound $\| \bG\|_{\op} \leq C\sqrt{\log{n}}$---which holds with probability at least $n^{-20}$---yields
\[
(\Xcol{\ell_1})^{\top} \bG\bS_{\backslash \ell_1} \bS_{\backslash \ell_1, \ell_2} \bG \Xcol{\ell_2} \leq C\bigl \| \bG \bS_{\backslash \ell_2} \bS_{\backslash \ell_1, \ell_2} \bG \Xcol{\ell_1} \bigr \|_2 \sqrt{\log{n}} \leq C \sqrt{n} \log^{3/2}{n},
\]
where in the final inequality we used the fact that the norm of a standard Gaussian vector is bounded as $\sqrt{n}$ (with probability $1 - n^{-25}$) and each inequality holds with probability at least $1 - n^{-25}$.  An identical argument yields an identical bound on the quantity $(\Xcol{\ell_1})^{\top} \bG\bS_{\backslash \ell_1, \ell_2}  \bG \Xcol{\ell_2}$.  Finally, note that the inclusion $(\Xcol{1}, \Xcol{2}, \dots, \Xcol{d}; \bG) \in \mathcal{S}_1$ holds with probability at least $1- n^{-20}$, whence $\| \bu_{\ell_1, \ell_2} \|_2^2 \geq c_1 n$ and by a similar argument to the previous step, $(\Xcol{\ell_2})^{\top} \bG\bS_{\backslash \ell_1, \ell_2} \bS_{\backslash \ell_2} \bG \Xcol{\ell_2} \leq C_1n$.  Applying the union bound and putting the pieces together yields the result. \\\\ 

We then turn to bound
\[
	\Prob\{ (\Xcol{1}, \Xcol{2}, \dots, \Xcol{d}; \bG) \in \mathcal{S}_3\}.
\]
We use an identical argument to the bound on $\Prob\{ (\Xcol{1}, \Xcol{2}, \dots, \Xcol{d}; \bG) \in \mathcal{S}_2\}$, using the fact that the conditional distribution $\by \mid \bG, \bV$ is sub-Gaussian, (by Assumption~\ref{assptn:Y-psi}).  For brevity, we omit the details.\\\\  

We finally turn to bound 
\[
	\Prob\{ (\Xcol{1}, \Xcol{2}, \dots, \Xcol{d}; \bG) \in \mathcal{S}_4\}.
\]
We restrict ourselves to proving the bound on the second term as the first follows from identical steps.  First, note that for any symmetric matrix $\bA \in \mathbb{R}^{n\times n}$ and vectors $\bu, \bv \in \mathbb{R}^n$,
\begin{align} \label{eq:E4-decomp}
	\bigl \langle \bu, \bA \bv \bigr \rangle = \frac{1}{2} \cdot \Bigl[ \bigl \langle \bu + \bv, \bA (\bu + \bv) \bigr \rangle - \bigl \langle \bu, \bA \bu \bigr \rangle - \bigl \langle \bv, \bA\bv \bigr\rangle \Bigr].
\end{align}
Note that identifying $\bu = \by -\beps - \EE \by, \bv = \Xcol{\ell}$, and $\bA = \bG\sum_{k \neq \{1, \ell\}} \bu_{1, \ell, k} \bu_{1, \ell, k}^{\top}\bG$, each of the three terms on the RHS is a quadratic form in zero-mean sub-Gaussian random vectors (conditionally on the matrices $\bG, \bV$).  We remark that removing the noise component in $\bu$ is without loss of generality as $\beps$ is independent of all other randomness in the problem.  The noise dependence is easily recovered from the ensuing bounds.  Each term is bounded in the same way, so without loss of generality, we restrict ourselves to studying the quadratic form $\langle \widebar{\by}, \bG\sum_{k \neq \{1, \ell\}} \bu_{1, \ell, k} \bu_{1, \ell, k}^{\top} \bG\widebar{\by}\rangle$, where $\widebar{\by} = \by - \beps - \EE \by$.  
Proceeding, we apply the Hanson--Wright inequality (noting that for either choice of the function $\psi$, the vector $\widebar{\by}$ is at most $(\| \bG \|_{\op} + \| \bV \|_{\op})$ sub-Gaussian conditionally on $\bG, \bV$) to obtain
\begin{align} \label{ineq:HW-bound-A-4}
	&\Prob\biggl\{\Bigl \lvert \by^{\top} \bG\Bigl[\sum_{k \neq \{1, \ell\}} \bu_{1, \ell, k}\bu_{1, \ell, k}^{\top}\Bigr]\bG\by - \EE\Bigl\{\by^{\top} \bG\Bigl[\sum_{k \neq \{1, \ell\}} \bu_{1, \ell, k}\bu_{1, \ell, k}^{\top}\Bigr]\bG\by \Bigm \vert \bG, \bV \Bigr\} \Bigr \rvert \geq t \Bigm \lvert \bG, \bV \biggr\} \nonumber\\
	&\qquad  \qquad\leq 2\exp\biggl\{ - c\min\biggl(\frac{t^2 \cdot (\| \bG \|_{\op} + \| \bV \|_{\op})^{-4}}{\bigl \|\bG\sum_{k \neq \{1, \ell\}} \bu_{1, \ell, k}\bu_{1,\ell, k}^{\top}\bG\bigr \|_{F}^2}, \frac{t \cdot (\| \bG \|_{\op} + \| \bV \|_{\op})^{-2}}{\bigl \|\bG\sum_{k \neq \{1, \ell\}} \bu_{1, \ell, k}\bu_{1, \ell, k}^{\top} \bG\bigr \|_{\mathsf{op}}}\biggr)\biggr\}.
\end{align}
We now bound the Frobenius and operator norms of the random matrix $\sum_{kk \neq \{1, \ell\}} \bu_{1, \ell, k}\bu_{1, ,\ell, k}^{\top}$ in turn.  To this end, we expand
\begin{align*}
	\Bigl\| \sum_{k \neq \{1, \ell\}} \bu_{1, \ell, k}\bu_{1, \ell, k}^{\top}\Bigr \|_F^2 = \sum_{k \neq \{1, \ell\}} \| \bu_{1, \ell,k} \|_2^4 + \sum_{(k_1, k_2): k_1 \neq k_2 \neq \{1, \ell\}} \langle \bu_{1, \ell, k_1}, \bu_{1, \ell, k_2} \rangle^2.
\end{align*}
We note that on the sets $\mathcal{S}_1$ and $\mathcal{S}_2$, the following hold
\[
\| \bu_{1, \ell, k}\|_2^2 \leq Cn \qquad \text{ and } \qquad \langle \bu_{1, \ell, k_1}, \bu_{1, \ell, k_2} \rangle \leq C\sqrt{n}\log^{3/2}{n}.
\]
Consequently, 
\begin{align*}
	\Bigl\| \sum_{k \neq \{1, \ell\}} \bu_{1, \ell, k}\bu_{1, \ell, k}^{\top}\Bigr \|_F^2 \leq Cn^3\log^{3}{n} \qquad \text{ with probability } \qquad \geq 1 - n^{-20}.
\end{align*}
We note that since $\bG$ is a diagonal matrix with entries bounded as $C\sqrt{\log{n}}$ with probability at least $1 - n^{-10}$, this immediately implies the bound 
\begin{align*}
	\Bigl\| \bG \sum_{k \neq \{1, \ell\}} \bu_{1, \ell, k}\bu_{1, \ell, k}^{\top} \bG\Bigr \|_F^2 \leq Cn^3\log^{5}{n} \qquad \text{ with probability } \qquad \geq 1 - n^{-20}.
\end{align*}
The operator norm bound follows with the same probability as 
\[
\Bigl\| \bG \sum_{k \neq \{1, \ell\}} \bu_{1, \ell, k}\bu_{1, \ell, k}^{\top} \bG\Bigr \|_{\mathsf{op}} \leq \Bigl\| \bG\sum_{k \neq \{1, \ell\}} \bu_{1, \ell, k}\bu_{1, \ell, k}^{\top}\bG\Bigr \|_F \leq C  n^{3/2} \log^{5/2}(n).
\]
Note as well that with probability at least $1 - n^{-20}$, $\| \bG \|_{\op} + \| \bV \|_{\op} \leq C \sqrt{\log{n}}$.  Substituting these bounds into the inequality~\eqref{ineq:HW-bound-A-4}, we deduce the inequality
\[
\Bigl \lvert \by^{\top} \bG\Bigl[\sum_{k \neq \{1, \ell\}} \bu_{1, \ell, k}\bu_{1, \ell, k}^{\top}\Bigr]\bG\by - \EE\Bigl\{\by^{\top} \bG\Bigl[\sum_{k \neq \{1, \ell\}} \bu_{1, \ell, k}\bu_{1, \ell, k}^{\top}\Bigr]\bG\by \Bigm \vert \bG, \bV \Bigr\} \Bigr \rvert \leq C(1 + \sigma) n^{3/2}\log^{7/2}(n),
\]
with probability at least $1 -n^{-20}$.  The proof is complete upon noting that for every realization of $\bG, \bV$, with $\bu = \by - \beps - \EE \by, \bv = \Xcol{\ell}$, and $\bA = \bG\sum_{k \neq \{1, \ell\}} \bu_{1, \ell, k} \bu_{1, \ell, k}^{\top}\bG$, the LHS of the decomposition~\eqref{eq:E4-decomp} is zero-mean, conditionally on $\bG, \bV$. 

\paragraph{Proving the consequence.} We first reduce a tail bound on $f$ to proving a tail bound the truncated functions $\fdown$, noting
\begin{align}\label{ineq:f-deviation}
	\Prob\bigl\{ \lvert f - \EE f \rvert \geq t \bigr\} &\leq \Prob\bigl\{ \lvert f - \EE f \rvert \geq t, (\bX, \bG) \in \mathcal{S} \bigr\} + \Prob\bigl\{(\bX, \bG) \notin \mathcal{S}\bigr\} \nonumber\\
	&= \Prob\bigl\{ \lvert \fdown - \EE f \rvert \geq t, (\bX, \bG) \in \mathcal{S} \bigr\} + \Prob\bigl\{(\bX, \bG) \notin \mathcal{S}\bigr\},
\end{align}
where the equality follows since by Lemma~\ref{lem:properties-fdown}, $\fdown$ and $f$ agree on the set $\mathcal{S}$.  Next, we decompose $\EE f$ as 
\begin{align*}
	\EE f = \EE f \mathbbm{1}\{(\bX, \bG) \in \mathcal{S}\} + \EE f \mathbbm{1}\{(\bX, \bG) \notin \mathcal{S}\} \leq \EE \fdown + \EE (\lvert f \rvert + \lvert \fdown \rvert) \mathbbm{1}\{(\bX, \bG) \notin \mathcal{S}\}. 
\end{align*}
Towards bounding the second term, we claim the inequality
\begin{align}\label{ineq:ub-f-squared}
	\EE (\fdown)^2 \leq \EE f^2 \leq C d^2 n,
\end{align}
postponing its proof to the end.  Applying the above inequality in conjunction with the Cauchy--Schwarz inequality, we obtain the bound
\[
\bigl \lvert \EE f - \EE \fdown \bigr \rvert \leq Cd^2n \cdot \Prob\bigl\{(\bX, \bG) \notin \mathcal{S}\bigr\} \leq Cn^{-17}.  
\]
Consequently, substituting $t \neq n^{17}$ and $ \Prob\bigl\{(\bX, \bG) \notin \mathcal{S}\bigr\} \leq n^{-20}$ into the RHS of Eq.~\eqref{ineq:f-deviation} yields the result.  It remains to prove inequality~\eqref{ineq:ub-f-squared}.

\paragraph{Proof of the upper bound~\eqref{ineq:ub-f-squared}.} First, expand $f$ according to its definition 
\[
\EE f^2 = \frac{1}{n^4} \sum_{k, j \geq 2} \EE \langle \Xcol{k}, \bv_k \rangle^2 \cdot \langle \Xcol{j}, \bv_j \rangle^2 \leq \frac{1}{n^4} \sum_{k, j \geq 2} \EE \Bigl\{\| \Xcol{k} \|_2^2 \cdot \| \bv_k \|_2^2 \cdot \| \Xcol{j} \|_2^2 \cdot \| \bv_j \|_2^2 \Bigr\},  
\]
where the inequality follows upon applying the Cauchy--Schwarz inequality to each inner product.  We next invoke sub-multiplicativity of the operator norm in conjunction with the fact that $\| \bS_{\backslash k} \|_{\mathsf{op}} = 1$---since the matrix $\bS_{\backslash k}$ is a projection matrix---to obtain the upper bound $\| \bv_k \|_2^2 \leq \|\bG \|_{\mathsf{op}}^2 \cdot \| \by \|_2^2$.  Substituting this bound into the RHS of the display above and noting the independence of the random variables $\Xcol{k} ,  \Xcol{j} , \bG$, and $\by$, we obtain the bound
\[
\EE f^2 \leq \frac{d^2}{n^4} \underbrace{\EE\bigl\{ \| \Xcol{2} \|_2^2\}^2}_{=n^2} \cdot \underbrace{\EE\bigl\{ \| \bG \|_{\mathsf{op}}^4 \bigr \}}_{\lesssim n} \cdot \underbrace{\EE \bigl\{ \| \by \|_2^4\bigr\}}_{\lesssim n^2} \leq Cd^2n,
\]
where to obtain the final inequality, we note that the random vector $\by$ consists of i.i.d. entries which have moments of all orders and the matrix $\bG$ is diagonal and whose entries are i.i.d. Gaussian random variables. 

\qed

\subsubsection{Proof of Lemma~\ref{lem:lower-truncation}} \label{sec:proof-lem-lower-truncation}
For brevity, we provide the proof assuming $\rho(\bX, \bX') = 1$, noting that the case when $\rho(\bX, \bX') = 2$ follows in an entirely parallel fashion.  We begin by expanding the difference
\begin{align*}
	\fdown(\bX) - \fdown(\bX') &= \frac{1}{n^2}\sum_{k = 2}^{d} \bigl\langle \Xcol{k}, \bv_k\bigr\rangle^2 - \bigl \langle (\Xcol{k})', \bv_k' \bigr\rangle^2 = A  + B,
\end{align*}
where $I$ denotes the index which is changed between $\bX$ and $\bX'$, and have defined terms $A$ and $B$ as
\begin{subequations}
	\begin{align}
		A &:= \frac{1}{n^2}\bigl \langle \bv_I, \Xcol{I} - \widebar{\bX}^{(I)} \bigr \rangle \cdot \bigl \langle \bv_I, \Xcol{I} + \widebar{\bX}^{(I)} \bigr \rangle \quad \text{ and } \quad \label{eq:term-A-orthogonal}\\
		B &:= \frac{1}{n^2}\sum_{k \neq \{1, I\}} \bigl \langle \Xcol{k}, \bv_k - \bv_k' \bigr \rangle \cdot \bigl \langle \Xcol{k}, \bv_k + \bv_k' \bigr \rangle. \label{eq:term-B-orthogonal}
	\end{align}
\end{subequations}
We claim that both terms $A$ and $B$ are bounded as $(1 + \sigma^2)\log^{15/2}{n}/n$.  

\paragraph{Bounding term $A$~\eqref{eq:term-A-orthogonal}}
We bound the first term in the expansion as the others follow similarly, writing
\[
\frac{1}{n^2} \langle \bv_I, \Xcol{I} \rangle^2 = \frac{1}{n^2} \langle \by, \bu_{I} \rangle^2 \leq \frac{(1 + \sigma^2)\log^{3}{n}}{n},
\]
where the first step follows by definition of $\bu_I, \bv_I$~\eqref{eq:u-notation} and the final inequality follows on event $\mathcal{E}$.  We turn now to the bound on term $B$.  

\paragraph{Bounding term $B$~\eqref{eq:term-B-orthogonal}}
First, we apply the rank one update~\eqref{eq:recursive-S} to obtain the characterization
\[
\bv_k - \bv_k' = \frac{1}{\| \bu_{k, I} \|_2^2} \bG \bu_{k, I}\bu_{k,I}^{\top} \by -  \frac{1}{\| \bu_{k, I}' \|_2^2} \bG \bu_{k, I}'\bu_{k,I}'^{\top} \by,
\]
whence by symmetry, term $B$ is bounded as 
\begin{align}\label{ineq:ineq1-termB-orthogonal}
	B \leq \frac{4}{n^2}\sum_{k \neq \{1, I\}} \bigl \langle \Xcol{k}, \bv_k \bigr \rangle \cdot \bigl \langle \Xcol{k}, \frac{1}{\| \bu_{k, I} \|_2^2} \bG \bu_{k, I} \bu_{k, I}^{\top} \by \bigr \rangle.
\end{align}
Note that 
\[
\bigl \langle \Xcol{k}, \bv_k \bigr \rangle = \langle \by, \bu_k \rangle, \quad \text{ and } \quad \bigl \langle \Xcol{k}, \bG \bu_{k, I}, \bu_{k, I}^{\top} \by \bigr \rangle = \langle \bu_{I, k}, \bu_{k, I} \rangle \cdot \langle \bu_{k, I}, \by \rangle,
\]
where to obtain the second relation we have exploited the idempotence of the projection matrix $\bS_{k, I}$.  Expanding $\bu_k$ as $\bu_k = \bu_{I, k} - \| \bu_{k, I} \|_2^{-2} \langle \bu_{k, I}, \bu_{I, k} \rangle \bu_{k, I}$ and substituting the two relations in the display above into the RHS of the inequality~\eqref{ineq:ineq1-termB-orthogonal} yields the inequality 
\[
B \leq \frac{4}{n^2} \underbrace{\by^{\top} \sum_{k \neq \{1, I\}} \frac{1}{\| \bu_{k, I} \|_2^2} \bu_{I, k} \bu_{I,k}^{\top} \bu_{k, I} \bu_{k, I}^{\top} \by}_{:= T_1} + \frac{4}{n^2} \underbrace{\by^{\top} \sum_{k \neq \{1, I\}} \frac{1}{\| \bu_{k, I} \|_2^4} \langle \bu_{k, I}, \bu_{I, k} \rangle \bu_{k, I} \bu_{I, k}^{\top} \bu_{k, I} \bu_{k, I}^{\top} \by}_{:= T_2}.
\]
We bound each of these in turn. 

\bigskip 
\noindent\underline{Bounding $T_1$} Note that each summand is PSD and positive.
Thus, since by definition of $\fdown$, we are working on the set $\mathcal{S}$, we deduce the following two bounds
\[
\| \bu_{k, I} \|_2^2 \geq cn \qquad \text{ and } \qquad \by^{\top} \bu_{k, I} \leq C (1 + \sigma)\log^{3/2}{n}\sqrt{n}.
\]
Thus
\[
T_1 \leq \frac{C(1 + \sigma)\log^{3/2}{n}}{n^{5/2}} \by^{\top} \sum_{k \neq \{1, I\}} \bu_{I, k} \bu_{I,k}^{\top} \bG \Xcol{I}.
\]
Applying the rank one update~\eqref{eq:recursive-S} once more to obtain $\bu_{I, k} = \bu_{1, I, k} - \| \bu_{I, k , 1} \|_2^{-2} \langle \bu_{1, I, k}, \bu_{I, k, 1} \rangle \bu_{I, k, 1}$ and expanding the RHS of the above display yields
\begin{align}\label{ineq:T1-ub-term-A-orth}
	T_1 \leq \frac{C(1 + \sigma)\log^{3/2}{n}}{n^{5/2}} \cdot \biggl[\by^{\top} \sum_{k \neq \{1, I\}} &\bu_{1, I, k} \bu_{1, I,k}^{\top} \bG \Xcol{I} +  \by^{\top} \sum_{k \neq \{1, I\}} \frac{\langle \bu_{1, I, k}, \bu_{I, k, 1} \rangle}{\| \bu_{I, k , 1} \|_2^{2}}\bu_{I, k, 1}\bu_{1, I,k}^{\top} \bG \Xcol{I} \nonumber\\
	&+ \by^{\top} \sum_{k \neq \{1, I\}}\frac{\langle \bu_{1, I, k}, \bu_{I, k, 1} \rangle^2}{\| \bu_{I, k , 1} \|_2^{4}}\bu_{I, k, 1} \bu_{I, k,1}^{\top} \bG\Xcol{I}\biggr].
\end{align}
On the set $\mathcal{S}$, the first term in~\eqref{ineq:T1-ub-term-A-orth} is upper bounded as $C\log^{4}{n}/n$.  Turning to the second term, and noting that on event $\mathcal{E}$, $\| \bu_{I, k, 1} \|_2^{-2} \by^{\top} \bu_{I, k, 1} \leq C ( 1 + \sigma)\log^{3/2}{n}$ as well as $(\Xcol{1})^{\top} \bG \sum_{k \neq \{1, I\}} \bu_{1, I, k}\bu_{1, I,k}^{\top} \bG \Xcol{I} \leq Cn^{3/2}\log^{7/2}{n}$, so that
\[
\frac{C (1 + \sigma)\log^{3/2}{n}}{n^{5/2}} \by^{\top} \sum_{k \neq \{1, I\}} \frac{\langle \bu_{1, I, k}, \bu_{I, k, 1} \rangle}{\| \bu_{I, k , 1} \|_2^{2}}\bu_{I, k, 1}\bu_{1, I,k}^{\top} \bG \Xcol{I} \leq \frac{C(1 + \sigma^2)\log^{13/2}{n}}{n}.
\]
Finally, to bound the third term on the RHS of inequality~\eqref{ineq:T1-ub-term-A-orth}, we note that the following inequalities hold on the set $\mathcal{S}$
\[
\langle \bu_{1, I, k}, \bu_{I, k, 1} \rangle \leq Cn^{1/2}\log^{3/2}{n} , \quad \| \bu_{I, k, 1} \|_2^2 \geq cn, \quad \text{ and } \quad  \by^{\top} \bu_{I, k, 1} \leq C(1 + \sigma)n\log^{3/2}{n},
\]
so that since $\bu_{I, k , 1}^{\top} \bG \Xcol{I} = \langle \bu_{I, k, 1}, \bu_{1, k, I}\rangle$, 
\[
\frac{C(1 + \sigma)\log^{3/2}{n}}{n^{5/2}} \by^{\top} \sum_{k \neq \{1, I\}}\frac{\langle \bu_{1, I, k}, \bu_{I, k, 1} \rangle^2}{\| \bu_{I, k , 1} \|_2^{4}}\bu_{I, k, 1} \bu_{I, k,1}^{\top} \bG\Xcol{I} \leq \frac{C(1 + \sigma^2)d\log^{15/2}{n}}{n^2}.
\]
Putting the pieces together yields $T_1 \leq C(1 + \sigma^2)\log^{15/2}{n}/n$ as desired.  

\bigskip
\noindent \underline{Bounding $T_2$} We more compactly write
\[
T_2 =  \frac{4}{n^2}\sum_{k \neq \{1, I\}} \frac{1}{\| \bu_{k, I} \|_2^4} \langle \bu_{k, I}, \bu_{I, k} \rangle^2  \langle \bu_{k, I}, \by\rangle^2
\]
Note that on the set $\mathcal{S}$, the following hold
\[
\| \bu_{k, I} \|_2^2 \geq cn, \qquad \langle \bu_{k, I}, \bu_{I, k} \rangle \leq C\sqrt{n}\log^{3/2}{n}, \qquad \text{ and } \qquad \langle \bu_{k, I}, \by \rangle \leq C(1 + \sigma)\sqrt{n}\log^{3/2}{n}.
\]
Applying the triangle inequality and substituting these bounds into the RHS of $T_2$ yields \sloppy\mbox{$T_2 \leq C(1  + \sigma^2)\log^{6}{n}/n$} as desired.  \qed

\subsubsection{Proof of Lemma~\ref{lem:expectation-leave-out-first-col}} \label{sec:proof-lem-expectation-leave-out-first-col}
We begin by decomposing the quantity in question as
\[
\EE\Bigl\{ \| \bG \bS_{\backslash k} \by \|_2^2 \Bigr\} = \EE\Bigl\{ \| \bG \bS \by \|_2^2 \Bigr\} + \EE\Bigl\{ \| \bG \bS_{\backslash k} \by \|_2^2  - \| \bG \bS \by \|_2^2\Bigr\},
\]
and subsequently factoring the difference of squares and applying the (reverse) triangle inequality to obtain the inequality 
\[
\Bigl \lvert \EE\Bigl\{ \| \bG \bS_{\backslash k} \by \|_2^2 \Bigr\} - \EE\Bigl\{ \| \bG \bS \by \|_2^2 \Bigr\} \Bigr \rvert \leq  \EE \Bigl\{ \| \bG (\bS - \bS_{\backslash k}) \by \|_2 \cdot \bigl(\| \bG \bS \by \|_2  + \| \bG \bS_{\backslash k} \by \|_2\bigr)  \Bigl\}.
\]
Focusing our attention on the RHS, we apply the rank one update~\eqref{eq:recursive-S} to obtain the equivalent expression 
\begin{align*}
	\EE \Bigl\{ \| \bG  (\bS - \bS_{\backslash k}) \by \|_2 \cdot\bigl(\| \bG \bS \by \|_2  + \| \bG \bS_{\backslash k} \by \|_2\bigr)  \Bigl\} &= \EE \Bigl\{ \| \bu_k \|_2^{-2} \cdot \| \bG \bu_k \bu_k^{\top} \by \|_2 \cdot  \bigl(\| \bG \bS \by \|_2  + \| \bG \bS_{\backslash k} \by \|_2\bigr) \Bigr\}\\
	&\leq 2\EE\Bigl\{ \| \bu_k \|_2^{-1} \cdot \lvert \bu_k^{\top} \by \rvert \cdot \| \bG \|_{\mathsf{op}}^2 \cdot \| \by \|_2 \Bigr\},
\end{align*}
where in the inequality we have used the fact that $\bS$ and $\bS_{\backslash k}$ are projection matrices whence have operator norms bounded above by one.  
We now further decompose this upper bound into its restrictions to the two elements of the partition formed by the truncation set $\mathcal{S}$~\eqref{eq:events-E}.  That is, with $\mathcal{E} = \bigl\{(\Xcol{1}, \Xcol{2}, \dots, \Xcol{d}; \bG) \in \mathcal{S} \bigr\}$,
\begin{align*}
	\EE\Bigl\{ \| \bu_k \|_2^{-1} \cdot \lvert \bu_k^{\top} \by \rvert \cdot \| \bG \|_{\mathsf{op}}^2 \cdot \| \by \|_2 \Bigr\} &= \underbrace{\EE\Bigl\{ \| \bu_k \|_2^{-1} \cdot \lvert \bu_k^{\top} \by \rvert \cdot \| \bG \|_{\mathsf{op}}^2 \cdot \| \by \|_2  \cdot \mathbbm{1}\{ \mathcal{E}\} \Bigr\}}_{A} \\
	&\qquad \qquad+ \underbrace{\EE\Bigl\{ \| \bu_k \|_2^{-1} \cdot \lvert \bu_k^{\top} \by \rvert \cdot \| \bG \|_{\mathsf{op}}^2 \cdot \| \by \|_2 \cdot \mathbbm{1}\{\mathcal{E}^{c} \} \Bigr\}}_{B}.
\end{align*}
The remainder of the proof consists of bounding the terms $A$ and $B$.  We begin with term $A$.  Note that each of the terms in the product is controlled by the defining inequalities of the events which make up the truncation event $\mathcal{E}$.  Note as well that $\bG$ is a diagonal matrix with independent Gaussian entries whence we obtain the bound $\EE \| \bG \|_{\mathsf{op}}^2 \leq C\log{n}$ and $\by$ consists of independent entries with bounded moments of all orders whence we obtain the bound $\EE \| \by \|_2^2 \leq C(1 + \sigma^2)n$.  Thus, combining this with the definition of event $\mathcal{E}$ and applying Cauchy--Schwarz yields the ultimate bound
\[
A \leq C (1 + \sigma)  \sqrt{n} \log^{5/2}(n).
\]
We now turn our attention to term $B$. We first apply the Cauchy--Schwarz inequality to the inner product $\bu_k^{\top} \by$ and subsequently to the expectation to obtain the inequality
\[
B \leq \EE \Bigl\{ \| \bG\|_{\mathsf{op}}^2 \| \by \|_2^2 \mathbbm{1}\{\mathcal{E}^{c}\}\Bigr\} \leq \Bigl(\EE \| \bG \|_{\mathsf{op}}^4 \cdot \EE \| \by \|_2^4\Bigr)^{1/4} \cdot \sqrt{\Prob\{\mathcal{E}^{c}\}} \leq \frac{1}{n},
\]
where in the final inequality we have used $\EE \| \by \|_2^4 \leq C(1 + \sigma^2)^2 n^2$, $\EE \| \bG\|_{\op}^4 \leq C\log^2{n}$ and applied Lemma~\ref{lem:truncation-probability}.  
Putting the pieces together yields the inequality   
\[
\Bigl \lvert \EE\Bigl\{ \| \bG \bS_{\backslash k} \by \|_2^2 \Bigr\} - \EE\Bigl\{ \| \bG \bS \by \|_2^2 \Bigr\} \Bigr \rvert \leq C(1 + \sigma)\sqrt{n} \log^{5/2}(n),
\]
from which the conclusion follows immediately. \qed

\subsection{Establishing the relations~\eqref{eq:loo-residual-tool} and~\eqref{eq:recursive-S}} \label{sec:proof-loo-tools}
We establish each in turn.
\paragraph{Proof of the residual relation~\eqref{eq:loo-residual-tool}.}
Whereas the proof of the rank one update relied on the closed-form solution to the updates, the residual relation rests on exploiting the KKT conditions.  In particular, we note the vector relations
\begin{align*}
	\bx_{i} \cdot \underbrace{\bigl( \bG_{i}^2 \langle \bx_{i}, \bcoefX \rangle - W_ii y_i \bigr)}_{R_{i}} + \sum_{j \neq i} \bx_{j} \cdot \bigl( \bG_{jj}^2 \langle \bx_{j}, \bcoefX \rangle - W_jj y_j \bigr) &= 0, \qquad \text{ and } \\
	\sum_{j \neq i} \bx_{j} \cdot \bigl( \bG_{jj}^2 \langle \bx_{j}, \bcoefX^{(i)} \rangle - W_jj y_j \bigr) &= 0.
\end{align*}
Combining these equations and re-arranging yields
\[
R_{i} \bx_{i} = \bSig_{i}^{-1} \cdot \bigl(\bcoefX^{(i)} - \bcoefX\bigr).
\]
Finally, note that 
\[
r_{i, (i)} - R_{i} = \bG_{ii}^2 \langle \bx_{i}, \bcoefX^{(i)} - \bcoefX \rangle.
\]
Combining the two preceding displays and re-arranging yields the desired result. \qed

\paragraph{Proof of the rank one update~\eqref{eq:recursive-S}.}
This relation follows straightforwardly from block matrix inversion~\citep[see, e.g.,][Section A.5.5]{boyd2004convex}.  We begin by writing the projection matrix $\bP$ explicitly as 
\[
\bP  = \bG \bX\bigl(\bX^{\top} \bG^2 \bX \bigr)^{-1} \bX^{\top} \bG.
\]
Note that without loss of generality, we write $\bX= \begin{bmatrix} \Xcol{k} & \bX_{\backslash k} \end{bmatrix}$; that is, we apply permutation matrices to swap the columns of the data so that the entries of the $k$th column are in the first column.  Thus, in matrix notation, we write
\begin{align} \label{eq:matrix-notation-projection}
	\bP  = \underbrace{\begin{bmatrix} (\Xcol{k})^{\top} \bG  \\ \bX_{\backslash k}^{\top} \bG  \end{bmatrix}^{\top}}_{\bG \bX} \cdot  \underbrace{\begin{bmatrix} \langle \Xcol{k}, \bG^2 \Xcol{k} \rangle & (\Xcol{k})^{\top} \bG^2 \bX_{\backslash k} \\
			\bX_{\backslash k}^{\top} \bG^2 \Xcol{k} & \bX_{\backslash k}^{\top} \bG^2 \bX_{\backslash k} \end{bmatrix}^{-1}}_{\bigl(\bX^{\top} \bG^2 \bX \bigr)^{-1} } \cdot \underbrace{\begin{bmatrix} (\Xcol{k})^{\top} \bG  \\ \bX_{\backslash k}^{\top} \bG  \end{bmatrix}}_{\bX^{\top} \bG}.
\end{align}
Applying the block matrix inversion formula to the middle term yields
\[
(\bX^{\top} \bG^2 \bX)^{-1} = \frac{1}{\| \bu_k \|_2^2} \cdot \begin{bmatrix} 1 & - \bigl(\bSig_{\backslash k}^{-1} \bX_{\backslash k}^{\top} \bG^2 \Xcol{k}\bigr)^{\top}  \\
	-\bSig_{\backslash k}^{-1} \bX_{\backslash k}^{\top} \bG^2 \Xcol{k} & \| \bu_k \|_2^2 \bSig_{\backslash k}^{-1} + \bSig_{\backslash k}^{-1} \bX_{\backslash k}^{\top} \bG^2 \Xcol{k} (\Xcol{k})^{\top} \bG^2 \bX_{\backslash k} \bSig_{\backslash k}^{-1},
\end{bmatrix}
\]
where we have used the  shorthand $\bSig_{\backslash k}$ to denote the matrix $\bX_{\backslash k}^{\top} \bG^2 \bX_{\backslash k}$.  Substituting the result of the previous display into the RHS of equation~\eqref{eq:matrix-notation-projection} yields the relation
\[
\bP = \bP_{\backslash k} + \frac{1}{\| \bu_k \|_2^2} \bu_k \bu_k^{\top}.
\]
Adding the identity matrix to both sides and re-arranging yields the result. \qed

\subsection{Deterministic equivalences in Examples~\ref{ex:lin} and~\ref{ex:nonlin}}\label{sec:proof-det-equivalence}

In this section, we prove the equivalences claimed in Examples~\ref{ex:lin} and~\ref{ex:nonlin}. Recall the maps $F_{\psi}$ and $G_{\psi}$ introduced in those examples for the functions $\psi \in \{ \mathsf{id}, \sign \}$.
\begin{lemma}\label{det-equivalence}
Consider $\psi \in \{\mathsf{id},\mathsf{sgn}\}$ and recall the definition of $F_{\psi}$ and $G_{\psi}$~\eqref{eq:idmaps} and~\eqref{eq:bitmaps}. The following hold. 
\[
	\parcompX_{t+1}^{\mathsf{det}} = F_{\psi}\big( \parcompZ_{t+1},\perpcompZ_{t+1} \big) \quad \text{and} \quad (\perpcompX_{t+1}^{\mathsf{det}})^{2} = G_{\psi}\big( \parcompZ_{t+1},\perpcompZ_{t+1} \big).
\]
\end{lemma}

\begin{proof}
We prove each part in turn, beginning with the identity map $\psi(x) = x$.  In this case, 
\[
Y =  \frac{\parcompZ_{t+1}}{\sqrt{\parcompZ_{t+1}^{2} + \perpcompZ_{t+1}^{2}}} \cdot G X + \frac{\perpcompZ_{t+1}}{\sqrt{\parcompZ_{t+1}^{2} + \perpcompZ_{t+1}^{2}}} \cdot VX.
\]
Using this representation in conjunction with the fact that $G, X, V \overset{\mathsf{i.i.d.}}{\sim} \mathsf{N}(0, 1)$ yields
\[
\EE\left\{ \frac{GX  Y}{1+\tau G^{2}} \right\} = \frac{\parcompZ_{t+1}}{\sqrt{\parcompZ_{t+1}^{2} + \perpcompZ_{t+1}^{2}}} \cdot \EE\left\{ \frac{G^{2}}{1+\tau G^{2}} \right\}.
\]
Substituting the above into the RHS of Equation~\eqref{par-equiv} yields
\[
\parcompX_{t+1}^{\mathsf{det}} = \frac{\parcompZ_{t+1}}{\parcompZ_{t+1}^{2} + \perpcompZ_{t+1}^{2}} = \parmapid \big( \parcompZ_{t+1},\perpcompZ_{t+1} \big).
\]
We next consider $(\perpcompX_{t+1}^{\mathsf{det}})^{2}$, first considering three terms in the numerator
\begin{align*}
	T_1 &= \EE\left\{\frac{G^{2}Y^{2}}{(1+\tau G^{2})^{2}}\right\}, \qquad T_2 = \parcompX_{t+1}^{\mathsf{det}} \cdot (\parcompZ_{t+1}^{2} + \perpcompZ_{t+1}^{2} )^{1/2} \cdot \EE\Big\{ \frac{G^{3}XY}{(1+\tau G^{2})^{2}} \Big\}, \qquad \text{ and }\\
	&\qquad \qquad \qquad \qquad T_3 = (\parcompX_{t+1}^{\mathsf{det}})^{2} \cdot (\parcompZ_{t+1}^{2} + \perpcompZ_{t+1}^{2}) \cdot \EE\Big\{ \frac{G^{4}}{(1+\tau G^{2})^{2}} \Big\}.
\end{align*}
Substituting the expressions of $Y$ and $\parcompdetX_{t+1}$ yields
\begin{align*}
T_1 &= \frac{\parcompZ_{t+1}^{2}}{\parcompZ_{t+1}^{2} + \perpcompZ_{t+1}^{2}} \cdot \EE\Bigl\{\frac{G^{4}}{(1+\tau G^{2})^{2}}\Bigr\} + \frac{\perpcompZ_{t+1}^{2}}{\parcompZ_{t+1}^{2} + \perpcompZ_{t+1}^{2}} \cdot \EE\Bigl\{\frac{G^{2}}{(1+\tau  G^{2})^{2}}\Bigr\}\\
T_2 &= \frac{\parcompZ_{t+1}^{2}}{\parcompZ_{t+1}^{2} + \perpcompZ_{t+1}^{2}} \cdot \EE\Big\{\frac{G^{4}}{(1+\tau G^{2})^{2}}\Bigr\}, \qquad \text{ and } \qquad T_3 = \frac{\parcompZ_{t+1}^{2}}{\parcompZ_{t+1}^{2} + \perpcompZ_{t+1}^{2}} \cdot \EE\left\{\frac{G^{4}}{(1+\tau G^{2})^{2}}\right\}.
\end{align*}
Consequently, 
\begin{align*}
	(\perpcompX_{t+1}^{\mathsf{int}})^{2} &= \frac{1}{C(\oversamp)} \cdot \frac{\perpcompZ_{t+1}^{2}}{(\parcompZ_{t+1}^{2} + \perpcompZ_{t+1}^{2})^{2}} + \frac{\sigma^{2}}{C(\oversamp)} \cdot \frac{1}{\parcompZ_{t+1}^{2} + \perpcompZ_{t+1}^{2}} \\&=
	\frac{1+\sigma^{2}}{C(\oversamp)} \cdot \frac{\perpcompZ_{t+1}^{2}}{(\parcompZ_{t+1}^{2} + \perpcompZ_{t+1}^{2})^{2}} + \frac{\sigma^{2}}{C(\oversamp)} \cdot \frac{\parcompZ_{t+1}^{2}}{(\parcompZ_{t+1}^{2} + \perpcompZ_{t+1}^{2})^{2}} = \perpmapid\big( \parcompZ_{t+1},\perpcompZ_{t+1} \big).
\end{align*}
We next prove the case for $\psi(x) = \sign(x)$. By definition, we have that
\[
Y =  \sign\big( X \cdot (\parcompZ_{t+1} G + \perpcompZ_{t+1} V ) \big).
\]
We first consider $\parcompdetX_{t+1}$.  To this end, we compute the quantity $\EE\left\{ GXY/(1+\tau\cdot G^{2})\right\}$.
Conditioning on $G$ and $X$ yields
\begin{align*}
	\EE\left\{ \frac{GXY}{1+\tau G^{2}} \; \Big | \; G,X \right\} = \EE\biggl\{ \frac{GX\sign(X)}{1+\tau G^{2}} \cdot \biggl[ \Prob\Bigl\{V>\frac{-\parcompZ_{t+1}}{\perpcompZ_{t+1}} G\Bigr\} - 
	\Prob\Bigl\{V<\frac{-\parcompZ_{t+1}}{\perpcompZ_{t+1}}G\Bigr\} \biggr] \; \Big | \;  G,X \biggr\}.
\end{align*}
Using the notation $\phi(x) = \int_{0}^{x} e^{-\frac{t^{2}}{2}} \mathrm{d}t$, we note that $G, V \overset{\mathsf{i.i.d.}}{\sim} \mathsf{N}(0, 1)$, whence
\[
G \cdot   \biggl[ \Prob\Bigl\{V>\frac{-\parcompZ_{t+1}}{\perpcompZ_{t+1}} G\Bigr\} - 
\Prob\Bigl\{V<\frac{-\parcompZ_{t+1}}{\perpcompZ_{t+1}}G\Bigr\} \biggr] = |G| \cdot \sqrt{\frac{2}{\pi}} \cdot \phi\biggl(\frac{\parcompZ_{t+1}|G|}{\perpcompZ_{t+1}}\biggr).
\]
Consequently, 
\[
\EE\left\{ \frac{GXY}{1+\tau G^{2}} \; \Big | \; G,X \right\} =  \sqrt{\frac{2}{\pi}} \cdot \EE\bigg\{ \frac{X\mathsf{sgn}(X) |G|}{1+\tau G^{2}} \phi\bigg(\frac{\parcompZ_{t+1}|G|}{\perpcompZ_{t+1}}\bigg)  \; \Big| \; G,X \bigg\}.
\]
Taking expectation over $G,X$  and noting that $\EE\{X\mathsf{sign}(X)\} = \EE\{|X|\} = \sqrt{2/\pi}$, we obtain 
\[
\EE\left\{ \frac{GXY}{1+\tau G^{2}} \right\} = \frac{2}{\pi} \cdot \EE\Bigg\{ \frac{ |G|  }{1+\tau G^{2}} \cdot \phi\biggl(\frac{\parcompZ_{t+1}|G|}{\perpcompZ_{t+1}}\biggr) \Bigg\}.
\]
Now recalling that $\tau = 1/C(\oversamp)$, we obtain
\[
\EE\left\{ \frac{G^{2}}{1+\tau G^{2}} \right\} = C(\oversamp) \cdot \EE\left\{ \frac{G^{2}}{C(\oversamp)+ G^{2}} \right\} = \frac{C(\oversamp)}{\oversamp},
\]
where the final step follows from the fixed point equation~\eqref{definition-of-C}. Putting the pieces together yields that $\parcompX_{t+1}^{\mathsf{det}} = \parmapbit \big( \parcompZ_{t+1},\perpcompZ_{t+1} \big)$. 

We turn now to computing $(\perpcompX_{t+1}^{\mathsf{det}})^{2}$. For the numerator, similar calculations yield the pair of equivalent relations
\begin{align}\label{eq:GY-expectations}
	\EE\bigg\{ \frac{G^{2} Y^{2}}{(1+\tau  G^{2})^{2}}\bigg\} = \EE\bigg\{ \frac{G^{2}}{(1+\tau G^{2})^{2}}\bigg\}\qquad \text{ and } \qquad \EE\bigg\{ \frac{G^{3}XY}{(1+\tau G^{2})^{2}} \bigg\}
	= \frac{2}{\pi} \cdot \EE\bigg\{ \frac{|G^{3}| }{(1+\tau\cdot G^{2})^{2}} \cdot \phi\bigg( \frac{\parcompZ_{t+1}|G|}{\perpcompZ_{t+1}}\bigg)\bigg\}.
\end{align}
Applying the first of the two relations in the previous display yields 
\[
\frac{  \EE\Big\{ \frac{G^{2}\cdot Y^{2}}{(1+\tau  G^{2})^{2}}\Big\} +  \EE\Big\{ \frac{\sigma^{2}G^{2}}{(1+\tau  G^{2})^{2}} \Big\} }{ (\parcompZ_{t+1}^{2} + \perpcompZ_{t+1}^{2}) \cdot  C(\oversamp) \cdot \EE\left\{ \frac{G^{2}}{(1+\tau G^{2})^{2}} \right\} } = \frac{1+\sigma^{2}}{C(\oversamp)} \cdot \frac{1}{\parcompZ_{t+1}^{2}+\perpcompZ_{t+1}^{2}}.
\]
Recalling $C_2(\oversamp)$~\eqref{eq:C2-C3} and applying the second of the pair of equations~\eqref{eq:GY-expectations} yields
\[
\frac{ 2\parcompX_{t+1}^{\mathsf{det}} \EE\Big\{ \frac{G^{3}XY}{(1+\tau G^{2})^{2}} \Big\} }{ (\parcompZ_{t+1}^{2} + \perpcompZ_{t+1}^{2})^{1/2} \cdot  C(\oversamp) \cdot \EE\left\{ \frac{G^{2}}{(1+\tau  G^{2})^{2}} \right\} } = \frac{4}{\pi} \cdot \frac{\parmapbit \big(\parcompZ_{t+1},\perpcompZ_{t+1}\big)}{C(\oversamp)\cdot\big(\parcompZ_{t+1},\perpcompZ_{t+1}\big)^{1/2}} \cdot  \EE\bigg\{\frac{|G^{3}| }{C_{2}(\oversamp)(C(\oversamp)+ G^{2})^{2}} \cdot \phi\bigg( \frac{\parcompZ_{t+1}|G|}{\perpcompZ_{t+1}}\bigg)\bigg\}
\]
Finally, recalling $C_3(\oversamp)$, we obtain 
\[
\frac{ (\parcompX_{t+1}^{\mathsf{det}})^{2}  \EE\Big\{ \frac{G^{4}}{(1+\tau G^{2})^{2}} \Big\} }{  C(\oversamp) \cdot \EE\left\{ \frac{G^{2}}{(1+\tau G^{2})^{2}} \right\} } = \frac{\parmapbit \big(\parcompZ_{t+1},\perpcompZ_{t+1}\big)^{2}}{C(\oversamp)} \cdot \frac{ \EE\left\{ \frac{G^{4}}{(C(\oversamp)+ G^{2})^{2}} \right\} }{ \EE\left\{ \frac{G^{2}}{(C(\oversamp)+ G^{2})^{2}} \right\} } = \frac{\parmapbit \big(\parcompZ_{t+1},\perpcompZ_{t+1}\big)^{2}}{C(\oversamp)} \cdot \frac{C_{3}(\oversamp)}{C_{2}(\oversamp)}
\]
 Putting the pieces together yields that $(\perpcompX_{t+1}^{\mathsf{det}})^{2} = \perpmapbit\big(\parcompZ_{t+1},\perpcompZ_{t+1}\big)$. \end{proof}
\section{Auxiliary proofs for convergence results}\label{sec:aux-convergence}
This section is dedicated to the proofs of Lemma~\ref{lem:initialization}, which we provide in Sections~\ref{sec:proof-lem-initialization} and~\ref{sec:proof-lem-initialization-b}; Lemma~\ref{lemma2-good-region}, which we provide in Section~\ref{sec:proof-lemma-good-region};  and Lemma~\ref{lem:h-sgn-properties}, which we provide in Section~\ref{sec:proof-lem-h-sgn-properties} (alongside several other properties of the function $\ratiomapbit$).  Finally, in Section~\ref{sec:global-convergence-large-oversamp}, we provide global convergence guarantees in the large sample regime when $\oversamp \geq \sqrt{d}$.

\subsection{Proof of Lemma~\ref{lem:initialization}(a)} \label{sec:proof-lem-initialization}
We prove the lemma for the linear model and nonlinear model in turn.  For both cases, we first bound the deviation of $\parcompZ_{t+1}$ and $\perpcompZ_{t+1}$ from their deterministic counterparts and use these to then bound the deviation of $\parcompX_{t+1}$ and $\perpcompX_{t+1}$ from their deterministic counterparts.  Throughout, we will use the notation 
\begin{align}
	\label{eq:notation-lem-initialization}
	\rho = \frac{C(\oversamp)}{1+\sigma^{2}}, \quad r_{t} = \frac{\perpcompX_{t}}{\parcompX_{t}}, \quad \Delta_{1} = (\parsigma) \sqrt{\frac{\log(d)}{\oversamp \cdot d}}, \quad \text{ and } \quad\Delta_{2} = (\perpsigma)\frac{\perplogd}{\sqrt{\oversamp \cdot d}},
\end{align}

\subsubsection{Linear observation model}
Recall the functions $\parmapid$~\eqref{eq:parmapid} and $\perpmapid$~\eqref{eq:perpmapid} and note that by definition
\[
\parcompdetZ_{t+1} = \parmapid(\parcompX_{t}, \perpcompX_{t}) = \frac{\parcompX_{t}}{\parcompX_{t}^{2} + \perpcompX_{t}^{2}} \quad \text{ and } \quad (\perpcompdetZ_{t+1})^2 = \perpmapid(\parcompX_{t}, \perpcompX_{t}) = \frac{1+\sigma^{2}}{C(\oversamp)} \cdot \frac{\perpcompX_{t}^{2}}{(\parcompX_{t}^{2} + \perpcompX_{t}^{2})^{2}} + \frac{\sigma^{2}}{C(\oversamp)} \cdot  \frac{\parcompX_{t}^{2}}{(\parcompX_{t}^{2} + \perpcompX_{t}^{2})^{2}}
\]
Straightforward computation yields the pair of sandwich inequalities 
\begin{align}\label{ineq1-lemma1-initia-linear-proof}
	(1-1/r_{t}^{2}) \cdot \frac{\parcompX_{t}}{\perpcompX_{t}^{2}} \leq \parcompZ_{t+1}^{\mathsf{det}} \leq \frac{\parcompX_{t}}{\perpcompX_{t}^{2}} \quad \text{and} \quad 
	\rho^{-1} \cdot  \perpcompX_{t}^{-2} \cdot (1+1/r_{t}^{2})^{-2} \leq (\perpcompZ_{t+1}^{\mathsf{det}})^{2} \leq \rho^{-1} \cdot  \perpcompX_{t}^{-2}.
\end{align}
On event $\mathcal{A}_{t+1}$ and using $\parcompX_{t}^{2} + \perpcompX_{t}^{2} \geq 0.5$, we obtain that
\begin{align}\label{ineq2-lemma1-initia-linear-proof}
	\bigl\lvert \parcompZ_{t+1} - \parcompZ_{t+1}^{\mathsf{det}} \bigr\rvert \leq C_{1} \cdot \Delta_{1} \quad \text{and} \quad \bigl \lvert \perpcompZ_{t+1}^{2} - (\perpcompZ_{t+1}^{\mathsf{det}})^{2} \bigr\rvert \leq C_{1} \cdot \Delta_{2}.
\end{align}
We turn now to bound $\parcompX_{t+1}$ and $\perpcompX_{t+1}$.  Combining inequalities~\eqref{ineq1-lemma1-initia-linear-proof} and~\eqref{ineq2-lemma1-initia-linear-proof} with the numeric inequality $(a + b)^2 \leq 2a^2 + 2b^2$ yields the upper bound
\begin{subequations}
	\label{ineq3-lemma1-initia-linear-proof}
	\begin{align}\label{ineq:ineq3-initia-linear-ub}
		\parcompZ_{t+1}^{2} + \perpcompZ_{t+1}^{2} &\leq 2\parcompX_{t}^{2}\perpcompX_{t}^{-4}  + 2C_{1}^{2}\Delta_{1}^{2} + \rho^{-1} \cdot  \perpcompX_{t}^{-2}  + C_{1}\Delta_{2}\nonumber\\
		&= 
		\perpcompX_{t}^{-2} \cdot (\rho^{-1} + 2/r_{t}^{2}) + 2C_{1}^{2}\Delta_{1}^{2} + C_{1}\Delta_{2}.
	\end{align}
	We similarly obtain the lower bound
	\begin{align}\label{ineq:ineq3-initia-linear-lb}
		\parcompZ_{t+1}^{2} + \perpcompZ_{t+1}^{2} \geq \perpcompX_{t}^{-2} \cdot \rho^{-1} \cdot 1/(1+1/r_{t}^{2})^{2} - C_{1}\Delta_{2}.
	\end{align}
\end{subequations}
Again, by definition of the maps $\parmapid$ and $\perpmapid$, 
\begin{align}\label{eq1-lemma1-initia-linear-proof}
	\parcompX_{t+1}^{\mathsf{det}} = \frac{\parcompZ_{t+1}}{\parcompZ_{t+1}^{2} + \perpcompZ_{t+1}^{2}}\quad \text{and} \quad (\perpcompX_{t+1}^{\mathsf{det}})^{2} = \frac{1+\sigma^{2}}{C(\oversamp)} \cdot \frac{\perpcompZ_{t+1}^{2}}{(\parcompZ_{t+1}^{2} + \perpcompZ_{t+1}^{2})^{2}} + \frac{\sigma^{2}}{C(\oversamp)} \cdot  \frac{\parcompZ_{t+1}^{2}}{(\parcompZ_{t+1}^{2} + \perpcompZ_{t+1}^{2})^{2}}.
\end{align}
Combining the lower bound on $\parcompZ_{t+1}^{\mathsf{det}}$~\eqref{ineq1-lemma1-initia-linear-proof}, the deviation inequality~\eqref{ineq2-lemma1-initia-linear-proof} and the upper bound~\eqref{ineq:ineq3-initia-linear-ub} yields the lower bound 
\[
\parcompX_{t+1}^{\mathsf{det}} \geq \frac{\parcompX_{t}\perpcompX_{t}^{-2}(1-r_{t}^{-2}) - C_{1}\Delta_{1}}{\perpcompX_{t}^{-2} \cdot (\rho^{-1} + 2/r_{t}^{2}) + 2C_{1}^{2}\Delta_{1}^{2} + C_{1}\Delta_{2}} = \frac{\parcompX_{t}(1-r_{t}^{-2}) - C_{1}\Delta_{1}\perpcompX_{t}^{2}}{(\rho^{-1} + 2/r_{t}^{2}) + \perpcompX_{t}^{2} \cdot (2C_{1}^{2}\Delta_{1}^{2} + C_{1}\Delta_{2})}.
\]
Further, note that by assumption, $\parcompX_{t} \geq 1/(50\sqrt{d}), \perpcompX_{t}^2 \leq 2, r_t^2 \geq 20\rho$, $C_0(1+\sigma^{2}) \log(d) \leq \oversamp \leq \sqrt{d}$ and $\rho \asymp \oversamp/(1+\sigma^{2})$ yields the pair of inequalities
\[
\parcompX_{t} (1 - r_t^{-2}) - C_1 \Delta_1 \perpcompX_{t}^2 \overset{\1}{\geq} \parcompX_{t}/2 \quad \text{ and } \quad (\rho^{-1} + 2/r_{t}^{2}) + \perpcompX_{t}^{2} \cdot (2C_{1}^{2}\Delta_{1}^{2} + C_{1}^{2}\Delta_{2}) \overset{\2}{\leq} 2\rho^{-1},
\]
where in step $\1$ and step $\2$ we let $C_{0}$ be a large enough constant and use
\begin{align}\label{Detla1-2-bound}
	\Delta_1 \leq \frac{2}{\sqrt{C_{0}d}},\; \frac{\Delta_{1}^{2}}{\rho^{-1}} \asymp \frac{\log(d)}{d} = o_{d}(1)\;\text{ and }\; \frac{\Delta_{2}}{\rho^{-1}}  \asymp \frac{\log^{8}(d)\sqrt{\oversamp}}{\sqrt{d}} = o_{d}(1),
\end{align}
Combining the previous three displays yields the lower bound $\parcompdetX_{t+1} \geq \rho/4 \cdot \parcompX_{t}$.  Similarly, we combine the inequalities~\eqref{ineq1-lemma1-initia-linear-proof},~\eqref{ineq2-lemma1-initia-linear-proof} and~\eqref{ineq:ineq3-initia-linear-lb} yields the upper bound 
\[
\parcompX_{t+1}^{\mathsf{det}} \leq \frac{\parcompX_{t} \perpcompX_{t}^{-2} + C_{1} \Delta_{1}}{ \rho^{-1} \perpcompX_{t}^{-2} \cdot (1+r_{t}^{-2})^{-2} - C_{1}\Delta_{2}} = 
\frac{\parcompX_{t} + C_{1}\Delta_{1}\perpcompX_{t}^{2} }{ \rho^{-1} \cdot (1+r_{t}^{-2})^{-2} - C_{1}\Delta_{2}\perpcompX_{t}^{2}}.
\]
Again, combining the assumptions yields the pair of inequalities 
\[
\parcompX_{t} + C_{1}\Delta_{1}\perpcompX_{t}^{2} \leq 2 \parcompX_{t} \quad \text{ and } \quad \rho^{-1} \cdot (1+r_{t}^{-2})^{-2} - C_{1}\Delta_{2}\perpcompX_{t}^{2} \geq 0.5 \cdot \rho^{-1},
\]
whence we obtain the inequality $\parcompdetX_{t+1} \leq 4\rho \cdot \parcompX_{t}$.  Taking stock, we have shown the sandwich relation
\begin{align}\label{ineq4-lemma1-initia-linear-proof}
	\frac{\rho}{4} \cdot \parcompX_{t} \leq \parcompX_{t+1}^{\mathsf{det}} \leq 4\rho \cdot \parcompX_{t}.
\end{align}
Towards bounding $\perpcompdetX_{t+1}$, we introduce the parallel notation $\widetilde{r}_{t+1} = \perpcompZ_{t+1}/\parcompZ_{t+1}$.  Similarly to the second inequality~\eqref{ineq1-lemma1-initia-linear-proof}, we note the inequality 
\[
\rho^{-1} \cdot \perpcompZ_{t+1}^{-2} \cdot (1+\widetilde{r}_{t+1}^{-2})^{-2} \leq (\perpcompX_{t+1}^{\mathsf{det}})^{2} \leq \rho^{-1} \cdot \perpcompZ_{t+1}^{-2}.
\]
Combining the inequalities~\eqref{ineq1-lemma1-initia-linear-proof} and~\eqref{ineq2-lemma1-initia-linear-proof} with the above display and the numeric inequalities $(1 + a)^{-1} \geq 1 - a$ and $(1 + a)^{-2} \geq 1 - 2a$, for $a > 0$ yields the lower bound 
\begin{align}\label{ineq:linear-perpdet-lb}
	(\perpcompX_{t+1}^{\mathsf{det}})^{2} \geq \rho^{-1} \cdot \perpcompZ_{t+1}^{-2} \cdot (1+\widetilde{r}_{t+1}^{-2})^{-2} &\geq
	\perpcompX_{t}^{2} \cdot (1+C_{1}\Delta_{2} \cdot \rho \cdot \perpcompX_{t}^{2})^{-1} \cdot (1+\widetilde{r}_{t+1}^{-2})^{-2} \\&\geq
	\perpcompX_{t}^{2} \cdot (1-C_{1}\Delta_{2} \cdot \rho \cdot \perpcompX_{t}^{2}) \cdot (1-2\widetilde{r}_{t+1}^{-2}).
\end{align} 
Towards bounding $\widetilde{r}_{t+1}$, note the pair of inequalities 
\[
\perpcompZ_{t+1}^{2} \geq \rho^{-1}\perpcompX_{t}^{-2}(1+r_{t}^{-2})^{2} - C_{1}\Delta_{2} \geq 0.5 \rho^{-1} \perpcompX_{t}^{-2} \quad \text{ and } \quad \parcompZ_{t+1}^{2} \leq 2\parcompX_{t}^{2}\perpcompX_{t}^{-4}+2C_{1}^{2}\Delta_{1}^{2}\leq 3\parcompX_{t}^{2}\perpcompX_{t}^{-4},
\]
whence we deduce the lower bound $\widetilde{r}_{t+1}^{2} \geq (6\rho)^{-1}r_{t}^{2}$.  Putting the pieces together, we obtain the lower bound
\begin{align*}
	(\perpcompX_{t+1}^{\mathsf{det}})^{2} \geq \perpcompX_{t}^{2} \cdot (1-2C_{1}\Delta_{2} \cdot \rho) \cdot (1-12\rho r_{t}^{-2}) &\geq \perpcompX_{t}^{2} \cdot (1-2C_{1}\Delta_{2}\rho -12 \rho r_{t}^{-2})\\
	&\geq \perpcompX_{t}^{2} - 12\rho \parcompX_{t}^{2}-C_{1}'\frac{\perplogd \sqrt{\oversamp}}{\sqrt{d}},
\end{align*}
where the first inequality follows by combining the inequality~\eqref{ineq:linear-perpdet-lb} with the lower bound $\perpcompX_{t}^2 \geq 0.5$, and the last inequality follows for some universal constant $C_1'$ since $\rho \asymp \oversamp/(1+\sigma^{2})$.  Similarly, we upper bound $\perpcompdetX_{t+1}$ as
\begin{align*}
	(\perpcompX_{t+1}^{\mathsf{det}})^{2} \leq \rho^{-1} \cdot \perpcompZ_{t+1}^{-2} \leq \rho^{-1} \cdot (\rho^{-1}\perpcompX_{t}^{-2}(1+r_{t}^{-2})^{-2} - C_{1}\Delta_{2})^{-1} &=
	\frac{\perpcompX_{t}^{2} \cdot (1+r_{t}^{-2})^{2}}{1-C_{1}\rho \Delta_{2} \perpcompX_{t}^{2} (1+r_{t}^{-2})^{2}}.
\end{align*}
We further upper bound the RHS 
\begin{align*}
	\frac{\perpcompX_{t}^{2} \cdot (1+r_{t}^{-2})^{2}}{1-C_{1}\rho \Delta_{2} \perpcompX_{t}^{2} (1+r_{t}^{-2})^{2}} &\leq
	\perpcompX_{t}^{2} \cdot (1+r_{t}^{-2})^{2} \cdot \bigl(1+2C_{1}\rho \Delta_{2} \perpcompX_{t}^{2} (1+r_{t}^{-2})^{2}\bigr) \\
	&\leq \perpcompX_{t}^2 (1+r_{t}^{-2})^{2} \Bigl(1 + \frac{C_1' \perplogd \sqrt{\oversamp}}{\sqrt{d}}\Bigr),
\end{align*}
where the first inequality follows from the numeric inequality $(1-a)^{-1} \leq 1+2a$ for $0\leq a \leq 0.5$, with $a = C_{1}\rho \Delta_{2} \perpcompX_{t}^{2} (1+r_{t}^{-2})^{2} \leq 0.5$ and the final inequality follows by using the assumptions $r_t^2 \geq 20\rho$ and $\perpcompX_t^2 \leq 2$.  Now, since $r_t > 1$, $(1 + r_t^{-2})^2 \leq 1 + 3r_t^{-2}$, 
\[
(1+r_{t}^{-2})^{2} \Bigl(1 + \frac{C_1' \log^{8}(d) \sqrt{\oversamp}}{\sqrt{d}}\Bigr) \leq 1 + 3r_t^{-2}  + \frac{C \perplogd \sqrt{\oversamp}}{\sqrt{d}},
\]
whence we immediately obtain the upper bound
\[
(\perpcompX_{t+1}^{\mathsf{det}})^{2} \leq \perpcompX_{t}^{2} \cdot (1+3r_{t}^{-2}+\frac{C_{1}''\perplogd \sqrt{\oversamp}}{\sqrt{d}}) \leq \perpcompX_{t}^{2} + 3\parcompX_{t}^{2} + \frac{2C_{1}''\perplogd \sqrt{\oversamp}}{\sqrt{d}}.
\]
Summarizing, we have obtained the sandwich inequality
\begin{align}\label{ineq5-lemma1-initia-linear-proof}
	\perpcompX_{t}^{2} - 12\rho \parcompX_{t}^{2}-C_{1}'\frac{\perplogd \sqrt{\oversamp}}{\sqrt{d}} \leq (\perpcompX_{t+1}^{\mathsf{det}})^{2} \leq \perpcompX_{t}^{2} + 3\parcompX_{t}^{2} + \frac{2C_{1}''\perplogd \sqrt{\oversamp}}{\sqrt{d}}.
\end{align}
Next, we use the inequality
\[
\parcompZ_{t+1}^{2} + \perpcompZ_{t+1}^{2} \geq \perpcompX_{t}^{-2} \cdot \rho^{-1} \cdot (1+1/r_{t}^{2})^{-2} - C_{1}\Delta_{2} \geq \rho^{-1}/4,
\]
in conjunction with the definition of event $\mathcal{B}_{t+1}$ to obtain the pair of deviation bounds
\begin{align*}
	\bigl\lvert \parcompX_{t+1} - \parcompX_{t+1}^{\mathsf{det}} \bigr\rvert  \leq \frac{C_{1} \Delta_1}{\bigl( \parcompZ_{t+1}^{2} + \perpcompZ_{t+1}^{2} \bigr)^{1/2}} \lesssim \sqrt{\frac{\log(d)}{d}} \quad \text{ and }\quad \bigl\lvert \perpcompX_{t+1}^{2}  - (\perpcompX_{t+1}^{\mathsf{det}})^{2} \bigr\rvert \leq \frac{C_{1} \Delta_2}{\parcompZ_{t+1}^{2} + \perpcompZ_{t+1}^{2}}  \lesssim \frac{\perplogd \sqrt{\oversamp}}{\sqrt{d}}.
\end{align*}
Note that $\sqrt{\log(d)/d} = o_{d}(\rho \parcompX_{t})$ by the assumption $\oversamp \geq C(1+\sigma^{2})\log(d)$ and $\parcompX_{t} \gtrsim 1/\sqrt{d}$. Consequently, combining this with the inequalities~\eqref{ineq4-lemma1-initia-linear-proof} and~\eqref{ineq5-lemma1-initia-linear-proof} yields the desired result.  \qed

\subsubsection{One-bit observation model}
We first obtain a sandwich relation on $\parcompdetZ_{t+1}$.  Using the functions $\parmapbit$~\eqref{eq:parmapbit} in conjunction with the inequality $\phi(x) = \int_{0}^{x} e^{-t^2/2}\mathrm{d}t \leq x$ yields the upper bound
\begin{subequations}\label{ineq1-lemma1-init-nonlinear-proof}
\begin{align}
\parcompdetZ_{t+1} = \parmapbit(\parcompX_{t}, \perpcompX_{t}) \leq \frac{2}{\pi} \cdot \frac{1}{\sqrt{\parcompX_{t}^2 + \perpcompX_{t}^2}} \cdot \frac{\parcompX_{t}}{\perpcompX_{t}} \leq \frac{2}{\pi} \cdot \frac{\parcompX_{t}}{\perpcompX_{t}^2}.
\end{align}
Next, applying Lemma~\ref{f-h-properties-nonlinear}(b) in conjunction with the upper bound $\parcompX_{t} /\perpcompX_{t} \leq 1$ yields the lower bound 
\begin{align}
\parcompZ_{t+1}^{\mathsf{det}} \geq \frac{1}{\sqrt{2\pi}} \cdot \frac{\parcompX_{t}}{\perpcompX_{t}} \cdot \frac{1}{\sqrt{\parcompX_{t}^{2} + \perpcompX_{t}^{2}}} = \frac{1}{\sqrt{2\pi}} \cdot \frac{\parcompX_{t}}{\perpcompX_{t}^{2}} \cdot \frac{1}{(1 + r_{t}^{-2})^{1/2}}.
\end{align}
We next find upper and lower bounds on $(\perpcompdetZ_{t+1})^2 = \perpmapbit(\parcompX_{t}, \perpcompX_{t})$~\eqref{eq:perpmapbit}.  To this end, we note that 
\[
\frac{C_3(\oversamp)}{C_2(\oversamp)} \leq \EE\{W^{4}\} \bigg/ \EE\left\{ \frac{W^{2}}{(1 + W^{2})^{2}}\right\} \leq 20,
\]
where $W \sim \mathsf{N}(0, 1)$.  Straightforward computation thus yields the upper bound
\begin{align}
	(\perpcompZ_{t+1}^{\mathsf{det}})^{2} \leq \frac{1+\sigma^{2}}{C(\oversamp)} \cdot \frac{1}{\parcompX_{t}^{2}+\perpcompX_{t}^{2}} + \frac{20(\parcompZ_{t+1}^{\mathsf{det}})^{2}}{C(\oversamp)} &\leq 
	\frac{1+\sigma^{2}}{C(\oversamp)} \cdot \frac{1}{\parcompX_{t}^{2}+\perpcompX_{t}^{2}} + \frac{20}{C(\oversamp)} \cdot \frac{4}{\pi^{2}} \cdot \frac{1}{\parcompX_{t}^{2}+\perpcompX_{t}^{2}}\cdot \frac{\parcompX_{t}^{2}}{\perpcompX_{t}^{2}} \nonumber\\
	&\leq \frac{1+\sigma^{2}}{C(\oversamp)} \cdot \frac{1}{\parcompX_{t}^{2}+\perpcompX_{t}^{2}} \cdot (1+9r_{t}^{-2}),
\end{align}
where in the penultimate inequality we have also used the upper bound on $\parcompdetZ_{t+1}$.  We additionally obtain the lower bound
\begin{align}
	(\perpcompZ_{t+1}^{\mathsf{det}})^{2} \geq  \frac{1+\sigma^{2}}{C(\oversamp)} \cdot \frac{1}{\parcompX_{t}^{2}+\perpcompX_{t}^{2}} - \frac{4}{\pi} \cdot \frac{\parcompZ_{t+1}^{\mathsf{det}}}{C(\oversamp)\sqrt{\parcompX_{t}^{2} + \perpcompX_{t}^{2}}} \cdot 20 &\geq \frac{1+\sigma^{2}}{C(\oversamp)} \cdot \frac{1}{\parcompX_{t}^{2}+\perpcompX_{t}^{2}} - \frac{160}{\pi^{2}C(\oversamp)} \cdot \frac{1}{\parcompX_{t}^{2} + \perpcompX_{t}^{2}} \cdot \frac{\parcompX_{t}^{2}}{\perpcompX_{t}^{2}}\nonumber \\&\geq\frac{1+\sigma^{2}}{C(\oversamp)} \cdot \frac{1}{\parcompX_{t}^{2}+\perpcompX_{t}^{2}} \cdot (1-18r_{t}^{-2}),
\end{align}
\end{subequations}
where to obtain the first inequality, we have used 
\[
\EE\biggl\{ \frac{|W|^{3}  \phi(\frac{\parcompX_t}{\perpcompX_t} |W|)}{C_{2}(\oversamp) (C(\oversamp) + W^{2})^{2}} \biggr\} \leq \frac{\parcompX_{t}}{\perpcompX_{t}}\frac{C_3(\oversamp)}{C_2(\oversamp)} \leq 20,
\]
which holds since $\parcompX_{t}/\perpcompX_{t} \leq 1$.  We obtain upper and lower bounds on $\parcompZ_{t+1}$ and $\perpcompZ_{t+1}$ since, on $\mathcal{A}_{t+1}$, 
\begin{align}\label{ineq2-lemma1-init-nonlinear-proof}
	\bigl\lvert \parcompZ_{t+1} - \parcompZ_{t+1}^{\mathsf{det}}\bigr\rvert \leq C_{1}\Delta_{1} \quad \text{and} \quad \bigl \lvert  \perpcompZ_{t+1}^{2} - (\perpcompZ_{t+1}^{\mathsf{det}})^{2} \bigr \rvert \leq C_{1}\Delta_{2},
\end{align}
where we have additionally used the fact that $\parcompX_{t}^{2} + \perpcompX_{t}^{2} \geq 0.5$.  With the bounds on $\parcompZ_{t+1}$ and $\perpcompZ_{t+1}$ in hand, we turn to bounding $\parcompX_{t+1}$ and $\perpcompX_{t+1}$.  We first note the upper bound
\[
\parcompdetX_{t+1} = \parmapbit(\parcompZ_{t+1}, \perpcompZ_{t+1}) \leq \frac{2}{\pi} \cdot \frac{1}{\sqrt{\parcompZ_{t+1}^{2} + \perpcompZ_{t+1}^{2}}} \cdot \frac{\parcompZ_{t+1}}{\perpcompZ_{t+1}} \leq \frac{2}{\pi} \cdot \frac{\parcompZ_{t+1}}{\perpcompZ_{t+1}^{2}}.
\]
Using inequalities~\eqref{ineq1-lemma1-init-nonlinear-proof} and~\eqref{ineq2-lemma1-init-nonlinear-proof} in conjunction with the assumptions $\parcompX_{t} \geq 1/(50\sqrt{d})$, $0.5\leq \perpcompX_{t}^{2}\leq2$, $r_{t}^{2} \geq 50 \rho$, $\rho \asymp \oversamp/(1+\sigma^{2})$, and inequality~\eqref{Detla1-2-bound}, we obtain the pair of inequalities
\begin{align}\label{ineq3-lemma1-init-nonlinear-proof}
	\parcompZ_{t+1} &\leq \frac{2}{\pi} \cdot \frac{\parcompX_{t}}{\perpcompX_{t}^{2}} + C_{1}\cdot \Delta_{1} \leq \frac{3}{\pi} \cdot \frac{\parcompX_{t}}{\perpcompX_{t}^{2}} \quad \text{ and }\quad \nonumber\\
	\perpcompZ_{t+1}^{2} &\geq \rho^{-1} \cdot \perpcompX_{t}^{-2} \cdot (1+r_{t}^{-2})^{-1} \cdot (1-18r_{t}^{-2}) - C_{1}\Delta_{2} \geq 0.5\cdot \rho^{-1}\cdot \perpcompX_{t}^{-2}.
\end{align}
Putting the pieces together yields the upper bound 
\[
\parcompX_{t+1}^{\mathsf{det}} \leq \frac{2}{\pi} \cdot \frac{\frac{3}{\pi} \cdot \frac{\parcompX_{t}}{\perpcompX_{t}^{2}}}{0.5\cdot \rho^{-1}\cdot \perpcompX_{t}^{-2}} \leq 2\rho \cdot \parcompX_{t}.
\]
We next derive the lower bound of $\parcompX_{t+1}^{\mathsf{det}}$. Using the parallel notation $\widetilde{r}_{t+1} = \perpcompZ_{t+1}/\parcompZ_{t+1}$ and applying inequality~\eqref{ineq3-lemma1-init-nonlinear-proof} yields the lower bound
\[
\widetilde{r}_{t+1} \geq \frac{\left(0.5\cdot \rho^{-1}\cdot \perpcompX_{t}^{-2}\right)^{1/2}}{\frac{3}{\pi} \cdot \frac{\parcompX_{t}}{\perpcompX_{t}^{2}}} \geq \frac{\pi}{6} \cdot \rho^{-1/2} \cdot r_{t}.
\]
Since $r_{t}\geq \sqrt{20\rho}$ by assumption, we conclude that $\widetilde{r}_{t+1} \geq 1$.  Next, we apply Lemma~\ref{f-h-properties-nonlinear}(b) in conjunction with the fact that $\frac{\parcompZ_{t+1}}{\perpcompZ_{t+1}} = 1/\widetilde{r}_{t+1} \leq 1$ to obtain the lower bound
\begin{align}\label{ineq3.5-lemma1-init-nonlinear-proof}
	\parcompX_{t+1}^{\mathsf{det}} \geq \frac{1}{\sqrt{2\pi}}\cdot \frac{\parcompZ_{t+1}}{\perpcompZ_{t+1}} \cdot \frac{1}{(\parcompZ_{t+1}^{2} + \perpcompZ_{t+1}^{2})^{1/2}} = \frac{1}{\sqrt{2\pi}}\cdot \frac{\parcompZ_{t+1}}{\perpcompZ_{t+1}^{2}} \cdot \frac{1}{(1+\widetilde{r}_{t+1}^{-2})^{1/2}} \geq \frac{1}{2\sqrt{\pi}} \cdot \frac{\parcompZ_{t+1}}{\perpcompZ_{t+1}^{2}}.
\end{align}
Once more applying inequalities~\eqref{ineq1-lemma1-init-nonlinear-proof} and~\eqref{ineq2-lemma1-init-nonlinear-proof} in conjunction with the assumptions $\parcompX_{t} \geq 1/(50\sqrt{d})$, $0.5\leq \perpcompX_{t}^{2}\leq2$, $r_{t}^{2} \geq 50 \rho$, $\rho \asymp \oversamp/(1+\sigma^{2})$, and~\eqref{Detla1-2-bound} yields
\begin{align}\label{ineq4-lemma1-init-nonlinear-proof}
	\parcompZ_{t+1} &\geq \frac{1}{\sqrt{2\pi}} \cdot \frac{\parcompX_{t}}{\perpcompX_{t}^{2}} \cdot \frac{1}{(1 + r_{t}^{-2})^{1/2}} - C_{1}\Delta_{1} \geq \frac{1}{4\sqrt{\pi}} \cdot \frac{\parcompX_{t}}{\perpcompX_{t}^{2}}\quad \text{and} \nonumber\\
	\perpcompZ_{t+1}^{2} &\leq \rho^{-1} \cdot \perpcompX_{t}^{-2} \cdot (1+r_{t}^{-2})^{-1} \cdot (1+9r_{t}^{-2}) + C_{1}\Delta_{2} \leq2\rho^{-1} \cdot \perpcompX_{t}^{-2}.
\end{align}
Putting the pieces together yields the lower bound
\[
\parcompX_{t+1}^{\mathsf{det}} \geq \frac{1}{2\sqrt{\pi}} \cdot \frac{\parcompZ_{t+1}}{\perpcompZ_{t+1}^{2}} \geq \frac{1}{2\sqrt{\pi}} \cdot \frac{ \frac{1}{4\sqrt{\pi}} \cdot \frac{\parcompX_{t}}{\perpcompX_{t}^{2}} }{ 2\rho^{-1} \cdot \perpcompX_{t}^{-2} } \geq \frac{\rho}{16\pi} \cdot \parcompX_{t}. 
\]
Towards bounding $(\perpcompdetX_{t+1})^2 = \perpmapbit(\parcompZ_{t+1}, \perpcompZ_{t+1})$~\eqref{eq:perpmapbit}, we obtain the inequality 
\begin{align*}
	(\perpcompX_{t+1}^{\mathsf{det}})^{2} &\leq \frac{1+\sigma^{2}}{C(\oversamp)} \cdot \frac{1}{\parcompZ_{t+1}^{2}+\perpcompZ_{t+1}^{2}} \cdot (1+9\widetilde{r}_{t+1}^{-2}) \leq \rho^{-1} \cdot \perpcompZ_{t+1}^{-2}\cdot (1+9\widetilde{r}_{t+1}^{-2}).
\end{align*}
Further, we apply inequalities~\eqref{ineq1-lemma1-init-nonlinear-proof} and~\eqref{ineq2-lemma1-init-nonlinear-proof} to obtain the lower bound
\[
\perpcompZ_{t+1}^{2} \geq \rho^{-1} \cdot \perpcompX_{t}^{-2} \cdot (1+r_{t}^{-2})^{-1} \cdot (1-18r_{t}^{-2}) - C_{1}\Delta_{2}.
\]
Putting the two pieces together yields 
\begin{align*}
	(\perpcompX_{t+1}^{\mathsf{det}})^{2} \leq \frac{1}{\rho}\cdot \frac{ (1+9\widetilde{r}_{t+1}^{-2})}{\rho^{-1} \cdot \perpcompX_{t}^{-2} \cdot (1+r_{t}^{-2})^{-1} \cdot (1-18r_{t}^{-2}) - C_{1}\Delta_{2}} &=
	\frac{\perpcompX_{t}^{2}\cdot (1+r_{t}^{-2})\cdot (1-18r_{t}^{-2})^{-1} \cdot (1+9\widetilde{r}_{t+1}^{-2}) }{ 1-C_{1}\cdot \rho \cdot \perpcompX_{t}^{2}\cdot (1+r_{t}^{-2})\cdot (1-18r_{t}^{-2})^{-1} \cdot \Delta_{2}}.
\end{align*}
Further upper bounding the RHS by applying both $\1$ $\perpcompX_{t}^{2}\cdot (1+r_{t}^{-2})\cdot (1-18r_{t}^{-2})^{-1} \leq 10$ since $\perpcompX_{t}^{2} \leq 2$ and $r_{t}^{2} \geq 50\rho$, and $\2$ $(1-a)^{-1} \leq 1+2a$ for $0 \leq a\leq 0.5$ with $a = 10C_{1}\cdot \rho \cdot \Delta_{2} \asymp \perplogd \sqrt{\oversamp}/\sqrt{d}$ yields the bound 
\begin{align*}
	(\perpcompX_{t+1}^{\mathsf{det}})^{2}	\leq \perpcompX_{t}^{2}\cdot (1+r_{t}^{-2})\cdot (1-18r_{t}^{-2})^{-1} \cdot (1+9\widetilde{r}_{t+1}^{-2}) \cdot (1+20C_{1}\rho\Delta_{2}).
\end{align*}
Note that $\widetilde{r}_{t+1} \geq \frac{\pi}{6} \cdot \rho^{-1/2} \cdot r_{t}$, whence we deduce the inequality
\begin{align*}
	(1+r_{t}^{-2})\cdot (1-18r_{t}^{-2})^{-1} \cdot (1+9\widetilde{r}_{t+1}^{-2}) \leq (1+r_{t}^{-2})\cdot(1+36r_{t}^{-2})\cdot(1+36\rho r_{t}^{-2}) \leq 1+40\rho r_{t}^{-2}.
\end{align*}
Combining the previous two displays yields the upper bound
\[
(\perpcompX_{t+1}^{\mathsf{det}})^{2} \leq \perpcompX_{t}^{2} \cdot (1+40\rho r_{t}^{-2}) \cdot (1+20C_{1}\rho\Delta_{2}) \leq \perpcompX_{t}^{2} + 40\rho\parcompX_{t}^{2} + \frac{C_{1}'\perplogd \sqrt{\oversamp}}{\sqrt{d}}.
\]
Next we derive the lower bound of $(\perpcompX_{t+1}^{\mathsf{det}})^{2}$. We obtain that
\begin{align*}
	(\perpcompX_{t+1}^{\mathsf{det}})^{2} &\geq \frac{1+\sigma^{2}}{C(\oversamp)} \cdot \frac{1}{\parcompZ_{t+1}^{2}+\perpcompZ_{t+1}^{2}} - \frac{160}{\pi^{2}C(\oversamp)} \cdot \frac{1}{\parcompZ_{t+1}^{2} + \perpcompZ_{t+1}^{2}} \cdot \frac{\parcompZ_{t+1}^{2}}{\perpcompZ_{t+1}^{2}} \\&\geq \rho^{-1} \cdot \perpcompZ_{t+1}^{-2} \cdot (1+\widetilde{r}_{t+1}^{-2})^{-1} \cdot (1-18\widetilde{r}_{t+1}^{-2}),
\end{align*}
Applying inequalities~\eqref{ineq1-lemma1-init-nonlinear-proof} and~\eqref{ineq2-lemma1-init-nonlinear-proof} yields 
\[
\perpcompZ_{t+1}^{2} \leq \rho^{-1} \cdot \perpcompX_{t}^{-2} \cdot (1+r_{t}^{-2})^{-1} \cdot (1+9r_{t}^{-2}) + C_{1}\Delta_{2} \leq \rho^{-1} \cdot \perpcompX_{t}^{-2} \cdot (1+9r_{t}^{-2}) + C_{1}\Delta_{2} .
\]
Putting the two pieces together yields 
\begin{align*}
	(\perpcompX_{t+1}^{\mathsf{det}})^{2} &\geq \frac{1}{\rho} \cdot \frac{(1+\widetilde{r}_{t+1}^{-2})^{-1} \cdot (1-18\widetilde{r}_{t+1}^{-2})}{\rho^{-1} \cdot \perpcompX_{t}^{-2}  \cdot (1+9r_{t}^{-2}) + C_{1}\Delta_{2}} &= \frac{\perpcompX_{t}^{2}\cdot(1+9r_{t}^{-2})^{-1}\cdot(1+\widetilde{r}_{t+1}^{-2})^{-1} \cdot (1-18\widetilde{r}_{t+1}^{-2})}{1 + C_{1}\perpcompX_{t}^{2}(1+9r_{t}^{-2})^{-1}\rho\Delta_{2}}.,
\end{align*}
Using $\1$ $\perpcompX_{t}^{2}(1+9r_{t}^{-2})^{-1} \leq 10$ since $\perpcompX_{t}^{2} \leq 2$ and $r_{t}^{2}\geq 20\rho$, and $\2$ $(1+a)^{-1} \geq 1-a$ with $a = 10C_{1}\rho\Delta_{2}$ yields 
\begin{align*}
	(1 + C_{1}\perpcompX_{t}^{2}(1+9r_{t}^{-2})^{-1}\rho\Delta_{2})^{-1} \geq 1 - 10C_{1}\rho\Delta_{2}.
\end{align*}
Continuing, using $\widetilde{r}_{t+1} \geq \frac{\pi}{6} \cdot \rho^{-1/2} \cdot r_{t}$, we obtain that
\[
(1+9r_{t}^{-2})^{-1}\cdot(1+\widetilde{r}_{t+1}^{-2})^{-1} \cdot (1-18\widetilde{r}_{t+1}^{-2}) \geq (1-9r_{t}^{-2})\cdot(1-\widetilde{r}_{t+1}^{-2}) \cdot (1-18\widetilde{r}_{t+1}^{-2}) \geq 1-100\rho\cdot r_{t}^{-2}.
\]
Putting the pieces together, we obtain the lower bound
\[
(\perpcompX_{t+1}^{\mathsf{det}})^{2} \geq \perpcompX_{t}^{2}\cdot (1-100\rho\cdot r_{t}^{-2})\cdot (1 - 10C_{1}\rho\Delta_{2}) \geq \perpcompX_{t}^{2} - 100\rho \cdot \parcompX_{t}^{2} - \frac{C_{1}'\perplogd \sqrt{\oversamp}}{\sqrt{d}}.
\]
Taking stock, we have shown the pair of sandwich relations
\begin{align}\label{ineq:penultimate-nonlinear-random-init}
	\frac{\rho}{16\pi} \cdot \parcompX_{t} &\leq \parcompX_{t+1}^{\mathsf{det}} \leq 2\rho \cdot \parcompX_{t} \quad \text{ and } \nonumber\\
	\perpcompX_{t}^{2} - 100\rho \cdot \parcompX_{t}^{2} - \frac{C_{1}'\perplogd \sqrt{\oversamp}}{\sqrt{d}} &\leq (\perpcompX_{t+1}^{\mathsf{det}})^{2} \leq \perpcompX_{t}^{2} + 40\rho\parcompX_{t}^{2} + \frac{C_{1}'\perplogd \sqrt{\oversamp}}{\sqrt{d}}.
\end{align}
To conclude, we lower bound the sum of squares $\parcompZ_{t+1}^2 + \perpcompZ_{t+1}^2$ as 
\[
\parcompZ_{t+1}^{2}+\perpcompZ_{t+1}^{2} \geq \perpcompZ_{t+1}^{2} \geq \rho^{-1} \cdot \perpcompX_{t}^{-2} \cdot (1+r_{t}^{-2})^{-1} \cdot (1-18r_{t}^{-2}) - C_{1}\Delta_{2} \geq \rho^{-1}/4,
\]
where in the last step we used inequality~\eqref{Detla1-2-bound}.  Finally, by the above inequality, we note that on event $\mathcal{B}_{t+1}$, 
\begin{align*}
	\bigl\lvert \parcompX_{t+1} - \parcompX_{t+1}^{\mathsf{det}}\bigr\rvert \lesssim \sqrt{\frac{\log(d)}{d}} \quad \text{and} \quad \bigl\lvert \perpcompX_{t+1}^{2} - (\perpcompX_{t+1}^{\mathsf{det}})^{2} \bigr\rvert \lesssim \frac{\perplogd \sqrt{\oversamp}}{\sqrt{d}}.
\end{align*}
Note that $\sqrt{\log(d)/d} = o_{d}(\rho \parcompX_{t})$ by the assumption $\oversamp \geq C(1+\sigma^{2})\log(d)$ and $\parcompX_{t} \gtrsim 1/\sqrt{d}$. The conclusion follows immediately upon combining the above display with the pair of inequalities~\eqref{ineq:penultimate-nonlinear-random-init}. \qed

\subsection{Proof of Lemma~\ref{lem:initialization}(b)}\label{sec:proof-lem-initialization-b}
We proceed by induction on $0 \leq t \leq T_{\star}$.  

\medskip
\noindent \underline{Base case $t=0$:} Assumption~\ref{asm:random-linear} guarantees that
\[
	\frac{1}{50\sqrt{d}} \leq \parcompX_{0} \leq 1,\;0.8 \leq \perpcompX_{0}^{2} \leq 1.2 \;\; \text{and} \;\; \frac{\perpcompX_{0}^{2}}{\parcompX_{0}^{2}} \geq \frac{20C(\oversamp)}{1+\sigma^{2}}.
\]
Consequently, applying Lemma~\ref{lem:initialization}(a) yields that the inequalities~\eqref{eq:strange-dream} hold for $t=0$. 

\medskip
\noindent \underline{Induction step:} Suppose inequalities~\eqref{eq:strange-dream} hold for all $t \leq T_{\star} - 1$.  We need to prove that these inequalities also hold at iteration $t+1$. Note that $\parcompX_{t+1} \leq 1/(30\sqrt{\rho})$ by definition of $T_\star$.  By the induction hypothesis and using Eq.~\eqref{eq:strange-dream-a}, we have $\parcompX_{t+1} \geq \frac{\rho}{60} \parcompX_{t} \geq \left( \frac{\rho}{60} \right)^{t + 1} \parcompX_{0}$. Putting these two together yields $\left( \frac{\rho}{60} \right)^{t + 1} \parcompX_{0} \leq 1/(30\sqrt{\rho})$, from which we obtain that
\begin{align}\label{ineq:ub-iteration-linear}
	t +1 \leq \log_{\rho/60}\Bigl(\frac{1}{30\sqrt{\rho} \cdot \parcompX_{0}} \Bigr) \overset{\1}{\leq} \log_{\rho/60}\Bigl(\frac{5\sqrt{d}}{3\sqrt{\rho} } \Bigr) \leq \log_{\rho/60}(d),
\end{align}
where in step $\1$ we used the assumption $\parcompX_{0} \geq d^{-1/2}/50$.  Next, since $\parcompX_{t + 1} \geq (\rho/60)^{t + 1 - \tau} \parcompX_{\tau}$ for $ 0 \leq \tau \leq t + 1$, we obtain the upper bound
\begin{align}
	\label{ineq:ub-sum-square-parcomp-linear}
	\sum_{\tau=0}^{t + 1} \parcompX_{\tau}^{2} \leq \sum_{\tau=0}^{t + 1} (\rho/60)^{2\tau-2(t + 1)}\parcompX_{t + 1}^{2} \overset{\1}{\leq} \parcompX_{t + 1}^{2} \cdot \sum_{\tau'=0}^{t + 1}(\rho/60)^{-2\tau'} \overset{\2}{\leq} 2\parcompX_{t + 1}^{2} \overset{\3}{\leq} \frac{1}{450\rho},
\end{align}
where in step $\1$ we have made the change of variables $\tau' = t+ 1 - \tau$, and step $\2$ holds provided $\oversamp \geq C$ for some large enough positive constant $C$. Step $\3$ holds since by assumption $\parcompX_{t + 1} \leq 1/(30\sqrt{\rho})$. 
Combining the inequalities~\eqref{ineq:ub-iteration-linear} and~\eqref{ineq:ub-sum-square-parcomp-linear} with Eq.~\eqref{eq:strange-dream-c} of the induction hypothesis yields the pair of inequalities
\begin{align*}
	\perpcompX_{t+1}^{2} &\leq \perpcompX_{0}^{2} + 40\rho\sum_{\tau=0}^{t + 1} \parcompX_{\tau}^{2} + (t + 1) \cdot \frac{C_{1}\perplogd \sqrt{\oversamp}}{\sqrt{d}} \leq 1.2+\frac{40}{450} + \log_{\rho/60}(d)  \frac{C_{1}\perplogd \sqrt{\oversamp}}{\sqrt{d}} \leq 1.4 \qquad \text{and}\\
	\perpcompX_{t+1}^{2} &\geq \perpcompX_{0}^{2} - 100\rho\sum_{\tau=0}^{t} \parcompX_{\tau}^{2} - (t+1) \cdot \frac{C_{1}\perplogd \sqrt{\oversamp}}{\sqrt{d}} \geq 0.8 - \frac{100}{450} - \log_{\rho/60}(d)  \frac{C_{1}\perplogd \sqrt{\oversamp}}{\sqrt{d}} \geq 0.5,
\end{align*}
where we have used $\oversamp \leq \sqrt{d}$. We have thus established Eq.~\eqref{eq:strange-dream-b} of the induction step.  In order to establish Eqs.~\eqref{eq:strange-dream-a} and~\eqref{eq:strange-dream-c}, we note that $\perpcompX_{t+1}^2/\parcompX_{t+1}^2 \geq 0.5 \cdot 900 \rho \geq 20 C(\oversamp)/(1 + \sigma^2)$, where we use $\parcompX_{t+1} \leq 1/(30\sqrt{\rho})$.
Taking stock, the following inequalities hold
\[
\frac{1}{50\sqrt{d}} \leq \parcompX_{0} \leq \parcompX_{t+1} \leq \frac{1}{30\sqrt{\rho}} \leq 1, \quad 0.5 \leq \perpcompX_{t+1}^{2} \leq 2 \quad \text{ and }\quad \frac{\perpcompX_{t+1}^{2}}{\parcompX_{t+1}^{2}} \geq \frac{20C(\oversamp)}{1+\sigma^{2}}.
\]
Applying Lemma~\ref{lem:initialization}(a) then proves the remaining claims in the induction step.  \qed

\subsection{Proof of Lemma~\ref{lemma2-good-region}}\label{sec:proof-lemma-good-region}
We require the following lemma, deferring its proof to Section~\ref{sec:proof-help-lemma2-GD}.
\begin{lemma}\label{help-lemma2-GD-linear}
	Let $\parcompX > 0$ and $\perpcompX\geq 0 $ satisfy $\perpcompX/\parcompX \lesssim \sqrt{\oversamp/(1+\sigma^{2})}$.  Further, suppose that with $\psi \in \{ \mathsf{id}, \sign \}$ and $(\parcompdetX, \perpcompdetX) = \bigl(F_{\psi}(\parcompX, \perpcompX), G_{\psi}(\parcompX, \perpcompX)\bigr)$, $\parcompX', \perpcompX'$ satisfy
	\[
	\bigl \lvert \parcompX' - \parcompdetX \bigr \rvert \lesssim \frac{\parsigma}{\sqrt{\parcompX^2 + \perpcompX^2}} \pardevn, \qquad \text{ and } \qquad \bigl \lvert \perpcompX'^{2} - (\perpcompdetX)^{2} \bigr \rvert \lesssim \frac{1+\sigma^{2}}{\parcompX^2 + \perpcompX^2} \frac{\perplogn}{\sqrt{n}}. 
	\]	
	Then, there exists a pair of universal, positive constants $(c,C)$ such that for $\oversamp \geq C(1+\sigma^{2})$,  and $\log(n)/d^{1/2} \leq c$, the following holds.
	\[
	\biggl \lvert \Bigl(\frac{\perpcompX'}{\parcompX'}\Bigr)^2 - \ratiomappsi\Bigl(\frac{\perpcompX^2}{\parcompX^2}\Bigr)\biggl\lvert \lesssim \Bigl(1 \vee \frac{\perpcompX^2}{\parcompX^2}\Bigr)\cdot \frac{\perplogn(1+\sigma^{2})}{\sqrt{n}} + \Bigl(1 \vee \frac{\perpcompX}{\parcompX}\Bigr)\cdot(\parsigma) \pardevn.
	\]
\end{lemma}

\subsubsection{Proof of Lemma~\ref{lemma2-good-region} for the linear observation model}
We begin by showing $\perpcompZ_{t+1}/\parcompZ_{t+1} \leq 51$.  Note that on the event $\mathcal{A}_{t+1}$, 
\[
\bigl \lvert \parcompZ_{t+1} - \parcompdetZ_{t+1} \bigr \rvert \lesssim \frac{1+\sigma}{\sqrt{\parcompX_{t}^2 + \perpcompX_{t}^2}} \cdot \pardevn \qquad \text{ and } \qquad\bigl \lvert (\perpcompZ_{t+1})^2 - (\perpcompdetZ_{t+1})^2 \bigr \rvert \lesssim \frac{1+\sigma^{2}}{\parcompX_{t}^2 + \perpcompX_{t}^2} \cdot \frac{ \perplogn}{\sqrt{n}}.  
\]
By assumption, $\perpcompX_t/\parcompX_t \leq 50 \sqrt{\rho} \lesssim \sqrt{\oversamp/(1+\sigma^{2})}$, whence we apply Lemma~\ref{help-lemma2-GD-linear} to obtain the bound 
\begin{align}\label{ineq1-lemma2-GD-linear}
	\begin{split}
		\biggl\lvert \Bigl( \frac{\perpcompZ_{t+1}}{\parcompZ_{t+1}} \Bigr)^{2} - \ratiomapid\Bigl( \frac{\perpcompX_{t}^{2}}{\parcompX_{t}^{2}}\Bigr) \biggr \rvert &\lesssim \Bigl(1 \vee \frac{\perpcompX^2}{\parcompX^2}\Bigr)\cdot(1+\sigma^{2}) \frac{\perplogn}{\sqrt{n}} + \Bigl(1 \vee \frac{\perpcompX}{\parcompX}\Bigr)\cdot (1+\sigma)\pardevn \\&\lesssim \oversamp \cdot \frac{\perplogn}{\sqrt{n}} + \sqrt{\oversamp} \pardevn.
	\end{split}
\end{align}
Consequently, using the assumption $\oversamp \leq \sqrt{d}$ and $\oversamp \geq C(1+\sigma^{2})$,
\begin{align*}
	\Bigl( \frac{\perpcompZ_{t+1}}{\parcompZ_{t+1}} \Bigr)^{2} &\leq \ratiomapid\left( \frac{\perpcompX_{t}^{2}}{\parcompX_{t}^{2}}\right) + \frac{C_{1}\perplogd \sqrt{\oversamp}}{\sqrt{d}} = \frac{1+\sigma^{2}}{C(\oversamp)} \cdot \frac{\perpcompX_{t}^{2}}{\parcompX_{t}^{2}} + \frac{\sigma^{2}}{C(\oversamp)} + \frac{C_{1}\perplogd}{d^{1/2}} \leq
	50^{2}+1.
\end{align*}
Equipped with this bound, we turn to bounding $\perpcompX_{t+1}/\parcompX_{t+1}$.  Proceeding in a parallel manner to above, we obtain the inequality
\begin{align}\label{ineq2-lemma2-GD-linear}
	\begin{split}
		\biggl\lvert \Bigl( \frac{\perpcompX_{t+1}}{\parcompX_{t+1}} \Bigr)^{2} - \ratiomapid\Bigl( \frac{\perpcompZ_{t+1}^{2}}{\parcompZ_{t+1}^{2}}\Bigr) \biggr \rvert &\lesssim \Bigl(1 \vee \frac{\perpcompZ_{t+1}^2}{\parcompZ_{t+1}^2}\Bigr)\cdot \frac{\perplogn(1+\sigma^{2})}{\sqrt{n}} + \Bigl(1 \vee \frac{\perpcompZ_{t+1}}{\parcompZ_{t+1}}\Bigr)\cdot (1+\sigma)\pardevn \\&\lesssim \frac{\perplogn (1+\sigma^{2}) }{\sqrt{n}}.
	\end{split}
\end{align}
We thus decompose and apply the triangle inequality to obtain
\begin{align*}
	\biggl\lvert \Bigl( \frac{\perpcompX_{t+1}}{\parcompX_{t+1}} \Bigr)^{2} - \ratiomapid \circ \ratiomapid\Bigl( \frac{\perpcompX_{t}^{2}}{\parcompX_{t}^{2}}\Bigr) \biggr\rvert &\leq
	\biggl\lvert \Bigl( \frac{\perpcompX_{t+1}}{\parcompX_{t+1}} \Bigr)^{2} - \ratiomapid\Bigl( \frac{\perpcompZ_{t+1}^{2}}{\parcompZ_{t+1}^{2}}\Bigr) \biggr\rvert + 
	\biggl\lvert \ratiomapid\Bigl( \frac{\perpcompZ_{t+1}^{2}}{\parcompZ_{t+1}^{2}}\Bigr) - \ratiomapid \circ \ratiomapid\Bigl( \frac{\perpcompX_{t}^{2}}{\parcompX_{t}^{2}}\Bigr) \biggr\rvert \\&\overset{\1}{=}
	\biggl\lvert \Bigl( \frac{\perpcompX_{t+1}}{\parcompX_{t+1}} \Bigr)^{2} - \ratiomapid\Bigl( \frac{\perpcompZ_{t+1}^{2}}{\parcompZ_{t+1}^{2}}\Bigr) \biggr\rvert + 
	\frac{1+\sigma^{2}}{C(\oversamp)} \cdot \biggl\lvert \frac{\perpcompZ_{t+1}^{2}}{\parcompZ_{t+1}^{2}} -  \ratiomapid\Bigl( \frac{\perpcompX_{t}^{2}}{\parcompX_{t}^{2}}\Bigr) \biggr\rvert \\&\overset{\2}{\lesssim}
	\bigg( \frac{\perplogn}{\sqrt{n}} + \pardevn \bigg)(1+\sigma^{2})+ \frac{1+\sigma^{2}}{\oversamp} \cdot \Bigl( \oversamp \cdot \frac{\perplogn}{\sqrt{n}} + \sqrt{\oversamp} \cdot  \pardevn \Bigr) \\&\lesssim
	\frac{\perplogn(1+\sigma^{2})}{\sqrt{n}},
\end{align*}
where in step $\1$ we use $\ratiomapid(x) = \frac{1+\sigma^{2}}{C(\oversamp)}\cdot x+\frac{\sigma^{2}}{C(\oversamp)}$, step $\2$ follows from inequalities~\eqref{ineq1-lemma2-GD-linear} and~\eqref{ineq2-lemma2-GD-linear}.  This completes the proof. \qed

\subsubsection{Proof of Lemma~\ref{lemma2-good-region} for the one-bit observation model}
We begin by estimating $\perpcompZ_{t+1}/\parcompZ_{t+1}$, noting that by applying Lemma~\ref{help-lemma2-GD-linear} inequality~\eqref{ineq1-lemma2-GD-linear} continues to hold, so that
\begin{align}\label{ineq1-lemma2-GD-nonlinear}
	\biggl\lvert \Bigl( \frac{\perpcompZ_{t+1}}{\parcompZ_{t+1}} \Bigr)^{2} - \ratiomapbit\Bigl( \frac{\perpcompX_{t}^{2}}{\parcompX_{t}^{2}}\Bigr) \biggr \rvert \lesssim \oversamp \cdot \frac{\perplogn}{\sqrt{n}} + \sqrt{\oversamp} \cdot \pardevn.
\end{align}
Applying Lemma~\ref{lem:h-sgn-properties} yields the upper bound 
\[
\ratiomapbit\left( \frac{\perpcompX_{t}^{2}}{\parcompX_{t}^{2}}\right) \leq \max\left\{ \frac{\pi^{2}}{2} \cdot \frac{1}{\rho} \cdot \frac{\perpcompX_{t}^{2}}{\parcompX_{t}^{2}} + \frac{20}{C(\oversamp)}, \frac{50\pi^{2}}{\rho} \right\} \leq C,
\]
where in the last step we use $\frac{\perpcompX_{t}}{\parcompX_{t}} \leq 50\sqrt{\rho}$.
Combining the previous two displays, we obtain the inequality
\begin{align*}
	\Bigl( \frac{\perpcompZ_{t+1}}{\parcompZ_{t+1}} \Bigr)^{2} &\leq \ratiomapbit\Bigl( \frac{\perpcompX_{t}^{2}}{\parcompX_{t}^{2}}\Bigr) + C_{1}\oversamp \cdot \frac{\perplogn}{\sqrt{n}} \leq C,
\end{align*}
where in the last step we used $\oversamp \leq \sqrt{d}$.
With this estimate of $\perpcompZ_{t+1}/\parcompZ_{t+1}$ in hand, we turn to bounding $\perpcompX_{t+1}/\parcompX_{t+1}$.  Proceeding in a parallel manner to above, we obtain the inequality (cf. inequality~\eqref{ineq2-lemma2-GD-linear})
\begin{align}\label{ineq2-lemma2-GD-nonlinear}
	\biggl\lvert \Bigl( \frac{\perpcompX_{t+1}}{\parcompX_{t+1}} \Bigr)^{2} - \ratiomapbit\Bigl( \frac{\perpcompZ_{t+1}^{2}}{\parcompZ_{t+1}^{2}}\Bigr) \biggr \rvert \lesssim \frac{\perplogn (1+\sigma^{2})}{\sqrt{n}}.
\end{align}
We next note the upper bound
\begin{align*}
	\biggl \lvert \ratiomapbit\Bigl(\frac{\perpcompZ_{t+1}^{2}}{\parcompZ_{t+1}^{2}}\Bigr) - \ratiomapbit \circ \ratiomapbit \Bigl(\frac{\perpcompX_{t}^{2}}{\parcompX_{t}^{2}}\Bigr) \biggr \rvert \leq \max_{0 \leq \lambda \leq 1}\; \biggl \lvert \ratiomapbit'\Bigl( \lambda \cdot \frac{\perpcompZ_{t+1}^{2}}{\parcompZ_{t+1}^{2}} + (1 - \lambda) \ratiomapbit\Bigl(\frac{\perpcompX_{t}^{2}}{\parcompX_{t}^{2}}\Bigr) \Bigr) \biggr \rvert \cdot \biggl \lvert \frac{\perpcompZ_{t+1}^{2}}{\parcompZ_{t+1}^{2}} - \ratiomapbit\Bigl(\frac{\perpcompX_{t}^{2}}{\parcompX_{t}^{2}}\Bigr) \biggr \rvert.
\end{align*}
Applying Lemma~\ref{f-h-properties-nonlinear}(a) in conjunction with the bounds $\perpcompZ_{t+1}^2/\parcompZ_{t+1}^2 \leq C$, $\ratiomapbit(\perpcompX_{t}^2/\parcompX_{t}^2) \leq C$, and $C(\oversamp) \asymp \oversamp$ yields the bound 
\[
\max_{0 \leq \lambda \leq 1}\; \biggl \lvert \ratiomapbit'\Bigl( \lambda \cdot \frac{\perpcompZ_{t+1}^{2}}{\parcompZ_{t+1}^{2}} + (1 - \lambda) \ratiomapbit\Bigl(\frac{\perpcompX_{t}^{2}}{\parcompX_{t}^{2}}\Bigr) \Bigr) \biggr \rvert \lesssim \frac{1+\sigma^{2}}{\oversamp},
\]
so that 
\begin{align}\label{ineq:final-local-nonlinear}
	\biggl \lvert \ratiomapbit\Bigl(\frac{\perpcompZ_{t+1}^{2}}{\parcompZ_{t+1}^{2}}\Bigr) - \ratiomapbit \circ \ratiomapbit \Bigl(\frac{\perpcompX_{t}^{2}}{\parcompX_{t}^{2}}\Bigr) \biggr \rvert \lesssim \frac{1+\sigma^{2}}{\oversamp} \cdot \biggl \lvert \frac{\perpcompZ_{t+1}^{2}}{\parcompZ_{t+1}^{2}} - \ratiomapbit\Bigl(\frac{\perpcompX_{t}^{2}}{\parcompX_{t}^{2}}\Bigr) \biggr \rvert.
\end{align}
To conclude, we decompose and apply the triangle inequality to obtain the inequality 
\begin{align*}
	\biggl\lvert \Bigl( \frac{\perpcompX_{t+1}}{\parcompX_{t+1}} \Bigr)^{2} - \ratiomapbit \circ \ratiomapbit\Bigl( \frac{\perpcompX_{t}^{2}}{\parcompX_{t}^{2}}\Bigr) \biggr\rvert &\leq
	\biggl\lvert \Bigl( \frac{\perpcompX_{t+1}}{\parcompX_{t+1}} \Bigr)^{2} - \ratiomapbit\Bigl( \frac{\perpcompZ_{t+1}^{2}}{\parcompZ_{t+1}^{2}}\Bigr) \biggr\rvert + 
	\biggl\lvert \ratiomapbit\Bigl( \frac{\perpcompZ_{t+1}^{2}}{\parcompZ_{t+1}^{2}}\Bigr) - \ratiomapbit \circ \ratiomapbit\Bigl( \frac{\perpcompX_{t}^{2}}{\parcompX_{t}^{2}}\Bigr) \biggr\rvert.
\end{align*}
We conclude by applying the inequalities~\eqref{ineq1-lemma2-GD-nonlinear},~\eqref{ineq2-lemma2-GD-nonlinear}, and~\eqref{ineq:final-local-nonlinear} to the RHS in the display above to obtain the inequality
\[
\biggl\lvert \Bigl( \frac{\perpcompX_{t+1}}{\parcompX_{t+1}} \Bigr)^{2} - \ratiomapbit \circ \ratiomapbit\Bigl( \frac{\perpcompX_{t}^{2}}{\parcompX_{t}^{2}}\Bigr) \biggr\rvert \lesssim
\frac{\perplogn(1+\sigma^{2})}{\sqrt{n}},
\]
which concludes the proof.\qed

\subsubsection{Proof of Lemma~\ref{help-lemma2-GD-linear}}\label{sec:proof-help-lemma2-GD}
We separate the two cases, $\psi(x) = x$ and $\psi(x) = \sign(x)$, proving each part in turn.  

\paragraph{Proof of Lemma~\ref{help-lemma2-GD-linear} with $\psi(x) = x$}
Applying the condition of the lemma and re-arranging yields
\begin{align*}
	\parcompX' \geq 
	\frac{\parcompX}{\parcompX^{2}+\perpcompX^{2}} - \frac{C_{1}(\parsigma)}{\sqrt{\parcompX^{2}+\perpcompX^{2}}} \cdot \pardevn &=
	\frac{1}{\sqrt{\parcompX^{2}+\perpcompX^{2}}}\cdot \biggl( \frac{1}{\sqrt{1+\perpcompX^{2}/\parcompX^{2}}} - C_{1}(1+\sigma)\pardevn \biggr).
\end{align*}
Continuing, using the assumption $\perpcompX/\parcompX \lesssim \sqrt{\oversamp/(1+\sigma^{2})}$, $\oversamp \geq C(1+\sigma^{2})$ and $\log(n)/d^{1/2} \leq c$, we obtain 
\begin{align*}
	\frac{1}{\sqrt{1 + \perpcompX^2/\parcompX^2}} \gtrsim \sqrt{\frac{1+\sigma^{2}}{\oversamp}} \text{ and }(1+\sigma)\pardevn \leq \sqrt{\frac{1+\sigma^{2}}{\oversamp}} \cdot c.
\end{align*} 
Putting the two pieces together and letting $c$ small enough yields
\[
	\parcompX' \geq \frac{1}{\sqrt{\parcompX^{2}+\perpcompX^{2}}}\cdot\frac{0.5}{\sqrt{1+\perpcompX^{2}/\parcompX^{2}}}.
\]
Decomposing and applying the triangle inequality, we obtain
\begin{align}\label{ineq1-help-lemma2-GD-linear}
	\biggl \lvert \Bigl(\frac{\perpcompX'}{\parcompX'} \Bigr)^{2} - \Bigl(\frac{\perpcompdetX}{\parcompdetX} \Bigr)^{2} \biggr \rvert &\leq \frac{\left| (\perpcompX')^{2} - (\perpcompdetX)^{2}\right|}{(\parcompX')^{2}} + 
	\biggl\lvert \Bigl( \frac{\perpcompdetX}{\parcompX'} \Bigr)^{2} -\Bigl( \frac{\perpcompdetX}{\parcompdetX}\Bigr)^{2} \biggr\rvert \nonumber\\
	&= \frac{\left| (\perpcompX')^{2} - (\perpcompdetX)^{2}\right|}{(\parcompX')^{2}} + \Bigl(\frac{\perpcompdetX}{\parcompdetX}\Bigr)^2 \cdot \frac{\bigl \lvert (\parcompX')^2 - (\parcompdetX)^2\bigr \rvert}{(\parcompX')^2} \nonumber\\
	&\leq \frac{\left| (\perpcompX')^{2} - (\perpcompdetX)^{2}\right|}{(\parcompX')^{2}} + \Bigl(\frac{\perpcompdetX}{\parcompdetX}\Bigr)^2 \cdot \biggl[\frac{\bigl (\parcompX' - \parcompdetX\bigr)^2}{(\parcompX')^2} + \frac{2\bigl \lvert \parcompX' - \parcompdetX\bigr \rvert}{\parcompX'}\biggr]
\end{align}
Note that 
\[
\Bigl( \frac{\perpcompdetX}{\parcompdetX} \Bigr)^2 = \ratiomapid\Bigl(\frac{\perpcompX^2}{\parcompX^2}\Bigr) = \frac{1 + \sigma^2}{C(\oversamp)} \frac{\perpcompX^2}{\parcompX^2} + \frac{\sigma^2}{C(\oversamp)} \lesssim  1,
\]
where the final inequality follows from the assumption $\perpcompX/\parcompX \lesssim \sqrt{\oversamp/(1+\sigma^{2})}$ and $C(\oversamp) \asymp \oversamp \geq (1+\sigma^{2})$.  Applying the assumption in conjunction with the lower bound on $\parcompX'$ yields the pair of inequalities
\[
\frac{\lvert (\perpcompX')^2 - (\perpcompdetX)^2 \rvert}{(\parcompX')^2} \lesssim \Bigl(1 \vee \frac{\perpcompX^2}{\parcompX^2} \Bigr) \cdot \frac{\perplogn(1+\sigma^{2})}{\sqrt{n}} \qquad \text{ and } \qquad \frac{\lvert \parcompX' - \parcompdetX \rvert}{\parcompX'} \lesssim \Bigl(1 \vee \frac{\perpcompX}{\parcompX} \Bigr)(1+\sigma) \pardevn.
\]
Putting the pieces together yields the inequality 
\[
\biggl \lvert \Bigl( \frac{\perpcompX'}{\parcompX'} \Bigr)^{2} -\Bigl( \frac{\perpcompdetX}{\parcompdetX}\Bigr)^{2} \biggr\rvert \lesssim \Bigl(1 \vee \frac{\perpcompX^2}{\parcompX^2} \Bigr) \cdot \frac{\perplogn(1+\sigma^{2})}{\sqrt{n}} +  \Bigl(1 \vee \frac{\perpcompX}{\parcompX} \Bigr)\cdot (1+\sigma)\pardevn.
\]
The result follows since $(\perpcompdetX)^2/(\parcompdetX)^2 = \ratiomapid(\perpcompX^2/\parcompX^2)$. \qed

\paragraph{Proof of Lemma~\ref{help-lemma2-GD-linear} with $\psi(x) = \sign(x)$}
The architecture of the proof is nearly identical to the previous paragraph, so we restrict ourselves to the differences.  First, applying part (b) of Lemma~\ref{f-h-properties-nonlinear}, we lower bound $\parcompX'$ as
\begin{align*}
	\parcompX' \geq \frac{1}{\pi \sqrt{2}} \cdot \frac{ (\parcompX/\perpcompX)\wedge 1}{\sqrt{\parcompX^2 + \perpcompX^2}} - \frac{C_1 (1+\sigma)}{\sqrt{\parcompX^2 + \perpcompX^2}} \pardevn \geq \frac{1}{\sqrt{\parcompX^2 + \perpcompX^2}} \cdot \frac{(\parcompX/\perpcompX) \wedge 1}{2\pi},
\end{align*}
where in the final inequality we have used the assumption $\perpcompX/\parcompX \lesssim \sqrt{\oversamp/(1+\sigma^{2})}$, $\oversamp \geq C(1+\sigma^{2})$ and $\log(n)/d^{1/2} \leq c$ and $c$ being a small enough constant, whence 
\[
	(\parcompX/\perpcompX)\wedge 1 \gtrsim \sqrt{\frac{1+\sigma^{2}}{\oversamp}} \text{ and }
	(1+\sigma)\pardevn \leq \sqrt{ \frac{1+\sigma^{2}}{\oversamp}} \cdot c \leq \frac{(\parcompX/\perpcompX)\wedge 1}{100C_{1}}.
\] 
We then note the inequality~\eqref{ineq1-help-lemma2-GD-linear}, which continues to hold, and apply Lemma~\ref{lem:h-sgn-properties} to obtain the bound 
\[
\Bigl(\frac{\perpcompdetX}{\parcompdetX}\Bigr)^{2} = \ratiomapbit\Bigl(\frac{\perpcompX^{2}}{\parcompX^{2}}\Bigr) \leq \max\left\{ \frac{\pi^{2}}{2} \cdot \frac{1}{\rho} \cdot \frac{\perpcompX^{2}}{\parcompX^{2}} + \frac{20}{C(\oversamp)}, \frac{50\pi^{2}}{\rho} \right\} \leq C,
\]
where $\rho = C(\oversamp)/(1+\sigma^{2})$ in the last step we use $\perpcompX/\parcompX \lesssim \sqrt{\oversamp/(1+\sigma^{2})} \asymp \sqrt{\rho}$.  The conclusion follows using identical steps to the proof of Lemma~\ref{help-lemma2-GD-linear}, with $\psi(x) = x$. \qed

\subsection{Proof of Lemma~\ref{lem:h-sgn-properties}}
This section is dedicated to various properties of the function $\ratiomapbit$.  We first state two lemmas before providing their proofs as well as the proof of Lemma~\ref{lem:h-sgn-properties}.  We first define the functions $h_1: \mathbb{R} \rightarrow \mathbb{R}$ and $h_2: \mathbb{R} \rightarrow \mathbb{R}$ as
\begin{align}\label{eq:def-h1-h2}
	h_1(x) = \oversamp \EE \biggl\{ \frac{\lvert W \rvert \phi(x \lvert W \rvert)}{C(\oversamp) + W^2}\biggr\} \qquad \text{ and } \qquad h_2(x) = C_3(\oversamp)^{-1}\EE\biggl\{\frac{\lvert W \rvert^3 \phi(x\lvert W \rvert)}{(C(\oversamp) + W^2)^2}\biggr\},
\end{align}
where $W \sim \mathsf{N}(0, 1)$, $\phi(x) = \int_{0}^{x} e^{-t^2/2}\mathrm{d}t$, and we recall $C_3(\oversamp)$~\eqref{eq:C2-C3}.  We will make use of the following technical lemma, whose proof we provide in Section~\ref{sec:proof-lemma1-proof-f-h-main-properties}.
\begin{lemma}\label{lemma1-proof-f-h-main-properties} 
	Consider the functions $h_1$ and $h_2$~\eqref{eq:def-h1-h2}.  The following hold. 
	\begin{itemize}
		\item[(a)] There is a universal positive constant $C$ such that for $x > 0$, 
		\[
		\frac{x}{(1 + x^2)^{1.5}} \leq h_1(x) \leq x \vee 5, \qquad \text{ and } \qquad h_2(x) \leq x \vee C.  
		\]
		\item[(b)] For $x > 0$, the derivatives satisfy
		\[
		\frac{1}{(1 + x^2)^{1.5}} \leq h_1'(x) \leq \frac{5}{(1 + x^2)^{1.5}} \qquad \text{ and } \qquad 0 \leq h_2'(x) \leq \frac{30}{(1 + x^2)^{2.5}}.
		\]
	\end{itemize}
\end{lemma}

The next lemma bounds the derivative of $\ratiomapbit$ as well as lower bounds the parallel component of the update.  We provide its proof in Section~\ref{sec:proof-f-h-properties-nonlinear}.  
\begin{lemma}\label{f-h-properties-nonlinear}
	Consider the functions $\parmapbit$~\eqref{eq:parmapbit}, $\perpmapbit$~\eqref{eq:perpmapbit}, and $\ratiomapbit$~\eqref{eq:ratiomaps}.  The following hold.
	\begin{itemize}
		\item[(a)] There exists a universal positive constant $C_1$ such that
		\[
		\bigl \lvert \ratiomapbit'(x) \bigr \rvert \leq C_1 \cdot \frac{(1 + x^{3/2})(1+\sigma^{2})}{C(\oversamp)} \qquad \text{ as long as } \qquad x > 0.
		\]
		\item[(b)] The function $\parmapbit$ is lower bounded as 
		\[
		\parmapbit(\parcompX,\perpcompX) \geq \frac{1}{\sqrt{2}\pi} \cdot \frac{\min\{\frac{\parcompX}{\perpcompX},1\}}{(\parcompX^{2} + \perpcompX^{2})^{1/2}},\qquad \text{ as long as } \qquad \frac{\parcompX}{\perpcompX} \geq 0.
		\]
	\end{itemize}
\end{lemma}
\subsubsection{Proof of Lemma~\ref{lemma1-proof-f-h-main-properties}}\label{sec:proof-lemma1-proof-f-h-main-properties}
Taking derivatives yields
\[
h_1'(x) = \oversamp \EE\biggl\{\frac{W^2 e^{-x^2 W^2/2}}{C(\oversamp + W^2)}\biggr\}, \qquad \text{ and } \qquad h_2'(x) = C_3(\oversamp)^{-1}\EE \biggl\{\frac{W^4 e^{-x^2W^2/2}}{(C(\oversamp) + W^2)^2}\biggr\}.
\]
Note that since $e^{-x^2W^2/2} \leq 1$, both $h_1'(x) \leq 1$ and $h_2'(x) \leq 1$.  Towards lower bounding the derivatives, note that 
\begin{align*}
	\EE_{W\sim\mathcal{N}(0,1)}\left\{ \frac{W^{2} e^{-x^{2}W^{2}/2}}{C(\oversamp) + W^{2}} \right\} = \frac{ \EE_{W\sim\mathcal{N}(0,\frac{1}{1+x^{2}})}\left\{ \frac{W^{2}}{C(\oversamp) + W^{2}} \right\} }{\sqrt{1+x^{2}}}  = \frac{ \EE_{W\sim\mathcal{N}(0,1)}\left\{ \frac{W^{2}}{ C(\oversamp) + \frac{W^{2}}{1+x^{2}}} \right\} }{(1+x^{2})^{1.5}}.
\end{align*}
Similarly, 
\[
\EE\left\{ \frac{W^{4}  e^{-x^{2}W^{2}/2}}{(C(\oversamp) + W^{2})^{2}} \right\} = (1+x^{2})^{-2.5}\EE \left\{ \frac{W^{4}}{ (C(\oversamp) + W^{2}/(1+x^{2}) )^{2}} \right\},
\]
whence the derivatives admit the equivalent representations
\begin{align}\label{h1'-h2'}
	h_{1}'(x) = \frac{ \EE\bigl\{ \frac{W^{2}}{ C(\oversamp) + W^{2}/(1+x^{2}) } \bigr\} }{(1+x^{2})^{1.5} \cdot \EE\bigl\{ \frac{W^{2}}{C(\oversamp) + W^{2}}\bigr\} } \qquad \text{ and }\qquad h_{2}'(x) = \frac{ \EE\bigl\{ \frac{W^{4}}{ ( C(\oversamp) + W^{2}/(1+x^{2}))^{2} } \bigr\} }{(1+x^{2})^{2.5} \cdot \EE\bigl\{ \frac{W^{4}}{(C(\oversamp) + W^{2} )^{2}}\bigr\} }.
\end{align}
We thus deduce that $h_1'(x) \geq (1 + x^2)^{-1.5}$ for positive $x$.  Applying the upper and lower bounds on $h_1'(x)$ in conjunction with the fact that $h_1(0) = 0$, we obtain the sandwich relation
\[
\frac{x}{(1 + x^2)^{1.5}} \leq h_1(x) = \int_{0}^{x}h_1'(t) \mathrm{d}t \leq x.
\]
In a similar manner, since $h_2'(x) \leq 1$ and $h_2(0) = 0$, we deduce that $h_2(x) \leq x$ for positive $x$.  Now, using the inequality $\int_{0}^{x \lvert W \rvert} e^{-t^2/2} \mathrm{d}t \leq \sqrt{\pi/2}$ for positive $x$, we obtain the upper bounds
\begin{align*}
	h_1(x) &\leq \oversamp \sqrt{\pi/2} \EE\left\{ \frac{|W|}{C(\oversamp) + W^{2}} \right\} \leq \sqrt{\pi/2} \cdot \EE\{|W|\}\cdot \EE\left\{ \frac{W^{2}}{1 + W^{2}}\right\}^{-1} \leq 5, \qquad \text{ and }\\
	h_2(x) &\leq C_3(\oversamp)^{-1} \sqrt{\pi/2} \EE\left\{ \frac{|W|^{3}}{(C(\oversamp) + W^{2})^{2}} \right\}\leq \sqrt{\pi/2} \cdot \EE\{|W|^{3}\}\cdot \EE\left\{ \frac{W^{4}}{(1 + W^{2})^{2}}\right\}^{-1} \leq C. 
\end{align*}
Finally, we proceed from the relations~\eqref{h1'-h2'} to obtain the upper bounds
\begin{align*}
	h_1'(x) &\leq \frac{ \EE\left\{ W^{2} \right\} }{(1+x^{2})^{1.5} \cdot \EE\bigl\{ \frac{W^{2}}{1 + W^{2}}\bigr\} } \leq \frac{5}{(1+x^{2})^{1.5}} \qquad \text{ and }\\
	h_2'(x) &\leq \frac{ \EE\left\{ W^{4} \right\} }{(1+x^{2})^{2.5} \cdot \EE\bigl\{ \frac{W^{4}}{(1 + W^{2} )^{2}}\bigr\} } \leq \frac{30}{(1+x^{2})^{2.5}},
\end{align*}
which completes the proof. \qed

\subsubsection{Proof of Lemma~\ref{f-h-properties-nonlinear}}\label{sec:proof-f-h-properties-nonlinear}
We prove each part in turn, starting with part (a). 
\paragraph{Proof of Lemma~\ref{f-h-properties-nonlinear}(a)}
Using the functions $h_1$ and $h_2$~\eqref{eq:def-h1-h2}, we write
\begin{align}\label{eq1-proof-f-h-properties-nonlinear}
	\ratiomapbit(x) = \frac{\pi^{2}}{4} \cdot \frac{1+\sigma^{2}}{C(\oversamp)} \cdot \frac{1}{h_{1}(1/\sqrt{x})^{2}} + \frac{1}{C(\oversamp)} \cdot \frac{ C_3(\oversamp) }{ C_2(\oversamp) } \cdot \left( 1 - 2\cdot \frac{h_{2}(1/\sqrt{x})}{h_{1}(1/\sqrt{x})} \right).
\end{align}
Since 
\[
\frac{C_3(\oversamp)}{C_2(\oversamp)} \leq \frac{\EE\{W^{4}\} }{ \EE\left\{ \frac{W^{2}}{ (1+W^{2})^{2}} \right\} } \leq 20,
\]
Computing the derivative and applying the triangle inequality in conjunction with the inequality $C_3(\oversamp)/C_2(\oversamp)  \leq 20$ yields the upper bound
\begin{align}\label{ineq1-h'-onebit-property-proof}
	|\ratiomapbit'(x)| \leq &\frac{\pi^{2}}{4} \cdot \frac{1+\sigma^{2}}{C(\oversamp)} \cdot \frac{|h_{1}'(1/\sqrt{x}) \cdot x^{-1.5}|}{h_{1}(1/\sqrt{x})^{3}} \nonumber\\
	&\qquad \qquad \qquad + \frac{20}{C(\oversamp)} \cdot \Bigl( \frac{|h_{2}'(1/\sqrt{x}) \cdot x^{-1.5}|}{h_{1}(1/\sqrt{x})} +  \frac{h_{2}(1/\sqrt{x})|h_{1}'(1/\sqrt{x}) \cdot x^{-1.5}|}{h_{1}(1/\sqrt{x})^{2}} \Bigr).
\end{align}
Applying Lemma~\ref{lemma1-proof-f-h-main-properties} yields 
\begin{align}\label{ineq2-h'-onebit-property-proof}
	(i.)\;\; |h_{1}'(1/\sqrt{x}) &\cdot x^{-1.5}| \leq \frac{5x^{-1.5}}{(1+1/x)^{1.5}} \leq 5, \qquad (ii.)\;\; h_2(x) \leq C \qquad \text{ and } \nonumber \\
	&(iii.)\;\; |h_{2}'(1/\sqrt{x}) \cdot x^{-1.5}| \leq \frac{30x^{-1.5}}{(1+1/x)^{2.5}} \leq 30x. 
\end{align}
We consider two cases: $x \leq 100$ and $x > 100$.

\bigskip
\noindent\underline{Case 1: $x \leq 100$.} Note that $h_{1}(x)$ is a monotonly increasing function, whence
\[
h_{1}(1/\sqrt{x}) \geq h_{1}(1/\sqrt{100}) \geq 0.1/1.01^{3/2} \qquad \text{ for } x \leq 100,
\]
Consequently, substituting the inequalities~\eqref{ineq2-h'-onebit-property-proof}, $h_{1}(1/\sqrt{x}) \geq 0.1/1.01^{3/2}$ and $h_{2}(x) \leq C$ into inequality~\eqref{ineq1-h'-onebit-property-proof}, we obtain that for $x\leq100$
\[
|\ratiomapbit'(x)| \leq \frac{\pi^{2}}{4} \cdot \frac{1+\sigma^{2}}{C(\oversamp)} \cdot \frac{5}{(0.1/1.01^{3/2})^{3}} + \frac{20}{C(\oversamp)} \cdot \left( \frac{30x}{0.1/1.01^{3/2}} +  \frac{5C}{(0.1/1.01^{3/2})^{2}} \right) \leq \frac{C(1+\sigma^{2})}{C(\oversamp)},
\]
where the last inequality follows since $x\leq100$ and $C$ a large enough constant.  This concludes the first case.

\bigskip
\noindent\underline{Case 2: $x > 100$.} By Lemma~\ref{lemma1-proof-f-h-main-properties}, $h_{1}(1/\sqrt{x}) \geq 1/(1.01^{1.5}\sqrt{x})$ for $x\geq 100$. Consequently, we obtain that for $x\geq 100$,
\[
|\ratiomapbit'(x)| \leq \frac{\pi^{2}}{4} \cdot \frac{1+\sigma^{2}}{C(\oversamp)} \cdot \frac{5}{(1/(1.01^{1.5}\sqrt{x}))^{3}} + \frac{20}{C(\oversamp)} \cdot \left( \frac{30x}{1/(1.01^{1.5}\sqrt{x})} +  \frac{5C}{(1/(1.01^{1.5}\sqrt{x}))^{2}} \right) \leq \frac{Cx^{1.5}(1+\sigma^{2})}{C(\oversamp)},
\]
which concludes the second case. 

Combining the two cases yields part (a), so we turn now to part (b).

\paragraph{Proof of Lemma~\ref{f-h-properties-nonlinear}(b)} Use the function $h_1$ to write
\[
\parmapbit(\parcompX,\perpcompX) = \frac{2}{\pi} \cdot (\parcompX^{2} + \perpcompX^{2})^{-1/2} \cdot h_{1}(\parcompX/\perpcompX).
\]
If $\parcompX/\perpcompX \geq 1$,then applying the monotone increasing nature of  $h_{1}(x)$ in conjunction with Lemma~\ref{lemma1-proof-f-h-main-properties} yields the lower bound $h_{1}(\parcompX/\perpcompX) \geq h_{1}(1) \geq 2^{-1.5}$.  Consequently, 
\[
\parmapbit(\parcompX,\perpcompX) \geq \frac{1}{\sqrt{2}\pi} \cdot (\parcompX^{2} + \perpcompX^{2})^{-1/2}, \qquad \text{ if }\qquad \frac{\parcompX}{\perpcompX} \geq 1.
\]
Conversely, if $0 \leq \parcompX/\perpcompX \leq 1$, then by Lemma~\ref{lemma1-proof-f-h-main-properties}, 
\[
h_{1}(\parcompX/\perpcompX) \geq \frac{\parcompX/\perpcompX}{(1+\parcompX^{2}/\perpcompX^{2})^{1.5}} \geq \parcompX/\perpcompX \cdot 2^{-1.5}.
\]
Consequently in this case, 
\[
\parmapbit(\parcompX,\perpcompX) \geq \frac{1}{\sqrt{2}\pi} \cdot (\parcompX^{2} + \perpcompX^{2})^{-1/2} \cdot \frac{\parcompX}{\perpcompX}, \qquad \text{ if }\qquad 0\leq\frac{\parcompX}{\perpcompX} \leq 1.
\]
Combining the two cases yields the conclusion. \qed

\subsubsection{Proof of Lemma~\ref{lem:h-sgn-properties}}\label{sec:proof-lem-h-sgn-properties}
Using the functions $h_1$ and $h_2$~\eqref{eq:def-h1-h2}, we write
\begin{align*}
	\ratiomapbit(x) = \frac{\pi^{2}}{4} \cdot \frac{1+\sigma^{2}}{C(\oversamp)} \cdot \frac{1}{h_{1}(1/\sqrt{x})^{2}} + \frac{1}{C(\oversamp)} \cdot \frac{ C_3(\oversamp) }{ C_2(\oversamp) } \cdot \left( 1 - 2\cdot \frac{h_{2}(1/\sqrt{x})}{h_{1}(1/\sqrt{x})} \right).
\end{align*}
Applying Lemma~\ref{lemma1-proof-f-h-main-properties} yields the pair of inequalities (which hold for $x > 0$)
\[
\frac{1}{\sqrt{x}} \cdot \frac{1}{(1+1/x)^{1.5}} \leq h_{1}\Bigl(\frac{1}{\sqrt{x}}\Bigr) \leq  \frac{1}{\sqrt{x}} \vee 5 \qquad \text{ and }\qquad h_{2}\Bigl(\frac{1}{\sqrt{x}}\Bigr) \leq \frac{1}{\sqrt{x}} \vee C
\]
Consequently, we note the lower bound $h_{1}(1/\sqrt{x}) \geq \frac{1}{1.01^{1.5}\sqrt{x}}$ for $x\geq 100$.  Note additionally that
\[
\frac{C_3(\oversamp)}{C_2(\oversamp)} \leq \frac{\EE\{W^{4}\} }{ \EE\left\{ \frac{W^{2}}{ (1+W^{2})^{2}} \right\} } \leq 20.
\]
Thus, we deduce the upper bound
\[
\ratiomapbit(x) \leq \frac{\pi^{2}}{4} \cdot \frac{1+\sigma^{2}}{C(\oversamp)} \cdot 1.01^{3} \cdot x + \frac{20}{C(\oversamp)} \leq \frac{\pi^{2}}{2} \cdot \frac{1+\sigma^{2}}{C(\oversamp)} \cdot x + \frac{20(1+\sigma^{2})}{C(\oversamp)}, \qquad \text{ for } x \geq 100,
\]
as well as the lower bound
\begin{align*}
	\ratiomapbit(x) \geq  \frac{\pi^{2}}{4} \cdot \frac{1+\sigma^{2}}{C(\oversamp)} \cdot x + \frac{1}{C(\oversamp)} \cdot \frac{ C_3(\oversamp) }{ C_2(\oversamp)}\cdot (1-2\cdot1.01^{3/2}) &\geq \frac{\pi^{2}}{4} \cdot \frac{1+\sigma^{2}}{C(\oversamp)} \cdot x - \frac{25}{C(\oversamp)} \\&\overset{\1}{\geq} \frac{\pi^{2}}{8}\cdot \frac{1+\sigma^{2}}{C(\oversamp)} \cdot x + \frac{20(1+\sigma^{2})}{C(\oversamp)} \quad \text{ for } x \geq 100,
\end{align*}
where we note that step $\1$ follows since $x \geq 100$.  This proves part (a) and we turn our attention to part (b).  To this end,  note that $h_{1}(x)$ is a monotone increasing function. Thus,
\[
h_{1}(1/\sqrt{x}) \geq h_{1}(1/\sqrt{100}) \geq 0.1/1.01^{3/2} \qquad \text{ for } x \leq 100,
\]
where the final inequality follows upon applying Lemma~\ref{lemma1-proof-f-h-main-properties}.  Finally, we note that when $x \leq 100$,
\[
\ratiomapbit(x) \leq \frac{\pi^{2}}{4} \cdot \frac{1+\sigma^{2}}{C(\oversamp)} \cdot \frac{1}{h_{1}(1/\sqrt{x})^{2}} + \frac{20}{C(\oversamp)} \leq \frac{\pi^{2}}{4} \cdot \frac{1+\sigma^{2}}{C(\oversamp)} \cdot \frac{1.01^{3}}{0.01} + \frac{20}{C(\oversamp)} \leq 50 \pi^{2} \cdot \frac{1+\sigma^{2}}{C(\oversamp)}.
\]
We are left to prove the last part. Recall from equation~\eqref{eq:C2-C3} that
\[
	C_{2}(\oversamp) = \EE\left\{ \frac{W^{2}}{(C(\oversamp) + W^{2})^{2}}\right\} \qquad \text{ and } \qquad C_{3}(\oversamp) = \EE\left\{ \frac{W^{4}}{(C(\oversamp) + W^{2})^{2}} \right\}.
\]
Let $T = \frac{2}{\pi} \cdot \oversamp \cdot \EE\bigg\{ \frac{|W|\phi\big( \frac{|W|}{\sqrt{x}}\big)}{C(\oversamp)+W^{2}}\bigg\}$. Straightforward calculation yields that
\[
	\ratiomapbit(x) \cdot T^{2} = \frac{1+\sigma^{2}}{C(\oversamp)} + \frac{1}{C(\oversamp)C_{2}(\oversamp)} \cdot \bigg(C_{3}(\oversamp)T^{2} - \frac{4T}{\pi} \EE\bigg\{ \frac{W^{3} \phi\big( \frac{|W|}{\sqrt{x}}\big) }{(C(\oversamp)+W^{2})^{2}}\bigg\} \bigg).
\]
We next find the lower bound of the RHS of the equation in the display above. Minimizing the term in parenthesis in the above display in $T$ (which is a quadratic function in $T$) yields the lower bound
\begin{align*}
	C_{3}(\oversamp)T^{2} - \frac{4T}{\pi} \EE\bigg\{ \frac{W^{3} \phi\big( \frac{|W|}{\sqrt{x}}\big) }{(C(\oversamp)+W^{2})^{2}}\bigg\} \geq -\frac{4}{\pi^{2}} \cdot \frac{1}{C_{3}(\oversamp)} \cdot \EE\bigg\{ \frac{W^{3} \phi\big( \frac{|W|}{\sqrt{x}}\big) }{(C(\oversamp)+W^{2})^{2}}\bigg\}^{2}.
\end{align*}
Next, we apply the numeric inequality $\phi(x) = \int_{0}^{x} e^{-t^{2}/2} \mathrm{d}t \leq \sqrt{\frac{\pi}{2}}$ for $x\geq 0$ to obtain the inequality
\[
	\EE\bigg\{ \frac{W^{3} \phi\big( \frac{|W|}{\sqrt{x}}\big) }{(C(\oversamp)+W^{2})^{2}}\bigg\}^{2} \leq \frac{\pi}{2} \EE\bigg\{ \frac{|W|^{3}}{(C(\oversamp)+W^{2})^{2}}\bigg\}^{2} \overset{\1}{\leq} \frac{\pi}{2} \cdot \EE\bigg\{ \frac{W^{2}}{(C(\oversamp)+W^{2})^{2}}\bigg\} \cdot \EE\bigg\{ \frac{W^{4}}{(C(\oversamp)+W^{2})^{2}}\bigg\} = \frac{\pi}{2}C_{2}(\oversamp)C_{3}(\oversamp),
\]
where in the step $\1$ we applied the Cauchy--Schwarz inequality $\EE\{ab\}^{2} \leq \EE\{a^{2}\}\EE\{b^{2}\}$, taking $a = \frac{|W|}{C(\oversamp)+W^{2}}$ and $b = \frac{W^{2}}{C(\oversamp)+W^{2}}$. Putting the pieces together thus yields the inequality
\[
	\frac{1}{C(\oversamp)C_{2}(\oversamp)} \cdot \bigg(C_{3}(\oversamp)T^{2} - \frac{4T}{\pi} \EE\bigg\{ \frac{W^{3} \phi\big( \frac{|W|}{\sqrt{x}}\big) }{(C(\oversamp)+W^{2})^{2}}\bigg\} \bigg) \geq - \frac{2}{\pi} \cdot \frac{1}{C(\oversamp)}.
\]
Consequently, we obtain that
\[
		\ratiomapbit(x) \cdot T^{2} \geq \frac{1+\sigma^{2}}{C(\oversamp)} - \frac{2}{\pi} \cdot \frac{1}{C(\oversamp)} \geq (1-2/\pi) \cdot \frac{1+\sigma^{2}}{C(\oversamp)}.
\]
We next obtain an upper bound of $T$. Once again using the numeric inequality $\phi(x) \leq \sqrt{\frac{\pi}{2}}$ for $x\geq 0$ yields the inequality
\[
	T = \frac{2}{\pi} \cdot \oversamp \cdot \EE\bigg\{ \frac{|W|\phi\big( \frac{|W|}{\sqrt{x}}\big)}{C(\oversamp)+W^{2}}\bigg\} \lesssim \frac{1}{\oversamp} \cdot \EE\bigg\{ \frac{|W|}{C(\oversamp)+W^{2}}\bigg\} \leq \frac{\oversamp}{C(\oversamp)} \EE\{|W|\} \lesssim 1,
\]
where in the last step we use $\oversamp \asymp C(\oversamp)$. We thus deduce that when $x>0$,
\[
	\ratiomapbit(x) \geq (1-2/\pi) \cdot \frac{1+\sigma^{2}}{C(\oversamp)} \cdot \frac{1}{T^{2}} \gtrsim \frac{1+\sigma^{2}}{C(\oversamp)}.
\]
This concludes the proof.  \qed

\subsection{Global convergence analysis when $\oversamp \geq \sqrt{d}$} \label{sec:global-convergence-large-oversamp}
This section is dedicated to the proof of the following proposition, which demonstrates that---from a random initialization---when $\oversamp \geq \sqrt{d}$, then two iterations of AM suffice to reach very small error in both the linear observation model and the nonlinear observation model. 
\begin{proposition}\label{prop:convergence-lambda-rootd}
	Suppose that $\frac{1}{50\sqrt{d}} \leq \parcompX_{0} \leq \frac{1}{\sqrt{d}}$, $\parcompX_{0}^{2} + \perpcompX_{0}^{2} = 1$ and $\oversamp \geq \sqrt{d}$.  Then, there exists a universal, positive constant $c$ such that for all $\sigma^{2} \leq d^{c}$ and $\log(n) \leq d^{c}$, and for both $\psi(w)=w$ and $\psi(w)=\sign(w)$, the following holds with probability at least $1-4n^{-10}$,
	\begin{align*}
		\frac{\perpcompX_{2}^{2}}{\parcompX_{2}^{2}} \lesssim \frac{(1+\sigma^{2})d}{n} + \frac{(1+\sigma^{2})\perplogn}{\sqrt{n}}.
	\end{align*}
\end{proposition}

\begin{proof}
Recall the definition of events $\mathcal{A}_{t}$ and $\mathcal{B}_{t}$~\eqref{good-event-A-B}. We will work on the event $\mathcal{A}_{1} \cap \mathcal{B}_{1} \cap \mathcal{A}_{2} \cap \mathcal{B}_{2} $, which holds with probability exceeding $1-4n^{-10}$.  We first provide the proof in the linear model, when $\psi(w) = w$, before turning to the one-bit model.

\paragraph{Convergence in the linear model:} After one step, we obtain the deterministic updates
\begin{align*}
	\parcompZ_{1}^{\mathsf{det}} = \parmapid(\parcompX_{0},\perpcompX_{0}) = \parcompX_{0} \qquad \text{ and } \qquad
	(\perpcompZ_{1}^{\mathsf{det}})^{2} = \perpmapid(\parcompX_{0},\perpcompX_{0}) = \frac{1+\sigma^{2}}{C(\oversamp)} \cdot \perpcompX_{0}^{2} + \frac{\sigma^{2}}{C(\oversamp)} \cdot \parcompX_{0}^{2} \leq \frac{1+\sigma^{2}}{C(\oversamp)}.
\end{align*}
Note that on event $\mathcal{A}_1$, 
\[
\lvert \parcompZ_{1} -\parcompZ_{1}^{\mathsf{det}} \rvert \lesssim (1+\sigma)\pardevn \qquad  \text{ and } \qquad 
\lvert \perpcompZ_{1}^{2} -(\perpcompZ_{1}^{\mathsf{det}})^{2} \rvert \lesssim \frac{\perplogn(1+\sigma^{2})}{\sqrt{n}}.
\]
Using $\sigma^{2} \leq d^{c}$, $\log(n) \leq d^{c}$ and $\parcompX_{0} \gtrsim d^{-1/2}$ and putting the two pieces together yields
\begin{align}\label{ineq1-convergence-proof-large-oversamp}
\parcompZ_{1} \asymp \parcompX_{0} \asymp d^{-1/2} \quad \text{and} \quad \perpcompZ_{1}^{2} \lesssim \frac{1+\sigma^{2}}{C(\oversamp)} + \frac{\perplogn(1+\sigma^{2})}{\sqrt{n}} \lesssim d^{c-1/2} + d^{9c-3/4},
\end{align}
where in the last step we use $\sigma^{2},\log(n) \leq d^{c}$ and $\oversamp \geq \sqrt{d}$.
Continuing, we let 
\[
	\parcompX_{1}^{\mathsf{det}} = \frac{ \parcompZ_{1} }{\parcompZ_{1}^{2} + \perpcompZ_{1}^{2}} \quad \text{and} \quad
	(\perpcompX_{1}^{\mathsf{det}})^{2} =  \frac{1+\sigma^{2}}{C(\oversamp)} \frac{ \perpcompZ_{1}^{2} }{(\parcompZ_{1}^{2} + \perpcompZ_{1}^{2})^{2}} + \frac{\sigma^{2}}{C(\oversamp)} \frac{ \parcompZ_{1}^{2} }{(\parcompZ_{1}^{2} + \perpcompZ_{1}^{2})^{2}}.
\]
On event $\mathcal{B}_{1}$, we obtain 
\[
	\lvert \parcompX_{1} -\parcompX_{1}^{\mathsf{det}} \rvert \lesssim \frac{1+\sigma}{\sqrt{\parcompZ_{1}^{2} + \perpcompZ_{1}^{2} }} \pardevn \quad  \text{ and } \quad 
\lvert \perpcompX_{1}^{2} -(\perpcompX_{1}^{\mathsf{det}})^{2} \rvert \lesssim  \frac{1+\sigma^{2}}{\parcompZ_{1}^{2} + \perpcompZ_{1}^{2}} \frac{\perplogn}{\sqrt{n}}.
\]
Using inequality~\eqref{ineq1-convergence-proof-large-oversamp}, we obtain 
\[
	\parcompX_{1}^{\mathsf{det}} \gtrsim \frac{d^{-1/2}}{d^{-1} + d^{c-1/2} + d^{9c-3/4} } \text{ and }
	\frac{1+\sigma}{\sqrt{\parcompZ_{1}^{2} + \perpcompZ_{1}^{2} }} \pardevn \lesssim d^{1/2 + 2c - 3/4} = o_{d}(\parcompX_{1}^{\mathsf{det}}).
\]
Consequently, we obtain that
\[
	\parcompX_{1} \gtrsim \parcompX_{1}^{\mathsf{det}} \gtrsim \frac{ \parcompZ_{1} }{\parcompZ_{1}^{2} + \perpcompZ_{1}^{2}}.
\]
Putting together the pieces yields
\begin{align*}
	\frac{ \perpcompX_{1}^{2}}{\parcompX_{1}^{2}} &\lesssim \frac{(\perpcompX_{1}^{\mathsf{det}})^{2} + \frac{1+\sigma^{2}}{\parcompZ_{1}^{2} + \perpcompZ_{1}^{2}} \frac{\perplogn}{\sqrt{n}} }{(\parcompX_{1}^{\mathsf{det}})^{2}}
	\\ & \leq \frac{1+\sigma^{2}}{C(\oversamp)} \frac{\perpcompZ_{1}^{2}}{\parcompZ_{1}^{2}} + \frac{\sigma^2}{C(\oversamp)} + 
	\frac{(1+\sigma^{2})(\parcompZ_{1}^{2} + \perpcompZ_{1}^{2})}{\parcompZ_{1}^{2}} \frac{\perplogn}{\sqrt{n}}
	\\& = \Big( \frac{1+\sigma^{2}}{C(\oversamp)} + \frac{\perplogn(1+\sigma^{2})}{\sqrt{n}} \Big) \cdot \frac{\perpcompZ_{1}^{2}}{\parcompZ_{1}^{2}} + \frac{(1+\sigma^{2})\perplogn}{\sqrt{n}}
	\\& \overset{\1}{\lesssim} (d^{c-1/2} + d^{9c-3/4})^{2} d + d^{9c-3/4} \lesssim d^{2c},
\end{align*}
where in step $\1$ we combined inequality~\eqref{ineq1-convergence-proof-large-oversamp} and $\sigma^{2},\log(n) \leq d^{c}, \oversamp \geq \sqrt{d}$, and in last step we let $c$ small enough. Note that $C(\oversamp)/(1+\sigma^{2}) \gtrsim d^{1/2-c}$. Consequently, if $c$ is small enough, we obtain that 
\[	
\frac{ \perpcompX_{1}^{2}}{\parcompX_{1}^{2}} \lesssim \frac{C(\oversamp)}{1+\sigma^{2}}.
\]
Consequently, on event $\mathcal{A}_{2}$, applying Lemma~\ref{help-lemma2-GD-linear} yields
\begin{align*}
	&\Big|\frac{\perpcompZ_{2}}{\parcompZ_{2}} - h_{\mathsf{id}}\Big( \frac{\perpcompX_{1}^{2}}{\parcompX_{1}^{2}}\Big) \Big| \lesssim d^{2c} \frac{(1+\sigma^{2}) \perplogn}{\sqrt{n}} \quad \text{and consequently}
	\\& \frac{\perpcompZ_{2}}{\parcompZ_{2}} \leq \frac{1+\sigma^{2}}{C(\oversamp)} \frac{\perpcompX_{1}^{2}}{\parcompX_{1}^{2}} + \frac{\sigma^{2}}{C(\oversamp)} + d^{2c} \frac{(1+\sigma^{2}) \perplogn}{\sqrt{n}} \lesssim 1,
\end{align*}
where the last step follows by $\sigma^{2},\log(n) \leq d^{c}$, $\oversamp \geq d^{1/2}$ and $\perpcompX_{1}^{2}/\parcompX_{1}^{2} \lesssim d^{2c}$. Consequently, on event $\mathcal{B}_{2}$, applying Lemma~\ref{help-lemma2-GD-linear} again yields
\begin{align*}
	\frac{\perpcompX_{2}}{\parcompX_{2}} \leq \frac{1+\sigma^{2}}{C(\oversamp)} \frac{\perpcompZ_{2}^{2}}{\parcompZ_{2}^{2}} + \frac{\sigma^{2}}{C(\oversamp)} + \frac{(1+\sigma^{2}) \perplogn}{\sqrt{n}} \lesssim \frac{1+\sigma^{2}}{C(\oversamp)} + \frac{(1+\sigma^{2}) \perplogn}{\sqrt{n}} \lesssim \frac{(1+\sigma^{2})d}{n}+\frac{(1+\sigma^{2}) \perplogn}{\sqrt{n}},
\end{align*}
where the last step follows from $C(\oversamp) \asymp \oversamp = n/d$.

We turn now to proving the claim in the nonlinear model. 
\paragraph{Convergence in the nonlinear model:} Using the function $\parmapbit$~\eqref{eq:parmapbit} in conjunction with the inequality $\phi(x) = \int_{0}^{x} e^{-t^2/2}\mathrm{d}t \leq x$ yields the upper bound
\begin{align*}
	\parcompdetZ_{1} = \parmapbit(\parcompX_{0}, \perpcompX_{0}) \leq \frac{2}{\pi} \cdot \frac{1}{\sqrt{\parcompX_{0}^2 + \perpcompX_{0}^2}} \cdot \frac{\parcompX_{0}}{\perpcompX_{0}} \leq \frac{4\parcompX_{0}}{\pi}.
\end{align*}
Next, applying Lemma~\ref{f-h-properties-nonlinear}(b) in conjunction with the upper bound $\parcompX_{0} /\perpcompX_{0} \leq 1$ yields the lower bound 
\begin{align*}
	\parcompZ_{1}^{\mathsf{det}} \geq \frac{1}{\sqrt{2\pi}} \cdot \frac{\parcompX_{0}}{\perpcompX_{0}} \cdot \frac{1}{\sqrt{\parcompX_{0}^{2} + \perpcompX_{0}^{2}}} \geq \frac{\parcompX_{0}}{\sqrt{2\pi}}.
\end{align*}
We next find upper and lower bounds on $(\perpcompdetZ_{1})^2 = \perpmapbit(\parcompX_{0}, \perpcompX_{0})$~\eqref{eq:perpmapbit}.  To this end, we note that 
\[
\frac{C_3(\oversamp)}{C_2(\oversamp)} \leq \EE\{W^{4}\} \bigg/ \EE\left\{ \frac{W^{2}}{(1 + W^{2})^{2}}\right\} \leq 20,
\]
where $W \sim \mathsf{N}(0, 1)$.  Straightforward computation thus yields the upper bound
\begin{align*}
	(\perpcompZ_{1}^{\mathsf{det}})^{2} \leq \frac{1+\sigma^{2}}{C(\oversamp)} \cdot \frac{1}{\parcompX_{0}^{2}+\perpcompX_{0}^{2}} + \frac{20(\parcompZ_{1}^{\mathsf{det}})^{2}}{C(\oversamp)} &\leq 
	\frac{1+\sigma^{2}}{C(\oversamp)} + \frac{20}{C(\oversamp)} \cdot \frac{4}{\pi^{2}} \cdot \frac{\parcompX_{0}^{2}}{\perpcompX_{0}^{2}} \leq \frac{2(1+\sigma^{2})}{C(\oversamp)}.
\end{align*}
On event $\mathcal{A}_{1}$, we obtain
\[
	\parcompZ_{1} \asymp d^{-1/2} \quad \text{and} \quad \perpcompZ_{1}^{2} \lesssim \frac{1+\sigma^{2}}{C(\oversamp)} + \frac{\perplogn (1+\sigma^{2})}{\sqrt{n}}.
\]
Consequently, using $\sigma^{2},\log(n) \leq d^{c}$ and $\oversamp \geq \sqrt{d}$, we obtain 
\[
	\frac{\perpcompZ_{1}}{\parcompZ_{1}} \lesssim d^{1/2}(d^{c/2} d^{-1/4} + d^{4.5c} d^{-3/8}) \lesssim d^{5c+1/4},
\]
where in the last step we let $c$ small enough. Continuing, we let 
\[
	\parcompX_{1}^{\mathsf{det}} = F_{\mathsf{sgn}}(\parcompZ_{1},\perpcompZ_{1}) \gtrsim \frac{\min(\parcompZ_{1}/\parcompZ_{1},1)}{\sqrt{\parcompZ_{1}^{2} + \perpcompZ_{1}^{2}}} \quad \text{and} \quad (\perpcompX_{1}^{\mathsf{det}})^{2} = G_{\mathsf{sgn}}(\parcompZ_{1},\perpcompZ_{1}).
\]
On event $\mathcal{B}_{1}$, we obtain that
\[
	\lvert \parcompX_{1} -\parcompX_{1}^{\mathsf{det}} \rvert \lesssim \frac{1+\sigma}{\sqrt{\parcompZ_{1}^{2} + \perpcompZ_{1}^{2} }} \pardevn \quad  \text{ and } \quad 
	\lvert \perpcompX_{1}^{2} -(\perpcompX_{1}^{\mathsf{det}})^{2} \rvert \lesssim  \frac{1+\sigma^{2}}{\parcompZ_{1}^{2} + \perpcompZ_{1}^{2}} \frac{\perplogn}{\sqrt{n}}.		
\]
Putting together the pieces, we note
\[
	\frac{1+\sigma}{\sqrt{\parcompZ_{1}^{2} + \perpcompZ_{1}^{2} }} \pardevn \lesssim \frac{d^{2c - 3/4}}{\sqrt{\parcompZ_{1}^{2} + \perpcompZ_{1}^{2}}} \quad \text{and} \quad  \parcompX_{1}^{\mathsf{det}} \gtrsim \frac{d^{-2c}}{\sqrt{\parcompZ_{1}^{2} + \perpcompZ_{1}^{2}}}.
\] 
Consequently, when $c$ is a small enough constant, we obtain $\parcompX_{1} \gtrsim \parcompX_{1}^{\mathsf{det}}$. Putting together the pieces yields
\begin{align*}
	\frac{\perpcompX_{1}^{2}}{\parcompX_{1}^{2}} &\lesssim \frac{(\perpcompX_{1}^{\mathsf{det}})^{2} + \frac{1+\sigma^{2}}{\parcompZ_{1}^{2} + \perpcompZ_{1}^{2}} \frac{\perplogn}{\sqrt{n}} }{(\parcompX_{1}^{\mathsf{det}})^{2}}
	\\& \overset{\1}{\lesssim} \Big( \frac{1+\sigma^{2}}{C(\oversamp)} + \frac{(1+\sigma^{2})\perplogn}{\sqrt{n}} \Big) \frac{1}{\min(1,\parcompZ_{1}^{2} / \perpcompZ_{1}^{2})} + \frac{1}{C(\oversamp)} 
	\\&\lesssim (d^{c-1/2} + d^{9c - 3/4}) d^{10c+1/2} + d^{-1/2},
\end{align*}
where in step $\1$ we applyed Lemma~\ref{f-h-properties-nonlinear}(b) to lower bound $\parcompX_{1}^{\mathsf{det}}$ and used the definition of $G_{\mathsf{sgn}}$~\eqref{eq:perpmapbit}. We consequently deduce the upper bound
\[
	\frac{\perpcompX_{1}^{2}}{\parcompX_{1}^{2}} \lesssim \frac{C(\oversamp)}{1+\sigma^{2}}.
\]
The rest of the proof follows indetical steps of the linear model case by applying Lemma~\ref{help-lemma2-GD-linear} and Lemma~\ref{lem:h-sgn-properties}. So we omit the remaining steps.
\end{proof}

\section{Non-asymptotic random matrix theory} \label{sec:non-asymptotic-rmt}
This section provides some non-asymptotic random matrix theory guarantees that are used throughout the proof of Theorem~\ref{thm:one-step}.  In Section~\ref{sec:bounds-minimum-eigenvalue}, we provide some useful bounds on the minimum eigenvalue of the random matrix $\bX^{\top} \bG^2 \bX$.  In Section~\ref{sec:concentration-trace-inverse}, we show that the trace of the inverse of the random matrix $\bX^{\top} \bG^2 \bX$ concentrates around the solution to the fixed point equation~\eqref{definition-of-C}.

\subsection{Bounds on the minimum eigenvalue} \label{sec:bounds-minimum-eigenvalue}
We first require bounds on the extremal eigenvalues of the random matrix $\bX^{\top} \bG^2 \bX$.  
\begin{lemma}
	\label{lem:eigenvalues-G}
	Let $\bG = \diag(G_1, G_2, \dots, G_n)$, with $(G_i)_{1 \leq i \leq n} \overset{\mathsf{i.i.d.}}{\sim} \mathsf{N}(0, 1)$, let the random matrix $\bX \in \mathbb{R}^{n \times d}$ consist of entries $(X_{ij})_{1 \leq i \leq n, 1 \leq j \leq d} \overset{\mathsf{i.i.d.}}{\sim} \mathsf{N}(0, 1)$, and consider the random matrix $\bX^{\top} \bG^2 \bX$.  As long as $n \geq 2d$, the following hold.
	\begin{itemize}
		\item[(a)] There exists a permutation $\pi: [n] \rightarrow [n]$, which depends only on the $G_{i}$'s such that
		\[
		\lambda_{\min}(\bX^{\top} \bG^2 \bX) \geq \mathsf{med}\Bigl(\bigr(G_i^2\bigr)_{i=1}^{n}\Bigr) \cdot \lambda_{\min}\biggl(\sum_{i = n/2}^{n} \bx_{\pi(i)} \bx_{\pi(i)}^{\top}\biggr).
		\]
		\item[(b)]
		There exists a pair of universal, positive constants $(c, c')$, such that with probability at least $1 - 2e^{-cn}$, 
		\[
		\lambda_{\min}(\bX^{\top} \bG^2 \bX)\geq c_1 n.
		\]
	\end{itemize}
\end{lemma}

\begin{proof}
	Let $\bx_i \in \mathbb{R}^d$ denote the $i$-th row of the random matrix $\bX$ and note the expansion
	\[
	\bX^{\top} \bG^2 \bX = \sum_{i=1}^{n} G_i^2 \bx_i \bx_i^{\top}.
	\]
	Next, let the random permutation $\pi: [n] \rightarrow [n]$ denote the ordering of the random variables $(G_i)_{i=1}^{n}$ (indexed so that $G_{\pi(n)} = \max\{G_1, G_2, \dots, G_n\}$).  Consequently,
	\[
	\sum_{i=1}^{n} G_i^2 \bx_i \bx_i^{\top} = \sum_{i=1}^{n} G_{\pi(i)}^2 \bx_{\pi(i)} \bx_{\pi(i)}^{\top} \succeq \sum_{i=n/2}^{n} G_{\pi(i)}^2 \bx_{\pi(i)} \bx_{\pi(i)}^{\top} \succeq \mathsf{med}\Bigl(\bigr(G_i^2\bigr)_{i=1}^{n}\Bigr) \sum_{i = n/2}^{n} \bx_{\pi(i)} \bx_{\pi(i)}^{\top},
	\]
	which proves part (a).  Proceeding to part (b), let $Q_{\star}$ denote the first quartile of a $\chi^2(1)$--distributed random variable and $I_i := \mathbbm{1}\{G_i^2 \leq Q_{\star}\}$.  Note that
	\[
	\Prob\Bigl\{\mathsf{med}\Bigl(\bigr(G_i^2\bigr)_{i=1}^{n}\Bigr) \leq Q_{\star}\Bigr\} \leq \Prob \Bigl\{\frac{1}{n} \sum_{i=1}^{n} I_i \geq \frac{1}{2}\Bigr\} \leq e^{-cn},
	\]
	where the final inequality follows by noting that $(I_i)_{1 \leq i \leq n}$ are i.i.d., sub-Gaussian, and have expectation $\EE I_i = 1/4$, and applying Hoeffding's inequality.  Thus, with probability at least $1 - e^{-cn}$, we deduce the lower bound
	\[
	\sum_{i=1}^{n} G_i^2 \bx_i \bx_i^{\top} \succeq Q_{\star} \cdot \sum_{i=n/2}^{n} \bx_{\pi(i)}\bx_{\pi(i)}^{\top}.
	\]
	The proof is complete upon noticing that the random vectors $(\bx_{i})_{1 \leq i \leq n}$ are independent of the random permutation $\pi$, whence we apply~\citet[Theorem 6.1]{wainwright2019high} in conjunction with the display above to obtain the inequality
	\[
	\lambda_{\min}(\bX^{\top} \bG^2 \bX)\geq c' n, \quad \text{ with probability } \geq 1 - 2e^{-cn}. 
	\]
\end{proof}
The next lemma uses the above result to bound quadratic forms involving a leave-one-out sequence.
\begin{lemma} \label{lem:general-trace-concentration}
	Consider a random matrix $\bX = [\bx_1 \mid \bx_2 \mid \dots \mid \bx_n]^{\top}$, where $(\bx_i)_{1 \leq i \leq n} \overset{\mathsf{i.i.d.}}{\sim} \mathsf{N}(0, \bI_d)$ as well as a random diagonal matrix $\bG = \diag(G_1, G_2, \dots, G_n)$, where $(G_i)_{1 \leq i \leq n} \overset{\mathsf{i.i.d.}}{\sim} \mathsf{N}(0, 1)$.  Consider the random matrix $\bSig_{i} = \sum_{j \neq i}^{n} G_j \bx_j \bx_j^{\top}$. Suppose $n\geq 2d$. There exists a universal, positive constant $c$ such that for all $i \in [n]$,
		\[
		\Prob \Bigl\{ \bigl\lvert \bx_{i}^{\top} \bSig_{i}^{-1} \bx_{i} - \trace(\bSig_i^{-1}) \bigr\rvert \geq t \Bigr\} \leq 2\exp\Bigl\{ -c n \cdot \min\bigl(t^2, t\bigr)\Bigr\} + e^{-cn}.
		\]
\end{lemma}
\begin{proof}
	Applying Lemma~\ref{lem:eigenvalues-G} yields the pair of inequalities
	\[
	\| \bSig_i^{-1} \|_F^2 \lesssim \frac{1}{n}, \quad \text{ and } \quad \| \bSig_i^{-1} \|_{\op} \lesssim \frac{1}{n}, \quad \text{ with probability } \geq 1 - e^{-cn}.
	\]
	The conclusion follows upon applying the Hanson--Wright inequality.
\end{proof}

\subsection{Concentration of the trace inverse} \label{sec:concentration-trace-inverse}
This section is dedicated to the proof of the following lemma.
\begin{lemma}
	\label{lem:tau-concentration}
	Let $\bG = \diag(G_1, G_2, \dots, G_n)$ consist of entries $(G_i)_{1 \leq i \leq n} \overset{\mathsf{i.i.d.}}{\sim} \mathsf{N}(0, 1)$ and let the random matrix $\bX\in \mathbb{R}^{n \times d}$ consist of entries $(X_{ij})_{1 \leq i \leq n, 1 \leq j \leq d} \overset{\mathsf{i.i.d.}}{\sim} \mathsf{N}(0, 1)$.  Recall the solution $C(\oversamp)$ to the fixed point equation~\eqref{definition-of-C}.
	There exist universal, positive constants $c$ and $C > 2$ such that for all $n\geq Cd$ and for all $t \geq C/\sqrt{n}$, 
	\[
	\Prob\bigl\{\bigl \lvert\mathsf{tr}\bigl((\bX^{\top} \bG^2\bX)^{-1}\bigr) - C(\oversamp)^{-1} \bigr \rvert \geq t\bigr\} \leq 2 \exp\bigl\{-cnt^2\bigr\}+e^{-cn}.
	\]
\end{lemma}

\paragraph{Proof of Lemma~\ref{lem:tau-concentration}}
Lemma~\ref{lem:tau-concentration} follows from three technical lemmas.  The first---whose proof we provide in Section~\ref{sec:proof-uniqueness}---shows that the fixed point equation defining $C(\oversamp)$ admits a unique solution.
\begin{lemma}
	\label{lem:uniqueness}
	Suppose $\Lambda \geq 1$ and let $W \sim \mathsf{N}(0, 1)$. Then, the equation 
	\begin{align*}
		\frac{1}{\Lambda} = \EE \Bigl\{\frac{W^2}{C(\oversamp) +  W^2}\Bigr\},
	\end{align*}
	admits a unique solution for $C(\oversamp)$. Moreover, letting $\tau = 1/C(\oversamp)$,  we have
	\begin{align}\label{eq:C-lambda-inverse}
		\frac{1}{\Lambda} = \EE \Bigl\{\frac{\tau W^2}{1 +  \tau W^2}\Bigr\}.
	\end{align}
\end{lemma}
The next lemma---whose proof we provide in Section~\ref{sec:proof-fluctuation-tau}---demonstrates that the trace concentrates around its expectation.
\begin{lemma}
	\label{lem:fluctuation-tau}
	Under the assumptions of Lemma~\ref{lem:tau-concentration}, there exist universal positive constants $c$ and $C$ such that for all $n \geq C$, the following holds for all $t \geq Ce^{-cn}$.  
	\[
	\Prob\Bigl\{ \bigl \lvert \trace\bigl((\bX^{\top} \bG^2 \bX)^{-1}\bigr) - \EE \trace\bigl((\bX^{\top} \bG^2 \bX)^{-1}\bigr) \bigr \rvert \geq t\Bigr\} \leq e^{-cn} + 2 \exp\{-cnt^2\}.
	\]
\end{lemma}
Finally, the next lemma---whose proof we provide in Section~\ref{sec:proof-deterministic-tau}---demonstrates that the expectation of the trace inverse is nearly the solution to the fixed point equation.  
\begin{lemma}
	\label{lem:deterministic-tau}
	Let $C(\Lambda)$ be as in equation~\eqref{definition-of-C}.  There exists a universal, positive constant $C$ such that 
	\[
	\bigl \lvert \EE \trace\bigl((\bX^{\top} \bG^2 \bX)^{-1}\bigr)  - C(\Lambda)^{-1} \bigr \rvert \leq \frac{C}{\sqrt{n}}.
	\]
\end{lemma}
The desired result follows immediately upon decomposing
\[
\bigl \lvert \trace\bigl((\bX^{\top} \bG^2 \bX)^{-1}\bigr) - C(\Lambda)^{-1} \bigr \rvert \leq \bigl \lvert \trace\bigl((\bX^{\top} \bG^2 \bX)^{-1}\bigr) - \EE \trace\bigl((\bX^{\top} \bG^2 \bX)^{-1}\bigr) \rvert + \bigl \lvert \EE \trace\bigl((\bX^{\top} \bG^2 \bX)^{-1}\bigr) - C(\Lambda)^{-1} \rvert,
\]
and subsequently applying Lemma~\ref{lem:fluctuation-tau} to control the first term on the RHS and Lemma~\ref{lem:deterministic-tau} to control the second term. \qed

\subsubsection{Proof of Lemma~\ref{lem:uniqueness}} \label{sec:proof-uniqueness}
To begin, define the function $f: \mathbb{R}_{ \geq 0} \rightarrow \mathbb{R}$ as 
\begin{align*}
	f(x) =  \EE \Bigl\{\frac{W^2}{x + W^2} \Bigr\}.
\end{align*}
It is straightforward to see that $f(x)$ is strictly monotone decreasing.  Moreover, $f(0) = 1$ and by dominated convergence, $\lim_{x \rightarrow \infty} f(x) = 0$.  Consequently, the equation $f(x) = 1/\oversamp$ admits a unique solution when $\oversamp \geq 1$. Note that
\[
\frac{1}{\oversamp} = \EE \Bigl\{\frac{ W^2}{C(\oversamp)+ W^2}\Bigr\} = \EE \Bigl\{\frac{\tau W^2}{\tau C(\oversamp)+ \tau W^2}\Bigr\} = \EE \Bigl\{\frac{\tau W^2}{1+ \tau W^2}\Bigr\}.
\]
The equation~\eqref{eq:C-lambda-inverse} follows immediately.
\qed

\subsubsection{Proof of Lemma~\ref{lem:fluctuation-tau}} \label{sec:proof-fluctuation-tau}
The proof uses~\citet[Theorem 1]{guntuboyina2009concentration}, which provides a concentration inequality for certain functionals of the empirical spectral measure of a Wishart matrix.  To begin, define the shorthand $\bSig = \frac{1}{n} \bX^{\top} \bG^2 \bX$, so that we are interested in the quantity $\frac{1}{n} \trace(\bSig^{-1})$.  Next, define the function $f: \mathbb{R}_{\geq 0} \rightarrow \mathbb{R}$ as $f(x) = x^{-1}$ and its $M$--truncation as $f_M(x) = f(x) \cdot \mathbbm{1}\{x \geq M\}$.  We subsequently note the decomposition 
\begin{align*}
	\frac{1}{n}\trace(\bSig^{-1}) - \frac{1}{n}\EE \trace(\bSig^{-1}) = \underbrace{\frac{1}{n} \sum_{i=1}^{d} f\bigl(\lambda_i(\bSig)\bigr) - f_M\bigl(\lambda_i(\bSig)\bigr) }_{T_{1}}+ \underbrace{\frac{1}{n}\sum_{i=1}^{d} f_M\bigl(\lambda_i(\bSig)\bigr)  - \EE f_M\bigl(\lambda_i(\bSig)\bigr)}_{T_{2}} \\
	+ \underbrace{\frac{1}{n} \sum_{i=1}^{d}  \EE f_M\bigl(\lambda_i(\bSig)\bigr) -  \EE f\bigl(\lambda_i(\bSig)\bigr)}_{T_{3}}.
\end{align*}
Towards controlling each of these terms, note that by Lemma~\ref{lem:eigenvalues-G}, $\lambda_{\min}(\bSig) \geq c_1$, with probability at least $1 - e^{-cn}$. Thus, set $M = c_1$, whence $T_{1}=0$. So we conclude that by setting $M = c_1$ we have
\begin{align}\label{ineq:bound-T1}
	\Prob \{ \lvert T_{1} \rvert \geq t \} \leq e^{-cn},\;\;\text{ for all } t>0.
\end{align}
Turning to the term $T_{2}$, note that the function $f_M$ has bounded variation~\citep[see, e.g.,][]{guntuboyina2009concentration} of at most $1/M$, whence we apply~\citet[Theorem 1]{guntuboyina2009concentration} to obtain
\begin{align} \label{ineq:bound-T2}
	\Prob\Bigl\{ \bigl \lvert T_{2} \bigr \rvert \geq t \Bigr\} \leq 2 \exp\bigl\{-cnt^2\bigr\},\;\;\text{ for all }\;t\geq 0
\end{align}
We finally bound the term $T_{3}$.  To this end, note the equivalence
\[
\frac{1}{n} \sum_{i=1}^{d}  \EE f\bigl(\lambda_i(\bSig)\bigr) -  \EE f_M\bigl(\lambda_i(\bSig)\bigr)  = \frac{1}{n} \EE\biggl\{\Bigl[\sum_{i=1}^{d}   f\bigl(\lambda_i(\bSig)\bigr) -   f_{M}\bigl(\lambda_i(\bSig)\bigr)\Bigr] \cdot \mathbbm{1} \bigl\{ \lambda_{\min}(\bSig) < M\bigr\}\biggr\},
\]
which holds by definition of the truncation $f_M$.  Next, uniformly upper bound each summand by $f(\lambda_{\min}(\bSig))$ and apply the Cauchy--Schwarz inequality to obtain the upper bound
\begin{align*}
	\frac{1}{n} \EE\biggl\{\Bigl[\sum_{i=1}^{d}   f\bigl(\lambda_i(\bSig)\bigr) -   f_{M}\bigl(\lambda_i(\bSig)\bigr)\Bigr] \cdot \mathbbm{1} \bigl\{ \lambda_{\min}(\bSig) < M\bigr\}\biggr\} &\leq \frac{1}{\oversamp}\cdot \sqrt{\EE\bigl\{ f(\lambda_{\min}(\bSig))^2\}} \cdot \sqrt{\Prob\{\lambda_{\min}(\bSig) < M\}}\\
	&\leq \frac{1}{\oversamp} \cdot\sqrt{\EE\bigl\{ f(\lambda_{\min}(\bSig))^2\}} \cdot e^{-cn},
\end{align*}
where the final inequality follows by setting $M = c_1$ and applying Lemma~\ref{lem:eigenvalues-G}(b).  Towards bounding $\EE\bigl\{ f\bigl(\lambda_{\min}(\bSig)\bigr)\bigr\}$, note that $f$ is a decreasing function and apply Lemma~\ref{lem:eigenvalues-G}(a) to obtain the bound
\begin{align*}
	\EE\bigl\{ f(\lambda_{\min}(\bSig))^2\} &\leq n^2 \cdot \EE\Bigl\{\mathsf{med}\Bigl(\bigl(G_i\bigr)_{i=1}^{n}\Bigr)^{-2}\Bigr\} \cdot \EE\biggl\{ \lambda_{\min}\biggl(\biggl(\sum_{i = n/2}^{n} \bx_{i \backslash k} \bx_{i \backslash k}^{\top}\biggr)^{-2}\biggr)\biggr\}\\
	&\overset{\1}{\lesssim}  \EE\Bigl\{\mathsf{med}\Bigl(\bigl(G_i\bigr)_{i=1}^{n}\Bigr)^{-2}\Bigr\} \leq C,
\end{align*}
where step $\1$ follows upon applying~\citet[Lemma 21]{chandrasekher2021sharp} and the final inequality follows from Lemma~\ref{lem:median-chi-square}(b).  Putting the pieces together, we obtain the bound on the expected truncation error
\begin{align} \label{ineq:bound-T3}
	T_{3} = \frac{1}{n} \sum_{i=1}^{d}  \EE f\bigl(\lambda_i(\bSig)\bigr) -  \EE f_M\bigl(\lambda_i(\bSig)\bigr) \leq \frac{C}{\Lambda} e^{-cn}.
\end{align}
Combining the inequalities~\eqref{ineq:bound-T1},~\eqref{ineq:bound-T2} and~\eqref{ineq:bound-T3}, we obtain that there exists a universal constant $C'$ such that for $t\geq C'e^{-cn}$, we have
\begin{align*}
	\Prob\Bigl\{ \bigl \lvert \frac{1}{n}\trace(\bSig^{-1}) - \frac{1}{n}\EE \trace(\bSig^{-1})  \bigr \rvert \geq t \Bigr\} &\leq \Prob \{ \lvert T_{1} \rvert \geq t/3 \} +\Prob \{ \lvert T_{2} \rvert \geq t/3 \} + \Prob \{ \lvert T_{3} \rvert \geq t/3 \} \\&\leq e^{-cn} + 2 \exp\bigl\{-cnt^2\bigr\}.
\end{align*}
The desired result follows immediately. \qed

\subsubsection{Proof of Lemma~\ref{lem:deterministic-tau}} \label{sec:proof-deterministic-tau}
As in the previous section, we will use the notation $\bSig = \frac{1}{n} \bX^{\top} \bG^2 \bX$.  We will write its leave-one-sample-out counterpart $\bSig_i$ as 
$\bSig_i = \frac{1}{n}\sum_{j \neq i} G_j^2 \bx_j \bx_j^{\top}$. We further define the random variable $\tau_{1} = \frac{1}{n} \bx_{1}^{\top}\bSig_{1}^{-1}\bx_{1}$. Applying the triangle inequality yields
\begin{align}\label{ineq:two-terms}
	\lvert \frac{1}{n} \EE\{ \trace(\bSig^{-1}) \} - \tau \rvert \leq  \underbrace{ \lvert \EE\{\tau_{1}\} - \tau \rvert}_{T_{1}} + \underbrace{ \lvert \frac{1}{n} \EE\{ \trace(\bSig^{-1}) \} - \EE\{\tau_{1}\} \rvert}_{T_{2}}.
\end{align}
We bound the terms $T_1$ and $T_2$ in turn.

\paragraph{Bounding $T_{1}$~\eqref{ineq:two-terms}:} Note that
\begin{align*}
	d = \trace \EE \left\{ \bSig^{-1}\bSig \right\} = \trace \EE \left\{ \sum_{i=1}^{n} \bSig^{-1}  \frac{1}{n} G_i^2 \bx_i \bx_i^{\top} \right\} \overset{\1}{=} 
	n \cdot \trace \EE \left\{  \bSig^{-1}  \frac{1}{n} G_1^2 \bx_1 \bx_1^{\top} \right\}  = n \cdot \EE \left\{ \frac{1}{n} \cdot \bx_{1}^{T} \bSig^{-1} \bx_{1} \right\},
\end{align*}
where in step $\1$ we first switch the expectation and summation and then exploit the i.i.d. nature of the rank one matrices $G_i^2 \bx_{i} \bx_{i}$. Applying the Sherman--Morrison formula to $\bSig^{-1} = \big( \bSig_{1} + \frac{1}{n} G_{1}^{2} \bx_{1}\bx_{1}^{\top} \big)^{-1}$ yields
\[
\frac{1}{n} \cdot \bx_{1}^{T} \bSig^{-1} \bx_{1} = \frac{G_{1}^{2}\frac{1}{n} \bx_{1}^{\top}\bSig_{1}^{-1} \bx_{1}}{1+G_{1}^{2}\frac{1}{n} \bx_{1}^{\top}\bSig_{1}^{-1} \bx_{1}} = \frac{G_{1}^{2}\tau_{1}}{1+G_{1}^{2}\tau_{1}}.
\]
Putting the two pieces together yields that
\[
\EE \left\{ \frac{G_{1}^{2}\tau_{1}}{1+G_{1}^{2}\tau_{1}} \right\} = \frac{d}{n} = \frac{1}{\oversamp}.
\]
Consequently, by the fixed point equation~\eqref{eq:C-lambda-inverse}, 
\begin{align}\label{eq1:trace-inverse}
	\EE \left\{ \frac{G_{1}^{2}\tau_{1}}{1+G_{1}^{2}\tau_{1}} \right\} = \EE \left\{ \frac{G_{1}^{2}\tau}{1+G_{1}^{2}\tau} \right\}.
\end{align}
A straightforward calculation then implies
\begin{align}\label{eq2:trace-inverse}
	\bigg|  \EE \left\{ \frac{G_{1}^{2}\tau_{1}}{1+G_{1}^{2}\tau_{1}} \right\} - \EE \left\{ \frac{G_{1}^{2}\EE\{\tau_{1}\}}{1+G_{1}^{2}\EE\{\tau_{1}\}} \right\} \bigg| = 
	\EE \left\{ \frac{G_{1}^{2} \cdot \big| \EE\{\tau_{1}\} - \tau_{1} \big|}{(1+G_{1}^{2}\EE\{\tau_{1}\}) \cdot ( 1+G_{1}^{2}\tau_{1} )} \right\} \leq \EE\left\{ \big| \tau_{1}-\EE\{\tau_{1}\} \big| \right\},
\end{align}
where the last step follows from (i.) $\tau_{1}\geq0$, whence the denominator is lower bounded by $1$ and (ii.) the independent nature of the random variables $G_{1}$ and $\tau_{1}$. Proceeding to bound the final expectation, we write
\[
\Prob \Bigl\{ \bigl\lvert \tau_{1} - \EE\{\tau_{1}\} \bigr\rvert \geq t \Bigr\} \leq \Prob \Bigl\{ \bigl\lvert \frac{1}{n} \bx_{1}^{\top}\bSig_{1}^{-1}\bx_{1}  - \frac{1}{n}\trace(\bSig_{1}^{-1}) \bigr\rvert \geq t/2 \Bigr\} + \Prob \Bigl\{ \bigl\lvert  \frac{1}{n}\trace(\bSig_{1}^{-1}) - \frac{1}{n} \EE\{ \trace(\bSig_{1}^{-1}) \} \bigr\rvert \geq t/2 \Bigr\}.
\]
We now apply Lemma~\ref{lem:fluctuation-tau} to bound the first term on the RHS and apply Lemma~\ref{lem:general-trace-concentration} to bound the second term on the RHS. All in all, for any $t\geq Ce^{-cn}$, we obtain the tail bound \sloppy\mbox{$\Prob\{\lvert \tau_{1} - \EE\{\tau_{1}\} \rvert \geq t\} \leq 4\exp\{-cn \min(t^2, t)\}+e^{-cn}$} so that 
\begin{align}\label{eq3:trace-inverse}
	\EE \{ \lvert \tau_{1}-\EE\{\tau_{1}\} \rvert \} \leq Ce^{-cn} + \int_{Ce^{-cn}}^{\infty} \Prob\{\lvert \tau_{1}-\EE\{\tau_{1}\} \rvert \geq t\} \mathrm{d}t \lesssim \frac{1}{\sqrt{n}}.
\end{align}
Putting the inequalities~\eqref{eq1:trace-inverse},~\eqref{eq2:trace-inverse} and~\eqref{eq3:trace-inverse} together yields
\begin{align}\label{eq4:trace-inverse}
	\EE \left\{ \frac{G_{1}^{2} \cdot \lvert \EE\{\tau_{1}\} - \tau \rvert}{(1+G_{1}^{2}\EE\{\tau_{1}\}) \cdot ( 1+G_{1}^{2}\tau )} \right\} = \bigg|  \EE \left\{ \frac{G_{1}^{2}\tau}{1+G_{1}^{2}\tau} \right\} - \EE \left\{ \frac{G_{1}^{2}\EE\{\tau_{1}\}}{1+G_{1}^{2}\EE\{\tau_{1}\}} \right\} \bigg| \lesssim \frac{1}{\sqrt{n}}.
\end{align}
We next upper bound $\tau$ and $\EE\{\tau_{1}\}$. By definition, we have $\tau = 1/C(\oversamp) \lesssim \frac{1}{\oversamp}$. To bound the expectation, note that
\begin{align}\label{eq5:trace-inverse}
	\begin{split}
		\EE\{\tau_{1}\} = \frac{1}{n}\EE\{\trace(\bSig_{1}^{-1})\} \leq \frac{d}{n} \EE\{\lambda_{\min}(\bSig_{1})^{-1}\} &\overset{\1}{\leq} \frac{n}{\oversamp} \EE\Bigl\{\mathsf{med}\Bigl(\bigl(G_i\bigr)_{i=1}^{n}\Bigr)^{-1}\Bigr\} \cdot \EE\biggl\{ \lambda_{\min}\biggl(\biggl(\sum_{i = n/2}^{n} \bx_{i \backslash k} \bx_{i \backslash k}^{\top}\biggr)^{-1}\biggr)\biggr\} \\ &\overset{\2}{\lesssim}
		\frac{1}{\oversamp} \cdot \EE\Bigl\{\mathsf{med}\Bigl(\bigl(G_i\bigr)_{i=1}^{n}\Bigr)^{-1}\Bigr\} \lesssim \frac{1}{\oversamp},
	\end{split}
\end{align}
where step $\1$ follows from Lemma~\ref{lem:eigenvalues-G}(a) and step $\2$ follows upon applying~\citet[Lemma 21]{chandrasekher2021sharp}. The final inequality follows from Lemma~\ref{lem:median-chi-square}(b). Taking stock, we have proved that $\tau \leq C$ and $\EE\{\tau_{1}\} \leq C$ for some universal constant $C$, so that
\[
\EE \left\{ \frac{G_{1}^{2} }{(1+G_{1}^{2}\EE\{\tau_{1}\}) \cdot ( 1+G_{1}^{2}\tau )} \right\} \geq \EE \left\{ \frac{G_{1}^{2} }{(1+CG_{1}^{2})^{2}} \right\} \gtrsim 1.
\]
Substituting the above inequality into inequality~\eqref{eq4:trace-inverse} yields that
\[
T_{1} = \lvert \EE\{\tau_{1}\} - \tau \rvert \lesssim \frac{1}{\sqrt{n}}.
\]
\paragraph{Bounding $T_{2}$~\eqref{ineq:two-terms}:} By definition, we obtain that
\begin{align*}
	\bigg \lvert \frac{1}{n} \EE\{ \trace(\bSig^{-1}) \} - \EE\{\tau_{1}\} \bigg \rvert = \bigg \lvert \frac{1}{n} \EE\{ \trace(\bSig^{-1}) \} - \frac{1}{n}\EE\{ \trace(\bSig_{1}^{-1}) \} \bigg \rvert \overset{\1}{=}
	\bigg \lvert \frac{1}{n} \EE \left\{  \frac{\frac{1}{n} G_{1}^{2}\bx_{1}^{\top} \bSig_{1}^{-2} \bx_{1} }{1+ \frac{1}{n} G_{1}^{2}\bx_{1}^{\top} \bSig_{1}^{-1}\bx_{1}} \right\} \bigg \rvert \leq \frac{1}{n^{2}} \cdot \EE\{\trace(\bSig_{1}^{-2}) \},
\end{align*} 
where step $\1$ follows upon applying the Sherman--Morrison formula and the last inequality follows since the denominator is lower bounded by $1$. Continuing, we have
\[
\frac{1}{n^{2}} \cdot \EE\{\trace(\bSig_{1}^{-2}) \} \leq \frac{d}{n^{2}} \cdot \EE\{\lambda_{\min}(\bSig_{1})^{-2} \} \lesssim \frac{1}{n},
\]
where the last inequality follows upon applying the same steps as in the proof of inequality~\eqref{eq5:trace-inverse}.\\ 
Putting the pieces together yields the desired result.\qed
\section{Ancillary lemmas} \label{sec:ancillary}
This section contains some useful lemmas and their proofs.  Section~\ref{sec:deferred-calculations} contains some calculations deferred from the main text and Section~\ref{sec:median-chi-square} contains a bound on the median of a collection of $\chi^2$--distributed random variables. 

\subsection{Calculations deferred from the main text}\label{sec:deferred-calculations}
In this section, we collect some miscellaneous items from the main text.  First, Lemma~\ref{lemma:C(lambda)-lambda} proves the claim that $\oversamp \asymp C(\oversamp)$.

\begin{lemma}\label{lemma:C(lambda)-lambda}
Let $C(\oversamp)$ be the solution the fixed point equation~\eqref{definition-of-C}. The following sandwich relation holds.
\[
0.3 \oversamp \leq C(\oversamp) \leq \oversamp, \qquad \text{ as long as } \qquad \oversamp \geq 10.
\]
\end{lemma}

\begin{proof} We first show the lower bound $C(\oversamp) \geq 0.3 \oversamp$. Note that for $0 \leq t\leq 1$, 
\[
	\EE \bigg\{ \frac{G^{2}}{t + G^{2}}\bigg\} \geq \EE \bigg\{ \frac{G^{2}}{1 + G^{2}}\bigg\} = 1 - \EE \bigg\{ \frac{1}{1 + G^{2}}\bigg\} \geq 0.3,
\]
where the last step follows since $\EE \big\{ \frac{1}{1 + G^{2}}\big\} \leq 0.7$ for $G \sim \mathsf{N}(0,1)$. Consequently, there is no solution of $t$ satisfying the fixed point equation:
\[
	\EE \bigg\{ \frac{G^{2}}{t + G^{2}}\bigg\} = \frac{1}{\oversamp},\qquad \text{ for } \qquad  \oversamp \geq 10 \qquad  \text{ and }\qquad 0 \leq t \leq 1.
\]
We thus deduce $C(\oversamp) \geq 1$ for $\oversamp \geq 10$. Applying this lower bound, we obtain the inequality
\[
	\frac{1}{\oversamp} = \EE \bigg\{ \frac{G^{2}}{C(\oversamp) + G^{2}}\bigg\} = \frac{1}{C(\oversamp)} \EE \bigg\{ \frac{G^{2}}{1 + G^{2}/ C(\oversamp) }\bigg\} \overset{\1}{\geq} \frac{1}{C(\oversamp)} \EE \bigg\{ \frac{G^{2}}{1 + G^{2} }\bigg\} \geq \frac{0.3}{C(\oversamp)},
\]
where step $\1$ follows from the lower bound $C(\oversamp) \geq 1$. Re-arranging yields the desired lower bound $C(\oversamp) \geq 0.3 \oversamp$. 

We turn next to the upper bound $C(\oversamp) \leq \oversamp$. By the definition of $C(\Lambda)$~\eqref{definition-of-C}, we obtain 
\[
	\frac{1}{\oversamp} = \EE \bigg\{ \frac{G^{2}}{C(\oversamp) + G^{2}}\bigg\} \leq \EE \bigg\{ \frac{G^{2}}{C(\oversamp)}\bigg\} \leq \frac{1}{C(\oversamp)}.
\]
Consequently, we obtain the upper bound $C(\oversamp) \leq \oversamp$. 
\end{proof}

The next lemma provides an upper bound of the deterministic update for the parallel component. 

\begin{lemma}\label{lemma:upper-bound-para-det}
Consider the parallel component $\parcompX_{t+1}^{\mathsf{det}}$~\eqref{par-equiv}. There exists a constant $C_{\psi}'$, depending only on the parameter $C_{\psi}$ of Assumption~\ref{assptn:Y-psi} such that
\[
	\parcompX_{t+1}^{\mathsf{det}} \leq \frac{C_{\psi}}{\sqrt{\parcompZ_{t+1}^{2} + \perpcompZ_{t+1}^{2}}}.
\]
\end{lemma}
\begin{proof}
Recall from formula~\eqref{par-equiv} that 
\[
	\parcompX_{t+1}^{\mathsf{det}} = \frac{1}{\sqrt{\parcompX_{t+1}^{2} + \perpcompX_{t+1}^{2}}} \cdot \EE\biggl\{ \frac{G X  Y}{1+\tau G^{2}} \biggr\} \bigg/ \EE\biggl\{ \frac{G^{2}}{1+\tau  G^{2}} \biggr\}.
\]
By Lemma~\ref{lemma:C(lambda)-lambda} that $C(\oversamp) \geq 0.3\oversamp$ for $\oversamp \geq 10$, whence $\tau = C(\oversamp)^{-1} \leq 1/3$ for $\oversamp \geq 10$. Consequently, we deduce the lower bound
\[
	\EE\left\{ \frac{G^{2}}{1+\tau G^{2}} \right\} \geq \EE\left\{ \frac{G^{2}}{1+  G^{2}/3} \right\} \gtrsim 1.
\]
Using $1+\tau G^{2} \geq 1$ and applying the Cauchy--Schwarz inequality yields the inequality
\[
	\EE\left\{ \frac{G X Y}{1+\tau G^{2}} \right\} \leq \EE\{G^{2}X^{2}\}^{1/2} \cdot \EE\{Y^{2}\}^{1/2} \lesssim C_{\psi},
\]
where the last step follows from Assumption~\ref{assptn:Y-psi}. Putting the pieces together yields the desired result.
\end{proof}

\subsection{Median of $\chi^2$ random variables}\label{sec:median-chi-square}
We require some properties of the median of a set of $\chi^2$-distributed random variables.
\begin{lemma}
	\label{lem:median-chi-square}
	Let $W_1, W_2, \dots W_n$ be a collection of i.i.d. $\chi^2(1)$ random variables.  The following holds.
	\begin{itemize}
		\item[(a)] There is a universal positive constant $c$ such that the expectation of the median satisfies the bound 
		\[
		\EE \Bigl\{ \mathsf{med}\bigl(W_1, W_2, \dots, W_n\bigr) \Bigr \} \geq c.
		\]
		\item[(b)] There exists a pair of universal, positive constants $(C, C')$ such that the following holds for all $n \geq C$,
		\[
		\EE \Bigl\{ \bigl[\mathsf{med}\bigl(W_1, W_2, \dots, W_n\bigr)\bigr]^{-2} \Bigr \} \leq C'.
		\]
	\end{itemize}
\end{lemma}
\begin{proof}

	\bigskip
	\noindent \underline{Proof of part (a).} 
	Let $Q_{\star}$ denote the first quartile of the $\xi^2(1)$ distribution and note that $Q_{\star} > 0$.  Now, note that 
	\begin{align*}
		\EE \Bigl\{ \mathsf{med}\bigl(W_1, W_2, \dots, W_n\bigr) \Bigr \} &\geq 	\EE \Bigl\{ \mathsf{med}\bigl(W_1, W_2, \dots, W_n\bigr) \mathbbm{1}\bigl\{\mathsf{med}\bigl(W_1, W_2, \dots, W_n\bigr) > Q_{\star}\bigr\}\Bigr \} \\
		&\geq Q_{\star} \Prob\bigl\{\mathsf{med}\bigl(W_1, W_2, \dots, W_n\bigr) > Q_{\star}\bigr\}\\
		&\geq Q_{\star} \cdot (1 - 2e^{-cn}),
	\end{align*}
	where the final step follows from Hoeffding's inequality.  The result follows upon taking $n$ large enough.  
	
	\bigskip
	\noindent 
	\underline{Proof of part (b).} We define the random variable $M$ as $M = \mathsf{med}\bigl(W_1, W_2, \dots, W_n\bigr)$.  The proof employs a truncation argument.  Specifically, we deduce note the bound
	\begin{align*}
		\EE\{M^{-2}\} &= \EE\bigl\{ M^{-2} \cdot \mathbbm{1} \{M \leq T\}\bigr\}  + \EE\bigl\{ M^{-2} \cdot \mathbbm{1} \{M > T\}\bigr\} \\
		&\leq \EE\bigl\{ M^{-2} \cdot \mathbbm{1} \{M \leq T\}\bigr\} + \frac{1}{T^2}.
	\end{align*}
	We write the first term on the RHS explicitly as 
	\[
	\EE\bigl\{ M^{-2} \cdot \mathbbm{1} \{M \leq T\}\bigr\} = \int_{0}^{T} x^{-2} f_{M}(x) \mathrm{d}x,
	\]
	where $f_{\mathsf{med}}$ denotes the density of the random variable $M$.  Next, we apply~\citet[Theorem 5.4.4]{casella2002statistical} to obtain the density of the median
	\[
	f_{M}(x) = \binom{n}{n/2} \cdot \frac{n \cdot f_{W_1}(x)}{2 \cdot F_{W_1}(x)} \cdot [F_{W_1}(x)]^{n/2} \cdot [1 - F_{W_1}(x)]^{n/2}.
	\]
	Stirling's inequalities yield the estimate
	\begin{align*}
		\binom{n}{n/2} \cdot \frac{n}{2} \asymp \sqrt{n} \cdot 2^n.
	\end{align*}
	Moreover, note that 
	\[
	1 - F_{W_1}(x) \leq \frac{2}{\sqrt{2\pi}} \qquad \text{ and } \qquad F_{W_1}(x) \leq \frac{2}{\sqrt{2\pi}} \cdot \sqrt{x}.
	\] 
	Consequently, we deduce the bound 
	\[
	f_{M}(x) \lesssim \sqrt{n} \cdot \Bigl(\frac{4}{\sqrt{2\pi}}\Bigr)^n \cdot x^{n/4 - 1}.
	\]
	This upper bound on the density yields the inequality 
	\[
	\EE\bigl\{ M^{-2} \cdot \mathbbm{1} \{M \leq T\}\bigr\} = \int_{0}^{T} x^{-2} f_{M}(x) \mathrm{d}x \leq \sqrt{n} \cdot \Bigl(\frac{4}{\sqrt{2\pi}}\Bigr)^n \cdot T^{n/4 - 2}
	\]
	Setting $T = 1/32$, recalling that by assumption $n \geq C$, and putting the pieces together yields the result.
	
\end{proof}
\end{document}